\documentclass{article}

\usepackage[utf8]{inputenc}
\usepackage{amsmath,amsthm, amssymb, amsfonts}
\usepackage{todonotes}
\usepackage{kantlipsum}
\usepackage{wrapfig}
\usepackage{paralist}
\usepackage{wasysym}
\usepackage{verbatim}
\usepackage{fullpage}
\usepackage[a4paper, top=20mm, bottom=25mm, lmargin=1in, rmargin=1in]{geometry}
\usepackage[pdfencoding=auto]{hyperref}
\usepackage{graphicx}
\usepackage{wrapfig}
\usepackage{mathrsfs}
\usepackage{enumitem}
\usepackage{dsfont}
\usepackage{units}
\usepackage[]{algorithm2e}
\allowdisplaybreaks[1]
\usepackage{sistyle} 
\newcommand\bSI[1]{{\small[\SI{}{#1}]}}

\makeatletter
\newlength\unitwdth
\newlength\numwdth
\newlength\tdima
\newcommand\SIdescr[2]{%
    \setlength\tdima{\linewidth}%
    \addtolength\tdima{\@totalleftmargin}%
    \addtolength\tdima{-\dimen\@curtab}%
    \addtolength\tdima{-\unitwdth}%
    \addtolength\tdima{-\numwdth}%
    \parbox[t]{\tdima}{%
        #1
        \leaders\hbox{$\m@th\mkern \@dotsep mu\hbox{\tiny.}\mkern \@dotsep mu$}%
        \hfill
        \ifhmode\strut\fi
        \makebox[0pt][l]{%
            \makebox[\unitwdth][l]{}%
            \makebox[\numwdth][r]{#2}}}}
\makeatother

\newcommand{\Z}{\mathbb{Z}}
\newcommand{\N}{\mathbb{N}}

\newcommand{\R}{\mathbb{R}}

\newcommand{\eps}{\varepsilon}
\DeclareMathOperator*{\esssup}{ess\,sup}

\let\emptyset\varnothing

\newcommand{\identity}{\mathrm{Id}}
\newcommand{\Indicator}{{\mathds{1}}}
\newcommand{\mybullet}{\bullet}

\newcommand{\conc}{{\raisebox{2pt}{\tiny\newmoon} \,}}
\newcommand{\sconc}{\odot}

\DeclareMathOperator{\Lip}{Lip}

\newcommand{\diag}{{\operatorname{diag}}}
\newcommand{\dist}{{\operatorname{dist}}}
\newcommand{\range}{{\operatorname{range}}}
\newcommand{\kernel}{{\operatorname{ker}}}
\newcommand{\rank}{{\operatorname{rank}}}

\newtheorem{theorem}{Theorem}[section]
\newtheorem*{theorem*}{Theorem}
\newtheorem{remark}[theorem]{Remark}
\newtheorem{definition}[theorem]{Definition}
\newtheorem{proposition}[theorem]{Proposition}
\newtheorem{lemma}[theorem]{Lemma}
\newtheorem{corollary}[theorem]{Corollary}

\newtheorem*{remark*}{Remark}
\newtheorem*{proposition*}{Proposition}

\numberwithin{equation}{section}

\DeclareMathOperator{\spann}{span \,}

\newcommand{\pp}[1]{#1}
\definecolor{darkcandyapplered}{rgb}{0.64, 0.0, 0.0}

\newcommand{\Felix}[1]{#1}

\newcommand{\Realization}{\mathrm{R}}

\title{Optimal approximation of piecewise smooth functions\\
using deep ReLU neural networks}
\author{Philipp Petersen\footnotemark[1]\ \footnotemark[3]
\and Felix Voigtlaender\footnotemark[2]\ \footnotemark[3]}
\date{}

\makeatletter
\newcommand*\wt[1]{\mathpalette\wthelper{#1}}
\newcommand*\wthelper[2]{%
        \hbox{\dimen@\accentfontxheight#1%
                \accentfontxheight#11.3\dimen@
                $\m@th#1\widetilde{#2}$%
                \accentfontxheight#1\dimen@
        }%
}

\newcommand*\accentfontxheight[1]{%
        \fontdimen5\ifx#1\displaystyle
                \textfont
        \else\ifx#1\textstyle
                \textfont
        \else\ifx#1\scriptstyle
                \scriptfont
        \else
                \scriptscriptfont
        \fi\fi\fi3
}
\makeatother

\begin{document}
\maketitle

\begin{abstract}
  We study the necessary and sufficient complexity of ReLU neural
  networks---in terms of depth and number of weights---which is required for
  approximating classifier functions in an $L^2$-sense.

  As a model class, we consider the set $\mathcal{E}^\beta (\R^d)$ of possibly
  discontinuous piecewise $C^\beta$ functions
  $f : [-\nicefrac{1}{2}, \nicefrac{1}{2}]^d \to \R$, where the different
  ``smooth regions'' of $f$ are separated by $C^\beta$ hypersurfaces.
  For given dimension $d \geq 2$, regularity $\beta > 0$, and accuracy
  $\varepsilon > 0$, we construct artificial neural networks with ReLU
  activation function that approximate functions from $\mathcal{E}^\beta(\R^d)$
  up to an $L^2$ error of $\varepsilon$.
  The constructed networks have a fixed number of layers, depending only on
  $d$ and $\beta$, and they have $\mathcal{O}(\varepsilon^{-2(d-1)/\beta})$
  many nonzero weights, which we prove to be optimal.
  For the proof of optimality, we establish a lower bound on the description
  complexity of the class $\mathcal{E}^\beta (\R^d)$.
  By showing that a family of approximating neural networks gives rise to an
  encoder for $\mathcal{E}^\beta (\R^d)$, we then prove that one cannot
  approximate a general function $f \in \mathcal{E}^\beta (\R^d)$
  using neural networks that are less complex than those produced by our
  construction.

  In addition to the optimality in terms of the number of weights,
  we show that in order to achieve this optimal approximation rate,
  one needs ReLU networks of a certain minimal depth.
  Precisely, for piecewise $C^\beta(\R^d)$ functions, this minimal depth
  is given---up to a multiplicative constant---by $\beta/d$.
  Up to a log factor, our constructed networks match this bound.
  This partly explains the benefits of depth for ReLU networks by showing
  that deep networks are necessary to achieve efficient approximation of
  (piecewise) smooth functions.

  Finally, we analyze approximation in high-dimensional spaces where the
  function $f$ to be approximated can be factorized into a smooth dimension
  reducing feature map $\tau$ and classifier function $g$---defined on a
  low-dimensional feature space---as $f = g \circ \tau$.
  We show that in this case the approximation rate depends only on the
  dimension of the feature space and not the input dimension.
\end{abstract}



\renewcommand{\thefootnote}{\fnsymbol{footnote}}
\footnotetext[1]{Institut f\"ur Mathematik, Technische Universit\"at Berlin,
10623 Berlin, Germany, \texttt{pc.petersen.pp@gmail.com}}

\footnotetext[2]{Institut f\"ur Mathematik, Technische Universit\"at Berlin,
10623 Berlin, Germany, \texttt{felix@voigtlaender.xyz}}

\footnotetext[3]{Both authors contributed equally to this work.}


\section{Introduction}

\emph{Neural networks} implement functions by connecting multiple simple
operations in complex patterns.
They were inspired by the architecture of the human brain and in that framework
probably first studied in 1943 in \cite{MP43}.
A special network model is that of a \emph{multi-layer perceptron}
\cite{Rumelhart:1986:LIR:104279.104293, Rosenblatt1962}, which can,
in mathematical terms, be understood as an alternating concatenation of
affine-linear functions and simple nonlinearities arranged in multiple layers.

Recently, especially deep networks, \Felix{that is}, those with many layers,
have received increased attention, due to the possibility to train them
efficiently.
In particular, given training data in the form of input and output pairs,
there exist highly efficient \Felix{training} algorithms \Felix{that} adapt a
network in such a way that the \Felix{trained} network \Felix{approximately}
implements an interpolation of the training data, and even generalizes
well to previously unseen data points---at least for many problems that occur
in practice.
This procedure is customarily referred to as \emph{deep learning}
\cite{LeCun2015DeepLearning, Goodfellow-et-al-2016}.

A small selection of spectacular applications of deep learning are image
classification \cite{Krizhevsky2012Imagenet},
speech recognition \cite{Hint2012acoustic},
or game intelligence \cite{David2016Go}.
While networks trained by deep learning prove
to be remarkably versatile and powerful classifiers, it is not entirely
understood why these methods work so well.
One aspect of the success of deep learning is certainly the powerful network
architecture.
In mathematical terms, this means that networks yield efficient
approximators for relevant function classes.
Note though that this ability to approximate a given function---or to
interpolate the training data---does in itself \emph{not} explain why neural
networks yield better generalization than other learning architectures.
This \Felix{question of generalization}, however,
is outside the scope of this paper.

\pp{
\Felix{In this paper, we investigate the} \emph{approximation properties of
neural networks}.
In other words, we \Felix{study how complex networks need to be in order
to approximate certain functions well.}
\Felix{For this}, we focus on networks that use a certain
activation function---which is possibly the most widely used in applications---%
the \emph{rectified linear unit} (ReLU).
\Felix{For such networks}, we determine the optimal trade-off between the
complexity, \Felix{measured in terms of the number of nonzero weights of the
network,} and the approximation fidelity of neural networks when approximating
\emph{piecewise constant} (or \emph{piecewise smooth}) functions.
As we will elaborate upon below, \Felix{these functions}
resemble the classifier functions that occur in classification problems.

\Felix{Roughly speaking, a piecewise constant function $f$ is of the form
$f = \sum_{i=1}^N a_i \, \chi_{K_i}$, where the sets $K_i \subset \R^d$
that determine the indicator functions $\chi_{K_i}$ have
a smooth boundary, say $\partial K_i \in C^\beta$.
For such a function $f$, we show that one can find a ReLU network $\Phi$
with $\mathcal{O}((1 + \beta/d) \cdot \log_2 (2+\beta))$ layers and
$\mathcal{O}(\varepsilon^{-2(d-1)/\beta})$ nonzero weights such that
$\|f - \Phi\|_{L^2} \leq \varepsilon$.
Moreover, we show under natural assumptions that networks with fewer than
$\mathcal{O}(\varepsilon^{-2(d-1)/\beta})$ weights cannot achieve the same
approximation accuracy.}

Additionally, we study the effect of depth of neural networks.
In particular, we show that to attain the optimal
\Felix{``complexity approximation--rate trade--off,''}
for (piecewise) smooth functions $f \in C^\beta(\mathbb{R}^d)$,
one needs networks with a minimal depth of $\mathcal{O}(1 + \beta/d)$ layers.
\Felix{This lower bound for the depth matches the depth of the networks that we
construct, up to a log-factor.}

Finally, we analyze to what extent the presented results provide insights for
the approximation of high-dimensional functions.
\Felix{In contrast to the approximation results from above, where the exponent
$2(d-1)/\beta$ in the number of weights
$\mathcal{O}(\varepsilon^{-2(d-1)/\beta})$ increases with $d$,
we will see for a certain class of highly structured functions that such
a \emph{curse of dimension} can be avoided.
More precisely, if the function $f$ can be factored as $f = g \circ \tau$
with a \pp{smooth} feature map $\tau : \R^d \to \R^k$, and a piecewise
constant (or piecewise smooth) classifier function $g$, then one can approximate
$f$ up to $L^2$-error $\varepsilon$ using a ReLU network $\Phi$ with
$\mathcal{O}(\varepsilon^{-2(k-1)/\beta})$ weights.
Therefore, the approximation rate only depends on the dimension $k$ of the
feature space, instead of the input space dimension $d$.
}
}


\medskip{}

In the remainder of this introduction, we first motivate our choice of the class
of piecewise constant and piecewise smooth functions as functions of interest.
Afterwards, we review related results concerning the approximation of
(piecewise) smooth functions, both by neural networks and more general function
classes. Then, we will clarify our notion of complexity of neural networks.
Finally, we describe our contribution \Felix{in greater detail},
and fix some standard and non-standard notation.

\subsection{Classification with neural networks}

Neural networks are used in a broad range of classification problems:
Examples include image classification \cite{Krizhevsky2012Imagenet},
digit \Felix{and character} classification
\cite{GUYON1991, KnerrDigits1992, Gale1991Digits, NIPS1989_293},
or even medical diagnosis \cite{Baxt1990Class, Burk1994Cancer}.
A comprehensive survey on classification by neural networks can be found in
\cite{Zhang00neuralnetworks}.

The networks employed in these tasks take high-dimensional input and assign
a simple label to each data point, thereby performing a classification.
Thus, we perceive a prototype classifier function as a map
$f: \R^d \to \{1, \dots, N\}$, where $N$ is the number of possible labels.
In other words, the \emph{function class of classifier functions is that of
piecewise constant functions}.
A special case of particular interest is that of binary classification---that
is, when $N=2$---which is extensively studied in Part 1 of
\cite{Anthony2009NNL}.

Admittedly, the model of a classifier function described above is not the only
conceivable model.
Indeed, another point of view is to consider the classifier function as
\Felix{assigning to each input} a conditional probability distribution
\Felix{that determines for each possible label the probability with which this
label is the correct one for the given input.}
In this regard, not piecewise constant functions but rather functions that admit
reasonably sharp but smooth phase transitions are the right model.
\Felix{However, if the application requires \pp{selecting} one particular label,
instead of a probability density on the set of labels, one will typically
select the label with the highest probability. The resulting map will then
again be a piecewise constant function.}

Which point of view one should adapt naturally depends on the application.
To justify our approach, we give an example where a classifier
should indeed be piecewise constant.
Consider the problem of
predicting if a material undergoing
some known stress breaks or remains intact.
If the underlying physical model is too complicated, it might be reasonable
to learn the behavior from data and apply a deep learning approach.
In this case, the classifier has two labels---broken and unbroken---and a
potentially very high-dimensional input of forces and material properties.
Nonetheless, there will be a sharp transition between parameter values that
describe stable configurations and those that yield breaks.
It is conceivable that one might want to optimize the forces that can be
applied, which means that the jump set should be finely resolved by the
learned function.

\subsection{Related work on approximation of piecewise smooth functions}

We give a short overview of related work on approximation with neural networks
and approximation of piecewise smooth functions.
In fact, piecewise \emph{smooth} functions form a superset of the previously
described set of piecewise \emph{constant} functions that describe classifiers;
but it will turn out that they admit the same approximation rates with respect
to ReLU neural networks.
\Felix{Therefore}, it is natural to focus on the larger set of piecewise
smooth functions.

One of the central results of approximation with neural networks is the
universal approximation theorem
\cite{Hornik1989universalApprox, Cybenko1989,PinkusUniversalApproximation}
stating that every continuous function on a compact domain can be arbitrarily
well approximated by a \emph{shallow} neural network, that is,
by a network with only one hidden layer.
These approximation results, however, only show the possibility of approximation,
but do not provide any information on the required size of a network to achieve
a given approximation accuracy.

Other works analyze the necessary and sufficient size of networks to
approximate
functions whose Fourier transform has a
bounded first moment \cite{Barron1994, Barron1993}.
In \cite{Mhaskar:1996:NNO:1362203.1362213}, \cite{pinkus_1999} it is shown that,
\emph{assuming a smooth activation function}, a shallow network with
$\Felix{\mathcal{O}}(\varepsilon^{-d/n})$ neurons can uniformly approximate a general
$C^n$-function on a $d$-dimensional set with infinitesimal error $\varepsilon$.
This approximation rate is also demonstrated to be optimal, in the sense that
if one insists that the weights of the approximating network should depend
\emph{continuously} on the approximated function, the derived rate can not be
improved.
Note though that in \cite[Section 3.3]{YAROTSKY2017103}, Yarotsky gives a
construction where the weights do \emph{not} depend continuously on the
approximated function, and where the ``optimal'' lower bound is improved by a
log factor.
He uses deep networks instead of shallow ones and the ReLU activation function
instead of a smooth one.
This result shows that the optimality can indeed fail if the weights are allowed
to depend discontinuously on the approximated function.

Except for the \Felix{recent paper} \cite{YAROTSKY2017103},
all the results \Felix{mentioned} above concern shallow networks.
However, in applications, one observes that \emph{deep networks appear to
perform better than shallow ones of comparable size}.
Nonetheless, at this point, there does not exist an entirely satisfactory
explanation of why this should be the case.
Still, from an approximation theoretical point of view, there are a couple of
results explaining the connection of depth to the expressive power of a network.
In \cite{Montufar:2014:NLR:2969033.2969153} it was demonstrated that deep
networks can partition a space into exponentially more linear regions than
shallow networks of the same size.
\Felix{The paper} \cite{NIPS2011_4350} analyzes special network architectures of
sum-product networks and establishes the advantage in the expressive power of
deep networks.
Moreover, \cite{Telgarsky2016BenefitsOfDepth, pmlr-v70-safran17a} study the
advantages of depth for networks with special piecewise polynomial activation
functions.
An overview of a large class of functions that can be well approximated with
deep but not with shallow networks can be found in \cite{Poggio2017}.

In \cite{YAROTSKY2017103}, \cite{Telgarsky2017} deep ReLU networks are employed
to achieve optimal approximation rates for smooth functions.
These results are closely related to the findings in this paper.
However, \cite{YAROTSKY2017103} and \cite{Telgarsky2017} consider approximation
in the $L^\infty$ norm, which would not be possible for functions with jumps,
since ReLU networks always implement continuous functions.
Finally, we mention \cite{BoeGKP2017}, where it is demonstrated that for the
case of \emph{two-dimensional} piecewise $C^\alpha$ smooth functions with
$C^\alpha$ jump curves, $\alpha \in (1,2]$, neural networks with certain
smooth activation functions achieve optimal $L^2$ approximation.
However, these results do not cover the case of networks with a ReLU activation
function and do not apply in dimensions $d \neq 2$.

To \Felix{complete} this overview of related work, we also give a review on
\Felix{results concerning the} approximation of piecewise smooth functions by
more general representation systems than neural networks.

Piecewise smooth functions are frequently employed as a model for images in
image processing \cite{CurveletsIntro, KLShearIntro2012,
DonohoSparseComponentsOfImages}, which is why a couple of representation systems
developed in that area are particularly well-suited \Felix{for representing
such functions}.
For instance, \emph{shearlets} and \emph{curvelets} provide optimal $N$-term
approximation rates for piecewise $C^2(\R^2)$ functions with $C^2$ jump curves,
\cite{candes1999curvelets, CurveletsIntro, KLcmptShearSparse2011, GuoL2007, Voi17}.

To obtain optimal approximation of two-dimensional functions with jump curves
smoother than $C^2$, the band\Felix{e}let system was developed, \cite{MallatBand2005},
which is a system consisting of properly smoothly-transformed boundary-adapted
wavelets that are optimally adapted to the smooth jump curves.

Another system, the so-called surflets \cite{SurfletsTechnicalReport},
even yields optimal approximation of piecewise smooth functions in $\R^d$.
This system is constructed by \Felix{using} a partition of unity, as well as
local approximation using so-called horizon functions.
These ideas are also central to the approximation results in this work.

\subsection{Our notion of optimality}

To claim that our approximation results are optimal, we need to specify a
notion of optimality.
First of all, we measure the size of networks mostly in terms of the
\emph{number of nonzero weights} of the network.
Then we adopt an information theoretical point of view, which was already
introduced in \cite{BoeGKP2017, spie-nn-2017}, but will be refined and improved
here.
The underlying idea is the following: Under some assumptions on the encodability
of the weights of a network, each neural network can be encoded with a bit
string the length of which depends only on the number of weights of the network.
For a given function class which can be well approximated by neural networks of
a \Felix{certain} complexity, this gives rise to a lossy compression algorithm
for the function class; the error introduced by this compression algorithm
depends on the quality of approximation that can be achieved by the given class
of networks over the function class.
This observation \Felix{yields} an encoding strategy for function classes
that are well-approximated by neural networks of limited complexity.
In this way, the description complexity of a function class---
which measures how well a general element of the class can be described using $\ell$ bits---
provides a lower bound on the size of the associated networks.
Similar ideas for deriving lower bounds for the approximation with certain
representation systems were used in \cite{DONOHO1993100, grohs2015optimally}.

Certainly, other means of establishing lower bounds exist.
For instance, in \cite{YAROTSKY2017103} known bounds on the Vapnik-Chervonenskis
dimension or fat-shattering dimension of networks \cite{Anthony2009NNL}
are used to obtain lower bounds on the achievable approximation rate
for a large variety of function classes.

The arguments in \cite{YAROTSKY2017103}, however, only yield a lower bound
regarding the approximation with respect to the $L^\infty$ norm.
This is not appropriate in our setting as we study $L^2$ approximation,
\Felix{or more generally $L^p$ approximation with finite $p$}.
Additionally, to obtain \emph{sharp} lower bounds on the approximation using
neural networks as in \cite{YAROTSKY2017103}, it is necessary to impose an upper
bound on the depth of the network.
Such an assumption is not required in our approach.
On the downside, we require an encodability condition on the weights.
A final argument in favor of our optimality criterion is that it is
\emph{independent of the chosen activation function $\varrho$} (as long as
$\varrho(0) = 0$), while the arguments in \cite{YAROTSKY2017103, Anthony2009NNL}
are specific to piecewise polynomial activation functions.
A more in-depth comparison of the two approaches is given in
Section \ref{sec:Optimality}.

A further notion of optimality concerns the number of layers which is necessary
to achieve a certain approximation rate by neural networks of that depth.
In \cite{Poggio2017} an overview is given about function classes that can be
approximated well by deep networks, \Felix{but not by} shallow networks.
Furthermore, Yarotsky \cite{YAROTSKY2017103} shows that a certain depth is
needed to approximate nonlinear $C^2$ functions with a given approximation rate
with respect to the $L^\infty$ norm.
A similar result is given in \cite{pmlr-v70-safran17a} for approximation with
respect to the $L^2$ norm.

We will discuss this notion in more detail in Section \ref{sec:Optimality}.
In particular, we will show that the result of Yarotsky (and the one in
\cite{pmlr-v70-safran17a}) can be generalized from $L^\infty$ approximation
(or $L^2$ approximation) to approximation in the $L^p$-sense,
for any $p \in (0,\infty)$.

\subsection{Our contribution}

\Felix{We establish the optimal rates for approximating piecewise
$C^\beta$ functions on $[-\nicefrac{1}{2}, \nicefrac{1}{2}]^d$ (where $d \in \N_{\geq 2}$ and $\beta > 0$)
by ReLU neural networks, measuring the complexity of the networks
in terms of the number of nonzero weights.}
As two special cases, our results cover the approximation of $C^\beta$ functions
and of piecewise constant functions for which the different ``constant regions''
are separated by hyper\Felix{surfaces} of regularity $C^\beta$.

A simplified but honest summary of our main results is the following:
For a given piecewise $C^\beta$ function
$f: [-\nicefrac{1}{2},\nicefrac{1}{2}]^d \to \R$ and approximation accuracy
$\varepsilon \in (0, \nicefrac{1}{2})$ we construct a ReLU neural
network $N_{\varepsilon,f}^{\mathrm{constr}}$ with no more than
$c \cdot \varepsilon^{-2(d-1)/\beta}$ nonzero weights and
$c' \cdot \log_2(\beta + 2) \cdot (1+\nicefrac{\beta}{d})$ layers,
such that $\|f - N_{\varepsilon,f}^{\mathrm{constr}}\|_{L^2} \leq \varepsilon$.
Here $c'$ is an absolute constant, while $c$ might depend on $d$ and $\beta$.
Furthermore, we show that the scaling behavior of the number of
weights with $\varepsilon$ is optimal,
that is, it cannot be improved if one insists that each weight
\Felix{of the approximating networks} can be encoded using
only $\mathcal{O}(\log_2 (\nicefrac1\varepsilon))$ bits.

Finally, \Felix{we show that} if $(N_\varepsilon)_{\varepsilon > 0}$ is a
family of networks (which are \emph{not} required to have encodable weights)
\Felix{such that $N_{\varepsilon}$ has at most
$c \cdot \varepsilon^{-2(d-1)/\beta}$ nonzero weights, while satisfying
$\|f-N_{\varepsilon}\|_{L^2} \leq \varepsilon$ for a nonlinear smooth function
$f$,} then $N_\varepsilon$ needs to have at least
$\max \{1, \beta/(4(d-1))\}$ layers, for $\varepsilon$ small enough.
Note that the depth of the networks $N_{\varepsilon,f}^{\mathrm{constr}}$
constructed above coincides (up to a log factor) with \Felix{this} lower bound
of $\max \{1, \beta / (4(d-1))\}$ layers.
%
%
%
We observe that the depth of the optimally approximating networks does
\emph{not} depend on the approximation accuracy, but is influenced only by the
dimension of the input space and by the regularity of the functions.

\Felix{These observations regarding the necessary and sufficient depth needed to
obtain good approximation rates} offer some explanation for the
efficiency of deep networks observed in practice:
\emph{With increasing structure or regularity of the underlying signal class,
the best achievable approximation rate gets better, but
\Felix{deeper networks are} required to achieve this optimal approximation rate.}

\medskip{}


All previously described results are based on classical function spaces, that is,
function spaces defined via their smoothness.
As a consequence, all approximation rates---while optimal---suffer from the
\emph{curse of dimension}.
In other words, for a function class defined over $\R^d$, the asymptotically
required size of neural networks to guarantee a certain approximation fidelity
$\varepsilon$ is essentially of the form $\varepsilon^{-d/s}$ where $s>0$
depends on the regularity of the function class.
In practice, such an influence of the dimension on the required size of the
networks is not observed.
Indeed, neural networks are usually successfully employed on
high-dimensional problems.
To model this, we propose a function class consisting of classifier functions
$f$ that can be factorized into a smooth dimension--reducing feature map
$\tau$ and a low-dimensional classifier function $g$, in the sense that
$f = g \circ \tau$. In this model, $\tau$ takes the role of a feature map which
exhibits application-specific invariances, such as, for example, translation,
dilation, and rotation invariances in image classification.
We then demonstrate that such functions \Felix{$f = g \circ \tau$} can be
approximated by ReLU neural networks at a rate independent of the ambient
dimension.
This approach is closely related to the analysis of compositional functions of
\cite{mhaskar2016learning}.

\medskip{}

The approximation results can be found in Section \ref{sec:ApproxOfClass},
and the lower bounds for the number of weights and the number of layers are
presented in Section \ref{sec:Optimality}.
In Section \ref{sec:NeuralNetworks}, we precisely define the notion of neural
networks, and we introduce a kind of calculus for these
networks, which in particular covers their composition.
This calculus will greatly simplify subsequent proofs.
Finally, in Section \ref{sec:CurseOfDIm}, we comment on the curse of dimension,
and introduce a novel function class, which can be approximated by ReLU neural
networks at a rate independent of the ambient dimension.

To not disrupt the flow of the presentation, all results and their
interpretations are presented on the first fifteen pages of the paper, 
and almost all proofs have been deferred to the appendix:
Appendix \ref{sec:ApproxWithPieceSmoothAppendix} contains the proofs related to
Section \ref{sec:ApproxOfClass}, \Felix{while} the proofs for
Section \ref{sec:Optimality} are presented in
Appendices \ref{sec:LowerBoundProofs} and \ref{sec:DepthMatters}.
Moreover, the proofs for Section \ref{sec:CurseOfDIm} can be found in
Appendices \ref{sec:SmoothSubmersions} and \ref{sec:HighDimApprox}.
Appendices \ref{sec:IntermediateDerivativesEstimate} and
\ref{sec:StandingAssumptionJustification} contain two technical auxiliary
results.

\medskip{}

Finally, we remark that our construction of approximating neural networks
relies on two technical ingredients which are possibly of independent interest
for future work:

First, we show (see Lemma \ref{lem:MultiplicationWithBoundedLayerNumber}) that
neural networks can realize an \emph{approximate multiplication}:
One can achieve $|xy - N(x,y)| \leq \varepsilon$ using a ReLU neural network $N$
with $L$ layers and $\mathcal{O}(\varepsilon^{-c/L})$ nonzero weights,
for a universal constant $c > 0$.
A similar result (see Lemma \ref{lem:ComputationalUnit}) then holds for general
polynomials.
We emphasize that it is not a new result that ReLU neural networks can realize
an approximate multiplication; this was already observed by
Yarotsky \cite{YAROTSKY2017103}.
What is new, however, is that \emph{the depth of the network is independent
of the approximation accuracy $\varepsilon$}; the depth only influences the
approximation rate.

Second, we show (see Lemma \ref{lem:MultiplicationWithACharacteristicFunction})
that neural networks can implement a ``cutoff'', that is, a multiplication with
an indicator function $\chi_{[a_1,b_1] \times \cdots \times [a_d,b_d]}$, using a
\emph{fixed number of layers and weights}, as long as the error is measured in
$L^p$, $p < \infty$.

By combining the two results, one sees that neural networks can well
approximate every function which is locally well approximated by polynomials.

\subsection{Notation}
Given a subset $A \subset X$ of a ``master set'' $X$ (which is usually implied
by the context), we define the \emph{indicator function} of $A$ as
\[
  \chi_A : X \to \{0,1\},
  x \mapsto \begin{cases}
              1 , & \text{if } x \in A, \\
              0, & \text{if } x \notin A.
            \end{cases}
\]
Moreover, if $X$ is a topological space, we write $\partial A$ for the boundary
of $A$.
We denote by $\N = \{1, 2, \dots\}$ the set of natural numbers, by
$\N_0 = \N \cup \{0\}$ the set of natural numbers including $0$, and for
$k\in \N$ we denote by $\N_{\geq k}$ all natural numbers larger or equal to $k$.
Occasionally, we also use the notation $\underline{n} := \{1,\dots, n\}$ for
$n \in \N$.
Furthermore, we write $\lfloor x \rfloor = \max \{k \in \Z \,:\, k \leq x\}$ and
$\lceil x \rceil = \min \{k \in \Z \,:\, k \geq x\}$ for $x \in \R$.

\pp{For a function $f : X \to \R$, we write $\| f \|_{\sup} := \sup_{x \in X} |f(x)| \in [0,\infty]$}, \Felix{while we set as usual
\[
  \|g\|_{L^\infty} := \esssup_{x \in \Omega} |g(x)|
  \quad \text{for} \quad
  \Omega \subset \R^d
  \quad \text{ and } \quad
  g : \Omega \to \R \text{ measurable}.
\]}

For a given norm $\|\cdot\|$ on $\R^d$, we denote by
\[
  B_r^{\|\cdot\|}(x) = B_r (x) = \{y \in \R^d \,:\, \|y-x\| < r\}
  \quad \text{ and } \quad
  \overline{B_r}^{\|\cdot\|} (x)
  = \overline{B_r} (x) = \{y \in \R^d \,:\, \|y-x\|\leq r\}
\]
the open and closed balls around $x \in \R^d$ of radius $r > 0$.
Similar notations are also used in general normed vector spaces,
not only in $\R^d$.

For a \emph{multiindex} $\alpha \in \N_0^d$, we write
$|\alpha| := \alpha_1 + \dots + \alpha_d$.
This creates a slight ambiguity with the notation $|x|$ for the euclidean norm
of $x \in \R^d$, but the context will always make clear which interpretation is
desired.
\Felix{For $x,y \in \R^d$, we write $\langle x,y \rangle :=\sum_{j=1}^d x_j w_j$
for the standard inner product of $x,y$.}

If $X,Y,Z$ are sets and $f: X \to Y$ and $g: Y \to Z$, then we denote by
$g \circ f$ the composition of $f$ and $g$, that is, $g \circ f(x) = g(f(x))$
for $x \in X$.
\Felix{Given functions $f_i : X_i \to Y_i$ for $i \in \{1,\dots,n\}$, we denote
the \emph{cartesian product} of $f_1,\dots,f_n$ by
$f_1 \times \cdots \times f_n :
X_1 \times \cdots \times X_n \to Y_1 \times \cdots \times Y_n,
(x_1,\dots,x_n) \mapsto (f_1(x_1),\dots,f_n(x_n))$.}

We denote by $|M|$ the cardinality $|M| \in \N_0 \cup \{\infty\}$ of a set $M$.
For $A \in \R^{n \times m}$, we denote by
$\| A \|_{\ell^0} := |\{(i,j) \, : \, A_{i,j} \neq 0\}|$ the number of nonzero
entries of $A$.
A similar notation is used for vectors $b \in \R^n$.
\Felix{Finally, we write $A^T \in \R^{m \times n}$ for the \emph{transpose} of
a matrix $A \in \R^{n \times m}$.}

\section{Neural networks}\label{sec:NeuralNetworks}

Below we present a \Felix{mathematical} definition of neural networks.
For our arguments, it will be crucial to emphasize the difference between a
network and the associated function.
Thus, we define a network as a structured set of weights and its
\emph{realization} as the associated function that results from alternatingly
applying the weights and a fixed activation function, which acts componentwise.

\begin{definition}\label{def:NeuralNetworks}
Let $d, L\in \N$.
A \emph{neural network $\Phi$ with input dimension $d$ and $L$ layers} is a
sequence of matrix-vector tuples
\[
  \Phi = \big( (A_1,b_1),  (A_2,b_2),  \dots, (A_L, b_L) \big),
\]
where $N_0 = d$ and $N_1, \dots, N_{L} \in \N$, and where each $A_\ell$ is an
$N_{\ell} \times N_{\ell-1}$ matrix, and $b_\ell \in \R^{N_\ell}$.

If $\Phi$ is a neural network as above, and if $\varrho: \R \to \R$ is arbitrary,
then we define the associated \emph{realization of $\Phi$ with activation
function $\varrho$} as the map $\mathrm{R}_{\varrho}(\Phi): \R^d \to \R^{N_L}$
such that
\[
  \mathrm{R}_{\varrho}(\Phi)(x) = x_L ,
\]
where $x_L$ results from the following scheme:
\begin{equation*}
  \begin{split}
    x_0 :&= x, \\
    x_{\ell} :&= \varrho(A_{\ell} \, x_{\ell-1} + b_\ell),
    \quad \text{ for } \ell = 1, \dots, L-1,\\
    x_L :&= A_{L} \, x_{L-1} + b_{L},
  \end{split}
\end{equation*}
where $\varrho$ acts componentwise, that is,
$\varrho(y) = (\varrho(y_1), \dots, \varrho(y_m))$ for
$y = (y_1, \dots, y_m) \in \R^m$.

We call $N(\Phi) := d + \sum_{j = 1}^L N_j$ the
\emph{number of neurons of the network} $\Phi$, \Felix{while}
$L(\Phi) := L$ \Felix{denotes} the \emph{number of layers} \Felix{of $\Phi$}.
\Felix{Moreover,}
$M(\Phi) := \sum_{j=1}^L (\| A_j\|_{\ell^0} + \| b_j \|_{\ell^0} )$ denotes the
\Felix{total} number of nonzero entries of all $A_\ell, b_\ell$, which we call
\emph{the number of weights of $\Phi$}.
\Felix{Finally}, we refer to $N_L$ as the \emph{dimension of the output layer}
of $\Phi$, \Felix{or simply as \emph{the output dimension} of $\Phi$}.
\end{definition}

To construct new neural networks from existing ones,
we will frequently need to concatenate networks or put them in parallel.
We first define the \emph{concatenation} of networks.

\begin{definition}
Let $L_1, L_2 \in \N$ and let
$\vphantom{\sum_j} \Phi^1 = ((A_1^1,b_1^1), \dots, (A_{L_1}^1,b_{L_1}^1))$ and
$\Phi^2 = ((A_1^2,b_1^2), \dots, (A_{L_2}^2,b_{L_2}^2))$
be two neural networks such that the input layer of $\Phi^1$ has the same
dimension as the output layer of $\Phi^2$.
Then, \emph{$\Phi^1 \conc \Phi^2$} denotes the following $L_1+L_2-1$ layer
network:
\[
  \Phi^1 \conc \Phi^2
  :=  ((A_1^2,b_1^2),
      \dots,
      (A_{L_2-1}^2,b_{L_2-1}^2),
      (A_{1}^1 A_{L_2}^2, A_{1}^1 b^2_{L_2} + b_1^1),
      ({A}_{2}^1, b_{2}^1),
      \dots,
      (A_{L_1}^1, b_{{L_1}}^1)
     ).
\]
We call $\Phi^1 \conc \Phi^2$ the \emph{concatenation of $\Phi^1$ and $\Phi^2$}.
\end{definition}
One directly verifies that $\mathrm{R}_\varrho(\Phi^1 \conc \Phi^2)
= \mathrm{R}_\varrho(\Phi^1) \circ \mathrm{R}_\varrho(\Phi^2)$,
which shows that the definition of concatenation is reasonable.

If the activation function $\varrho: \R \to \R$ is the ReLU---that is,
$\varrho(x)  = \max\{0,x \}$---then, \Felix{based on the identity
$x = \varrho(x) - \varrho(-x)$ for $x \in \R$,}
one can construct a simple two-layer network
whose realization is the identity \Felix{$\identity_{\R^d}$} on $\R^d$.

\begin{lemma}\label{lem:identity}
Let $\varrho$ be the ReLU, let $d\in \N$, and define
\[
  \Phi^{\identity}_{\Felix{d}} := \big( (A_1, b_1), (A_2, b_2) \big)
\]
with
\begin{align*}
        A_1 &:= \begin{pmatrix}
                   \identity_{\R^d} \\
                  -\identity_{\R^d}
                \end{pmatrix}
  \quad b_1 := 0,
  \quad A_2 := \begin{pmatrix}
                 \identity_{\R^d} & - \identity_{\R^d}
               \end{pmatrix}
  \quad b_2 := 0.
\end{align*}
Then $\mathrm{R}_\varrho(\Phi^{\identity}_{\Felix{d}}) = \identity_{\R^d}$.
\end{lemma}
\begin{remark}\label{rem:DeepIdentity}
  In generalization of Lemma \ref{lem:identity}, for each $d \in \N$, and each
  $L \in \N_{\geq 2}$, one can construct a network $\Phi^{\identity}_{d,L}$
  with $L$ layers and with \Felix{at most} $2d \cdot L$ nonzero,
  $\{1,-1\}$-valued weights such
  that $R_\varrho (\Phi^{\identity}_{d,L}) = \identity_{\R^d}$.
  In fact, one can choose
  \[
    \Phi^{\identity}_{d,L}
    := \left(
          \left(
            \begin{pmatrix}
              \identity_{\R^d} \\ - \identity_{\R^d}
            \end{pmatrix}
            , 0
          \right),
          \underbrace{
            (\identity_{\R^{2d}}, 0),
            \dots,
            (\identity_{\R^{2d}}, 0)
          }_{L-2 \text{ times}},
          \left(
            \left[
              \identity_{\R^d} \,\middle|\, - \identity_{\R^d}
            \right],
            0
          \right)
       \right).
  \]
  For $L=1$, one can achieve the same bounds, simply by setting
  $\Phi_{d,1}^{\identity} := ((\identity_{\R^d}, 0))$.
\end{remark}

Lemma \ref{lem:identity} enables us to define an alternative concatenation where
one can precisely control the number of weights of the resulting network.
Note though, that this only works for the ReLU \Felix{activation function}.

\begin{definition}\label{def:SpecialConc}
Let  $\varrho : \R \to \R$ be the ReLU, let $L_1, L_2 \in \N$, and let
$\Phi^1 = ((A_1^1,b_1^1), \dots, (A_{L_1}^1,b_{L_1}^1))$ and
$\Phi^2 = ((A_1^2,b_1^2), \dots, (A_{L_2}^2,b_{L_2}^2))$
be two neural networks such that the input layer of $\Phi^1$ has the same
dimension $d$ as the output layer of $\Phi^2$.
Let $\Phi^{\identity}_{\Felix{d}}$ be as in Lemma \ref{lem:identity}.
Then, the \emph{sparse concatenation of $\Phi^1$ and $\Phi^2$} is defined as
\[
  \Phi^1 \sconc \Phi^2
  := \Phi^1 \conc \Phi^{\mathrm{id}}_{\Felix{d}} \conc \Phi^2.
\]
\end{definition}
\begin{remark}\label{rem:SpecialConc}
It is easy to see that
\[
  \Phi^1 \sconc \, \Phi^2
  \! = \!\!
       \left( \!
         (A^2_1, b_1^2), \dots, (A^2_{L_2-1}, b^2_{L_2 - 1}),
         \left( \!
           \begin{pmatrix}
               A_{L_2}^2 \\[0.15cm]
             - A_{L_2}^2
           \end{pmatrix} \! ,
           \begin{pmatrix}
             b_{L_2}^2 \\[0.15cm]
             -b_{L_2}^2
           \end{pmatrix}
         \right) \! ,
         \left(
           \left[ A^1_{1} \, \middle| \, -A^1_{1}\right], b_1^1 \,
         \right),
         (A^1_2, b_2^1),
         \dots,
         (A^1_{L_1}, b^1_{L_1}) \!
       \right)
\]
has $L_1+L_2$ layers and that $\mathrm{R}_\varrho(\Phi^1 \sconc \Phi^2)
= \mathrm{R}_\varrho(\Phi^1) \circ \mathrm{R}_\varrho(\Phi^2)$
and $M(\Phi^1 \sconc \Phi^2) \leq 2M(\Phi^1) + 2M(\Phi^2)$.
\Felix{From this, and since $a+b \leq 2 \max\{a,b\}$ for $a,b \geq 0$,
it follows inductively that}
\[
  \Felix{
  M(\Phi^1 \sconc \cdots \sconc \Phi^n)
  \leq 4^{n-1} \cdot \max \{ M(\Phi^1), \dots, M(\Phi^n) \} \, .
  }
\]
\end{remark}

\Felix{In addition to concatenating networks,}
one can put two networks in parallel by using the following procedure.

\begin{definition}\label{def:Parallelization}
Let $L \in \N$ and let $\Phi^1 = ((A_1^1,b_1^1), \dots, (A_{L}^1,b_{L}^1))$ and
$\Phi^2 = ((A_1^2,b_1^2), \dots, (A_{L}^2,b_{L}^2))$
be two neural networks with $L$ layers and with $d$-dimensional input.
We define
\[
  \mathrm{P}(\Phi^1, \Phi^2)
  := \big(
       (\widetilde{A}_1,\widetilde{b}_1),
       \dots,
       (\widetilde{A}_{L},\widetilde{b}_{L})
     \big),
\]
where
\begin{align*}
  \widetilde{A}_1 := \begin{pmatrix}
                       A_1^1 \\[0.1cm]
                       A_1^2
                     \end{pmatrix},
  \quad
  \widetilde{b}_1 := \begin{pmatrix}
                       b_1^1 \\[0.1cm]
                       b_1^2
                     \end{pmatrix}
  \quad \text{ and } \quad
  \widetilde{A}_\ell := \left(
                      \begin{array}{l l}
                        A_\ell^1 & 0        \\
                        0        & A_\ell^2
                      \end{array}
                    \right),
  \quad
  \widetilde{b}_\ell: = \begin{pmatrix}
                          b_\ell^1 \\[0.1cm]
                          b_\ell^2
                        \end{pmatrix}
  \quad \text{ for } 1 < \ell \leq L.
\end{align*}
Then, $\mathrm{P}(\Phi^1, \Phi^2)$ is a neural network with $d$-dimensional input
and $L$ layers, called the \emph{parallelization of $\Phi^1$ and $\Phi^2$}.
\end{definition}

One readily verifies that $M(P(\Phi^1, \Phi^2)) = M(\Phi^1) + M(\Phi^2)$, and
\begin{equation}
  \Realization_\varrho\big(\mathrm{P}(\Phi^1,\Phi^2)\big) (x)
  = \big(\Realization_\varrho(\Phi^1)(x), \Realization_\varrho(\Phi^2)(x)\big)
  \qquad \text{ for all } x \in \R^d.
  \label{eq:ParallelizationDoesTheRightThing}
\end{equation}

\begin{remark}
With the above definition, parallelization is only defined for networks with the
same number of layers.
However, since we will be working with ReLU networks only,
Remark \ref{rem:DeepIdentity} and Definition \ref{def:SpecialConc} enable a
reasonable definition of the
parallelization of two networks $\Phi^1, \Phi^2$ of \emph{different} sizes
$L_1 < L_2$:
One first sparsely concatenates $\Phi^1$ with a network with $L_2-L_1$ layers
whose realization is the identity; that is, one defines
$\widetilde{\Phi}^1 :=  \Phi^1 \sconc \Phi^{\mathrm{Id}}_{d,L_2 - L_1}$.
We then define
$\mathrm{P}(\Phi^1,\Phi^2) := \mathrm{P}(\widetilde{\Phi}^1,\Phi^2)$.
It is not hard to verify that with this new definition,
Equation \eqref{eq:ParallelizationDoesTheRightThing} still holds.
Of course, a similar construction works for $L_1 > L_2$.
\end{remark}

\Felix{When implementing a neural network on a typical computer, one only has
a fixed number of bits for storing each weight of the network.
Generalizing from this restrictive condition, in the remainder of the paper}
we will be especially interested in neural networks whose weights
are \emph{bounded and quantized}, since these networks can be stored on a
computer with controllable memory requirements.
\Felix{However, instead of allowing for each weight only a number of bits that is
fixed a priori, we allow the number of bits per weight to increase in a
controlled way as the approximation accuracy gets better and better.}
This notion of quantized weights is made precise in the following definition:
\begin{definition}\label{def:Quantisation}
  Let $\varepsilon \in (0, \infty)$ and let $s\in \N$.
  A neural network $\Phi = ((A_1, b_1),\dots, (A_L, b_L))$ is said to possess
  \emph{$(s, \varepsilon)$-quantized weights}, if all weights
  (that is, all entries of $A_1,\dots,A_L$ and $b_1,\dots,b_L$) are elements of
  $[-\varepsilon^{-s}, \varepsilon ^{-s}]
   \cap 2^{-s \lceil \log_2 ( \nicefrac{1}{\varepsilon} )\rceil}\Z$.
\end{definition}

\begin{remark}\label{rem:QuantisationConversion}
  \setlength{\leftmargini}{0.5cm}
\begin{itemize}
  \item Assume that $\varepsilon \in (0,\nicefrac12)$,
        $q \in (0, \infty)$, $C \geq 1$, and $s \in \N$.
        If $\Phi$ is a network with $(s,\varepsilon^q/C)$-quantized weights,
        then the weights are also $(\tilde{s}, \varepsilon)$-quantized,
        where $\tilde{s} = \lceil q s + s\log_2(C) \rceil+s$.
        This is because
        \begin{align*}
               \varepsilon^{-\tilde{s}} \geq \varepsilon^{-q s - s\log_2(C)}
          =    \varepsilon^{-q s}
               \cdot \left(\frac{1}{\varepsilon}\right)^{s\log_2(C)}
          \geq \varepsilon^{-q s} \cdot 2^{s\log_2(C)}
          = \varepsilon^{-q s} \cdot C^s
          = \left(\frac{\varepsilon^{q}}{C}\right)^{-s},
        \end{align*}
        and
        \begin{align*}
          \frac{s \cdot \lceil \log_2 (1 / (\varepsilon^q / C))\rceil}
               {\tilde{s} \cdot \lceil \log_2(\frac{1}{\varepsilon}) \rceil}
          \leq \frac{s (q \log_2(\frac{1}{\varepsilon}) + \log_2(C) + 1)}
                    { (q s + s\log_2(C) + s)  \log_2(\frac{1}{\varepsilon})}
          = \frac{s q \log_2(\frac{1}{\varepsilon}) + s\log_2(C) + s}
                 {s q \log_2(\frac{1}{\varepsilon})
                  + s\log_2(C)\log_2(\frac{1}{\varepsilon})
                  + s\log_2(\frac{1}{\varepsilon})}
          \leq 1.
        \end{align*}

  \item It was shown in \cite[Lemma 3.7]{BoeGKP2017} that for a Lipschitz
        continuous activation function $\varrho$, each neural network
        $\Phi$ with all weights bounded in absolute value by
        $\varepsilon^{-s_0}$ (where $s_0 \in \N$ and
        $\varepsilon \in (0,\nicefrac12)$) can be well approximated by a neural
        network with quantized weights.

        Specifically, if $\Felix{p},\sigma, \Felix{\theta,C} \in (0,\infty)$ and
        $M(\Phi) \leq C \cdot \eps^{-\theta}$
        and $L(\Phi) = L$,
        \Felix{and if $\Omega \subset \R^d$ is bounded, with $d$ denoting the
        input dimension of $\Phi$,}
        then there is a constant
        $s = s (\theta,\sigma, C, s_0,\Omega,p, L, \varrho)\in \N$, such that
        there exists a network $\Psi$ with $L(\Phi) = L(\Psi)$ and
        $M(\Phi) = M(\Psi)$, and such that $\Psi$ has
        $(s, \varepsilon)$-quantized weights and \pp{satisfies $\|
            \mathrm{R}_\varrho(\Phi) - \mathrm{R}_\varrho(\Psi)
          \|_{L^p(\Omega)}
          \leq \varepsilon^\sigma$.}
        \Felix{Therefore, if one can achieve
        $\|f - \Realization_\varrho(\Phi)\|_{L^p (\Omega)}
        \leq \varepsilon^\sigma$
        for a ReLU network $\Phi$ with weights bounded in absolute value by
        $\varepsilon^{-s_0}$, then also
        $\|f - \Realization_\varrho(\Psi)\|_{L^p(\Omega)}
        \leq 2^{\max\{1,p^{-1}\}} \varepsilon^\sigma$ for a network $\Psi$
        with $(s,\varepsilon)$-quantized weights.
        Hence, for the approximation results that we are interested in,}
        requiring quantized weights is essentially equivalent to
        requiring the weights to be bounded (in absolute value) by
        $\varepsilon^{-s_0}$ for some $s_0 \in \N$.
        \Felix{There is only one caveat: The constant $s$ depends on the number
        of layers $L$ of the network $\Phi$, so that the argument is only
        effective if $L \leq L_0$ for a fixed $L_0 \in \N$.
        This is satisfied in many, but not all interesting cases.}
    \end{itemize}
\end{remark}

\section{Approximation of classifier functions} \label{sec:ApproxOfClass}

In this section, we will provide the main approximation results of the paper.
We will only state the results without the underlying proofs,
which would otherwise distract from the essentials.
All proofs can be found in Appendix \ref{sec:ApproxWithPieceSmoothAppendix}.
In this entire section, we assume that
$\varrho : \R \to \R, x \mapsto \max\{0,x\}$ is the ReLU.

\subsection{Approximation of horizon functions}
\label{sub:HorizonApproximation}

For $\beta \in (0,\infty)$ with $\beta = n + \sigma$, where
$n \in \mathbb{N}_0$ and $\sigma \in (0,1]$ and $d \in \N$,
we define for $f \in C^n ([-\nicefrac{1}{2}, \nicefrac{1}{2}]^d)$ the norm
\begin{equation*}
  \| f \|_{C^{0, \beta}}
  := \max \left\{
             \max_{|\alpha| \leq n} \| \partial^\alpha f \|_{\sup} \, , \,\,
             \max_{|\alpha| = n} \,\,
               \Felix{\mathrm{Lip}_\sigma (\partial f)}
          \right\}
  \in [0,\infty],
\end{equation*}
\Felix{where we used the notation
\[
  \mathrm{Lip}_\sigma (g)
  := \sup_{x,y \in \Omega, x \neq y} \frac{|g(x) -g (y)|}{|x-y|^\sigma}
  \quad \text{for} \quad g : \Omega \subset \R^d \to \R \, .
\]
Then,} for $B > 0$, we define the following class of smooth functions:
\begin{equation}
    \mathcal{F}_{\beta,d, B}
    := \left\{
         f \in C^{n}\left(\left[-\nicefrac{1}{2},\nicefrac{1}{2}\right]^d\right)
         \,:\,
         \|f\|_{C^{0,\beta}} \leq B
       \right\}.
    \label{eq:SmoothFunctionClassDefinition}
\end{equation}
It should be observed that for $\beta = n + 1$, we do \emph{not} require
$f \in \mathcal{F}_{\beta,d,B}$ to be $n+1$ times continuously differentiable.
Instead, we only require $f \in C^n$, where all derivatives of order $n$ are
assumed to be Lipschitz continuous.
Of course, if $f \in C^{n+1} ([-\nicefrac{1}{2} , \nicefrac{1}{2}]^d)$ with
$\| \partial^{\Felix{\beta}} f \|_{L^\infty} \leq B$ for all
$|\Felix{\beta}| \leq n+1 \vphantom{\sum_j}$, then it easily follows that
$\partial^\alpha f$ is Lipschitz continuous, with Lipschitz constant
$\mathrm{Lip}_{\Felix{1}} (\partial^\alpha f) \leq \sqrt{d} \cdot B$ for all
$|\alpha| = n$, so that $f \in \mathcal{F}_{n+1, d, \sqrt{d}B}$.
In this sense, our assumptions in case of $\beta = n+1$ are slightly weaker
than assuming $f \in C^{n+1}$.

The following theorem establishes optimal approximation rates by ReLU neural
networks for the function class $\mathcal{F}_{\beta, d, B}$.
It is proved in the appendix as Theorem \ref{thm:ApproxOfSmoothFctnAppendix}.

\begin{theorem}\label{thm:ApproxOfSmoothFctn}
For any $d\in \N$, and $\beta, B, \Felix{p} > 0$, there exist constants
$s = s(d,\beta,B \Felix{,p})\in \N$ and $c= c(d,\beta, B) >0$ such that for any
function $f\in \mathcal{F}_{\beta,d, B}$ and any
$\varepsilon \in (0, \nicefrac{1}{2})$, there is a neural network
$\Phi^f_\varepsilon$ with at most
\Felix{$(2 + \lceil \log_2 \beta \rceil) \cdot (11 + \nicefrac{\beta}{d})$}
layers, and at most $c \cdot \varepsilon^{-d/\beta}$ nonzero,
$(s,\varepsilon)$-quantized weights such that
\[
  \|
    \Realization_\varrho(\Phi^f_\varepsilon) - f
  \|_{L^p \Felix{([-\nicefrac{1}{2}, \nicefrac{1}{2}]^d)}}
  < \varepsilon
  \quad \text{and} \quad
  \|\Realization_\varrho(\Phi^f_\varepsilon)\|_{\sup} \leq \lceil B \rceil.
\]
\end{theorem}

\begin{remark}
  Approximation of functions in $ \mathcal{F}_{\beta,d, B}$ by ReLU networks
  was already considered in \cite[Theorem 1]{YAROTSKY2017103},
  which provides a result similar to Theorem \ref{thm:ApproxOfSmoothFctn}.
  The two theorems differ mainly in two points:
  First of all, the approximation is with respect to the $L^\infty$ norm in
  \cite[Theorem 1]{YAROTSKY2017103}, whereas we provide an approximation result
  in $L^p$, $p\in (0,\infty)$.
  Additionally, \cite[Theorem 1]{YAROTSKY2017103} requires the number of
  \emph{layers} of the network to grow logarithmically in
  $\nicefrac{1}{\varepsilon}$, which is not necessary for our result.
  Overcoming the dependence of the number of layers on $\varepsilon$ is
  achieved by using a refined construction of a multiplication operator,
  which is given in Lemma \ref{lem:MultiplicationWithBoundedLayerNumber},
  and the fact that (approximate) multiplications with indicator functions can
  be much more efficiently implemented if only $L^{\Felix{p}}$ approximation
  \Felix{with $p \in (0,\infty)$} is required,
  see Lemma \ref{lem:MultiplicationWithACharacteristicFunction}.

\end{remark}

One of the main function classes of interest in the subsequent analysis is that
of \emph{horizon functions}.
These are $\{0,1\}$-valued functions with a jump along a hypersurface and such
that the jump surface is the graph of a smooth function.
Formally, we define the class of horizon functions as follows:
\begin{definition}\label{def:HorizonFunctions}
Let $d \in \N_{\geq 2}$, and $\beta, B>0$.
Furthermore, let $H := \chi_{[0,\infty) \times \R^{d-1}}$ be the Heaviside
function. We define
\[
  \mathcal{HF}_{\beta,d, B}
  \! := \!
  \left\{
    \! f \! \circ \! T \!
    \in \! L^\infty \! \left( \left[-\nicefrac{1}{2},\nicefrac{1}{2}\right]^d\right)
    :
    f(x) \! = \! H(x_1 \! + \! \gamma(x_2, \dots, x_d), x_2, \dots, x_d),
    \gamma \! \in \! \mathcal{F}_{\beta,d-1, B},
    T \! \in  \! \Pi(d,\R) \!
  \right\},
\]
where $\Pi(d,\R) \subset GL(d, \R)$ denotes the group of permutation matrices.
\end{definition}

Concerning approximation by neural networks of functions in the class
$\mathcal{HF}_{\beta, d, B}$, we achieve the following result,
which is proved in the appendix as Lemma \ref{lem:HorFunctAppendix}.

\begin{lemma}\label{lem:HorFunct}
  For any $\Felix{p},\beta,B > 0$ and $d \in \N_{\geq 2}$
  there exist constants $c= c(d, \beta, B \Felix{,p}) > 0$,
  and $s = s(d, \beta, B \Felix{,p}) \in \N$, such that for every function
  $f\in \mathcal{HF}_{\beta,d, B}$ and every
  $\varepsilon \in (0, \nicefrac{1}{2})$ there is a neural network
  $\Phi^f_\varepsilon$ with at most
  \Felix{$(2 + \lceil \log_2 \beta \rceil) \cdot (14 + \nicefrac{2\beta}{d})$}
  layers, and at most $c \cdot \varepsilon^{-\Felix{p}(d-1)/\beta}$ nonzero,
  $(s, \varepsilon)$-quantized weights, such that
  $\|
     \Realization_\varrho(\Phi^f_\varepsilon) - f
   \|_{L^{\Felix{p}}([-\nicefrac{1}{2}, \nicefrac{1}{2}]^d)}
   < \varepsilon$.
  Moreover, $0 \leq \Realization_\varrho(\Phi^f_\varepsilon)(x) \leq 1$ for all
  $x \in \Felix{\R^d}$.
\end{lemma}

At first, the approximation of horizon functions might seem a bit arbitrary as
this is not a function class of interest that is typically considered.
However, \Felix{Lemma \ref{lem:HorFunct}} directly enables the optimal
approximation of piecewise constant and even of piecewise smooth functions,
as we will see in the next subsection.

\subsection{Approximation of piecewise smooth functions}

In this subsection, we present approximation rates for piecewise smooth
functions $f$, depending on the smoothness of the jump surfaces and
on the smoothness of $f$ on each of the ''smooth pieces''.
We first observe that if one is able to approximate indicator functions
$\chi_K$ of compact sets $K\subset [-\nicefrac{1}{2},\nicefrac{1}{2}]^d$ with
say $\partial K \in C^\beta$, then---up to a constant depending on the number
$N$ of ``pieces''---one can achieve the same approximation quality for functions
$f = \sum_{k \leq N} c_k  \, \chi_{K_k}$, where $\partial K_k \in C^\beta$ for
all $1 \leq k \leq N$.

Thus, we will only demonstrate how to approximate indicator functions with a
condition on the smoothness of the jump surface.
We start by introducing a set of domains with smooth boundaries:
Let $r \in \N$, $d \in \N_{\geq 2}$, and $\beta, B >0$.
Then we define
\begin{align*}
  \mathcal{K}_{r, \beta, d, B}
  &:= \left\{
        K  \! \subset \! \left[-\nicefrac{1}{2},\nicefrac{1}{2}\right]^d
        :
        \forall \, x \! \in \! \left[-\nicefrac{1}{2},
                                     \nicefrac{1}{2}\right]^d \,
          \exists \, f_x \in \mathcal{HF}_{\beta,d, B}:
            \chi_{K} = f_x
            \text{ on } \left[-\nicefrac{1}{2}, \nicefrac{1}{2}\right]^d \!
                        \cap
                        \overline{B_{2^{-r}}}^{\|\cdot\|_{\ell^\infty}}\! (x)
      \right\}.
\end{align*}
Although the definition of $\mathcal{K}_{r, \beta, d, B}$ is strongly tailored
to our needs, it is not overly restrictive.
In fact, for every closed set $K' \subset [-\nicefrac{1}{2}, \nicefrac{1}{2}]^d$
such that $\partial K'$ is locally the graph of a $C^\beta$ function of all but
one coordinate, it follows by compactness of $[-\nicefrac12, \nicefrac12]^d$
that $K' \in \mathcal{K}_{r, \beta, d, B}$, for sufficiently large $r$ and large
enough $B$.

We obtain the following approximation result, which is proved in the appendix
as Theorem \ref{thm:ApproximationOfPiecewiseConstantFunctionsAppendix}.

\begin{theorem}\label{thm:ApproximationOfPiecewiseConstantFunctions}
  For $r \in \N$, $d \in \N_{\geq 2}$, and $\Felix{p}, \beta, B > 0$,
  there are constants $c= c(d,r,\Felix{p}, \beta, B) > 0$
  and $s = s(d, r, \Felix{p}, \beta, B)\in \N$,
  such that for any $K\in \mathcal{K}_{r, \beta,d, B}$ and any
  $\varepsilon \in (0, \nicefrac{1}{2})\vphantom{\sum_j}$,
  there is a neural network $\Phi^K_\varepsilon$ with at most
  \Felix{$(3 + \lceil \log_2 \beta \rceil) \cdot (11 + \nicefrac{2\beta}{d})$}
  layers, and at most $c \cdot \varepsilon^{-\Felix{p}(d-1)/\beta}$ nonzero,
  $(s, \varepsilon)$-quantized weights such that
  \[
    \|
      \Realization_\varrho(\Phi^K_\varepsilon) - \chi_{K}
    \|_{L^{\Felix{p}} \Felix{([-\nicefrac{1}{2}, \nicefrac{1}{2}]^d)}}
    < \varepsilon
    \quad \text{ and } \quad
    \|\Realization_\varrho(\Phi^K_\varepsilon)\|_{\sup} \leq  1.
  \]
\end{theorem}
\begin{remark}
  Theorem \ref{thm:ApproximationOfPiecewiseConstantFunctions} establishes
  approximation rates for piecewise constant functions.
  It should be noted that the number of required layers is fixed and
  only depends on the dimension $d$ and the regularity parameter $\beta$;
  in particular, it does \emph{not} depend on the
  approximation accuracy $\varepsilon$.
  %
\end{remark}

A simple extension of Theorem \ref{thm:ApproximationOfPiecewiseConstantFunctions}
allows us to also approximate piecewise smooth functions optimally.
First, let us introduce a suitable class of piecewise smooth functions:
\Felix{Since the approximation rate for a piecewise constant function with
boundary surface of regularity $C^\beta$ is $p(d-1)/\beta$, while the
approximation rate for $C^\beta$ functions is $d/\beta$, we will consider
piecewise smooth functions for which the smoothness of the boundary surfaces is
potentially different from that of the smooth regions.}
Precisely, for $r \in \N$, $d \in \N_{\geq 2}$,
and $\Felix{p}, \beta, B>0$ we define $\beta' := (d \beta)/(\Felix{p}(d-1))$ and
\[
  \mathcal{E}_{r, \beta, d, B}^{\Felix{p}}
  := \left\{
        f = \chi_{K} \cdot g
        \,:\,
        g \in \mathcal{F}_{\beta',d, B}
        \text{ and }
        K \in \mathcal{K}_{r, \beta, d, B}
     \right\}.
\]
In terms of this new function class of piecewise smooth functions,
we get the following result, which is proven in the appendix as
Corollary \ref{cor:PiecewiseSmoothFunctionsAppendix}.

\begin{corollary}\label{cor:PiecewiseSmoothFunctions}
  Let $r \in \N$, $d\in \N_{\geq 2}$, and $\Felix{p}, B, \beta > 0$.
  Let $\beta'$ as above, and set $\beta_0 := \max \{\beta,\beta'\}$.
  Then there exist constants $c = c(d,\Felix{p}, \beta,r, B) > 0$ and
  $s = s(d, \Felix{p}, \beta,r,B) \in \N$,
  such that for all $\varepsilon \in (0, \nicefrac{1}{2})$ and all
  $f \in \mathcal{E}_{r, \beta, d, B}^{\Felix{p}}$ there is a neural network
  $\Phi^f_\varepsilon$ with at most
  \Felix{$(4 + \lceil \log_2 \beta_0 \rceil) \cdot (12 + \nicefrac{3\beta_0}{d})$}
  layers, and at most $c \cdot \varepsilon^{-\Felix{p} (d-1)/\beta}$ nonzero,
  $(s, \varepsilon)$-quantized weights, such that
  \[
    \|
      \Realization_{\varrho}(\Phi^f_\varepsilon) - f
    \|_{L^{\Felix{p}} \Felix{([-\nicefrac{1}{2}, \nicefrac{1}{2}]^d)}}
    \leq \varepsilon
    \quad \text{ and } \quad
    \|\Realization_\varrho(\Phi^f_\varepsilon)\|_{\sup} \leq \lceil B \rceil.
  \]
\end{corollary}\pp{
\begin{remark}
 We will see in Section \ref{sec:Optimality} that the given number of layers is optimal (up to \Felix{a factor of the form
  $c' \cdot (1 + \lceil \log_2 \beta_0 \rceil)$}) if one wants to achieve the approximation rate stated in the theorem.
\end{remark}}

\section{Optimality}\label{sec:Optimality}

In this section, we study two notions of optimality:
First of all, we establish in the upcoming subsection a lower bound on the
number of weights that neural networks need to have in order to achieve a
given approximation accuracy for the class of horizon functions
of regularity $\beta > 0$.
These results are valid for arbitrary activation functions $\varrho$,
\Felix{as long as $\varrho(0) = 0$}.
In the second subsection, we study lower bounds on the number of layers that a
ReLU neural network needs to have in order to achieve a given approximation
rate in terms of the number of weights or neurons.
Overall, we will see that the constructions from the previous section
achieve the optimal number of weights and have the optimal number of layers,
both up to logarithmic factors.

\subsection{Optimality in terms of \texorpdfstring{\Felix{the}}{the}
number of weights}

In this subsection, we show that the approximation results from the
preceding section are \emph{sharp}.
More precisely, we show that in order to approximate functions from the class
$\mathcal{HF}_{\beta, d, B}$ of horizon functions up to an error of
$\varepsilon > 0$ with respect to the $L^{\Felix{p}}$ norm, one generally needs
a network with at least $\Omega(\varepsilon^{-\Felix{p}(d-1)/\beta})$ nonzero
weights, independent of the employed activation functions,
\Felix{as long as $\varrho(0) = 0$}.
This claim is still somewhat imprecise; the precise statements are contained in
the theorems below.
Here, we mention the  following \Felix{five} most important points that should
be observed:
\begin{itemize}
  \item We have for all $d \in \N_{\geq 2}$, $r \in \N$, and
        $\beta, B \Felix{,p} >0$ that
        \[
          \mathcal{HF}_{\beta, d, B}
          \subset \{ \chi_K: K \in \mathcal{K}_{r, \beta, d, B}\}
          \subset \frac{1}{B} \cdot \mathcal{E}_{r, \beta, d, B}^{\Felix{p}}.
        \]
        Thus all lower bounds established for horizon functions also hold for
        the function classes of piecewise constant and piecewise
        smooth functions.

  \item The statement ``one generally needs a network with at least
        $\Omega(\varepsilon^{-\Felix{p}(d-1)/\beta})$ nonzero weights''
        suppresses a $\log$ factor.
        Actually, we show that one needs a network with at least
        $c \cdot \varepsilon^{-\Felix{p}(d-1)/\beta}
                 \big/ \log_2 (\nicefrac{1}{\varepsilon})$
        nonzero weights, for a suitable constant
        $c = c(d,\beta,B \Felix{,p}) > 0$.

  \item In \cite[Theorem 4]{YAROTSKY2017103}, Yarotsky also derive\Felix{d}
        lower bounds for approximating functions using ReLU networks,
        by using known bounds for the VC dimension of such networks.
        The most obvious difference of this result to ours is that
        Yarotsky considers $L^\infty$ approximation of smooth functions,
        while we consider $L^{\Felix{p}}$ approximation of piecewise smooth,
        possibly discontinuous functions.
        Apart from these obvious differences, there are also more subtle ones:

        \Felix{On the one hand}, our lower bounds are more general than those in
        \cite{YAROTSKY2017103} in the sense that they hold for
        \emph{arbitrary} activation functions $\varrho : \R \to \R$,
        as long as $\varrho(0) = 0$. In contrast, the results of Yarotsky only
        apply for piecewise linear activation functions with a finite number
        of ``pieces''.

        On the other hand, our results are \emph{less} general than those in
        \cite{YAROTSKY2017103}, since we impose (as in \cite{BoeGKP2017})
        a restriction on the \emph{complexity of the weights} of the network.
        Put briefly, we assume that each weight of the networks $\Phi$ that we
        consider can be encoded with at most
        $\lceil C_0 \cdot \log_2 (\nicefrac{1}{\varepsilon}) \rceil$ bits,
        where $\varepsilon$ denotes the allowed approximation error, that is,
        $\| f - \Realization_{\varrho} (\Phi)\|_{L^{\Felix{p}}} \leq \varepsilon$.
        This assumption might appear somewhat restrictive and artificial at
        first glance, but we believe it to be quite natural,
        for \Felix{two} reasons:

        \begin{enumerate}
          \item The assumption is reasonable if one wants to understand the
                behavior of networks that are used in practice.
                Here, the weights of the network have to be stored in the memory
                of a computer and thus have to be of limited complexity.
                Note that our results, in particular, apply for the usual
                floating point numbers, since these only use
                a fixed number of bits per weight, independent of $\varepsilon$.

          \item Our results apply for general (arbitrary, but fixed) activation
                functions $\varrho$ \Felix{with $\varrho(0) = 0$}.
                In this generality, it is \emph{impossible} to derive
                nontrivial lower bounds without restricting the size and
                complexity of the weights:
                Indeed, \cite[Theorem 4]{Maiorov1999LowerBounds} shows that
                there exists an activation function $\varrho:\R \to \R$ that is
                analytic, strictly increasing, and sigmoidal
                (that is, $\lim_{x\to-\infty}\varrho(x) = 0$ and
                $\lim_{x\to\infty} \varrho(x) = 1$) such that for any $d\in \N$,
                any $f\in C([0,1]^d)$ and any $\varepsilon>0$ there exists a
                neural network $\Phi$ with two hidden layers of dimensions $3d$
                and $6d+3$ such that
                $\|f-\Realization_{\varrho}(\Phi)\|_{L^\infty} \leq\varepsilon$.
                Thus, if one uses this (incredibly complex) activation function
                $\varrho$, then one can approximate \emph{arbitrary} continuous
                functions to an \emph{arbitrary} precision, using a constant
                number of layers, neurons and weights.
                From this, it is not too hard to see that a similar result holds
                for functions in $\mathcal{HF}_{\beta, d, B}$, when the error
                is measured in $L^{\Felix{p}}$.
                Our bounds show that the weights used in such networks have to
                be incredibly complex and/or numerically large.
        \end{enumerate}

  \item There are two different settings in which one can derive lower bounds:
        \begin{enumerate}
          \item For \emph{optimality in a uniform setting}, we are given
                $\varepsilon > 0$ and want to find the smallest
                $M_{\varepsilon,\Felix{p}} \in \N$ such that for \emph{every}
                $f \in \mathcal{HF}_{\beta, d, B}$ there is a neural network
                $\Phi_{\varepsilon, f}$ with at most $M_{\varepsilon,\Felix{p}}$
                nonzero weights (and such that each weight can be encoded with
                at most
                $\lceil C_0 \cdot \log_2 (\nicefrac{1}{\varepsilon}) \rceil$
                bits) satisfying
                $\|f-\Realization_{\varrho}(\Phi_{\varepsilon,f})\|_{L^{\Felix{p}}}
                \leq \varepsilon$.

                Put differently, for each sufficiently small $\varepsilon >0$,
                there is some ``hard to approximate'' function
                $f_{\varepsilon} \in \mathcal{HF}_{\beta, d, B}$ such that
                $f_\varepsilon$ \emph{cannot} be approximated up to $L^p$ error
                $\varepsilon$ with a network using less than
                $M_{\varepsilon,\Felix{p}}$ nonzero weights.

                In Theorem \ref{thm:Optimality}, we will show
                $M_{\varepsilon,\Felix{p}}
                 \geq C \cdot \varepsilon^{-\Felix{p}(d-1)/\beta}
                              \big/ \log_2 (\nicefrac{1}{\varepsilon})$
                for some $C = C(d, \Felix{p}, \beta,B,C_0) \! > \!0$.

          \item In the setting of \emph{instance optimality}, we consider for
                each $f \in \mathcal{HF}_{\beta, d, B}$ the minimal number
                $M_{\varepsilon,\Felix{p}} (f) \Felix{\in\N}$ of nonzero weights
                (of limited complexity, as above) that a neural network needs to
                have in order to approximate this specific function $f$ up to an
                $L^{\Felix{p}}$ error of at most $\varepsilon$.
                Note $M_{\varepsilon,\Felix{p}}
                = \sup_{f \in \mathcal{HF}_{\beta, d, B}}
                    M_{\varepsilon,\Felix{p}} (f)$.

                Of course, for some $f$, it can be that
                $M_{\varepsilon,\Felix{p}} (f)$ grows much slower than
                $\varepsilon^{-\Felix{p} (d-1)/\beta}$, for example if the
                boundary surfaces of $f$ are much smoother than $C^\beta$.
                Indeed, if, for example, $f \in \mathcal{HF}_{\beta + 10, d, B}$,
                then Lemma \ref{lem:HorFunct} shows $M_\varepsilon (f)
                \lesssim \varepsilon^{-\Felix{p} (d-1)/(\beta + 10)}
                \ll \varepsilon^{-\Felix{p} (d-1)/\beta}$.

                Now, note that our lower bounds from the preceding point yield
                for each $\varepsilon \in (0, \nicefrac{1}{2})$ a function
                $f_\varepsilon$ with $M_{\varepsilon,\Felix{p}} (f_\varepsilon)
                \geq C \cdot \varepsilon^{-\Felix{p}(d-1)/\beta}
                             \big/ \log_2 (\nicefrac{1}{\varepsilon})
                \gg \varepsilon^{-\gamma}$,
                for fixed but arbitrary
                $\gamma < \Felix{p} (d-1)/\beta =: \gamma^\ast$.
                Nevertheless, since the choice of the  function $f_\varepsilon$
                might depend heavily on the choice of $\varepsilon$,
                this does \emph{not} rule out the possibility that we could have
                $M_{\varepsilon,\Felix{p}} (f)
                \in \mathcal{O}(\varepsilon^{-\gamma})$
                as $\varepsilon \downarrow 0$ for all
                $f \in \mathcal{HF}_{\beta, d, B}$ and some
                $\gamma < \gamma^\ast$.
                But as we will see in Theorem \ref{thm:SingleFunctionOptimality}
                and in Corollary \ref{cor:IndividualFunctionLowerBound},
                there is a \emph{single function}
                $f \in \mathcal{HF}_{\beta, d, B}$ such that
                $M_{\varepsilon,\Felix{p}} (f)
                \notin \mathcal{O}(\varepsilon^{-\gamma})$
                for all $\gamma < \gamma^\ast$.

                This shows that the exponent
                $\gamma^\ast = \Felix{p}(d-1)/\beta$
                from Theorem \ref{thm:ApproximationOfPiecewiseConstantFunctions}
                is the best possible, not only in a uniform sense, but
                even for a single (judiciously chosen) function
                $f \in \mathcal{HF}_{\beta, d, B}$.
        \end{enumerate}

  \item \pp{While we do not \Felix{discuss} such constructions in detail to
        keep technicalities limited, our results also hold if we allow
        realizations to apply \Felix{in each layer a different activation
        function, as long as all of these activation functions are chosen
        from a fixed, finite set of functions.}
        One particularly notable example is that of allowing an
        application of a soft-max \Felix{or an arg-max} function
        \Felix{in the last layer, as is commonly done in networks used
        for classification}.
        Essentially, this \Felix{means} that---as long as we have weights with
        reasonable complexity \Felix{and consider piecewise smooth
        functions---one cannot improve the approximation rate by also
        allowing other activation functions in addition to the ReLU}.}
\end{itemize}

After this overview of our optimality results, we state the precise theorems;
for the sake of clarity, we deferred the proofs to
Appendix \ref{sec:LowerBoundProofs}.
The first order of business is to make precise the assumption that
``the weights of a network can be encoded with $K$ bits''.

\begin{definition}\label{def:WeightEncodability}
  A \emph{coding scheme for real numbers} is a sequence
  $\mathcal{B} = (B_\ell)_{\ell \in \N}$ of maps $B_\ell : \{0,1\}^\ell \to \R$.

  We say that the coding scheme is \emph{consistent} if ``each number that can
  be represented with $\ell$ bits can also be represented with $\ell+1$ bits``,
  that is, if $\mathrm{Range}(B_\ell) \subset \mathrm{Range}(B_{\ell+1})$
  for all $\ell \in \N$.

  Given a (not necessarily consistent) coding scheme for real numbers
  $\mathcal{B} = (B_\ell)_{\ell \in \N}$, and integers $M, K, \Felix{d} \in \N$,
  we denote by $\mathcal{NN}_{M,K,d}^{\mathcal{B}}$ the class of all neural
  networks $\Phi$ with $d$-dimensional input and one-dimensional output,
  with at most $M$ nonzero weights and such that the value of
  each \emph{nonzero} weight of $\Phi$ is contained in $\mathrm{Range}(B_K)$.
  In words, $\mathcal{NN}_{M,K,d}^{\mathcal{B}}$ is the class of all neural
  networks with at most $M$ nonzero weights, each of which can be encoded with
  $K$ bits, using the coding scheme $\mathcal{B}$.
  If the coding scheme is implied by the context, we simply write
  $\mathcal{NN}_{M,K,d}$ instead of $\mathcal{NN}_{M,K,d}^{\mathcal{B}}$.
\end{definition}

Now, given a fixed activation function $\varrho : \R \to \R$ and a fixed
coding scheme of real numbers $\mathcal{B}$, it makes sense to ask for a given
function $f \in L^{\Felix{p}} ([-\nicefrac{1}{2}, \nicefrac{1}{2}]^d)$ how
quickly the minimal error $\| f - \Realization_\varrho (\Phi) \|_{L^{\Felix{p}}}$
(with $\Phi \in \mathcal{NN}_{M,K,d}^{\Felix{\mathcal{B}}}$) decays,
as $M, K \to \infty$.
More precisely, given a fixed $C_0 > 0$, we are interested in the behavior of
\begin{equation}
  M_{\varepsilon,\Felix{p}} (f)
  := M_{\varepsilon,\Felix{p}}^{\mathcal{B}, \varrho, C_0} (f)
  := \inf
     \left \{
        M \in \N
        \, : \,
        \exists \,\,
        \Phi \in \mathcal{NN}_{M,
                               \lceil
                                 C_0 \cdot \log_2 (\nicefrac{1}{\varepsilon})
                               \rceil,
                               d}^{\mathcal{B}}
            \, : \, \| f - \Realization_\varrho (\Phi) \|_{L^{\Felix{p}}}
                    \leq \varepsilon
     \right\} \, ,
  \label{eq:MEpsilonDefinitionSingleFunction}
\end{equation}
as $\varepsilon \downarrow 0$.
In words, $M_{\varepsilon,\Felix{p}} (f)$ describes the minimal number of
nonzero weights that a neural network (with activation function $\varrho$ and
with weights that can be encoded with
$\lceil C_0 \cdot \log_2 (\nicefrac{1}{\varepsilon}) \rceil$
bits using the coding scheme $\mathcal{B}$) needs to have in order to
approximate $f$ up to an $L^{\Felix{p}}$-error of at most $\varepsilon$.
Of course, for a badly chosen activation function
(for instance, for $\varrho \equiv 0$), it might happen that the set over
which the infimum is taken in
Equation \eqref{eq:MEpsilonDefinitionSingleFunction} is empty;
in this case, $M_{\varepsilon,\Felix{p}} (f) := \infty$.

The quantity $M_{\varepsilon,\Felix{p}} (f)$ describes how well a \emph{single}
function $f$ can be approximated.
In contrast, for optimality in a uniform setting, we are given a whole
\emph{function class}
$\mathcal{C} \subset L^{\Felix{p}} ([-\nicefrac{1}{2}, \nicefrac{1}{2}]^d)$,
and we are interested in the behavior of
\begin{equation*}
  M_{\varepsilon,\Felix{p}} (\mathcal{C})
  := M_{\varepsilon,\Felix{p}}^{\mathcal{B}, \varrho, C_0} (\mathcal{C})
  := \sup_{f \in \mathcal{C}}
       M_{\varepsilon,\Felix{p}}^{\mathcal{B}, \varrho, C_0} (f)
\end{equation*}
as $\varepsilon \downarrow 0$.
Note that $M_{\varepsilon,\Felix{p}} (\mathcal{C}) \leq M$ if and only if
\emph{every} function $f \in \mathcal{C}$ can be approximated with a neural
network
$\Phi_{f,\varepsilon} \in \mathcal{NN}_{M,
                                        \lceil
                                          C_0 \cdot \log_2 (\nicefrac{1}{\varepsilon})
                                        \rceil,
                                        d}^{\mathcal{B}}$
up to an $L^{\Felix{p}}$ error of $\varepsilon$.


The following theorem establishes a lower bound on
$M_{\varepsilon,\Felix{p}} (\mathcal{HF}_{\beta, d, B})$.
This lower bound shows that the size of the networks that are constructed in
Theorem \ref{thm:ApproximationOfPiecewiseConstantFunctions}
and Corollary \ref{cor:PiecewiseSmoothFunctions} is optimal,
up to a logarithmic factor in $\nicefrac{1}{\varepsilon}$.

\begin{theorem}\label{thm:Optimality}
  Let $d \in \N_{\geq 2}$ and $\Felix{p}, \beta, B, C_0 > 0$.
  Then there exist constants $C = C(d,\Felix{p}, \beta,B, C_0) > 0$ and
  $\varepsilon_0 = \varepsilon_0(d, \Felix{p}, \beta, B)>0$, such that for each
  encoding scheme of real numbers $\mathcal{B}$ and any activation function
  $\varrho : \R \to \R$ with $\varrho(0) = 0$, we have
  \[
    M_{\varepsilon,\Felix{p}}^{\mathcal{B}, \varrho, C_0}
      (\mathcal{HF}_{\beta, d, B} )
    \geq C \cdot \varepsilon^{- \frac{\Felix{p} (d-1)}{\beta}} \,
                 \bigg/ \, \log_2 \left(\frac{1}{\varepsilon} \right)
    \qquad \text{for all} \quad \varepsilon \in (0, \varepsilon_0).
  \]

\end{theorem}

The preceding theorem establishes a lower bound in the uniform setting that
was discussed \Felix{above}.
In general, given such a lower bound for the \emph{uniform} error,
it is \emph{not} clear that there is also a specific \emph{single}
function $f \in \mathcal{HF}_{\beta, d, B}$ for which
$M_{\varepsilon,\Felix{p}} (f) \gtrsim \varepsilon^{-\Felix{p} (d-1)/\beta}$
(up to log factors).
As the following theorem---our main optimality result---shows,
this \Felix{nevertheless} turns out to be true.

\begin{theorem}\label{thm:SingleFunctionOptimality}

  Let $d \in \N_{\geq 2}$, and $\Felix{p}, \beta, B, C_0 > 0$.
  Let $\varrho : \R \to \R$ be arbitrary with $\varrho(0) = 0$, and let
  $\mathcal{B}$ be a \emph{consistent} encoding scheme of real numbers.
  Then there is some $f \in \mathcal{HF}_{\beta,d,B}$
  (potentially depending on $\varrho, d, \mathcal{B}, \Felix{p}, \beta, B, C_0$)
  and a null-sequence $(\varepsilon_k)_{k \in \N}$ in $(0, \nicefrac{1}{2})$
  satisfying
  \[
    M_{\varepsilon_k, \Felix{p}}^{\mathcal{B}, \varrho, C_0} (f)
    \geq \frac{\varepsilon_k^{-\frac{\Felix{p} (d-1)}{\beta}}}
              {\log_2 \left(\frac{1}{\varepsilon_k}\right)
               \cdot
               \log_2 \left(\log_2 \left(\frac{1}{\varepsilon_k}\right)\right)}
    \qquad \text{for all} \quad k \in \N.
  \]
\end{theorem}

%

Although it is a trivial consequence of Theorem \ref{thm:SingleFunctionOptimality},
we note the following corollary which shows that the networks constructed in
Theorem \ref{thm:ApproximationOfPiecewiseConstantFunctions} and
Corollary \ref{cor:PiecewiseSmoothFunctions} are of (almost) optimal complexity,
even if one is only interested in approximating a single (judiciously chosen)
horizon function $f \in \mathcal{HF}_{\beta, d, B}$.

\begin{corollary}\label{cor:IndividualFunctionLowerBound}
  The function $f \in \mathcal{HF}_{\beta,d,B}$ from
  Theorem \ref{thm:SingleFunctionOptimality} satisfies
  $M_{\varepsilon,\Felix{p}} (f) \notin \mathcal{O}(\varepsilon^{-\gamma})$
  as $\varepsilon \downarrow 0$, for every $\gamma < \Felix{p} (d-1)/\beta$.
\end{corollary}
\begin{remark*}
  Thus, the rate obtained in Theorem
  \ref{thm:ApproximationOfPiecewiseConstantFunctions} is (almost) \emph{optimal}
  in the sense that there is one fixed (but unknown) horizon function
  $f \in \mathcal{HF}_{\beta,d,B}$ such that as $\varepsilon \downarrow 0$,
  one \emph{cannot} achieve
  $\|f - \Realization_{\varrho} (\Phi_\varepsilon)\|_{L^{\Felix{p}}}
   \leq \varepsilon$
  with a network $\Phi_\varepsilon$ that has only
  $\mathcal{O}(\varepsilon^{-\gamma})$ nonzero weights,
  for some $\gamma < \Felix{p} (d-1)/\beta$, at least if one insists that the
  weights of $\Phi_\varepsilon$ can be encoded with at most
  $\lceil C_0 \cdot \log_2 (\nicefrac{1}{\varepsilon}) \rceil$ bits.
\end{remark*}

\subsection{\texorpdfstring{\Felix{Optimality in terms of the number of layers}}
{Optimality in terms of the number of layers}}
\label{sub:OptimalDepth}

We now establish a lower bound on the number of layers $L(\Phi_\varepsilon)$
that a family of ReLU neural network $(\Phi_\varepsilon)_{\varepsilon > 0}$
needs to have to achieve a given approximation rate for approximating smooth
functions.
In this subsection, we again assume that the activation function $\varrho$
is the ReLU, that is, $\varrho(x) = \max \{0, x\} = x_+$.

Shortly after the first version of the present paper appeared on the arXiv, we
became aware of \cite[Theorem 4]{pmlr-v70-safran17a}, which yields a statement
close to the following result, and which was published almost a year before.
Nevertheless, our result still yields a generalization of that in
\cite[Theorem 4]{pmlr-v70-safran17a}:
First, we are able to cover approximation in $L^p$ for arbitrary
\Felix{$p \in (0,\infty)$}, while in \cite{pmlr-v70-safran17a}, only the case
$p=2$ is considered.
Second, our proof is more elementary, since it does not rely on
Legendre polynomials, which are used crucially in \cite{pmlr-v70-safran17a}.

The following theorem will be proven in the appendix as
Theorem \ref{thm:DepthLowerBoundAppendix}.
\begin{theorem}
\label{thm:DepthLowerBound}Let $\Omega\subset\R^{d}$ be nonempty,
open, bounded, and connected. Furthermore, let $f\in C^{3}\left(\Omega\right)$
be nonlinear, \Felix{and let $p \in (0,\infty)$}.
Then there is a constant $C_{f,\Felix{p}}>0$ satisfying
\begin{align*}
  \left\Vert
    f - \Realization_{\varrho}(\Phi)
  \right\Vert _{L^{p}}
  & \geq C_{f,\Felix{p}} \cdot \left(N(\Phi) - 1\right)^{-2 L (\Phi)},\\
  \left\Vert
    f - \Realization_{\varrho}(\Phi)
  \right\Vert _{L^{p}}
  & \geq C_{f,\Felix{p}} \cdot \left(M(\Phi) + d\right)^{-2L(\Phi)}
\end{align*}
for each ReLU neural network $\Phi$
with input dimension $d$ and output dimension $1$.
\end{theorem}
\begin{remark}
The theorem (and also its proof) is inspired by \cite[Theorem 6]{YAROTSKY2017103},
where it is shown that if $f\in C^{2}\left(\smash{\left[0,1\right]^{d}}\right)$
is nonlinear and $L\in\N$ is fixed, and if
$\left\Vert
   f-\Realization_{\varrho}\left(\Phi\right)
 \right\Vert_{L^{\infty}\left(\left[0,1\right]^{d}\right)}
 \leq\varepsilon$
with $\varepsilon \in \left(0,1\right)$ for a neural network $\Phi$
with $L\left(\Phi\right)=L\geq2$, then
$\min\left\{ M\left(\Phi\right), \, N\left(\Phi\right) \right\}
 \geq c\cdot\varepsilon^{-1/\left(2\left(L-1\right)\right)}$
with $c=c\left(f,L\right)$. Note that Yarotsky uses a slightly different
definition of neural networks, but the given formulation of his result
is already adapted to our definition of neural networks.

The main difference \Felix{between the two results}
is that Yarotsky considers approximation in $L^{\infty}$,
while we consider approximation in $L^{p}$ for \Felix{$p \in (0,\infty)$},
where it is harder to reduce the $d$-dimensional case to the one-dimensional
case, as seen in the proof of
Proposition \ref{prop:PPiecewiseAffineFunctionsAreBad}.

Furthermore, there is a difference in the \emph{sharpness} of the
results: As we saw in Section \ref{sec:ApproxOfClass},
to approximate a function $f\in\mathcal{F}_{\beta,d,B}$ of regularity
$C^{\beta}$ up to error $\varepsilon$ in the $L^{p}$ norm, one
can take a neural network $\Phi$ with
$\mathcal{O}\left(\varepsilon^{-d/\beta}\right)$ nonzero weights and a given
fixed depth
\Felix{$L \leq (2 + \lceil \log_2 \beta \rceil) \cdot (11 + \nicefrac{\beta}{d})$}.
In this sense, up to a logarithmic multiplicative factor,
our constructed networks have an \emph{optimal depth}.

In contrast, the networks $\Phi$ constructed in \cite[Theorem 1]{YAROTSKY2017103}
\Felix{for approximating} a given function $f\in\mathcal{F}_{n,d,1}$ up
to error $\varepsilon$ in the $L^{\infty}$ norm \Felix{have}
$\mathcal{O}\left(
   \varepsilon^{-d/n} \cdot \log_{2}\left(\nicefrac{1}{\varepsilon}\right)
 \right)$
nonzero weights and neurons, \Felix{and they} have
$\Theta\left(\log_{2}\left(\nicefrac{1}{\varepsilon}\right)\right)$
\Felix{layers}; that is,
\emph{the depth grows with increasing accuracy of the approximation}.

Finally, note that it is necessary to assume a certain regularity
of $f$ to get the result, since there are nonlinear functions (like
the ReLU $\varrho$) which can be approximated arbitrarily well using
ReLU networks with a \emph{fixed} number of weights, neurons and layers.
\end{remark}

The following corollary states the connection between the number of weights or
neurons and the number of layers more directly.
It is proven in the appendix as
Corollary \ref{cor:OptimalRateRequiresDepthAppendix}.

\begin{corollary}
\label{cor:OptimalRateRequiresDepth}Let $\Omega\subset\R^{d}$ be nonempty,
open, bounded, and connected. Furthermore, let $f\in C^{3}\left(\Omega\right)$
be nonlinear, and let $p \in (0,\infty)$.
If there are constants $C,\theta > 0$, a null-sequence
$\left(\varepsilon_{k}\right)_{k \in \N}$ of positive numbers, and
a sequence $\left(\Phi_{k}\right)_{k \in \N}$ of ReLU neural networks satisfying
\[
  \left\Vert
    f - \Realization_{\varrho}\left(\Phi_{k}\right)
  \right\Vert _{L^{p}}
  \leq C\cdot\varepsilon_{k}
  \quad\text{ and }\quad
  \left[
    M\left(\Phi_{k}\right) \leq C\cdot\varepsilon_{k}^{-\theta}
    \text{ or }
    N\left(\Phi_{k}\right)\leq C\cdot\varepsilon_{k}^{-\theta}
  \right]
\]
for all $k \in \N$, 
then
\[
  \liminf_{k\to\infty} \:
    L\left(\Phi_{k}\right)
  \geq \frac{1}{2\theta}.
\]
\end{corollary}
\begin{remark}
The corollary demonstrates that a specific approximation rate in terms of
numbers of neurons or weights \emph{cannot} be achieved if the depth of the
network is too small.
In fact, suppose we are given $f \in \mathcal{E}_{r, \beta, d, B}^{\Felix{p}}$
where $r \in \N$, $d \in \N_{\geq 2}$, $\beta, B > 0$ and such that $f$ is
non-linear and $C^3$ when restricted to an open, connected set
$A \subset [-\nicefrac12,\nicefrac12]^d$, and let $(\varepsilon_k)_{k \in \N}$
be a null-sequence of positive numbers.
Then we conclude by Corollary \ref{cor:PiecewiseSmoothFunctions}
that there is a sequence of neural networks $\Phi_k$ such that
\[
  \|
    f - \Realization_\varrho(\Phi_k)
  \|_{L^{\Felix{p}}}
  \leq \varepsilon_k
  \quad \text{ and } \quad
  M(\Phi_k) \leq C \cdot \varepsilon_k^{-\frac{\Felix{p} (d-1)}{\beta}}.
\]
for all $k\in \N$.
Consequently, Corollary \ref{cor:OptimalRateRequiresDepth} applied to
$f\Felix{|_A}$
demonstrates that there is a lower bound on the number of layers of the
constructed networks given by $\beta/(\Felix{2p} (d-1))$.
\Felix{Therefore}, the neural networks constructed in
Corollary \ref{cor:PiecewiseSmoothFunctions} have
the optimal number of layers, up to a multiplicative factor which is
logarithmic in $\beta$.
\end{remark}

\section{Curse of dimension}\label{sec:CurseOfDIm}


The results of the previous sections demonstrate that piecewise smooth functions
in $\R^d$, $d\in \N$, with jump curves of regularity $C^\beta$ ($\beta > 0$) can
be approximated up to an $L^{\Felix{p}}$-error of
$\mathcal{O}(M^{-\beta/(\Felix{p}(d-1))})$ by
realizations of ReLU networks with $M$ nonzero weights.
While this is the optimal rate, we observe that this rate suffers from the
curse of dimension.
In fact, for \Felix{large input dimensions $d$}, only a very slow approximation
rate can be guaranteed.
Nevertheless, even though in practice data is usually high-dimensional,
\Felix{neural} networks appear to model the involved function classes well.
This raises the question whether the classical function spaces are an
appropriate model.

\Felix{Specifically, classifier functions that occur in practice exhibit
\emph{invariances}, while such invariances are not incorporated into the
classical function spaces.}
For instance, an image classifier should be translation invariant,
scaling invariant, invariant to small smooth deformations, and invariant
to small changes in color, brightness, or contrast;
see \cite{Mallat2012, WiatowskiBDeepConcNets}.

A way to model such a function class is by resorting to a two-step procedure:
\Felix{The occurrence of invariances}
described above can be interpreted 
\Felix{as stating} that every classifier function $f$ can be decomposed as
$f = g \circ \tau$, where $\tau$ is a smooth dimension-reducing ``feature map''
that incorporates \Felix{the invariances},
and $g$ is a piecewise smooth function responsible for the classification.

To translate this intuition into a solid mathematical framework,
we start by introducing a function class that models the smooth
dimension-reducing ``feature maps'' $\tau$.

\begin{definition}
Let $d, D \in \N$ with $d \leq D$, and let $\kappa, \Felix{p} > 0$ and
$a = (a_n)_{n\in \N}$, with $a_n\in \Felix{(0,\infty)}$ for all $n\in \N$.
Then we define
\begin{align*}
\mathcal{S}_{\kappa, d, D, a}^{\Felix{p}}
: =
& \left\{
     \tau :     \left[-\nicefrac{1}{2}, \nicefrac{1}{2}\right]^D
            \to \left[-\nicefrac{1}{2}, \nicefrac{1}{2}\right]^d
     \,:\,
     \tau_i \in \mathcal{F}_{n, D, a_n}
     \text{ for all } i = 1, \dots, d \text{ and } n \in \N, \right. \\
   & \quad \text{and } \left.
     \|
       g \circ \tau
     \|_{L^{\Felix{p}} \left(\left[-\nicefrac{1}{2}, \nicefrac{1}{2}\right]^D\right)}
     \leq \kappa \cdot
          \|g\|_{L^{\Felix{p}}\left(\left[-\nicefrac{1}{2}, \nicefrac{1}{2}\right]^d\right)}
     \text{ for all } g \in L^{\Felix{p}}\left(\left[-\nicefrac{1}{2}, \nicefrac{1}{2}\right]^d\right)
  \right\}.
\end{align*}
\end{definition}

The assumption
$\|g \circ \tau \|_{L^{\Felix{p}}\left(\left[-\nicefrac{1}{2}, \nicefrac{1}{2}\right]^D\right)}
\leq \kappa \cdot \|g \|_{L^{\Felix{p}}\left(\left[-\nicefrac{1}{2}, \nicefrac{1}{2}\right]^d\right) }$
might seem quite restrictive at first sight, but is in fact satisfied by all
smooth submersions $\tau$, for an appropriate constant
$\kappa = \kappa(\tau,\Felix{p})$.
Indeed, we have the following lemma, which will be proven in
Appendix \ref{sec:SmoothSubmersions}.

\begin{lemma}\label{lem:CompositionWithSubmersion}
  Let $n,m\in \N$ with $n \leq m$, let $U \subset \R^m$ be open,
  and let $\tau : U \to \R^n$ be continuously differentiable.
  Finally, let $\emptyset \neq K \subset U$ be compact and assume that
  $D \tau (x) \in \R^{n \times m}$ has full rank for all $x \in K$.

  Then there is a constant $C > 0$ satisfying
  \[
    \int_{K} (f \circ \tau)(x) \, dx
    \leq C \cdot \int_{\tau (K)} f(y) \, dy
    \qquad \text{for all Borel measurable } \quad f : \tau (K) \to \R_+ \, .
  \]
  \Felix{In particular,
  \[
    \|f \circ \tau\|_{L^p (K)} \leq C^{1/p} \cdot \|f\|_{L^p (\tau(K))}
    \quad \text{for all } 0 < p < \infty \text{ and }
    f : \tau(K) \to \R \text{ measurable}.
  \]}
\end{lemma}
\begin{remark*}
  As a consequence of Lemma \ref{lem:CompositionWithSubmersion},
  if $U \subset \R^D$ is open with $U \supset [-\nicefrac12, \nicefrac12]^D$,
  and if
  \begin{itemize}
    \item $\tau \in C^\infty (U; \R^d)$ with
          $\tau([-\nicefrac12,\nicefrac12]^D) \subset [-\nicefrac12,\nicefrac12]^d$
          and $d \leq D$;

    \item $\tau_i \in \mathcal{F}_{n,D,a_n}$ for all $i =1,\dots, d$ and $n \in \N$; and

    \item $D\tau(x)$ has full rank for all $x \in [-\nicefrac{1}{2},\nicefrac{1}{2}]^D$,
  \end{itemize}
  then $\tau \in \mathcal{S}_{\kappa,d,D,a}^{\Felix{p}}$
  for a suitable constant $\kappa = \kappa(\tau, \Felix{p}) > 0$,
  \Felix{where $p \in (0,\infty)$ can be chosen arbitrarily}.
\end{remark*}

We proceed to define a function class modelling precisely the behavior we
described before.
\begin{definition}
  Let $d,D, r \in \N$ with $D \geq d \geq 2$,
  let $B, \beta, \kappa, \Felix{p} > 0$, and $a = (a_n)_{n\in \N}$, with
  $a_n \in \Felix{(0,\infty)}$ for all $n\in \N$. Then we define
  \[
    \mathcal{SE}_{r, \beta, a, B, \kappa, d, D}^{\Felix{p}}
    := \left\{
         f \in L^\infty \left(\left[-\nicefrac{1}{2}, \nicefrac{1}{2}\right]^D\right)
         \,:\,
         f = g \circ \tau,
         \text{ where } g \in \mathcal{E}_{r, \beta, d, B}^{\Felix{p}},
         \text{ and } \tau \in \mathcal{S}_{\kappa, d, D, a}^{\Felix{p}}
       \right\}.
  \]
\end{definition}

We can now describe the size of networks which suffices for the approximation
of arbitrary functions
$f \in \mathcal{SE}_{r, \beta, a, B, \kappa, d, D}^{\Felix{p}}$,
up to an $L^{\Felix{p}}$-error of $\eps$.
The proof for the theorem below is given in the appendix in
Section \ref{sec:HighDimApprox}.

\begin{theorem}\label{thm:ApproximationOfSymmetricFunctions}
  For $r,d,D \in \N$ with $D \geq d \geq 2$,
  for $\Felix{p}, \kappa, \beta, B > 0$, and $a = (a_n)_{n\in \N}$ with
  $a_n \in \Felix{(0,\infty)}$ for all $n \in \N$,
  there are constants $c= c(d,D, \Felix{p}, \kappa, r,\beta, B, a) > 0$,
  $L = L(d, D, \Felix{p}, \kappa, r,\beta, B, a) \in \N$ and
  $s = s(d,D, \Felix{p}, \kappa, r,\beta, B, a) \in \N$, such that for any
  $f \in \mathcal{SE}_{r, \beta, a, B, \kappa, d, D}^{\Felix{p}}$ and any
  $\varepsilon \in (0, \nicefrac{1}{2})\vphantom{\sum_j}$,
  there is a neural network $\Phi^f_\varepsilon$ with at most $L$ layers,
  and at most $c \cdot \varepsilon^{-\Felix{p} (d-1)/\beta}$ nonzero,
  $(s,\varepsilon)$-quantized weights such that
  \begin{align} \label{eq:TargetApproxSymmetricFunctions}
    \left\|\Realization_\varrho(\Phi^f_\varepsilon) - f\right\|_{L^{\Felix{p}}}
    < \varepsilon.
  \end{align}
\end{theorem}
\begin{remark}\begin{itemize}
  \item Contrary to all previous results, we do not give an explicit bound on
        $L$ here. This is because the proof requires a very large non-explicit
        number of layers, which we believe to be highly suboptimal.

  \item We observe that the established approximation rate for
        $\mathcal{SE}_{r, \beta, a, B, \kappa, d, D}^{\Felix{p}}$ matches the
        optimal rate of Corollary \ref{cor:PiecewiseSmoothFunctions}
        and \Felix{Theorem} \ref{thm:Optimality}
        for $\mathcal{E}_{r, \beta, d, B}^{\Felix{p}}$.
        In particular, it is independent of the ambient dimension $D$.

  \item \Felix{Even though the approximation \emph{rate}---that is, the exponent
    $-p(d-1)/\beta$ of $\varepsilon$---is independent of the input dimension
    $D$, it should be observed that the number of neurons is bounded by
    $c \cdot \varepsilon^{-p(d-1)/\beta}$, where the constant $c$ \emph{does}
    depend on $D$.}
\end{itemize}
\end{remark}

\appendix

\section{Approximation of piecewise smooth functions}
\label{sec:ApproxWithPieceSmoothAppendix}

In this section, we prove all results stated in Section \ref{sec:ApproxOfClass},
as well as a couple of auxiliary lemmas.
\pp{Throughout the entire section, we assume that $\varrho$ is the ReLU,}
\Felix{that is, $\varrho : \R \to \R, x \mapsto \max \{0,x\}$.}

We start with a lemma that will be used often to obtain approximating networks
with \emph{bounded} realization.

\begin{lemma}\label{lem:boundedApprox}
There is a universal constant $c > 0$ such that the following holds:

For arbitrary $d, s \Felix{, k} \in \N$, $B >0$,
$\varepsilon \in (0,\nicefrac{1}{2})$, and any neural network $\Psi$ with
$d$-dimensional input and \Felix{$k$}-dimensional output and with
$(s,\varepsilon)$-quantized weights, there exists a neural network $\Phi$ with
the same input/output dimensions as $\Psi$ and with the following properties:
\begin{itemize}
\item $M(\Phi) \leq \Felix{2} M(\Psi) + c \Felix{k}$,
      and $L(\Phi) \leq L(\Psi) + 2$;

\item all weights of $\Phi$ are $(s_0, \varepsilon)$-quantized,
      where $s_0 := \max\{\lceil\log_2(\lceil B \rceil)\rceil, s\}$;

\item $\Realization_\varrho(\Phi)
       = \Felix{(\tau_{B} \times \cdots \times \tau_B)}
       \circ \Realization_\varrho(\Psi)$,
      where the function
      \[
        \tau_{B} :
        \R \to     \Felix{\big[ -\lceil B \rceil \,,\, \lceil B \rceil \, \big]},
        y  \mapsto \mathrm{sign}(y) \cdot \min\{|y|, \lceil B \rceil\}
      \]
      is $1$-Lipschitz and satisfies $\tau_{B} (y) = y$ for all $y \in \R$ with
      $|y| \leq \lceil B \rceil$.
\end{itemize}
\end{lemma}
\begin{proof}
\Felix{Consider the neural network $\Phi^B := \big( (A_1, b_1) , (A_2, b_2) \big)$,
where $A_1 \in \R^{2k \times k}$ is the matrix associated
(via the standard basis) to the linear map
$\R^k \to \R^{2k}, (x_1, \dots, x_k) \mapsto (x_1,x_1, x_2,x_2,\dots,x_k,x_k)$,
while $A_2 \in \R^{2k \times k}$ is associated to the linear map
$\R^{2k} \to \R^k,
(x_1,\dots,x_{2k}) \mapsto (x_1 - x_2, x_3 - x_4, \dots, x_{2k-1} - x_{2k})$.
Furthermore, $b_1 := \lceil B \rceil \cdot (1,-1,1,-1,\dots,1,-1)^T \in \R^{2k}$
and $b_2 := - \lceil B \rceil \cdot (1,\dots,1)^T \in \R^k$.

It is not hard to see $\|A_i\|_{\ell^0} \leq 2k$ for $i \in \{1,2\}$,
and furthermore $\|b_1\|_{\ell^0} \leq 2k$ and $\|b_2\|_{\ell^0} \leq k$,
so that $M(\Phi^B) \leq 7k$, and clearly $L(\Phi^B) = 2$.
In addition, we note that the weights of $\Phi^B$ are
$(\lceil \log_2(\lceil B \rceil) \rceil, \varepsilon)$-quantized,
for any $\varepsilon \in (0, \nicefrac{1}{2})$.

Finally, a direct calculation shows}
$\Realization_\varrho(\Phi^B) = \tau_{B} \Felix{\times \cdots \times \tau_B}$,
\Felix{where the \pp{C}artesian product has $k$ factors.}
All in all, setting $\Phi := \Phi^B \sconc \Psi$ yields the claim,
\Felix{thanks to Remark \ref{rem:SpecialConc}. In fact, that remark shows
that we can take $c=14$.}
\end{proof}

\subsection{Approximation of the Heaviside function}

As a first step towards approximating horizon functions,
it is necessary to recreate a sharp jump.
To this end, we show that the Heaviside function can be approximately
created with a network of fixed size.

\begin{lemma}\label{lem:ApproxOfHeaviside}
Let $d\in \N_{\geq 2}$ and $H := \chi_{[0,\infty) \times \R^{d-1}}$.
For every $\varepsilon > 0$ there exists a neural network $\Phi_\varepsilon^H$
with two layers and five (nonzero) weights which
only take values in $\{\Felix{\varepsilon^{-1}}, 1,-1\}$, such that
\[
  0 \leq \Realization_\varrho (\Phi_\varepsilon^H) \leq 1
  \quad \text{and} \quad
  |H(x) - \Realization_\varrho(\Phi_\varepsilon^H)(x)|
  \leq \chi_{\Felix{[0,\varepsilon] \times \R^{d-1}}}(x)
  \quad \text{for all} \quad x \in \R^d \, .
\]
Moreover,
$\|
   H - \Realization_\varrho(\Phi_\varepsilon^H)
 \|_{L^{\Felix{p}}([-\nicefrac{1}{2},\nicefrac{1}{2}]^d)}
 \leq \varepsilon^{\Felix{1/p}}$
\Felix{for all $p \in (0,\infty)$.}
\end{lemma}

\begin{proof}
Let $\Phi_\varepsilon^H := \big( (A_1, b_1), (A_2, b_2) \big)$ with
\begin{align*}
  A_1 := \left(
           \begin{array}{l l l l}
             \Felix{\varepsilon^{-1}} & 0 & 0 & \dots \\
             \Felix{\varepsilon^{-1}} & 0 & 0 & \dots
           \end{array}
         \right)
      \in \R^{2\times d},
  & \quad
  b_1 := \begin{pmatrix} 0 \\ -1 \end{pmatrix} \in \R^2, \\
  A_2 := \left(\begin{array}{l l} 1 & -1 \end{array}\right)
      \in \R^{1 \times 2},
& \quad
  b_2 := 0 \in \R^1.
\end{align*}
Then
\begin{align*}
  \Realization_\varrho(\Phi_\varepsilon^H)(x)
  = \varrho \left(\frac{x_1}{\Felix{\varepsilon}}\right)
    -\varrho \left(\frac{x_1}{\Felix{\varepsilon}}-1\right)
  \text{ for } x = (x_1, \dots, x_d) \in \R^d .
\end{align*}
From this, we directly compute
\begin{align*}
  \Realization_\varrho(\Phi_\varepsilon^H)(x) &= 0
  \text{ for } x_1 < 0,
  \quad
  \Realization_\varrho(\Phi_\varepsilon^H)(x) = \frac{x_1}{\Felix{\varepsilon}}
  \text{ for } 0 \leq x_1 \leq \Felix{\varepsilon},
  \quad \text{ and }
  \Realization_\varrho(\Phi_\varepsilon^H)(x) = 1
  \text{ for } \Felix{ \varepsilon } \pp{ < x_1}.
\end{align*}
We conclude that indeed
$|H(x) - \Realization_\varrho (\Phi_\varepsilon^H)(x)|
\leq \chi_{\Felix{[0,\varepsilon] \times \R^{d-1}}}(x)$ and
$0 \leq \Realization_\varrho (\Phi_\varepsilon^H) \leq 1$, and therefore also
\[
  \left\|
    H - \Realization_\varrho(\Phi_\varepsilon^H)
  \right\|_{L^p ([-\nicefrac{1}{2}, \nicefrac{1}{2}]^d)}^p
   \leq \int_{[0,\Felix{\varepsilon}] \times [-\nicefrac{1}{2},\nicefrac{1}{2}]^{d-1}}
            1 \, dx
   =    \Felix{\varepsilon}. \qedhere
\]
\end{proof}

\subsection{Approximation of smooth functions}

The second cornerstone of our approximation results is the approximation of
smooth functions.
The argument proceeds as follows:
We start by \Felix{showing that one can approximate} a multiplication operator
with a ReLU network (Lemma \ref{lem:MultiplicationWithBoundedLayerNumber}).
With such an operator in place, one can construct networks realizing approximate
monomials (Lemma \ref{lem:MonomialsWithBoundedNumberOfLayers}).
From there on, it is quite clear that for a given function $f$ one is also able
to approximate Taylor polynomials of $f$ at various root points
(Lemma \ref{lem:ComputationalUnit}).
In combination with an approximate partition of unity
(Lemmas \ref{lem:MultiplicationWithACharacteristicFunction} and
\ref{lem:MultiplicationWithACharacteristicFunctionArray}), one can thus
approximate $C^\beta$ functions (Theorem \ref{thm:ApproxOfSmoothFctnAppendix}).

We start by constructing the approximate multiplication operator.
Already in \cite[Proposition 3]{YAROTSKY2017103}, it was shown that ReLU
networks can compute an approximate multiplication map with error at most
$\varepsilon$, using $\log_2 (\nicefrac{1}{\varepsilon})$ layers and nodes.
\Felix{However}, \Felix{this means that} the number of layers of the network
grows indefinitely as $\varepsilon \downarrow 0$.
The following lemma offers a compromise between the number of layers and the
growth of the number of weights, thereby allowing a construction with a
\emph{fixed} number of layers.

\begin{lemma}\label{lem:MultiplicationWithBoundedLayerNumber}
Let $\theta > 0$ be arbitrary.
Then, for every $L \in \N$ with $L > \Felix{(2\theta)^{-1}}$ and each
$M \geq 1$, there are constants $c = c(L,M,\theta) \in \N$ and $s = s(M) \in \N$
with the following property:

For each $\varepsilon \in (0, \nicefrac{1}{2})$, there is a neural network
$\wt{\times}$ with the following properties:
\begin{itemize}
  \item $\wt{\times}$ has at most $c \cdot \varepsilon^{-\theta}$
        nonzero, $(s, \varepsilon)$-quantized weights;

  \item $\wt{\times}$ has \Felix{$2L + 8$} layers;

  \item for all $x,y \in [-M,M]$, we have
        $|xy - \Realization_{\varrho} (\wt{\times})(x,y)| \leq \varepsilon$;

  \item for all $x,y \in [-M,M]$ with $x \cdot y = 0$, we have
        $\Realization_{\varrho} (\wt{\times})(x,y) = 0$.
\end{itemize}
\end{lemma}

\begin{proof}
Our proof is heavily based on that of \cite[Propositions 2 and 3]{YAROTSKY2017103}.
\Felix{The basic idea is to first approximate the square function, and then use
the polarization identity $xy = \frac{1}{2} \cdot \big( (x+y)^2 - x^2 - y^2 \big)$
to get an approximate multiplication operator.}

\Felix{As a preparation for approximating the square function, we define as in
\cite{YAROTSKY2017103} the function}
\[
    g : [0,1] \to [0,1], \
    x \mapsto \begin{cases}
                2x,     & \text{if } x < \nicefrac{1}{2}, \\
                2(1-x), & \text{if } x \geq \nicefrac{1}{2}.
              \end{cases}
\]
Next, for $t \in \N$, we let
$g_t := \underbrace{g \circ \cdots \circ g}_{t \text{ times}}$ be the $t$-fold
composition of $g$.
In the proof of \cite[Proposition 2]{YAROTSKY2017103}, it was shown that
\[
  g_t (x)
  = \begin{cases}
      2^t \cdot \left( x - \frac{2k}{2^t} \right),
      & \text{if } x \in \left[ \frac{2k}{2^t}, \frac{2k+1}{2^t} \right]
        \text{ for some } k \in \{0,1,\dots,2^{t-1} - 1\}, \\[0.2cm]
      - 2^t \cdot \left( x - \frac{2k}{2^t} \right),
      & \text{if } x \in \left[ \frac{2k - 1}{2^t}, \frac{2k}{2^t} \right]
        \text{ for some } k \in \{1,\dots,2^{t-1}\},
    \end{cases}
\]
so that each function $g_t$ is continuous and piecewise affine-linear
with $2^t$ ``pieces''. From this, it is not hard to see that
\[
  g_t (x)
  = 2^{t} \cdot
    \left(
      \varrho (x)
      + \sum_{k=1}^{2^{t-1} - 1}
          2 \cdot \varrho\left( x - \frac{2k}{2^{t}} \right)
      - \sum_{\ell=1}^{2^{t-1}}
          2 \cdot \varrho \left( x - \frac{2\ell - 1}{2^t} \right)
    \right)
  \qquad \text{for all} \quad x \in [0,1] \, .
\]
Therefore, for each $t \in \N$, there is a neural network $\Phi_t$ with
one-dimensional input and output, with two layers,
\pp{and $1 + 2^{t} + 1 \leq \Felix{4} \cdot 2^t$ neurons and at most
$2\cdot 2^{t-1} + 2 \cdot 2^{t-1} + 2^{t-1} + 2^{t-1} \leq \Felix{4} \cdot 2^t$}
nonzero weights, such that $g_t = \Realization_{\varrho} (\Phi_t)\Felix{|_{[0,1]}}$.
Furthermore, all weights of $\Phi_t$ can be chosen to be elements of
$[-2^{t+1}, 2^{t+1}] \cap 2^{-t} \Z \subset [-2^{m+1}, 2^{m+1}] \cap 2^{-m}\Z$
as long as $1 \leq t \leq m$ \Felix{for some $m \in \N$}.
Setting $g_0 := \identity_{[0,1]}$, we see also for $t = 0$ that
\Felix{there is a network $\Phi_t$ with all of the properties just stated;}
see Lemma \ref{lem:identity}.

Next, set
\[
  s_0 := 1 + \lceil \log_2 M \rceil \in \N \, ,
  \qquad 
  M_0 := 2^{s_0} \, ,
  \qquad 
  m := s_0 + \left\lceil
               \log_2 (\nicefrac{1}{\varepsilon})/2
             \right\rceil
    \in \N \, ,
  \quad \text{and} \quad
  N := \lceil m / L \rceil \in \N \, ,
\]
so that $2M \leq M_0 \leq 4M$.
Now, \Felix{by division with remainder}, we can write each $1 \leq t \leq m$ as
$t = k N + r$ for certain $k \in \N_0$ and $r \in \{0, \dots, N-1\}$.
Note $k = (t-r)/N \leq t/N \leq m/N \leq L$, and observe
$g_t = \underbrace{g_N \circ \cdots \circ g_N}_{k \text{ factors}} \circ g_r$,
so that we get $g_t = \Realization_{\varrho} (\Phi^{(t)}_0)\Felix{|_{[0,1]}}$,
where $\Phi^{(t)}_0
:= \underbrace{\Phi_N\sconc\cdots\sconc\Phi_N}_{k \text{ factors}} \sconc \Phi_r$
is a neural network with $2(k+1) \leq 2(L+1) < \Felix{2L + 3}$ layers,
\Felix{and with
\[
  M(\Phi_0^{(t)})
  \leq 4^{k} \cdot \max \{ M(\Phi_N) \,,\, M(\Phi_r) \}
  \leq 4^{k} \cdot 4 \cdot 2^N
  \leq 4^{L+1} \cdot 2^N \, ,
\]
see Remark \ref{rem:SpecialConc}.}
Therefore, with $\Phi_{1, \lambda}^{\mathrm{Id}}$
as in Remark \ref{rem:DeepIdentity}, the network
$\Phi^{(t)} := \Phi_{1, 2L+3 - 2(k+1)}^{\mathrm{Id}} \sconc \Phi^{(t)}_0$
satisfies $\Realization_\varrho (\Phi^{(t)})\Felix{|_{[0,1]}} = g_t$,
and $\Phi^{(t)}$ has precisely $\Felix{2L+3}$ layers, and at most
\[
  \Felix{
  2 \, M(\Phi_0^{(t)}) + 2 \, M(\Phi_{1, 2L+3 - 2(k+1)}^{\mathrm{Id}})
  \leq 4^{L+2} \cdot 2^N + 2 \cdot (2L+3)
  \leq (2L + 3 + 4^{L+2}) \cdot 2^N
  =: c_1 \cdot 2^N
  }
\]
nonzero weights, all of which lie in $[-2^{m+1}, 2^{m+1}] \cap 2^{-m} \Z$.
%

\medskip{}

We now use the functions $g_t$ to construct an approximation to the square
function.
Precisely, in the proof of \cite[Proposition 2]{YAROTSKY2017103},
it is shown that
\[
  f_m : [0,1] \to [0,1], x \mapsto x - \sum_{t=1}^m \frac{g_t (x)}{2^{2t}}
  \quad \text{satisfies} \quad
  \| (x \mapsto x^2) - f_m \|_{L^\infty \Felix{([0,1])}} \leq 2^{-2 - 2m} \, .
\]
Now, set
\[
  \Psi := P(\Phi_{1, \Felix{2L+3}}^{\mathrm{Id}},
            P(\Phi^{(1)},
              P(\Phi^{(2)},
                \dots,
                P(\Phi^{(m-1)}, \Phi^{(m)})
                \dots
               )
             )
           ) \, ,
\]
and $\Phi_{\mathrm{sum}} := ( (A_{\mathrm{sum}}, 0) )$ with
$A_{\mathrm{sum}} := (1, -2^{-2 \cdot 1}, -2^{-2 \cdot 2}, \dots, -2^{-2 m})
\in \R^{1 \times (m+1)}$.
Then, the neural network $\Phi_0 := \Phi_{\mathrm{sum}} \sconc \Psi$ satisfies
$\Realization_\varrho (\Phi_0)\Felix{|_{[0,1]}} = f_m \vphantom{\sum_j}$,
and $\Phi_0$ has \Felix{$(2L+3) + 1 = 2L + 4$} layers, and not more than
\Felix{$2 \cdot ( 2 \cdot (2L+3) + m \cdot c_1 \cdot 2^N) + 2(m+1)
\leq c_2 \cdot m \cdot 2^N$}
nonzero weights, which all lie in
$[-2^{m+1}, 2^{m+1}] \cap 2^{-2m}\Z$. Here, $c_2 = c_2 (L) > 0$.

As in the proof of \cite[Proposition 3]{YAROTSKY2017103},
we now use the polarization identity
$x\cdot y = \frac{1}{2} \cdot ((x+y)^2 - x^2 - y^2)$ and the approximation
$f_m$ of the square function to obtain an approximate multiplication.
Precisely, define
\[
    h : \left[- \Felix{\frac{M_0}{2}}, \Felix{\frac{M_0}{2}} \right]^2 \to \R,
                 (x,y) \mapsto \frac{M_0^2}{2} \cdot
                               \left[
                                  f_m \left( \frac{|x+y|}{M_0} \right)
                                  - f_m \left( \frac{|x|}{M_0} \right)
                                  - f_m \left( \frac{|y|}{M_0} \right)
                               \right].
\]
Because of $|x| = \varrho(x) + \varrho(-x)$, and given our implementation of
$f_m = \Realization_\varrho(\Phi_0)\Felix{|_{[0,1]}}$,
it is easy to see
that
$h = \Realization_{\varrho} (\wt{\times})\Felix{|_{[-M_0/2, M_0/2]}}$
for a neural network $\wt{\times}$ with \Felix{$(2L+4) + 4 = 2L + 8$} layers,
and at most $c_3 \cdot m \cdot 2^N$ nonzero weights,
all of which are in $[-2^{2m+2 s_0}, 2^{2m+2 s_0}] \cap 2^{-2m - s_0} \Z$
for some constant $c_3 = c_3 (L) \in \N$.
Next, since $f_m (0) = 0$, we easily get
$\Realization_{\varrho} (\wt{\times})(x,y) = h(x,y) = 0$ if
$x,y \in [-M,M] \subset [-M_0/2, M_0/2]$ with $x \cdot y = 0$.

Finally, for $x,y \in [-M, M] \subset [-M_0/2, M_0/2]$,
we have $|x+y| \leq |x| + |y| \leq 2M \leq M_0$, and hence
\begin{align*}
  & |h(x,y) - xy|
  = \left| h(x,y) - M_0^2 \cdot \frac{x}{M_0} \cdot \frac{y}{M_0} \right| \\
  ({\scriptstyle{\text{polarization}}})
  &= M_0^2 \left|
             \frac{1}{2} \left[
                           f_m \! \left( \frac{|x+y|}{M_0} \right)
                           \! - \! f_m \! \left( \frac{|x|}{M_0} \right)
                           \! - \! f_m \! \left( \frac{|y|}{M_0} \right)
                         \right] \!
             - \! \frac{1}{2} \left[
                                \left(
                                  \frac{x}{M_0} \! + \! \frac{y}{M_0}
                                \right)^2
                                \! \!-\! \left( \frac{x}{M_0} \right)^2
                                \! \!-\! \left( \frac{y}{M_0} \right)^2
                              \right]
           \right| \\
  ({\scriptstyle{\text{since } z^2 = |z|^2}})
  &\leq \frac{M_0^2}{2}
        \left(
          \left|
            f_m \left( \frac{|x+y|}{M_0} \right)
            \!-\! \left(\frac{|x+y|}{M_0}\right)^2
          \right|
          + \left|
              f_m \left( \frac{|x|}{M_0} \right)
              \!-\! \left( \frac{|x|}{M_0} \right)^2
            \right|
          + \left|
              f_m \left( \frac{|y|}{M_0}\right)
              \!-\! \left( \frac{|y|}{M_0} \right)^2
            \right|
        \right) \\
  &\leq \frac{M_0^2}{2} \cdot (2^{-2-2m} + 2^{-2-2m} + 2^{-2-2m})
   \leq \left( \frac{M_0}{2^m} \right)^2
   \leq \varepsilon.
\end{align*}
Here, the last step used that by choice of $m$, we have
$2^m \geq 2^{s_0} \cdot 2^{\log_2 (\nicefrac{1}{\varepsilon}) / 2}
= M_0 \cdot \varepsilon^{-1/2}$.
Thus, all that remains to be proven is that $\wt{\times}$ has the required
number of layers and nonzero weights, and that these weights are
$(s,\varepsilon)$-quantized for some $s = s(M) \in \N$.

To this end, first recall that $\wt{\times}$ has \Felix{$2L + 8$} layers.
Next, we saw above that all weights of $\wt{\times}$ lie in
$[-2^{2m+2s_0}, 2^{2m+2s_0}] \cap 2^{-2m-s_0} \Z$, where
$m = s_0 + \lceil \log_2 (\nicefrac{1}{\varepsilon}) / 2 \rceil
\leq 1 + s_0 + \nicefrac{1}{2} \log_2 (\nicefrac{1}{\varepsilon})$.
Because of $0 < \varepsilon < \nicefrac{1}{2}$, this implies
$2^{2m + 2s_0} \leq 2^{2 + 4 s_0 + \log_2 (\nicefrac{1}{\varepsilon})}
= 2^{2 + 4 s_0} \cdot \varepsilon^{-1} \leq \varepsilon^{-s}$
for $s := 3 + 4 s_0$.
Note that indeed $s = s(M)$, since $s_0 = s_0 (M)$.
Next, we observe $\log_2 (\nicefrac{1}{\varepsilon}) \geq 1$, which implies
$2m + s_0 \leq 3 s_0 + 2 + \log_2 (\nicefrac{1}{\varepsilon})
\leq (4 s_0 + 3) \cdot \log_2 (\nicefrac{1}{\varepsilon})
\leq s \lceil \log_2 (\nicefrac{1}{\varepsilon})\rceil $, and hence
$2^{-2m-s_0} \Z
\subset 2^{-s \lceil \log_2 (\nicefrac{1}{\varepsilon}) \rceil} \Z$.

\smallskip{}

Finally, we note that the number $M (\wt{\times})$ of nonzero weights of the
network $\wt{\times}$ satisfies
\begin{align*}
  M( \wt{\times})
  &\leq c_3 \cdot m \cdot 2^N
  =     c_3 \cdot \left(
                    s_0
                    + \left\lceil \frac{\log_2 (1/\varepsilon)}{2} \right\rceil
                  \right)
            \cdot 2^{\lceil m / L \rceil} \\
  &\leq 4 c_3 \cdot (1 + s_0) \cdot \log_2 (1/\varepsilon) \cdot 2^{m/L}
  \leq  8 c_3 \cdot (1 + s_0) 2^{s_0} \cdot \log_2 (1/\varepsilon)
              \cdot 2^{\log_2 (1/\varepsilon)/(2L)} \\
  &=    8 c_3 \cdot (1 + s_0) 2^{s_0} \cdot \log_2 (1/\varepsilon)
              \cdot \varepsilon^{-1/(2L)}
  \leq  c_{L,M,\theta} \cdot \varepsilon^{-\theta}.
\end{align*}
Here, we used in the last step that $s_0 = s_0 (M)$ and that
$1 / (2L) < \theta$, \Felix{whence}
$\log_2 (\nicefrac{1}{\varepsilon}) \cdot \varepsilon^{-1/(2L)}
\leq C_{L,\theta} \cdot \varepsilon^{-\theta}$,
for a suitable constant $C_{L,\theta} > 0$ and all
$\varepsilon \in (0, \nicefrac{1}{2})$.
\end{proof}

We will be especially interested in the following consequence of
Lemma \ref{lem:MultiplicationWithBoundedLayerNumber}, which demonstrates that
monomials can be (approximately) reproduced by neural networks with a
fixed number of layers.

\begin{lemma}\label{lem:MonomialsWithBoundedNumberOfLayers}
  Let $n, d, \ell \in \N$ be arbitrary.
  Then, there are constants $s = s(n) \in \N$, $c = c(d,n,\ell) \in \N$,
  and $L = L(d, n, \ell) \in \N$ such that
  $L \leq \Felix{(1 + \lceil \log_2 n \rceil ) \cdot (10 + \nicefrac{\ell}{d})}$
  with the following property:

  For each $\varepsilon \in (0, \nicefrac{1}{2})$ and $\alpha \in \N_0^d$ with
  $|\alpha| \leq n$, there is a neural network
  $\Phi^\alpha_\varepsilon \vphantom{\sum_j}$ with $d$-dimensional input and
  one-dimensional output, with at most $L$ layers,
  and with at most $c \cdot \varepsilon^{-d/\ell}$ nonzero,
  $(s, \varepsilon)$-quantized weights, and such that $\Phi^\alpha_\varepsilon$
  satisfies
  \begin{align}\label{eq:ThisShouldBeSatisfiedByPhiAlpha}
    | \mathrm{R}_{\varrho}(\Phi^\alpha_\varepsilon)(x) - x^\alpha |
    \leq \varepsilon
    \quad \text{ for all } x \in \left[ -\frac{1}{2}, \frac{1}{2}\right]^d.
  \end{align}
\end{lemma}
\begin{proof}
  Let $d\in \N$ be fixed, and let $s = s(2) \in \N$ denote the constant from
  Lemma \ref{lem:MultiplicationWithBoundedLayerNumber} for the choice $M=2$.
  We prove the claim by induction over $n \in \N$.

  For $n=1$, we either have $\alpha = 0$, so that
  $x^\alpha = 1 = \Realization_\varrho (\Phi_\varepsilon^\alpha) (x)$ for a $1$-layer
  network $\Phi_\varepsilon^\alpha$ that has only one nonzero
  \Felix{(properly quantized)} weight, or there exists
  $j \in \{1, \dots, d\}$ such that $x^\alpha = x_j$ for all $x$ in
  $[-\nicefrac{1}{2}, \nicefrac{1}{2}]^d$.
  But also in this case, there is a one-layer, one-weight, quantized network
  $\Phi_\varepsilon^\alpha$ with $\Phi_\varepsilon^\alpha (x) = x_j = x^\alpha$
  for all $x \in \R^d$, so that the claim holds.

  Now, let us assume that the claim holds for all $1 \leq n < k$,
  for some $k \in \N_{\geq 2}$.
  We want to show that the claim also holds for $n = k$.
  First, in case of $|\alpha| < k$, it is easy to see that the claim follows
  from the one for the case $n = |\alpha| < k$.
  Therefore, we can assume $|\alpha| = k$.
  Now, pick $\alpha^{(1)}, \alpha^{(2)}\in \N_0^d$ such that
  $|\alpha^{(2)}| =  2^{\lceil  \log_2 k \rceil - 1}$ and
  $\alpha^{(1)}  + \alpha^{(2)} = \alpha$.
  Note that indeed $2^{\lceil \log_2 k \rceil - 1} \in \N$ with
  $2^{\lceil \log_2 k \rceil - 1} < k = |\alpha|$, so that such a choice of
  $\alpha^{(1)}, \alpha^{(2)}$ is possible.
  Next, observe $|\alpha^{(1)}| \leq |\alpha^{(2)}| < k$, and
  $\log_2 |\alpha^{(2)}| = \lceil \log_2 k \rceil - 1$.

  Thus, by applying the inductive claim with $n = |\alpha^{(2)}|$,
  we conclude that there are $s_1 = s_1 (k) \in \N$,
  $c_1 = c_1 (d,k,\ell) \in \N$, and $L_0 = L_0 (d,k,\ell)\in \N$ with
  $L_0 \leq \Felix{(1 + \lceil \log_2 k \rceil -1) (10 + \nicefrac{\ell}{d})}$
  such that for all $\varepsilon \in (0,\nicefrac{1}{2})$ there exist
  two neural networks $\Phi^1_\varepsilon, \Phi^2_\varepsilon$ satisfying
  \[
    | \Realization_{\varrho}(\Phi^1_\varepsilon)(x) - x^{\alpha^{(1)}} |
    \leq \nicefrac{\varepsilon}{6}
    \quad \text{ and } \quad
    | \Realization_{\varrho}(\Phi^2_\varepsilon)(x) - x^{\alpha^{(2)}} |
    \leq \nicefrac{\varepsilon}{6}
    \quad \text{ for all }
          x \in \left[ -\nicefrac{1}{2}, \nicefrac{1}{2}\right]^d \, ,
  \]
  and $\Phi^1_\varepsilon, \Phi^2_\varepsilon$ both have at most $L_0$ layers,
  and at most $c_1 \cdot \varepsilon^{-d/\ell}$ nonzero,
  $(s_1, \nicefrac{\varepsilon}{6})$-quantized weights.
  Note by Remark \ref{rem:QuantisationConversion} that the weights of
  $\Phi^1_\varepsilon$ and $\Phi^2_\varepsilon$ are also
  $(s_2, \varepsilon)$-quantized for a suitable $s_2 = s_2(k) \in \N$.
  Next, by possibly replacing $\Phi_\varepsilon^t$ by
  $\Phi^{\mathrm{Id}}_{1, \lambda_t} \sconc \Phi_\varepsilon^t$
  with $\Phi^{\mathrm{Id}}_{1, \lambda_t}$ as in Remark \ref{rem:DeepIdentity}
  and for $\lambda_t = L_0 - L(\Phi_\varepsilon^t)$,
  we can assume that both $\Phi_\varepsilon^1 , \Phi_\varepsilon^2$ have exactly
  $L_0$ layers.
  Note in view of Remark \ref{rem:SpecialConc} and because of $L_0
  = L_0 (d,k,\ell)$ that this will not change the quantization of the weights,
  and that the number of weights of $\Phi_\varepsilon^t$ is still bounded by
  $c_1 ' \cdot \varepsilon^{-d/\ell}$ for a suitable $c_1 ' = c_1 ' (d,k,\ell)$.
  For simplicity, we will write $c_1$ instead of $c_1 '$ in what follows.

  Now, let $\wt{\times}$ be the network of
  Lemma \ref{lem:MultiplicationWithBoundedLayerNumber}
  with accuracy $\delta := \nicefrac{\varepsilon}{6}$ and with $M = 2$,
  and $\theta = \nicefrac{d}{\ell}$.
  Note \Felix{$(2\theta)^{-1} = \nicefrac{\ell}{2d}$}, so that we can choose
  $L = 1 + \lfloor \nicefrac{\ell}{\Felix{2d}}\rfloor$ in
  Lemma \ref{lem:MultiplicationWithBoundedLayerNumber}.
  Thus, $\wt{\times}$ can be chosen to have at most
  $c_2 \cdot \varepsilon^{-d/\ell}$ nonzero, $(s,\delta)$-quantized weights, and
  \Felix{$8 + 2 \cdot ( 1 + \lfloor \nicefrac{\ell}{2d}\rfloor)$} layers,
  with $s$ as chosen at the start of the proof,
  and for a suitable constant $c_2 = c_2(d,\ell)$.
  Again by Remark \ref{rem:QuantisationConversion} we see that the
  weights of $\wt{\times}$ are also $(s_3, \varepsilon)$-quantized for a
  suitable $s_3 = s_3(k) \in \N$.

  We now define
  \[
    \Phi^\alpha_\varepsilon
    := \wt{\times} \sconc P(\Phi^1_\varepsilon, \Phi^2_\varepsilon).
  \]
  By construction, $\Phi^\alpha_\varepsilon$ has not more than
  \[
    \Felix{
    8 + 2 \cdot \left(
                  1 + \left\lfloor \frac{\ell}{2d}\right\rfloor
                \right)
      + L_0
    \leq 10 + \frac{2\ell}{d}
            + \lceil \log_2 k \rceil \cdot \left(10 + \frac{\ell}{d}\right)
    = (1 + \lceil \log_2 k \rceil) \cdot \left(10 + \frac{\ell}{d}\right)
    }
  \]
  many layers, as desired. Next, we estimate by the triangle inequality
  \begin{align*}
    &| \Realization_{\varrho}(\Phi^\alpha_\varepsilon)(x) - x^\alpha | \\
    &\quad \leq
        \left|
          \Realization_{\varrho}(\wt{\times})
            \big(\Realization_{\varrho}(\Phi^1_\varepsilon)(x),
                 \Realization_{\varrho}(\Phi^2_\varepsilon)(x)\big)
          - \Realization_{\varrho}(\Phi^1_\varepsilon)(x)
            \cdot \Realization_{\varrho}(\Phi^2_\varepsilon)(x)
        \right|
        + \left|
            \Realization_{\varrho}(\Phi^1_\varepsilon)(x)
            \cdot \Realization_{\varrho}(\Phi^2_\varepsilon)(x)
            - x^\alpha
          \right| \\
    &\quad \leq \frac{\varepsilon}{6}
                + \left|
                    \Realization_{\varrho}(\Phi^1_\varepsilon)(x)
                    \cdot \Realization_{\varrho}(\Phi^2_\varepsilon)(x)
                    - \Realization_{\varrho}(\Phi^1_\varepsilon)(x)
                      \cdot x^{\alpha^{(2)}}
                  \right|
                + \left|
                    \Realization_{\varrho}(\Phi^1_\varepsilon)(x)
                    \cdot x^{\alpha^{(2)}}
                    - x^\alpha
                  \right| \\
    &\quad \leq \frac{\varepsilon}{6}
                + \left|
                    \Realization_{\varrho}(\Phi^1_\varepsilon)(x)
                  \right|
                  \cdot \frac{\varepsilon}{6}
                + |x^{\alpha^{(2)}}| \cdot \frac{\varepsilon}{6}
           \leq \varepsilon,
  \end{align*}
  \Felix{where the last three steps are justified since
  $x \in [-\nicefrac{1}{2}, \nicefrac{1}{2}]^d$ and
  $|\Realization_{\varrho}(\Phi^t_\varepsilon)(x)|
  \leq |x^{\alpha^{(t)}}| + \nicefrac{\varepsilon}{6} < 2$ for $t \in \{1,2\}$.}
  Finally, it is easy to see from Remark \ref{rem:SpecialConc} that there exist
  $c_3 = c_3(d,k,\ell) > 0$ and $s_4 = s_4(k) \in \N$ such that
  $\Phi^\alpha_\varepsilon$ has not more than $c_3 \cdot \varepsilon^{-d/\ell}$
  nonzero, $(s_4, \varepsilon)$-quantized weights. This concludes the proof.
\end{proof}

\Felix{Being able} to reproduce monomials, we can \Felix{now} construct networks
that reproduce polynomials up to a given degree.
Moreover, this can be achieved with a fixed and controlled number of layers.
\Felix{In fact, the main point of the following lemma is that for implementing
$m$ different polynomials, one does \emph{not} need $m \cdot W$ weights,
where $W$ denotes the number of weights needed to implement one polynomial.
In contrast, one only needs $\mathcal{O}(m + W)$ weights, which is much smaller.}

\begin{lemma}\label{lem:ComputationalUnit}
Let $d,m \in \N$, let $B, \beta>0$, let
$\{ c_{\ell,\alpha} : \ell \in \{ 1, \dots, m\} , \alpha \in \N_0^d, |\alpha| < \beta\}
\subset [-B,B]$ be a sequence of coefficients, and let
$(x_\ell)_{\ell = 1}^m \subset [-\nicefrac{1}{2}, \nicefrac{1}{2}]^d$
be a sequence of base points.

Then, there exist constants $c = c(d, \beta,B) > 0$, $s = s(d, \beta, B) \in \N$,
and $L = L(d, \beta) \in \N$ with
\Felix{$L \leq 1 + (1 + \lceil \log_2 \beta \rceil) \cdot (11 + \nicefrac{\beta}{d})$}
such that for all $\varepsilon \in (0, \nicefrac{1}{2})$ there is a neural
network $\Phi^{\mathrm{p}}_\varepsilon$ with at most
$c \cdot (\varepsilon^{-d/\beta} + m)$ many nonzero,
$(s, \varepsilon)$-quantized weights, at most $L$ layers,
and with an $m$-dimensional output such that
\begin{align}\label{eq:approxOfPolynomials}
  \left|\vphantom{\sum}
      [\Realization_{\varrho}(\Phi^{\mathrm{p}}_\varepsilon)]_\ell (x)
      - \smash{\sum_{|\alpha| < \beta}}
          c_{\ell, \alpha} \cdot (x-x_\ell)^\alpha
  \right|
  < \varepsilon
  \quad \text{ for all } \quad \ell \in \{ 1, \dots, m \}
        \text{ and } x \in \left[-\nicefrac{1}{2}, \nicefrac{1}{2}\right]^d.
\end{align}
\end{lemma}
\begin{proof}
Write $\beta = n + \sigma$, with $n\in \N_0$ and $\sigma \in (0,1]$, let
$\{ c_{\ell,\alpha} : \ell \in \{ 1, \dots, m \} , \alpha \in \N_0^d, |\alpha| < \beta\}$
and $(x_\ell)_{\ell = 1}^m$ be as in the statement of the lemma,
and let $\varepsilon \in (0, \nicefrac{1}{2})$.
By the $d$-dimensional binomial theorem
(cf.\@ \cite[Chapter 8, Exercise 2]{FollandRA}), we have\vspace{-0.2cm}
\[
     (x- x_\ell)^\alpha
   = \sum_{\gamma \leq \alpha}
        \binom{\alpha}{\gamma}
         (-x_\ell)^{\alpha - \gamma} x^\gamma
   \quad \text{for all } x \in \R^d \text{ and } \alpha \in \N_0^d .
\]

Note \Felix{for $\alpha \in \N_0^d$} that $|\alpha| < \beta$ is equivalent to
$|\alpha| \leq n$.
Thus we have for all $x \in \R^d$ \Felix{and $\ell \in \{1,\dots,m\}$}
that\vspace{-0.2cm}
\[
  \sum_{|\alpha| \Felix{< \beta}}
      c_{\ell, \alpha} (x-x_\ell)^\alpha
  = \sum_{|\alpha| \leq n}
        \left[
          c_{\ell, \alpha}
          \sum_{\gamma \leq \alpha}
             \binom{\alpha}{\gamma}
             x^\gamma (-x_\ell)^{\alpha - \gamma}
        \right]
  = \sum_{|\gamma| \leq n}
        \left[
            \, \, x^\gamma
            \vphantom{\sum_j}
            \smash{
            \underbrace{
                    \sum_{\substack{|\alpha|\leq n \\ \alpha \geq \gamma}}
                        c_{\ell, \alpha}
                        \binom{\alpha}{\gamma}
                        (-x_\ell)^{\alpha - \gamma}
                }_{=: \tilde{c}_{\ell, \gamma}}
                  }
              \,\,
        \right] .
        \vphantom{
            \underbrace{
                    \sum_{\substack{|\alpha|\leq n \\ \alpha \geq \gamma}}
                        c_{\ell, \alpha}
                        \binom{\alpha}{\gamma}
                        (-x_\ell)^{\alpha - \gamma}
                }_{=: \tilde{c}_{\ell, \gamma}}
        }
\]
It is easy to see that there is a constant $C = C(d,\beta, B) \geq 1$ such that
for all $\ell \in \{1, \dots, m\}$ and $\gamma \in \N_0^d$ with $|\gamma| \leq n$,
we have $|\tilde{c}_{\ell, \gamma}| \leq C$.
Furthermore, we just saw that
\begin{equation}
  \sum_{|\alpha| \Felix{< \beta}}
      c_{\ell, \alpha} (x-x_\ell)^\alpha
  = \sum_{|\gamma| \leq n}
       \tilde{c}_{\ell, \gamma} \, x^\gamma
  \quad \text{ for all } x \in \R^d.
  \label{eq:ComputationalUnit1}
\end{equation}

Since $\varepsilon \in (0, \nicefrac{1}{2})$, so that
$\vphantom{\sum_j} \varepsilon^{-s} > 2^s$ for $s \in \N$, there clearly exists
some $s_1 = s_1 (d,\beta, B) \in \N$ (\emph{independent} of $\varepsilon$)
such that there are $\tilde{\tilde{c}}_{\ell, \gamma, \varepsilon}
\in [-\varepsilon^{-s_1}, \varepsilon^{-s_1}]
\cap 2^{-s_1\lceil \log_2(\nicefrac{1}{\varepsilon}) \rceil}\Z$
with $|\tilde{c}_{\ell, \gamma} - \tilde{\tilde{c}}_{\ell, \gamma, \varepsilon}| \leq 1$ for all $\gamma \in \N_0^d$ with
$|\gamma| \leq n$ and all $1 \leq \ell \leq m$, and such that
\begin{equation}
  \Big| \,\,
      \sum_{|\gamma| \leq n}
          \tilde{c}_{\ell, \gamma} \, x^\gamma
    - \sum_{|\gamma| \leq n}
          \tilde{\tilde{c}}_{\ell, \gamma, \varepsilon} \, x^\gamma
  \,\,\Big|
  < \frac{\varepsilon}{2}
  \quad \text{for all } x \in \left[-\nicefrac{1}{2}, \nicefrac{1}{2} \right]^d.
  \label{eq:ComputationalUnit2}
\end{equation}

Write $\{ \gamma \in \N_0^d \,:\, | \gamma | \leq n\} = \{\gamma_1, \dots, \gamma_N\}$ with distinct $\gamma_i$,
for some $N = N(d,n) = N(d,\beta) \in \N$. 
With this choice, we define for $\ell \in \{1,\dots, m\}$ the network
\[
  \Phi^{\ell, \varepsilon} := ((A^{\ell, \varepsilon}, b^\ell))
  \quad \text{where} \quad
  A^{\ell, \varepsilon}
  := ({\tilde{\tilde{c}}}_{\ell, \gamma_1, \varepsilon}, \dots,  {\tilde{\tilde{c}}}_{\ell, \gamma_N, \varepsilon}) \in \R^{1 \times N},
  \quad \text{ and } \quad
  b^\ell := 0 \in \R^1 .
\]

An application of Lemma \ref{lem:MonomialsWithBoundedNumberOfLayers}
(with $\ell = n+1 \in \N$ and with $\lceil \beta \rceil = n+1 \in \N$ instead of $n$)
shows for arbitrary $\delta \in (0,\nicefrac{1}{2})$ and $\gamma \in \N_0^d$
with $|\gamma| \leq n+1$ that there exists a network $\Phi_\delta^\gamma$ with
$d$-dimensional input and one-dimensional output,
at most $c_1 \cdot \delta^{-d/(n+1)}$ nonzero, $(s_2,\delta)$-quantized weights,
and at most $L_1$ layers, such that
\begin{equation}
  |\Realization_{\varrho}(\Phi_\delta^\gamma)(x) - x^\gamma| \leq \delta
  \quad \text{ for all } \quad
  x \in \left[-\nicefrac{1}{2}, \nicefrac{1}{2}\right]^d.
  \label{eq:ComputationalUnit3}
\end{equation}
Here, $c_1 = c_1(d,n) = c_1 (d,\beta) > 0$, $s_2 = s_2 (n) = s_2 (\beta) \in \N$,
and $L_1 = L_1(d, n) = L_1(d, \beta) \in \N$ are constants, and
\Felix{$L_1 \leq (1 + \lceil \log_2 \beta \rceil )
                 \cdot (10 + \nicefrac{(n+1)}{d})
            \leq (1 + \lceil \log_2 \beta \rceil)
                 \cdot (11 + \nicefrac{\beta}{d})$}.
\Felix{To get this bound on $L_1$, we used that
$2^{\lceil \log_2 \beta \rceil} \geq \beta$
and thus $2^{\lceil \log_2 \beta \rceil} \geq \lceil \beta \rceil$,
whence $\lceil \log_2 \beta \rceil \geq \log_2 (\lceil \beta \rceil)$,
which finally implies
$\lceil \log_2 \beta \rceil \geq \lceil \log_2 (\lceil \beta \rceil) \rceil$.}

\Felix{As usual, by possibly replacing the network $\Phi_\delta^\gamma$ by the
network $\Phi_{1,\lambda_\gamma}^{\identity} \sconc \Phi_\delta^\gamma$
with $\Phi_{1,\lambda_\gamma}^{\identity}$ as in Remark \ref{rem:DeepIdentity}
and with $\lambda_\gamma = L_1 - L(\Phi_\delta^\gamma)$, we can assume
that the networks $\Phi_\delta^\gamma$ all have \pp{exactly} $L_1$ layers.
This might require changing the constant $c_1$, but otherwise leaves the
complexity of the networks $\Phi_\delta^\gamma$ unchanged.}

We now choose $\delta := \varepsilon/(4CN)$ and define
\begin{align*}
    \Phi^a_\varepsilon
    := P(\Phi^{1, \varepsilon},
         P(\Phi^{2, \varepsilon},
           \dots,
           P(\Phi^{m-1,\varepsilon},
             \Phi^{m,\varepsilon}
            )
           \dots
          )
        )
    \quad \text{ and } \quad
    \Phi^b_\varepsilon
    := P(\Phi^{\gamma_1}_\delta,
         P(\Phi^{\gamma_2}_\delta,
           \dots,
           P(\Phi_\delta^{\gamma_{N-1}},
             \Phi_\delta^{\gamma_N}
            )
           \dots
          )
        ) .
\end{align*}
\pp{Finally, we set $\Phi^{\mathrm{p}}_\varepsilon := \Phi^a_\varepsilon \sconc \Phi^b_\varepsilon$.} By construction \Felix{and by choice of $\delta = \varepsilon / (4CN)$,
by combining
Equations \eqref{eq:ComputationalUnit1}--\eqref{eq:ComputationalUnit3},
and by using $|\tilde{\tilde{c}}_{\ell,\gamma,\varepsilon}|
\leq 1 + |\widetilde{c}_{\ell,\gamma}| \leq 1 + C \leq 2C$, we see that
Equation} \eqref{eq:approxOfPolynomials} holds.
Moreover, the weights were chosen quantized
(see also Remark \ref{rem:QuantisationConversion} and note
$\delta \geq \varepsilon / C_2$ for a constant $C_2 = C_2 (d,\beta, B) > 0$),
and the number of weights of $\Phi^a_\varepsilon$ satisfies
$M(\Phi^a_\varepsilon) \leq mN$, while the number of weights of
$\Phi^b_\varepsilon$---up to a multiplicative constant depending on
$n=n(\beta)$, $d$ and $B$---is bounded by
$\varepsilon^{-d/(n+1)} \leq \varepsilon^{-d/\beta}$.
Therefore, Remark \ref{rem:SpecialConc} shows that
$\Phi_\varepsilon^{\mathrm{p}}$ has the required number of properly
quantized weights.

Additionally, since $\Phi^a_\varepsilon$ has one layer and $\Phi^b_\varepsilon$
has at most
\Felix{$L_1 \leq (1 + \lceil \log_2 \beta \rceil )
                 \cdot (11 + \nicefrac{\beta}{d})$}
layers, we conclude that $\Phi^{\mathrm{p}}_\varepsilon$ has at most
\Felix{$1 + (1 + \lceil \log_2 \beta \rceil) \cdot (11 + \nicefrac{\beta}{d})$}
layers.
This completes the proof.
\end{proof}

As the next step of our construction we show that one can construct a network
that \Felix{approximates a ``cutoff'' of a given network to an interval}.
\Felix{
\pp{We start by collecting} two estimates concerning the $L^p$ (quasi)-norms,
which we will use \pp{frequently}.
First, since the set $[-\nicefrac{1}{2}, \nicefrac{1}{2}]^d$ with the
Lebesgue measure is a probability space, Jensen's inequality
(see \cite[Theorem 10.2.6]{DudleyRealAnalysisProbability}) shows
\begin{equation}
  \|f\|_{L^p ([-\nicefrac{1}{2}, \nicefrac{1}{2}]^d)}
  \leq \|f\|_{L^q ([-\nicefrac{1}{2}, \nicefrac{1}{2}]^d)}
  \quad \text{for} \quad
  0 < p \leq q < \infty
  \quad \! \text{and} \quad \!
  f : [-\nicefrac{1}{2},\nicefrac{1}{2}]^d \to \R \text{ measurable} \, .
  \label{eq:JensenEstimate}
\end{equation}
\pp{Second}, if $p \in (0,1)$ then the (quasi)-norm $\|\mybullet\|_{L^p}$
does \emph{not} satisfy the triangle inequality.
However, as shown for example in \cite[Example 2.2.6]{Megginson}, we have
$\|f+g\|_{L^p}^p \leq \|f\|_{L^p}^p + \|g\|_{L^p}^p$.
Combining this with the elementary estimate
$\sum_{i=1}^N a_i^p \leq N \cdot \max\{a_i \,:\, i = 1,\dots,N\}^p$, we see
\begin{equation}
  \Big\|
    \sum_{i=1}^N f_i \,
  \Big\|_{L^p}
  \leq N^{\max\{1,p^{-1}\}} \cdot \max\{ \|f_i\|_{L^p} \,:\, i = 1,\dots,N \},
  \label{eq:PseudoTriangleInequality}
\end{equation}
which remains valid also in case of $p \geq 1$.
With these preparations, we can prove \pp{the previously announced} ``cutoff'' result.
}

\begin{lemma}\label{lem:MultiplicationWithACharacteristicFunction}
Let $d \in \N$, \Felix{$p \in (0,\infty)$}, and $B \geq 1$.
Let $-\nicefrac{1}{2}\leq a_{i} \leq b_i \leq \nicefrac{1}{2}$
for $i = 1, \dots, d$,
and let $\varepsilon \in (0, \nicefrac{1}{2})$ be arbitrary.
Then there exist constants $c = c(d) \in \N$, $s = s(d, B, \Felix{p}) \in \N$,
and a neural network $\Lambda_\varepsilon$ with \pp{a $d+1$-dimensional input},
at most four layers, and at most $c$ nonzero, $(s, \varepsilon)$-quantized
weights such that for each neural network $\Phi$ with one-dimensional output
layer and $d$-dimensional input layer, and with
$\|\Realization_\varrho(\Phi)\|_{L^\infty([-\nicefrac{1}{2},\nicefrac{1}{2}]^d)}
\leq B$, we have
\[
  \left\|
    \Realization_{\varrho}(\Lambda_\varepsilon)
      \big(\bullet, \Realization_{\varrho}({\Phi})(\bullet)\big)
    - \chi_{\prod_{i = 1}^d [a_i, b_i]} \cdot \Realization_{\varrho}(\Phi)
  \right\|_{L^p ([-\nicefrac{1}{2}, \nicefrac{1}{2}]^d)}
  \leq \varepsilon.
\]
\end{lemma}
\begin{proof}
  \Felix{In order to obtain a network with quantized weights, we first
  construct modified interval boundaries $\widetilde{a}_i,\widehat{b}_i$.
  To this end, let $p_0 := \lceil p \rceil \in \N$, and set
  $s_1 := s_1 (d,B,p) := 8d + p_0 (1 + 4 \lceil B \rceil) \in \N$, and
  $\widetilde{\varepsilon} := 2^{-s_1 \lceil \log_2(1/\varepsilon) \rceil}$.
  Then, on the one hand, 
  $\widetilde{\varepsilon}^{-1} = 2^{s_1 \lceil \log_2 (1/\varepsilon) \rceil}
  \leq 2^{s_1 (1+ \log_2(1/\varepsilon))} \leq 2^{2s_1 \log_2(1/\varepsilon)}
  = \varepsilon^{-2 s_1}$.
  On the other hand, since $\lceil \log_2 (1/\varepsilon) \rceil \geq 1$ and
  $2^{2x} = 4^x \geq e^x \geq 1+x \geq x$ for $x > 0$, we see
  \begin{equation}
    \widetilde{\varepsilon}
    \leq 2^{-2 \cdot 4d}
         \cdot \left(
                 2^{-2
                    \cdot 2 \lceil B \rceil
                    \cdot \lceil \log_2 (1/\varepsilon) \rceil}
                 \cdot 2^{-\lceil \log_2 (1/\varepsilon) \rceil}
               \right)^{p_0}
    \leq \frac{1}{4d} \cdot \left(
                              2^{-2 \cdot 2 \lceil B \rceil} \varepsilon
                            \right)^{p_0}
    \leq \frac{(\varepsilon / (2B))^{p_0}}{4d}
    \leq 1 \, .
    \label{eq:CutoffEpsilonEstimate}
  \end{equation}
  Finally, for each $i \in \{1,\dots,d\}$ we can choose
  $\widetilde{a_i}, \widetilde{b_i} \in [-\nicefrac{1}{2}, \nicefrac{1}{2}]
  \cap 2^{-s_1 \lceil \log_2 (1/\varepsilon) \rceil} \Z$ with
  $|a_i - \widetilde{a_i}| \leq \widetilde{\varepsilon}$
  and $|b_i - \widetilde{b_i}| \leq \widetilde{\varepsilon}$.

  Now, note that $\widetilde{\varepsilon}^{-1},
  \widetilde{\varepsilon}^{-1} \cdot \widetilde{a_i},
  \widetilde{\varepsilon}^{-1} \cdot \widetilde{b_i}$
  are all elements of $\Z \cap [-\varepsilon^{-2 s_1},\varepsilon^{2 s_1}]
  \subset \Z \cap [-\varepsilon^{-3 s_1}, \varepsilon^{-3 s_1}]$
  and likewise that
  $1 + \widetilde{\varepsilon}^{-1} \widetilde{a_i},
  1 + \widetilde{\varepsilon}^{-1} \widetilde{b_i}$
  are all elements of
  $\Z \cap [-\varepsilon^{-(1 + 2 s_1)}, \varepsilon^{-(1 + 2 s_1)}]
  \subset \Z \cap [-\varepsilon^{-3 s_1}, \varepsilon^{-3 s_1}]$.
  Therefore, the function
  %
  \[
    t_i : \left[-\frac{1}{2}, \frac{1}{2}\right] \to \R,
    x \mapsto
    \varrho\left(
             \frac{x-\widetilde{a}_i}{\widetilde{\varepsilon}}
           \right)
    - \varrho\left(
               \frac{x - \widetilde{a}_i - \widetilde{\varepsilon}}
                    {\widetilde{\varepsilon}}
             \right)
    - \varrho\left(
               \frac{x - \widetilde{b}_i + \widetilde{\varepsilon}}
                    {\widetilde{\varepsilon}}\right)
    + \varrho\left(
                \frac{x - \widetilde{b}_i}{\widetilde{\varepsilon}}
             \right)
  \]
  is the realization of a two-layer network with at most $12$ nonzero,
  $(3s_1, \varepsilon)$-quantized weights.

  A simple computation yields that if
  $\widetilde{b}_i - \widetilde{a}_i > 2\widetilde{\varepsilon}$, then
  \[
  t_i(x)
  = \begin{cases}
      0 ,
      & \text{for } x \in \R \setminus [\widetilde{a}_i, \widetilde{b}_i], \\
      \frac{x - \widetilde{a}_i}{\widetilde{\varepsilon}},
      & \text{for } x \in [\widetilde{a}_i,
                           \widetilde{a}_i + \widetilde{\varepsilon}], \\
      1 ,
      & \text{for } x \in [\widetilde{a}_i + \widetilde{\varepsilon},
                           \widetilde{b}_i - \widetilde{\varepsilon}], \\
      1 - \frac{x - (\widetilde{b}_i - \widetilde{\varepsilon})}
               {\widetilde{\varepsilon}}
      & \text{for } x \in [\widetilde{b}_i - \widetilde{\varepsilon},
                           \widetilde{b}_i] \, .
    \end{cases}
  \]
  }

  We continue defining the function $n_\varepsilon: \R^{d} \times \R \to \R$
  which will be the realization of $\Lambda_\varepsilon$.
  First, we set $B_0 := 2^{\lceil \log_2 B \rceil}$.
  \pp{If $\widetilde{b}_i - \widetilde{a}_i\geq 2\Felix{\widetilde{\varepsilon}}$}
  holds for all $i = 1, \dots, d$ then we define
  \[
    n_\varepsilon(x,y)
    := B_0 \cdot \varrho\left(\sum_{i = 1}^d t_i(x_i)
       + \varrho \left(\frac{y}{ B_0}\right) - d\right)
       - B_0 \cdot \varrho\left(\sum_{i=1}^d t_i (x_i)
       + \varrho\left(- \frac{y}{B_0}\right) - d\right).
  \]
  \pp{If
  $\widetilde{b}_i - \widetilde{a}_i < 2 \Felix{\widetilde{\varepsilon}}$} for
  some $i \in \{1,\dots, d\}$, we set $n_\varepsilon \equiv 0$.
  In both cases, it is easy to see that $n_\varepsilon$ is the realization of a
  four layer neural network $\Lambda_\varepsilon$ with at most $c = c(d)$
  nonzero, $(s_{\Felix{2}} , \varepsilon)$-quantized weights, for some
  $s_{\Felix{2}} = s_{\Felix{2}} (d,B,p) \in \N$.
  Further, in both cases, for all $y \in [-B,B] \subset [-B_0, B_0]$,
  the following hold:
  If $x \in \prod_{i = 1}^d [\widetilde{a}_i+\widetilde{\varepsilon},
                             \widetilde{b}_i-\widetilde{\varepsilon}]$,
  then $n_\varepsilon(\Felix{x,} y) = y$; and if
  $x \in \R^d \setminus \prod_{i = 1}^d [\widetilde{a}_i,\widetilde{ b}_i]$,
  then $n_\varepsilon(\Felix{x,} y) = 0$.
  Moreover, $\prod_{i = 1}^d [\widetilde{a}_i, \widetilde{b}_i]
  \setminus \prod_{i = 1}^d [\widetilde{a}_i+\widetilde{\varepsilon},
                             \widetilde{b}_i-\widetilde{\varepsilon}]$
  has Lebesgue measure bounded by \Felix{$2 d \, \widetilde{\varepsilon}
  \leq (\varepsilon / (2B))^{p_0}\pp{/2}$,
  see Equation \eqref{eq:CutoffEpsilonEstimate}.}
  Finally, since the ReLU $\varrho$ is $1$-Lipschitz, we have
  $|n_\varepsilon (x,y)| \leq B_0 \cdot |\varrho(y/B_0) - \varrho(-y/B_0)|
  = |y| \leq B$ for arbitrary $y \in [-B,B]$.
  Therefore, for any measurable
  $f : [-\nicefrac{1}{2}, \nicefrac{1}{2}]^d \to [-B, B]$, we have
  \begin{align*}
    \|
      n_\varepsilon(\bullet, f(\bullet))
      - \chi_{\prod_{i = 1}^d [a_i, b_i]} \cdot f
    \|_{L^{\Felix{p_0}}}
    \leq & \|
             n_\varepsilon(\bullet, f(\bullet))
             - \chi_{\prod_{i = 1}^d [\widetilde{a}_i, \widetilde{b}_i]} \cdot f
           \|_{L^{\Felix{p_0}}} \\
    &\qquad
         + \|
             \chi_{\prod_{i = 1}^d [\widetilde{a}_i, \widetilde{b}_i]} \cdot f
             - \chi_{\prod_{i = 1}^d [a_i, b_i]}  \cdot f
           \|_{L^{\Felix{p_0}}}.
  \end{align*}
  By the previous considerations, and since $|f| \leq B$, we can estimate
  \[
    \|
      n_\varepsilon(\bullet, f(\bullet))
      - \chi_{\prod_{i = 1}^d [\widetilde{a}_i, \widetilde{b}_i]} \cdot f
    \|_{L^{\Felix{p_0}}}
    \Felix{\leq 2B \cdot 2 d \widetilde{\varepsilon}}
    \leq \Felix{\left( \Big(\frac{\varepsilon}{2B}\Big)^{p_0} \right)^{1/p_0}}
         \cdot B
    \leq \frac{\varepsilon}{2}.
  \]
  Since $2d \widetilde{\varepsilon} \leq \Felix{(\varepsilon / (2B))^{p_0}}$,
  and since $a_i, \widetilde{a_i}, b_i , \widetilde{b_i}
  \in [-\nicefrac{1}{2}, \nicefrac{1}{2}]$ with
  $|\widetilde{a}_i - a_i| \leq \widetilde{\varepsilon}$
  and $|\widetilde{b}_i - b_i| \leq \widetilde{\varepsilon}$ for all
  $i \in \{1, \dots, d\}$, we also have
  \[
    \|
      \chi_{\prod_{i = 1}^d [\widetilde{a}_i, \widetilde{b}_i]} \cdot f
      - \chi_{\prod_{i = 1}^d [a_i, b_i]}  \cdot f
    \|_{L^{\Felix{p_0}}}
    \leq \Felix{\left( \Big( \frac{\varepsilon}{2B}\Big)^{p_0} \right)^{1/p_0}} \cdot B
    \leq \frac{\varepsilon}{2}.
  \]
  In combination, these estimates imply the result 
  \Felix{for the $L^{p_0}$ norm instead of the $L^p$ (quasi)-norm.
  In view of Equation \eqref{eq:JensenEstimate}, this implies the claim.}
\end{proof}

For technical reasons we require the following refinement of
Lemma \ref{lem:MultiplicationWithACharacteristicFunction}.

\begin{lemma}\label{lem:MultiplicationWithACharacteristicFunctionArray}
Let $d,m,s\in \N$, \Felix{$p \in (0,\infty)$},
and $\varepsilon\in (0,\nicefrac{1}{2})$,
and let $\Phi$ be a neural network with $d$-dimensional input and
$m$-dimensional output, and with $(s, \varepsilon)$-quantized weights.
Furthermore, let $B \geq 1$ with
$\|
   [\Realization_\varrho(\Phi)]_\ell
 \|_{L^\infty ([-\nicefrac{1}{2}, \nicefrac{1}{2}]^d)} \leq B$
for all $\ell = 1, \dots, m$.
Finally, let
$-\nicefrac{1}{2}\leq a_{i, \ell} \leq b_{i, \ell} \leq \nicefrac{1}{2}$
for $i = 1, \dots, d$ and $\ell = 1, \dots, m$.

Then, there exist constants $c = c(d) > 0$, $s_0 = s_0(d, B, \Felix{p}) \in \N$,
and a neural network $\Psi_\varepsilon$ with $d$-dimensional input layer and
$1$-dimensional output layer, with at most $6 + L(\Phi)$ layers,
and at most $c \cdot (m + L(\Phi) + M(\Phi))$ nonzero,
$(\max\{s, s_0\}, \nicefrac{\varepsilon}{m})$-quantized weights,
such that
\[
  \left\|
      \Realization_{\varrho}(\Psi_\varepsilon)
      - \sum_{\ell = 1}^m
          \chi_{\prod_{i = 1}^d [a_{i,\ell}, b_{i,\ell}]}
          \cdot [\Realization_{\varrho}(\Phi)]_\ell
  \right\|_{L^{\Felix{p}} ([-\nicefrac{1}{2}, \nicefrac{1}{2}]^d)}
  \leq \varepsilon.
\]
\end{lemma}

\begin{proof}
  \Felix{First}, let $L := L(\Phi)$, and set
  $\widetilde{\Phi} := P(\Phi_{d,L}^{\mathrm{Id}}, \Phi)$, where
  $\Phi_{d,L}^{\mathrm{Id}}$ is as in Remark \ref{rem:DeepIdentity}, so that
  $\Phi_{d,L}^{\mathrm{Id}}$ has $L = L(\Phi)$ layers and at most
  $2d \cdot L(\Phi)$ nonzero, $(1, \varepsilon)$-quantized weights, and
  satisfies $\Realization_\varrho(\Phi_{d,L}^{\mathrm{Id}}) = \identity_{\R^d}$.
  We conclude that $\widetilde{\Phi}$ has $L = L(\Phi)$ layers, and at most
  $M(\Phi) + 2dL(\Phi)$ nonzero, $(s, \varepsilon)$-quantized weights.

  \Felix{Second, set $p_0 := \max \{1 , p\}$,}
  and for each $\ell \in \{1, \dots, m\}$ let $\Lambda^\ell_\varepsilon$ be the
  neural network provided by
  Lemma \ref{lem:MultiplicationWithACharacteristicFunction} applied with
  $a_i = a_{i,\ell}$, $b_i = b_{i, \ell}$ and with
  \Felix{$p_0$ instead of $p$ and}
  $\nicefrac{\varepsilon}{m}$ instead of $\varepsilon$.
  There exist $c_0 = c_0(d) \in \N$ and $s_0 = s_0(d, B, p) \in \N$ such that
  $\Lambda^\ell_\varepsilon$ has four layers and at most $c_0$ nonzero,
  $(s_0, \nicefrac{\varepsilon}{m})$-quantized weights.

  \Felix{Third}, let $P_\ell \in \R^{(d+1) \times (d+m)}$ be the matrix
  associated (via the standard basis) to the linear map
  $\R^d \times \R^m \ni (x, y) \mapsto (x, y_{\ell}) \in \R^d \times \R^{1}$,
  and let $\Phi_\ell := ((P_\ell, 0))$ be the associated $1$-layer network.
  Clearly, $\Phi_\ell$ has $d+1$ nonzero, $(1,\varepsilon)$-quantized weights.

  \Felix{Fourth}, define
  $\Phi^{\mathrm{sum}} := ((A^{\mathrm{sum}}, b^{\mathrm{sum}}))$ where
  \[
    A^{\mathrm{sum}} := (1, 1, \dots, 1) \in \R^{1\times m},
    \quad \text{ and } \quad
    b^{\mathrm{sum}} := 0.
  \]
  $\Phi^{\mathrm{sum}}$ has exactly $m$ nonzero, $(1, \varepsilon)$-quantized
  weights and one layer.

  \Felix{With all these preparations, we can finally define}
  \[
    \Psi_\varepsilon
    := \Phi^{\mathrm{sum}}
        \sconc P(\Lambda^1_{\varepsilon} \sconc \Phi_1,
                 P(\Lambda^2_\varepsilon \sconc \Phi_2,
                   \dots,
                   P(\Lambda^{m-1}_{\varepsilon} \sconc \Phi_{m-1},
                     \Lambda^m_{\varepsilon} \sconc \Phi_{m})
                   \dots
                  )
                )
        \sconc \widetilde{\Phi}.
  \]
  By Remark \ref{rem:SpecialConc} we see that $\Psi_{\varepsilon}$ has
  $1 + 5 + L(\Phi)$ layers and at most
  \[
    \Felix{
    16 \cdot \max \{ m \,,\, m \cdot 2(d+1+c_0) \,,\, M(\Phi) + 2d \, L(\Phi) \}
    }
    \leq c \cdot (m+L(\Phi) + M(\Phi))
  \]
  nonzero, $(\max\{s_0, s\}, \nicefrac{\varepsilon}{m})$\Felix{-quantized}
  weights for a constant $c = c(d)>0$.

  We observe that
  \[
    \Realization_\varrho (\Psi_\varepsilon)(x)
    = \sum_{\ell=1}^m
         \Realization_\varrho (\Lambda^{\ell}_{\varepsilon})
           \left(x, [\Realization_\varrho (\Phi)(x)]_\ell\right)
    \quad \text{ for } x \in \R^d \, .
  \]
  Thus, by the triangle inequality, which is valid since $p_0 \geq 1$, we see
  \begin{align*}
    &\left\|
        \Realization_\varrho (\Psi_\varepsilon)
        - \sum_{\ell \leq m}
            \chi_{\prod_{i = 1}^d [a_{i,\ell}, b_{i,\ell}]}
            \cdot [\Realization_{\varrho}(\Phi)]_\ell
     \right\|_{L^{\Felix{p_0}}([-\nicefrac{1}{2}, \nicefrac{1}{2}]^d)} \\
    &\qquad \leq \sum_{\ell \leq m}
                    \left\|
                        \Realization_\varrho (\Lambda_{\varepsilon}^{\ell})
                          \left(
                            \bullet,
                            [\Realization_\varrho (\Phi)(\bullet)]_\ell
                          \right)
                        - \chi_{\prod_{i = 1}^d [a_{i,\ell}, b_{i,\ell}]}
                          \cdot [\Realization_{\varrho}(\Phi)]_\ell
                    \right\|_{L^{\Felix{p_0}} ([-\nicefrac{1}{2}, \nicefrac{1}{2}]^d)}
    \overset{(\ast)}{\leq} \sum_{\ell \leq m}
                             \frac{\varepsilon}{m}
    = \varepsilon,
  \end{align*}
  where the step marked with $(\ast)$ holds by choice of the neural networks
  $\Lambda^\ell_\varepsilon$,
  see Lemma \ref{lem:MultiplicationWithACharacteristicFunction}.
  \Felix{Finally, we apply Equation \eqref{eq:JensenEstimate} to get the desired
  estimate for the $L^p$ (quasi)-norm instead of the $L^{p_0}$ norm.}
\end{proof}

Our next larger goal is to show that neural networks can well approximate
smooth functions with respect to the $L^{\Felix{p}}$ norm, in such a way that
the number of layers does \emph{not} grow with the approximation accuracy,
only with the smoothness of the function.
A central ingredient for the proof is the local approximation of smooth
functions via their Taylor polynomials.
Precisely, we need the following result, which is probably folklore:

\begin{lemma}\label{lem:HoelderTaylor}
  Let $\beta \in (0,\infty)$, and write $\beta = n + \sigma$ with
  $n \in \N_0$ and $\sigma \in (0,1]$, and let $d \in \N$.
  Then there is a constant $C = C(\beta,d) > 0$ with the following property:

  For each $f \in \mathcal{F}_{\beta, d, B}$ and arbitrary
  $x_0 \in (-\nicefrac{1}{2}, \nicefrac{1}{2})^d$, there is a polynomial
  $p(x) = \sum_{|\alpha| \leq n} c_\alpha (x-x_0)^\alpha$ with
  $c_\alpha \in [- C \cdot B, C \cdot B]$ for all
  $\alpha \in \N_0^d$ with $|\alpha| \leq n$ and such that
  \[
    |f(x) - p(x)| \leq C \cdot B \cdot |x - x_0|^\beta
    \qquad \text{ for all } x \in \left[-\nicefrac{1}{2}, \nicefrac{1}{2}\right]^d.
  \]
  In fact, $p = p_{f, x_0}$ is the Taylor polynomial of $f$ of degree $n$.
\end{lemma}
\begin{proof}
  In case of $n = 0$, \Felix{so that $\beta = \sigma$,} the claim is trivial
  \Felix{for $C = 1$}:
  If we set $p(x) := f(x_0)$, then
  $|f(x_0)| \leq \| f \|_{C^{0,\beta}} \leq B$, and
  \[
      |f(x) - p(x)|
      = |f(x) - f(x_0)|
      \leq \| f \|_{C^{0,\beta}} \cdot |x - x_0|^\sigma
      \leq B \cdot |x - x_0|^\beta,
  \]
  as desired. Therefore, we can from now on assume $n \in \N$.

  In the following, we use a slightly different multi-index notation
  \Felix{than in the rest of the paper}, to be compatible with the notation in
  \cite{LeeSmoothManifolds}: We write $\underline{d} := \{1,\dots,d\}$, and for
  $I = (i_1,\dots, i_m) \in \underline{d}^m$ with $m \in \N$, we write
  $\partial_I f := \partial_{i_1} \cdots \partial_{i_m} f$ and
  $y^I = y^{i_1} \cdots y^{i_m}$ for $y \in \R^d $.
  Using this notation, the Taylor polynomial of $f$ of degree $n-1$ at $x_0$ is
  given by
  \[
    p_0(x) := f(x_0)
            + \sum_{m=1}^{n-1}
                \frac{1}{m!} \sum_{I \in \underline{d}^m}
                                (\partial_I f)(x_0) \cdot (x-x_0)^I .
  \]
  Taylor's theorem with integral remainder
  (see \cite[Theorem C.15]{LeeSmoothManifolds}) shows
  for $x \in (-\nicefrac{1}{2}, \nicefrac{1}{2})^d$ that
  \begin{align*}
    &f(x) - p_0 (x) \\
    &= \frac{1}{(n-1)!} \cdot
      \sum_{I \in \underline{d}^n}
          (x - x_0)^I \!
          \int_0^1
            (1-t)^{n-1} \partial_I f (x_0 + t(x-x_0))
          dt \\
    &= \frac{1}{(n-1)!}
       \left(
              \sum_{I \in \underline{d}^n}
                (x - x_0)^I
                \int_0^1
                  (1\!-\!t)^{n-1} \partial_I f (x_0)
                dt\right.\\
     &\qquad \qquad \qquad
       \left. + \! \sum_{I \in \underline{d}^n}
                     (x - x_0)^I \!
                     \int_0^1
                       (1\!-\!t)^{n-1}
                       [ \partial_I f (x_0 \! + \! t(x-x_0))
                         - \partial_I f (x_0) ]
                     dt
       \right)\\
    &= \frac{1}{n!} \cdot
      \sum_{I \in \underline{d}^n}
          \partial_I f (x_0) \cdot (x - x_0)^I
    + \frac{1}{(n-1)!} 
          \sum_{I \in \underline{d}^n}
              (x - x_0)^I
              \int_0^1
                (1-t)^{n-1}
                [ \partial_I f (x_0 + t(x-x_0)) - \partial_I f (x_0) ]
              dt \\
    &=: q(x) + R (x).
  \end{align*}
  But $p := p_0 + q$ is the Taylor polynomial of $f$ of degree $n$ at $x_0$, and
  $p(x) = \sum_{|\alpha|\leq n} c_\alpha (x-x_0)^\alpha$ for certain
  $c_\alpha \in \R$, which are easily seen to satisfy
  \[
    |c_\alpha|
    \leq \sum_{I \in \underline{d}^{|\alpha|}
               \text{ with } \alpha = \Felix{e_{i_1} + \dots + e_{i_{|\alpha|}}}}
            |\partial_I f(x_0)|
    \leq d^{|\alpha|} \cdot B \leq d^n \cdot B,
  \]
  \Felix{where $(e_1, \dots, e_d)$ denotes the standard basis of $\R^d$.}

  Finally, since $\partial_I f$ is $\sigma$ Hölder continuous with
  $\Lip_\sigma (\partial_I f) \leq \| f \|_{C^{0,\beta}} \leq B$
  for each $I \in \underline{d}^n$, we get
  \begin{align*}
      |f(x) - p(x)| = |R (x)|
      &\leq \frac{1}{(n-1)!} \cdot
           \sum_{I \in \underline{d}^n}
              |(x - x_0)^I| \cdot
              \int_0^1
                (1-t)^{n-1} \cdot B \cdot |t (x - x_0)|^\sigma
              dt \\
      & \leq \frac{d^n}{n!} |x-x_0|^n \cdot B \cdot |x - x_0|^\sigma
      \leq C \cdot B \cdot |x-x_0|^\beta \, ,
  \end{align*}
  for $C = C(d,n) = C(d,\beta) := d^n$.
  By continuity, this estimate holds for all
  $x \in [-\nicefrac{1}{2}, \nicefrac{1}{2}]^d$,
  not just for $x \in (-\nicefrac{1}{2}, \nicefrac{1}{2})^d$.
\end{proof}

Now, we can finally prove our main result about the
$L^{\Felix{p}}$-approximation of smooth functions using ReLU networks.

\begin{theorem}\label{thm:ApproxOfSmoothFctnAppendix}
  For $d\in \N$, and $\beta, B\Felix{, p} > 0$, there are constants
  $L = L(\beta,d) \in \N$ and $c= c(d,\beta, B) >0$,
  $s = s(d,\beta,B\Felix{,p})\in \N$ with
  \Felix{$L \leq 11 + \big(1 + \lceil \log_2 \beta \rceil \big)
                      \cdot \big(11 + \nicefrac{\beta}{d}\big)$}
  and such that for any function $f \in \mathcal{F}_{\beta,d, B}$ and any
  $\varepsilon \in (0, \nicefrac{1}{2})$, there is a neural network
  $\Phi^f_\varepsilon$ with at most $L$ layers, and at most
  $c \cdot \varepsilon^{-d/\beta}$ nonzero,
  $(s,\varepsilon)$-quantized weights such that 
  \[
    \|
      \Realization_\varrho(\Phi^f_\varepsilon) - f
    \|_{L^p ([-\nicefrac{1}{2}, \nicefrac{1}{2}]^d)}
    \leq \varepsilon
    \quad \text{ and } \quad
    \|
      \Realization_\varrho(\Phi^f_\varepsilon)
    \|_{\sup}
    \leq \lceil B \rceil.
  \]
\end{theorem}

\begin{proof}
As in the proof of Lemma \ref{lem:MultiplicationWithACharacteristicFunction},
\Felix{setting $p_0 := \lceil p \rceil \in \N$, it suffices to consider
approximation in $L^{p_0}$ instead of $L^p$, thanks to
Equation \eqref{eq:JensenEstimate}.}
Let $\beta = n + \sigma$ with $n\in \N_0$ and $\sigma \in (0,1]$.
Further, let $C = C(d,\beta) > 0$ denote the constant from
Lemma \ref{lem:HoelderTaylor}, and define
\[
  N := \left\lceil
         \left(\frac{\varepsilon}{4 C B d^\beta} \right)^{-\frac{1}{\beta}}
       \right\rceil \in \N .
\]
Finally, for $\lambda \in \{1, \dots, N\}^d$, set
\[
  I_\lambda
  := \prod_{i = 1}^d
      \left[
        \frac{\lambda_i - 1}{N} -\frac{1}{2}, \frac{\lambda_i}{N} -\frac{1}{2}
      \right].
\]
As a result, we have (with disjointness up to null-sets) that
\begin{equation}
  \left[-\frac{1}{2}, \frac{1}{2}\right]^d
  = \overset{\bullet}{\bigcup_{\lambda\in \{1, \dots, N\}^d}} I_\lambda
  \quad \text{and} \quad
  I_\lambda \subset \overline{B}_{1/n}^{\|\cdot\|_{\ell^\infty}}(x)
            \subset \overline{B}_{d/N}^{|\cdot|}(x)
  \quad \text{for all} \quad x \in I_\lambda.
  \label{eq:TheSizeOfILambda}
\end{equation}
\Felix{Let us write $\{1,\dots,N\}^d = \{\lambda_1, \dots, \lambda_{N^d}\}$,
and for each $i \in \{1,\dots,N^d\}$ choose}
a point $x_i$ in the interior of $I_{\lambda_i}$,  and set
$c_{i, \alpha} := \partial^\alpha f (x_i)/\alpha!$ for
$\alpha \in \N_0^d$ with $|\alpha| \leq n$.
Note $|c_{i, \alpha}| \leq B$.

In view of Lemma \ref{lem:HoelderTaylor} and
Equation \eqref{eq:TheSizeOfILambda}, we see
\begin{equation}
  \sup_{\substack{i \in \{1, \dots, N^d\} \\ x \in I_{\lambda_i}}}
      \left|
          f(x) - p_{i,\alpha}(x)
      \right|
  \leq C B \left(\frac{d}{N}\right)^\beta
  \quad \text{with} \quad
  p_{i,\alpha}(x)
  := \sum_{|\alpha| \leq n}
       \frac{\partial^\alpha f (x_i)}{\alpha!} (x-x_i)^\alpha \, .
  \label{eq:LocalPolynomialApproximation}
\end{equation}
\Felix{In particular, this implies
\begin{equation}
  |p_{i,\alpha}(x)|
  \leq |f(x)| + C d^\beta B
  \leq \left\lceil (1 + C d^\beta) \cdot B \right\rceil
  =: B_1
  \quad \text{for all} \quad
  x \in I_{\lambda_i} \, .
  \label{eq:SmoothFunctionApproximationB1Definition}
\end{equation}
}
Next, with the base points $(x_i)_{i=1,\dots,N^d}$ and the coefficients
$(c_{i,\alpha})$ from above, take $\Phi^{\mathrm{p}}_{\varepsilon/4}$ as in
Lemma \ref{lem:ComputationalUnit} with accuracy $\nicefrac{\varepsilon}{4}$
instead of $\varepsilon$, and with $m = N^d$.
By Lemma \ref{lem:ComputationalUnit}, the network
$\Phi_{\varepsilon/4}^{\mathrm{p}}$ has at most $L_1 = L_1(d, \beta)$ layers,
with
\Felix{$L_1 \leq 1 + (1 + \lceil \log_2 \beta \rceil ) \cdot (11 + \nicefrac{\beta}{d})$}
and at most  $c_1 (\varepsilon^{-d/\beta} + N^d)$ nonzero,
$(s_1,\varepsilon)$-quantized weights
(see also Remark \ref{rem:QuantisationConversion}),
for certain $s_1 = s_1(d,\beta, B) \in \N$ and $c_1 = c_1(d, \beta,B) > 0$.

\Felix{Now, Lemma \ref{lem:boundedApprox} (applied with $B_1$ instead of $B$)
yields a network $\Psi_{\eps/4}^{\mathrm{p}}$ with
\[
  \Realization_\varrho (\Psi_{\eps/4}^{\mathrm{p}})
   = (\tau_{B_1} \times \cdots \times \tau_{B_1})
     \circ \Realization_\varrho (\Phi_{\eps/4}^{\mathrm{p}}),
\]
where $\tau_{B_1} : \R \to [-B_1,B_1]$ is $1$-Lipschitz and satisfies
$\tau_{B_1} (x) = x$ for all $x \in [-B_1,B_1]$.
Furthermore, Lemma \ref{lem:boundedApprox} shows that
$\Psi_{\eps/4}^{\mathrm{p}}$ has at most
$2 c_1 \cdot (\varepsilon^{-d/\beta} + N^d) + c_2 \cdot N^d
 \leq c_3 \cdot (\varepsilon^{-d/\beta} + N^d)$ nonzero,
$(s_2, \varepsilon)$-quantized weights for an absolute constant $c_2 > 0$ and
suitable $s_2 = s_2 (d,\beta,B) \in \N$ and $c_3 = c_3 (d,\beta,B) > 0$.
Finally, Lemma \ref{lem:boundedApprox} also yields
\[
  L_2
  := L(\Psi_{\eps/4}^{\mathrm{p}})
  \leq 2 + L_1
  \leq 3
       + (1 + \lceil \log_2 \beta \rceil )
         \cdot \left(11 + \frac{\beta}{d}\right) \, .
\]
By combining Equation \eqref{eq:SmoothFunctionApproximationB1Definition}
with the properties of the function $\tau_{B_1}$ and with the properties of
the networks $\Phi_{\eps/4}^{\mathrm{p}}$ stated in
Lemma \ref{lem:ComputationalUnit}, we see
\[
  \left|
    [\Realization_\varrho (\Psi_{\eps/4}^{\mathrm{p}})(x)]_i - p_{i,\alpha}(x)
  \right|
  = \left|
      \tau_{B_1}(
                  \Realization_\varrho (\Phi_{\eps/4}^{\mathrm{p}}) (x)
                )
      - \tau_{B_1} (p_{i,\alpha} (x))
    \right|
  \leq \left|
         \Realization_\varrho (\Phi_{\eps/4}^{\mathrm{p}}) (x)
         - p_{i,\alpha} (x)
       \right|
  \leq \frac{\eps}{4}
\]
for all $x \in I_{\lambda_i}$ and $i \in \{1,\dots,N^d\}$.
Therefore, recalling Equation \eqref{eq:LocalPolynomialApproximation} and our
choice of $N$ from the beginning of the proof, we get}
\begin{align*}
  \left\|
    f
    - \sum_{\Felix{i \in \{1, \dots, N^d\}}}
        \chi_{I_{\Felix{\lambda_i}}}
        [\Realization_{\varrho}(\Psi^{\mathrm{p}}_{\varepsilon/4})]_{\Felix{i}}
  \right\|_{L^\infty}
  & \leq \sup_{\Felix{\substack{i \in \{1, \dots, N^d\} \\ x \in I_{\lambda_i}}}}
           |
            f(x)
            - [\Realization_{\varrho}(\Psi^{\mathrm{p}}_{\eps/4})]_i (x)
           | \\
  & \Felix{
    \leq \frac{\varepsilon}{4}
         + \sup_{\substack{i \in \{1, \dots, N^d\} \\ x \in I_{\lambda_i}}}
             |f(x) - p_{i,\alpha} (x)|
    \leq \frac{\eps}{4} + C B \left(\frac{d}{N}\right)^\beta
    \leq \frac{\eps}{2} \, .}
\end{align*}
By the triangle inequality, \Felix{and since
$\|\mybullet\|_{L^{p_0} ([-\nicefrac{1}{2},\nicefrac{1}{2}]^d)}
 \leq \|\mybullet\|_{L^{\infty} ([-\nicefrac{1}{2},\nicefrac{1}{2}]^d)}$,
we see that we are done---at least if we ignore the bound
$\|\Realization_\varrho (\Phi_\eps^f)\|_{\sup} \leq \lceil B \rceil$ for
the moment---if we can find a network}
$\Psi_\varepsilon$ with properly quantized weights, at most
\Felix{$L \leq 9 + \big( 1 + \lceil \log_2 \beta \rceil \big)
                   \cdot \big( 11 + \nicefrac{\beta}{d}\big)$}
layers,
and at most $c_6 \cdot \varepsilon^{-d/\beta}$ nonzero weights satisfying
$\|\Realization_{\varrho}(\Psi_\varepsilon)
   - \sum_{\Felix{i \in \{1, \dots, N^d\}}}
       \chi_{I_{\Felix{\lambda_i}}}
       [\Realization_{\varrho}(\Psi^{\mathrm{p}}_{\eps/4})]_{\Felix{i}}
 \|_{L^{\Felix{p_0}}}
\leq \nicefrac{\varepsilon}{2}$.
We will construct such a network using
Lemma \ref{lem:MultiplicationWithACharacteristicFunctionArray}.

Indeed, if we apply that lemma, with $\nicefrac{\varepsilon}{2}$ instead of
$\varepsilon$ \Felix{and $p_0$ instead of $p$}, with
$\Phi = \Felix{\Psi}_{\varepsilon/4}^{\mathrm{p}}$ and $m = N^d$,
and with the intervals $I_{\Felix{\lambda_i}}$,
\Felix{$i \in \{1, \dots, N^d\}$}, then we get
a neural network $\Psi_\varepsilon$ which satisfies the desired estimate.

Furthermore, $\Psi_{\varepsilon}$ has at most
\Felix{$6 + L(\Psi_{\varepsilon/4}^{\mathrm{p}})
\leq 9 + \big(1 + \lceil \log_2 \beta \rceil \big)
         \cdot \big(11 + \nicefrac{\beta}{d}\big)$}
layers. Moreover, $\Psi_\varepsilon$ has at most
\[
  \Felix{c_4} \cdot \big(
                       N^d
                       + L(\Felix{\Psi}_{\varepsilon/4}^{\mathrm{p}})
                       + M(\Felix{\Psi}_{\varepsilon/4}^{\mathrm{p}})
                    \big)
  \leq \Felix{c_4} \cdot \big(
                              N^d + \Felix{L_2}
                              + M(\Felix{\Psi}_{\varepsilon/4}^{\mathrm{p}})
                         \big)
  \leq \Felix{c_5} \cdot \big(N^d + \Felix{c_3} (\varepsilon^{-d/\beta} + N^d) \big)
\]
nonzero, $(\max\{\Felix{s_2}, s_0\}, \varepsilon/(2N^d))$-quantized weights,
with constants \Felix{$c_4 = c_4(d) > 0$, $c_5 = c_5 (d,\beta) > 0$}, and
$s_0 = s_0 (d\Felix{,p},B) \in \N$.
By choice of $N$, this shows that $\Psi_{\varepsilon}$ has the correct number
of nonzero weights.

Finally, we have $\varepsilon/(2N^d) \geq c_7 \cdot \varepsilon^{1+d/\beta}$,
for $c_7 = c_7( d, \beta, B)$ so that Remark \ref{rem:QuantisationConversion}
shows that the weights of $\Psi_{\varepsilon}$ are quantized as stated in the
theorem.

\medskip{}

\Felix{We yet have to modify the network $\Psi_\eps$ to obtain a network
$\Phi_\eps^f$ with
$\|\Realization_\varrho(\Phi_\eps^f)\|_{\sup} \leq \lceil B \rceil$.
To this end, we apply}
Lemma \ref{lem:boundedApprox} to $\Psi_\eps$.
This yields a network $\Phi^f_\eps$ with the required number of layers and of
(correctly quantized) nonzero weights, which satisfies
$\Realization_\varrho (\Phi_\eps^f)
= \tau_B \circ \Realization_\varrho (\Psi_\eps)$,
\Felix{with $\tau_B$ as in Lemma \ref{lem:boundedApprox}}.
By the properties of $\tau_B$, and since $\|f\|_{\sup}\leq B$,
so that $f = \tau_B \circ f$, we obtain
\[
  \|
    \Realization_\varrho(\Phi^f_\eps)
    - f
  \|_{L^{\Felix{p_0}}([-\nicefrac{1}{2}, \nicefrac{1}{2}]^d)}
  \leq \|
         \Realization_\varrho(\Psi_\eps)
         - f
       \|_{L^{\Felix{p_0}}([-\nicefrac{1}{2}, \nicefrac{1}{2}]^d)}
  \leq \eps.
\qedhere
\]
\end{proof}

\subsection{Approximation of horizon functions}

We proceed to construct networks that yield good approximations of horizon
functions.
The underlying idea is relatively straightforward:
We have already seen in Lemma \ref{lem:ApproxOfHeaviside} that networks yield
approximate realizations of Heaviside functions.
Since a horizon function is simply a smoothly transformed Heaviside function,
we only need to realize this smooth transformation with a network.
This is possible using Theorem \ref{thm:ApproxOfSmoothFctnAppendix}.
The following lemma makes these arguments rigorous.

\begin{lemma}\label{lem:HorFunctAppendix}
  For $\Felix{p},\beta,B > 0$ and $d \in \N_{\geq 2}$ there are constants
  $L = L(d, \beta) \in \N$, $c= c(d, \beta, \Felix{B, p}) > 0$, and
  $s = s(d, \beta, B \Felix{,p}) \in \N$
  with
  \Felix{$L \leq 14 + \big( 1 + \lceil \log_2 \beta \rceil \big)
                      \big( 11 + \nicefrac{2\beta}{d} \big)$}
  and such that for every function $f\in \mathcal{HF}_{\beta,d, B}$ and every
  $\varepsilon \in (0, \nicefrac{1}{2})$ there is a neural network
  $\Phi^f_\varepsilon$ with at most $L$ layers and at most
  $c \cdot \varepsilon^{-\Felix{p}(d-1)/\beta}$ nonzero,
  $(s, \varepsilon)$-quantized weights, satisfying
  \[
     \|
       \Realization_\varrho(\Phi^f_\varepsilon) - f
     \|_{L^{\Felix{p}}([-\nicefrac{1}{2}, \nicefrac{1}{2}]^d)}
     < \varepsilon.
  \]
  Moreover, $0 \leq \Realization_\varrho(\Phi^f_\varepsilon)(x) \leq 1$
  for all $x \in \Felix{\R^d}$.
\end{lemma}
\begin{proof}
Since multiplying $A_1$ in the definition of a neural network
$\Phi = ((A_1, b_1),\dots, (A_L, b_L))$ by a permutation matrix does not change
the number of layers or weights, or the possible values of the nonzero weights,
we can certainly restrict ourselves to horizon functions
$f \in \mathcal{HF}_{\beta,d, B}$ for which the permutation matrix
$T$ from Definition \ref{def:HorizonFunctions} is the identity matrix.
Choose $\gamma \in \mathcal{F}_{\beta,d-1,B}$ such that
$f = H \circ \widetilde{\gamma}$, where $H = \chi_{[0,\infty) \times \R^{d-1}}$
is the Heaviside function, and where
\[
  \widetilde{\gamma}(x) = (x_1 + \gamma(x_2, \dots, x_d), x_2, \dots, x_d ),
  \quad \text{ for } \quad
  x = (x_1, \dots, x_d) \in \left[-\nicefrac{1}{2}, \nicefrac{1}{2} \right]^d.
\]
Theorem \ref{thm:ApproxOfSmoothFctnAppendix}
(applied with $p=1$, with $d-1$ instead of $d$ and with
\Felix{$\frac{1}{2} \cdot (\varepsilon/4)^p$} instead of $\varepsilon$)
yields a network $\Phi^\gamma_\varepsilon$ with at most
\Felix{
\[
  L = L(d,\beta)
  \leq 11 + \big(1 + \lceil \log_2 \beta \rceil \big)
            \cdot \Big(11 + \frac{\beta}{d-1}\Big)
  \leq 11 + \big(1 + \lceil \log_2 \beta \rceil \big)
            \cdot \Big(11 + \frac{2\beta}{d}\Big)
\]}%
layers, and at most \Felix{$c \cdot \varepsilon^{-{p(d-1)}/\beta}$} nonzero
weights (where $c = c(d,\beta,B\Felix{,p}) \in \N$) such that
$\gamma_\varepsilon := \Realization_\varrho(\Phi^\gamma_\varepsilon)$
approximates $\gamma$ with an $L^1$-error of less than
\Felix{$\frac{1}{2} \cdot (\eps/4)^p$}.
We also recall (by invoking Remark \ref{rem:QuantisationConversion})
that it is possible to construct this network with
$(s, \varepsilon)$-quantized weights, for some
$s = s(d,\beta, B\Felix{,p}) \in \N$.

Clearly, one can construct a network $\Phi^{\tilde{\gamma}}_\varepsilon$
of the same complexity (number of nonzero weights and quantization)
up to multiplicative constants that depend only on $d$, which satisfies
\[
  \Realization_{\varrho} (\Phi^{\tilde{\gamma}}_\varepsilon) (x)
  = (x_1 + \gamma_\varepsilon (x_2, \dots, x_d), x_2, \dots, x_d)
  \quad \text{for all} \quad
  x \in \R^d \, ,
\]
\Felix{and furthermore
$L(\Phi_\eps^{\tilde{\gamma}}) \leq 1 + L(\Phi_\eps^{\gamma})$}.

As a second step, \Felix{we choose $\eps' \in 2^{-\N}$ with
$\frac{1}{4} \cdot (\varepsilon/4)^p \leq \eps'
\leq \frac{1}{2} \cdot (\varepsilon/4)^p$, and}
invoke Lemma \ref{lem:ApproxOfHeaviside}
(with $\varepsilon'$ instead of $\varepsilon$)
to obtain a neural network $\Phi_{\varepsilon'}^H$
with two layers and five weights such that
$|H(x) - \Realization_\varrho(\Phi_{\varepsilon'}^H)(x)|
\leq \chi_{\Felix{[0,\varepsilon'] \times \R^{d-1}}}(x)$
and $0 \leq \Realization_\varrho (\Phi_{\varepsilon'}^H) (x) \leq 1$
for all $x\in \R^d$.
Furthermore, Lemma \ref{lem:ApproxOfHeaviside} shows that all weights of
$\Phi_{\varepsilon'}^H$ are elements of
\Felix{$[-4 (\varepsilon/4)^{-p}, 4 \cdot (\varepsilon/4)^{-p}] \cap \Z
 \subset [-\varepsilon^{-s'}, \varepsilon^{-s'}] \cap \Z$ for
$s' := 2 + 3 \lceil p \rceil$.
Here, we used that $\varepsilon \leq \nicefrac{1}{2}$, so that
$(\varepsilon')^{-1} \leq 4 \cdot (\varepsilon/4)^{-p}
\leq 4 \cdot \varepsilon^{-3p} \leq \varepsilon^{-(2+3p)}$.}

Remark \ref{rem:SpecialConc} shows that there is
a constant $\widetilde{c} = \widetilde{c}(d, \beta, B) \in \N$
such that $\Phi_{\varepsilon'}^H \sconc \Phi_\varepsilon^{\tilde{\gamma}}$
is a neural network with at most
\Felix{
$\widetilde{L} \leq 2 + L(\Phi_\varepsilon^{\tilde{\gamma}})
               \leq 14 + \big( 1 + \lceil \log_2 \beta \rceil \big)
                         \cdot (11 + \nicefrac{2\beta}{d})$}
layers, and not more than
$\widetilde{c} \cdot \varepsilon^{-{\Felix{p} (d-1)}/\beta}$
nonzero $(\Felix{s''}, \varepsilon)$-quantized weights for a suitable
\Felix{$s'' = s'' (d,\beta,B,p) \in \N$.}
Furthermore, we have
$0 \leq \Realization_{\varrho}
        (\Phi_{\varepsilon'}^H \sconc \Phi_\varepsilon^{\tilde{\gamma}})
   \leq 1$,
since $0 \leq \Realization_{\varrho} (\Phi_{\varepsilon'}^H)\leq 1$.
Thus, to complete the proof, it remains to show that
$\Realization_\varrho(\Phi_{\varepsilon'}^H
\sconc \Phi_\varepsilon^{\tilde{\gamma}})$
indeed approximates $f = H \circ \widetilde{\gamma}$ with
an $L^{\Felix{p}}$-error of at most $\varepsilon$.

To this end, we use Equation \eqref{eq:PseudoTriangleInequality} to deduce
because of $\max\{1,p^{-1}\} \leq 1 + p^{-1} =: q$ that
\begin{align*}
  \|
    H \circ \widetilde{\gamma}
    - \Realization_\varrho
      (\Phi_{\varepsilon'}^H \sconc \Phi_\varepsilon^{\tilde{\gamma}})
  \|_{L^{\Felix{p}}}
  =    &~  \|
             H \circ \widetilde{\gamma}
             - \Realization_\varrho (\Phi_{\varepsilon'}^H)
               \circ \Realization_\varrho(\Phi_\varepsilon^{\tilde{\gamma}})
           \|_{L^{\Felix{p}}} \\
  \leq &~ \Felix{2^q
                 \cdot \max \big\{
                               \|
                                 H \circ \tilde{\gamma}
                                 - H \circ \Realization_\varrho
                                             (\Phi_\varepsilon^{\tilde{\gamma}})
                               \|_{L^p} \, ,
                               \|
                                 H \circ \Realization_\varrho
                                           (\Phi_\varepsilon^{\tilde{\gamma}})
                                 - \Realization_\varrho
                                     (\Phi_{\varepsilon'}^H)
                                   \circ \Realization_\varrho
                                           (\Phi_\varepsilon^{\tilde{\gamma}})
                               \|_{L^{p}}
                            \big\} } \\
  \pp{=}&: ~ \Felix{2^q \cdot \max \{ \mathrm{I} \,,\, \mathrm{II} \}.}
\end{align*}
First, we estimate term $\mathrm{I}$.
For this, we use the shorthand notation $\chi_{\widetilde{\gamma}_1 >0}$ for the
indicator function of the set
$\{x\in [-\nicefrac{1}{2},\nicefrac{1}{2}]^d: \widetilde{\gamma}_1(x) > 0\}$
and variations thereof.
Moreover, we denote by
$\Realization_\varrho(\Phi_\varepsilon^{\widetilde{\gamma}})_1$ the first
coordinate of the $\R^d$-valued function
$\Realization_\varrho(\Phi_\varepsilon^{\widetilde{\gamma}})$.
Recall that with our choice of $\Phi_\varepsilon^{\widetilde{\gamma}}$, we have
that $\Realization_{\varrho} (\Phi_\varepsilon^{\widetilde{\gamma}})_1(x)
= x_1 + \gamma_\varepsilon (x_2,\dots, x_d)$
for all $x \in [-\nicefrac{1}{2}, \nicefrac{1}{2}]^d$.
Having set the notation, we estimate
\begin{align*}
  & \Felix{(2^q \cdot \mathrm{I})^p}
  = \Felix{2^{1+p}}
    \cdot \|
            H \circ \widetilde{\gamma}
            - H \circ \Realization_\varrho
                        (\Phi_\varepsilon^{\widetilde{\gamma}})
          \|_{L^{\Felix{p}}}^{\Felix{p}} \\
  \quad
  & = \Felix{2^{1+p}}
      \int_{[-\nicefrac{1}{2},\nicefrac{1}{2}]^d}
        |
         \chi_{\widetilde{\gamma}_1 \geq 0}(x)
         - \chi_{\Realization_\varrho
                   (\Phi_\varepsilon^{\widetilde{\gamma}})_1 \geq 0}
             (x)
        |^{\Felix{p}}
      dx \\
  \quad
  &= \Felix{2^{1+p}}
     \int_{[-\nicefrac{1}{2},\nicefrac{1}{2}]^{d-1}}
       \int_{-\nicefrac{1}{2}}^{\nicefrac{1}{2}}
           \chi_{\widetilde{\gamma}_1 \geq 0,
                 \Realization_\varrho
                   (\Phi_\varepsilon^{\widetilde{\gamma}})_1 < 0}
             (x_1,\dots,x_d)
         + \chi_{\widetilde{\gamma}_1 < 0,
                 \Realization_\varrho
                   (\Phi_\varepsilon^{\widetilde{\gamma}})_1 \geq 0}
             (x_1,\dots,x_d)
       d x_1 \,
     d(x_2, \dots, x_d).
\end{align*}
Now, we observe for fixed
$(x_2,\dots, x_d) \in [-\nicefrac{1}{2}, \nicefrac{1}{2}]^{d-1}$
the following equivalence:
\begin{align*}
    \chi_{\widetilde{\gamma}_1 \geq 0,
          \Realization_\varrho(\Phi_\varepsilon^{\widetilde{\gamma}})_1 < 0}(x)
    = 1
  & \Longleftrightarrow
    x_1 + \gamma (x_2,\dots,x_d) \geq 0
    \quad \text{ and } \quad
    x_1 + \gamma_\varepsilon (x_2,\dots,x_d) < 0 \\
  & \Longleftrightarrow
    x_1 \in [- \gamma (x_2,\dots, x_d), - \gamma_\varepsilon (x_2,\dots,x_d) ).
\end{align*}
This implies
\[
  \int_{-\nicefrac{1}{2}}^{\nicefrac{1}{2}}
      \chi_{\widetilde{\gamma}_1 \geq 0,
            \Realization_\varrho(\Phi_\varepsilon^{\widetilde{\gamma}})_1 < 0}
        (x_1,\dots,x_d)
  d x_1
  \leq \max \{0, \gamma (x_2,\dots, x_d) - \gamma_\varepsilon (x_2,\dots,x_d)\}.
\]
By the same reasoning,
$\int_{-\nicefrac{1}{2}}^{\nicefrac{1}{2}}
   \chi_{\widetilde{\gamma}_1 < 0,
         \Realization_\varrho(\Phi_\varepsilon^{\widetilde{\gamma}})_1 \geq 0}
     (x_1,\dots,x_d)
 d x_1
 \leq \max \{0, \gamma_\varepsilon (x_2,\dots,x_d) - \gamma (x_2,\dots,x_d)\}$.
In total, we get because of $\max\{0,y\} + \max\{0, -y\} = |y|$
that
\begin{align*}
    (\Felix{2^q} \cdot \mathrm{I})^{\Felix{p}}
  &= \Felix{2^{1+p}}
     \int_{[-\nicefrac{1}{2},\nicefrac{1}{2}]^{d-1}}
       \int_{-\nicefrac{1}{2}}^{\nicefrac{1}{2}}
            \chi_{\widetilde{\gamma}_1 \geq 0,
                  \Realization_\varrho
                    (\Phi_\varepsilon^{\widetilde{\gamma}})_1 < 0}
              (\Felix{x_1,y})
          + \chi_{\widetilde{\gamma}_1 < 0,
                  \Realization_\varrho
                    (\Phi_\varepsilon^{\widetilde{\gamma}})_1 \geq 0}
              (\Felix{x_1,y})
       \, d x_1 \,
     \, d \Felix{y} \\ 
  &\!\leq
       \Felix{2^{1+p}}
       \int_{[-\nicefrac{1}{2},\nicefrac{1}{2}]^{d-1}}
         \max \{0, \gamma(\Felix{y}) - \gamma_\varepsilon(\Felix{y})\}
         +
         \max \{0, \gamma_\varepsilon(\Felix{y}) - \gamma(\Felix{y})\}
     \, d \Felix{y} \\ 
  &= \Felix{2^{1+p}} \,
     \|
       \gamma - \gamma_\varepsilon
     \|_{L^1 ([-\nicefrac{1}{2}, \nicefrac{1}{2}]^{d-1})}
  \leq \Felix{(\varepsilon/2)^p} \, ,
\end{align*}
and hence \Felix{$2^{q} \cdot \mathrm{I} \leq \nicefrac{\varepsilon}{2}$}.


To estimate the term $\mathrm{II}$, we recall that
$|H(x) - \Realization_\varrho(\Phi_{\varepsilon'}^H)(x)|
\leq \chi_{\Felix{[0,\varepsilon'] \times \R^{d-1}}}(x)
\leq \chi_{\Felix{[0,2^{-1} (\varepsilon/4)^p] \times \R^{d-1}}} (x)$
for all $x\in \R^d$.
Therefore,
\begin{align*}
  \Felix{(2^q \cdot \mathrm{II})^p}
  =   &  \Felix{2^{1+p} \,}
         \|
           H \circ \Realization_\varrho(\Phi_\varepsilon^{\tilde{\gamma}})
           - \Realization_\varrho(\Phi_{\varepsilon'}^H)
             \circ \Realization_\varrho( \Phi_\varepsilon^{\tilde{\gamma}})
         \|_{L^{\Felix{p}}}^\Felix{p}
  \leq  \Felix{2^{1+p}} \ %
        \int_{[-\nicefrac{1}{2},\nicefrac{1}{2}]^d}
           \chi_{0 \leq \Realization_\varrho
                           (\Phi_\varepsilon^{\tilde{\gamma}})_1
                   \leq \Felix{\frac{1}{2} \cdot (\varepsilon/4)^p}}(x)
        \, dx\\
  = & \Felix{2^{1+p}}
      \int_{[-\nicefrac{1}{2},\nicefrac{1}{2}]^{d-1}}
         \int_{-\nicefrac{1}{2}}^{\nicefrac{1}{2}}
            \chi_{0 \leq x_1 + \gamma_\varepsilon (x_2,\dots, x_d)
                    \leq \Felix{\frac{1}{2} \cdot (\varepsilon/4)^p}}
         \, dx_1
      \, d (x_2,\dots, x_d) \\
  \leq & \Felix{2^{p}}
       \int_{[-\nicefrac{1}{2}, \nicefrac{1}{2}]^{d-1}}
         \Felix{(\varepsilon/4)^p}
       \, d (x_2,\dots, x_d)
  = \Felix{(\varepsilon/2)^p} \, ,
\end{align*}
\Felix{and hence $2^q \cdot \mathrm{II} \leq \varepsilon/2$.}
In conclusion, we obtain
\[
  \|
    H \circ \tilde{\gamma}
    - \Realization_\varrho
      (
        \Phi_{\varepsilon'}^H \sconc \Phi_\varepsilon^{\tilde{\gamma}}
      )
  \|_{L^{\Felix{p}}}
  \leq 2^q \cdot \max \big\{ \mathrm{I} \,,\, \mathrm{II} \big\}
  \leq \frac{\varepsilon}{2} 
  < \varepsilon.
  \qedhere
\]
\end{proof}

\subsection{Approximation of piecewise constant
\texorpdfstring{\Felix{and piecewise smooth}}{and piecewise smooth} functions}
\label{sub:AppendixPiecewiseSmoothApproximation}

Since for $K \in \mathcal{K}_{r, \beta, d, B}$ the indicator function $\chi_K$
is locally a horizon function, we can use Lemma \ref{lem:HorFunctAppendix}
to construct neural networks \Felix{that approximate these indicator functions}.

\begin{theorem}\label{thm:ApproximationOfPiecewiseConstantFunctionsAppendix}
Let $r \in \N$, $d\in \N_{\geq 2}$ and $\Felix{p}, \beta, B > 0$ be arbitrary.
There are constants $L=L(\beta,d) \in \N$,
$c=c(d,\Felix{p}, \beta,r,B)>0$, and $s= s(d,\Felix{p}, \beta,r,B) \in \N$ with
\Felix{$L \leq 22 + \big(1 + \lceil \log_2 \beta \rceil \big)
                    \cdot (11 + \nicefrac{2\beta}{d})$}
and such that for all $\varepsilon \in (0, \nicefrac{1}{2})$ and arbitrary
$K \in \mathcal{K}_{r,\beta,d,B}$ there exists a neural network
$\Phi^K_\varepsilon$ with at most $L$ layers and at most
$c \cdot \varepsilon^{-\Felix{p}(d-1)/\beta}$ nonzero,
$(s,\varepsilon)$-quantized weights such that
\[
  \|
    \Realization_{\varrho}(\Phi^K_\varepsilon) - \chi_K
  \|_{L^{\Felix{p}} ([-\nicefrac{1}{2}, \nicefrac{1}{2}]^d)}
  \leq \varepsilon
  \quad \text{and} \quad
  \|\Realization_\varrho(\Phi^K_\varepsilon)\|_{\sup} \leq 1.
\]
\end{theorem}

\begin{proof}
For $\lambda = (\lambda_1, \dots, \lambda_d) \in \{1,\dots, 2^r\}^d$, define
\[
  I_\lambda
  := \prod_{i=1}^d
       \left[
         (\lambda_i - 1) \cdot 2^{-r}-\frac{1}{2},
         \lambda_i \cdot 2^{-r} - \frac{1}{2}
       \right].
\]
We have by construction (with disjointness up to null sets) that
\[
  \left[-\nicefrac{1}{2}, \nicefrac{1}{2}\right]^d
  = \overset{\bullet}{\bigcup_{\lambda \in \{1, \dots, 2^r\}^d}}
       I_\lambda \, ,
  \quad \text{ and } \quad
  I_\lambda \subset \overline{B}_{2^{-r}}^{\| \cdot \|_{\ell^\infty}}(x)
  \quad \text{for all} \quad x \in I_\lambda.
\]
As a consequence of the definition of $\mathcal{K}_{r,\beta,d,B}$,
there is for each $\lambda \in \{1,\dots, 2^r\}^d$ a horizon function
$f_\lambda \in \mathcal{HF}_{\beta,d, B}$ such that
$\chi_{I_{\lambda}} \chi_K = \chi_{I_\lambda} f_\lambda$.


\Felix{For brevity, let us set $q := \max\{1,p^{-1}\}$.}
Now, for each $\lambda \in \{1, \dots, 2^r\}^d$,
Lemma \ref{lem:HorFunctAppendix} yields a neural network
$\Phi^{\lambda}_\varepsilon$ such that
\[
  \|
    \Realization_{\varrho}(\Phi^{\lambda}_\varepsilon)
    - f_\lambda
  \|_{L^{\Felix{p}}}
  \leq \frac{\varepsilon}{2^{1 \Felix{+ q} + r d \Felix{q}}}
  \quad \text{and such that} \quad
  0 \leq \Realization_\varrho (\Phi_\varepsilon^{\lambda}) (x) \leq 1
  \text{ for all }
  \Felix{x \in \R^d} \, .
\]
By Lemma \ref{lem:HorFunctAppendix} and Remark \ref{rem:QuantisationConversion}
there exists $c_1 = c_1(d, \beta, B, r \Felix{,p}) > 0$,
$s_1=s_1(d,\beta,B, r \Felix{,p}) \in \N$, and $L_1 = L_1(d,\beta) \in \N$
with
\Felix{$L_1 \leq 14 + \big(1 + \lceil \log_2 \beta \rceil \big)
                      \cdot (11 + \nicefrac{2\beta}{d})$},
such that $\Phi^{\lambda}_\varepsilon$ has at most $L_1$ layers and
at most $c_1 \cdot \varepsilon^{-\Felix{p} (d-1)/\beta}$ nonzero,
$(s_1, \varepsilon)$-quantized weights.

Next, by possibly replacing $\Phi^{\lambda}_\varepsilon$ by
$\Phi^{\mathrm{Id}}_{1, L_\lambda} \sconc \Phi_\varepsilon^{\lambda}$
with $\Phi^{\mathrm{Id}}_{1, L_\lambda}$ as in Remark \ref{rem:DeepIdentity}
and for $L_\lambda = L_1 - L(\Phi_\varepsilon^{\lambda})$, we can assume that
each network $\Phi_\varepsilon^{\lambda}$ has exactly $L_1$ layers.
Note in view of Remark \ref{rem:SpecialConc} and because of
$L_1 = L_1 (d,\beta)$ that this will not change the quantization of the weights,
and that the number of weights of $\Phi_\varepsilon^{\lambda}$ is still bounded
by $c_1 ' \cdot \varepsilon^{-\Felix{p} (d-1)/\beta}$ for a suitable
constant $c_1 ' = c_1 ' (d,\beta,B,r \Felix{,p})$.
For simplicity, we will write $c_1$ instead of $c_1 '$ in what follows.

Now, write $\{1, \dots, 2^r\}^d = \{\lambda_1, \dots, \lambda_{2^{rd}} \}$,
and set
\[
  \Phi := P(\Phi_\varepsilon^{\lambda_1},
            P(\Phi_\varepsilon^{\lambda_2},
              \dots,
                P(\Phi_\varepsilon^{\lambda_{2^{rd} - 1}},
                  \Phi_{\varepsilon}^{\lambda_{2^{rd}}})
              \dots)
           ) \, .
\]
Note that $\Phi$ has $L_1$ layers, and at most
$2^{rd} \cdot c_1 \cdot \varepsilon^{-\Felix{p} (d-1)/\beta}
\leq c_2 \cdot \varepsilon^{-\Felix{p} (d-1)/\beta}$ nonzero,
$(s_1, \varepsilon)$-quantized weights, for a suitable constant
$c_2 = c_2 (d,\beta, B, r \Felix{,p} ) > 0$.

Finally, an application of
Lemma \ref{lem:MultiplicationWithACharacteristicFunctionArray}
with $m = 2^{rd}$ and $B=1$, with $\nicefrac{\varepsilon}{2^{1+q}}$ instead of
$\varepsilon$, and with the intervals $I_{\lambda_\ell}$,
$\ell \in \{1,\dots, 2^{rd}\}$ yields a network $\Psi$
\Felix{which satisfies---thanks to Equation
\eqref{eq:PseudoTriangleInequality}---the following estimate:}
\begin{align*}
  \|
    \Realization_\varrho (\Psi) - \chi_K
  \|_{L^{\Felix{p}}}
  &\leq \Felix{2^q} \,
        \left\|
          \Realization_\varrho (\Psi)
          - \sum_{\ell =1,\dots, 2^{rd}}
               \chi_{I_{\lambda_\ell}} [\Realization_\varrho(\Phi)]_{\ell}
       \right\|_{L^{\Felix{p}}}
       + \Felix{2^q} \,
         \left\|
            \sum_{\ell = 1,\dots, 2^{rd}}
                \chi_{I_{\lambda_\ell}}
                \cdot ([\Realization_\varrho(\Phi)]_{\ell} - f_{\lambda_\ell})
          \right\|_{L^{\Felix{p}}} \\
  &\leq \frac{\varepsilon}{2}
        \Felix{
        + 2^q \, 2^{rdq} \cdot
          \max \{ \|
                    \Realization_\varrho (\Phi_\varepsilon^{\lambda_\ell})
                    - f_{\lambda_\ell}
                  \|_{L^p}
                  \,:\,
                  \ell = 1, \dots, 2^{rd}
               \}
        }
  \leq \varepsilon \, .
\end{align*}
Here, we used that
$\chi_K = \sum_{\ell = 1}^{2^{rd}} \chi_{I_{\lambda_\ell}} \chi_K
=\sum_{\ell = 1}^{2^{rd}} \chi_{I_{\lambda_\ell}} f_{\lambda_\ell}$,
with equality almost everywhere, and that $[\Realization_\varrho(\Phi)]_{\ell}
= \Realization_\varrho (\Phi_\varepsilon^{\lambda_\ell})$, by construction
of $\Phi$.

To complete the proof, it remains to verify that $\Psi$ has the required
complexity, \Felix{and to modify $\Psi$ slightly in order to ensure
$\|\Realization_\varrho (\Psi)\|_{\sup} \leq 1$.}
But Lemma \ref{lem:MultiplicationWithACharacteristicFunctionArray} shows that
$\Psi$ has at most $6 + L(\Phi) = 6 + L_1$ layers.
The same lemma also shows that the weights of $\Psi$
are $(\max\{s_0, s_1\}, \varepsilon/2^{1 + \Felix{q}})$-quantized
for a constant $s_0 = s_0 (d \Felix{,p}) \in \N$, so that
Remark \ref{rem:QuantisationConversion} shows that $\Psi$ has
$(s_2, \varepsilon)$-quantized weights, for a suitable constant
$s_2 = s_2 (d,\beta, B, r \Felix{,p}) \in \N$.
Finally, Lemma \ref{lem:MultiplicationWithACharacteristicFunctionArray}
also shows
\[
  M(\Psi) \leq c \cdot (2^{rd} + L_1 + M(\Phi))
          \leq c_3 \cdot \varepsilon^{-\Felix{p} (d-1)/\beta},
\]
for suitable constants $c = c(d) > 0$ and
$c_3 = c_3 (d,\beta,B,r \Felix{,p}) > 0$.

Finally, an application of Lemma \ref{lem:boundedApprox} to $\Psi$
as at the end of the proof of Theorem \ref{thm:ApproxOfSmoothFctnAppendix}
yields the network $\Phi_\varepsilon^K$ satisfying all desired properties.
\end{proof}

Theorem \ref{thm:ApproximationOfPiecewiseConstantFunctionsAppendix}
yields an approximation result by neural networks for functions that are
piecewise constant.
However, a simple extension allows us to also approximate piecewise smooth
functions.
\begin{corollary}\label{cor:PiecewiseSmoothFunctionsAppendix}
  Let $r \in \N$, $d\in \N_{\geq 2}$, and $B, \beta \Felix{,p} > 0$.
  \Felix{Define
  \[
    \beta' := \frac{d \beta}{p(d-1)} \, ,
    \qquad 
    \beta_0 := \max\{\beta,\beta'\},
    \quad \text{and} \quad
    \mathcal{E}_{r,\beta,d,B}^p
    := \{
         \chi_K \cdot g
         \,:\,
         g \in \mathcal{F}_{\beta',d,B}
         \text{ and } K \in \mathcal{K}_{r,\beta,d,B}
       \} \, .
  \]}
  Then there exist constants $c = c(d,\beta,r, \Felix{p}, B) > 0$,
  $s = s(d,\beta,r \Felix{,p} ,B) \in \N$, and $L = L(d, \beta) \in \N$ with
  \Felix{$L \leq 34 + \big(1 + \lceil \log_2 \beta_0 \rceil \big)
                      \cdot \big( 11 + \nicefrac{3 \beta_0}{d} \big)$}
  such that for all $\varepsilon \in (0, \nicefrac{1}{2})$ and all
  $f \in \mathcal{E}_{r, \beta, d, B}^{\Felix{p}}$ there exists
  a neural network $\Phi^f_\varepsilon$ with at most $L$ layers,
  and at most $c \cdot \varepsilon^{-\Felix{p} (d-1)/\beta}$ nonzero,
  $(s, \varepsilon)$-quantized weights, such that
  \[
    \|
      \Realization_{\varrho}(\Phi^f_\varepsilon) - f
    \|_{L^{\Felix{p}} ([-\nicefrac{1}{2}, \nicefrac{1}{2}]^d)}
    \leq \varepsilon
    \quad \text{ and } \quad
    \| \Realization_\varrho(\Phi^f_\varepsilon) \|_{\sup} \leq \lceil B \rceil.
  \]
\end{corollary}

\begin{proof}
\Felix{Set $q := \max\{1,p^{-1}\}$.}
Let $\varepsilon \in (0,\nicefrac{1}{2})$ and
$f = \chi_{K} \cdot g$ with $g \in \mathcal{F}_{\beta',d, B}$ and
$K \in \mathcal{K}_{r, \beta, d, B}$.
We start by constructing the following three networks:
\smallskip{}

\Felix{First,} Theorem \ref{thm:ApproximationOfPiecewiseConstantFunctionsAppendix}
combined with Remark \ref{rem:QuantisationConversion} yields certain constants
$c_1 = c_1(d,\beta,r, \Felix{p}, B) > 0$,
$s_1 = s_1(d,\beta,r, \Felix{p}, B)\in\N$, and $L_1 = L_1(\beta,d) \in \N$ with
\Felix{
\[
  L_1
  \leq 22 + \big(1 + \lceil \log_2 \beta \rceil \big)
             \cdot (11 + \nicefrac{2\beta}{d})
  \leq 22 + \big(1 + \lceil \log_2 \beta_0 \rceil \big)
             \cdot (11 + \nicefrac{2\beta_0}{d})
\]}%
and a network $\Phi^K_\varepsilon$ with no more than $L_1$ layers and at most
$c_1 \cdot \varepsilon^{-\Felix{p} (d-1)/\beta}$ nonzero,
$(s_1,\varepsilon)$-quantized weights, such that
\[
  \|
    \Realization_{\varrho}(\Phi^K_\varepsilon)
    - \chi_K
  \|_{L^{\Felix{p}}\Felix{([-\nicefrac{1}{2}, \nicefrac{1}{2}]^d)}}
  \leq \frac{\varepsilon}{3 \cdot \Felix{4^q} \cdot B}
  \quad \text{and} \quad
  \Felix{\|\Realization_{\varrho} (\Phi_\varepsilon^K)\|_{\sup} \leq 1} \, .
\]

\Felix{Second},
Theorem \ref{thm:ApproxOfSmoothFctnAppendix}
combined with Remark \ref{rem:QuantisationConversion} yields
$L_2 = L_2(\beta,d) \in \N$ and $c_2 = c_2(d,\beta,B \Felix{,p}) > 0$,
$s_2 = s_2(d,\beta,B \Felix{,p}) \in \N$ with
\Felix{
\[
  L_2
  \leq 11 + \big(1 + \lceil \log_2 \beta' \rceil \big)
            \cdot \big(11 + \nicefrac{\beta'}{d}\big)
  \leq 11 + \big(1 + \lceil \log_2 \beta_0 \rceil \big)
            \cdot \big(11 + \nicefrac{2\beta_0}{d}\big)
\]}%
and a network $\Phi^g_\varepsilon$ with no more than $L_2$ layers and at most
$c_2 \cdot \varepsilon^{-d/\beta'}
= c_2 \cdot \varepsilon^{-\Felix{p} (d-1)/\beta}$ nonzero,
$(s_2,\varepsilon)$-quantized weights, such that
\[
  \|
    \Realization_{\varrho}(\Phi^g_\varepsilon) - g
  \|_{L^{\Felix{p}}\Felix{([-\nicefrac{1}{2},\nicefrac{1}{2}]^d)}}
  \leq \frac{\varepsilon}{3 \cdot \Felix{4^q}}
  \Felix{
  \quad \text{and} \quad
  \|\Realization_\varrho (\Phi_\varepsilon^g)\|_{\sup} \leq \lceil B \rceil \, .
  }
\]

\smallskip{}

As usual, we can assume
\[
  L(\Phi_\varepsilon^K)
  = L(\Phi_\varepsilon^g)
  = \max\{L_{\Felix{1}}, L_{\Felix{2}}\}
  \Felix{
  \leq 22 + \big( 1 + \lceil \log_2 \beta_0 \rceil \big)
            \cdot \big(11 + \nicefrac{2\beta_0}{d}\big) \, ,}
\]
by possibly switching from $\Phi_\varepsilon^K$ or $\Phi_\varepsilon^g$ to
$\Phi_{1, \lambda_1}^{\mathrm{Id}} \sconc \Phi_\varepsilon^K$ or
$\Phi_{1, \lambda_2}^{\mathrm{Id}} \sconc \Phi_\varepsilon^g$
for $\lambda_1 = \max\{L_{\Felix{1}}, L_{\Felix{2}}\} - L(\Phi_\varepsilon^K)$
and $\lambda_2 = \max\{L_{\Felix{1}}, L_{\Felix{2}}\} - L(\Phi_\varepsilon^g)$.
This might necessitate changing the constants $c_1$ and $c_2$,
but these constants stay of the required form.



\smallskip{}

\Felix{Third}, Lemma \ref{lem:MultiplicationWithBoundedLayerNumber}
(applied with $\theta = \pp{\nicefrac{\Felix{p} (d-1)}{\beta} = \nicefrac{d}{\beta'}
\geq \nicefrac{d}{\beta_0}}$, with
$L_3^{(0)} := 1 + \lfloor \nicefrac{\beta_0}{2d} \rfloor$ instead of $L$,
with $3^{-1} \cdot 2^{-q} \cdot \varepsilon$ instead of $\varepsilon$,
and with $M = \lceil B \rceil$),
combined with Remark \ref{rem:QuantisationConversion}
yields constants $c_3 = c_3(d,\beta, \Felix{p}, B)$, $s_3 = s_3(B) \in \N$, and
$L_3 = L_3(\beta,d) \in \N$ with
\Felix{$L_3 \leq 8 + 2 \cdot L_3^{(0)} \leq 10 + \nicefrac{\beta_0}{d}$}
and a network $\wt{\times}$ with at most $L_3$ layers and at most
$c_3 \cdot \varepsilon^{-\theta} = c_3 \cdot \varepsilon^{-\Felix{p} (d-1)/\beta}$
nonzero, $(s_3,\varepsilon)$-quantized weights such that
\[
  | x y -\Realization_{\varrho}(\wt{\times})(x,y)|
  \leq \frac{\varepsilon}{3 \cdot \Felix{2^q}}
  \quad \text{ for all }
        x,y \in \left[- \lceil B \rceil, \lceil B \rceil \right].
\]

Now, we set $\Psi^f_\varepsilon
:= \wt{\times} \sconc P(\Phi^K_\varepsilon, \Phi^g_\varepsilon)$.
By Remark \ref{rem:SpecialConc}, $\Psi^f_\varepsilon$ has at most
\[
  \max\{L_1,L_2\} + L_3
  \Felix{
  \leq 32 + \big(1 + \lceil \log_2 \beta_0 \rceil \big)
            \cdot \big( 11 + \nicefrac{3 \beta_0}{d} \big)}
\]
layers and
$c_4 \cdot  \varepsilon^{-\Felix{p} (d-1)/\beta}$ nonzero,
$(\max\{s_1,s_2,s_3\}, \varepsilon)$-quantized weights,
where $c_4 = c_4 (d,\beta, r, \Felix{p}, B) > 0$.

Finally, we show that $\Psi_\varepsilon^f$ satisfies the claimed error bound.
To this end, we recall Equation \eqref{eq:PseudoTriangleInequality}
and the identity $f = g \cdot \chi_K$ in order to estimate
\begin{align*}
  \| \Realization_{\varrho}(\Psi^f_\varepsilon) - f \|_{L^{\Felix{p}}}
  = &~\|
        \Realization_{\varrho}(\wt{\times})
          (\Realization_{\varrho}(\Phi^K_\varepsilon),
           \Realization_{\varrho}(\Phi^g_\varepsilon))
        - f
      \|_{L^{\Felix{p}}} \\
  \leq &~ \Felix{2^q} \cdot
          \|
            \Realization_{\varrho}(\wt{\times})
              (\Realization_{\varrho}(\Phi^K_\varepsilon),
               \Realization_{\varrho}(\Phi^g_\varepsilon))
            - \Realization_{\varrho}(\Phi^K_\varepsilon)
              \cdot \Realization_{\varrho}(\Phi^g_\varepsilon)
          \|_{L^{\Felix{p}}}
          + \Felix{2^q} \cdot
            \|
              \Realization_{\varrho}(\Phi^K_\varepsilon)
              \cdot \Realization_{\varrho}(\Phi^g_\varepsilon)
              - f
            \|_{L^{\Felix{p}}} \\
  \leq &~ \frac{\varepsilon}{3}
          + \Felix{4^q} \cdot
            \|
              \Realization_{\varrho}(\Phi^K_\varepsilon)
              \cdot [ \Realization_{\varrho}(\Phi^g_\varepsilon) - g ]
          \|_{L^{\Felix{p}}}
          + \Felix{4^q} \cdot
            \|
              g \cdot [ \Realization_{\varrho}(\Phi^K_\varepsilon) - \chi_K ]
            \|_{L^{\Felix{p}}}.
\end{align*}
We continue by recalling
$\| \Realization_\varrho (\Phi_\varepsilon^K)\|_{\sup} \leq 1$, so that
\begin{align*}
  \Felix{4^q} \cdot
  \|
    \Realization_{\varrho}(\Phi^K_\varepsilon)
    \cdot [ \Realization_{\varrho}(\Phi^g_\varepsilon) - g]
  \|_{L^{\Felix{p}}}
  \leq \Felix{4^q} \cdot
       \|\Realization_{\varrho}(\Phi^g_\varepsilon) - g\|_{L^{\Felix{p}}}
  \leq \frac{\varepsilon}{3}.
\end{align*}
Moreover, since $g \in \mathcal{F}_{\beta', d, B}$, so that
$\| g \|_{\sup} \leq B$, we also have
\[
  \Felix{4^q} \cdot
  \|
    g \cdot [ \Realization_{\varrho}(\Phi^K_\varepsilon) - \chi_K ]
  \|_{L^{\Felix{p}}}
  \leq \Felix{4^q} \cdot B \cdot
       \|
         \Realization_{\varrho}(\Phi^K_\varepsilon) - \chi_K
       \|_{L^{\Felix{p}}}
  \leq \frac{\varepsilon}{3}.
\]
Combining all estimates above yields
$\|\Realization_{\varrho}(\Psi^f_\varepsilon) - f \|_{L^{\Felix{p}}}
\leq \varepsilon$.
An application of Lemma \ref{lem:boundedApprox} to $\Psi^f_\varepsilon$ as at
the end of the proof of Theorem \ref{thm:ApproxOfSmoothFctnAppendix} yields the
network $\Phi^f_\varepsilon$ satisfying all desired properties.
\end{proof}

\section{Lower bounds for the approximation of horizon functions}
\label{sec:LowerBoundProofs}

In this section, we give the proofs of Theorem \ref{thm:Optimality}, which
establishes a lower bound for approximation uniformly over the class of horizon
functions, and of Theorem \ref{thm:SingleFunctionOptimality},
which establishes a similar lower bound for the approximation of a \emph{single}
judiciously chosen horizon function $f$.

Since the proof of the lower bound for the uniform setting is simpler but
contains most of the crucial ideas, we begin with this setting.
The improvement to a lower bound for the approximation of a single function is
then obtained by a suitable application of the Baire category theorem.

\subsection{Lower bounds for the uniform setting}

The general idea is as follows: In Lemma \ref{lem:NNEncoding},
we will show that if we denote by
\[
  \mathcal{NN}_{M,K,d}^{\mathcal{B}, \varrho}
  := \{
       \Realization_\varrho (\Phi)
       \,:\,
       \Phi \in \mathcal{NN}_{M,K,d}^{\mathcal{B}}
     \}
\]
the set of all realizations (with activation function $\varrho$) of networks in
$\mathcal{NN}_{M,K,d}^{\mathcal{B}}$, then each function
$f \Felix{= \Realization_\varrho(\Phi)}
   \in \mathcal{NN}_{M,K,d}^{\mathcal{B}, \varrho}$
can be encoded with $\ell := C \cdot M \cdot (K + \lceil \log_2 M \rceil)$ bits,
for a universal constant $C = C(d) \in \N$.
\Felix{More precisely}, there is an injective map
$\Gamma : \mathcal{NN}_{M,K,d}^{\mathcal{B}, \varrho} \to \{0,1\}^\ell$,
with suitable left inverse
$\Theta : \{0,1\}^\ell \to \mathcal{NN}_{M,K,d}^{\mathcal{B}, \varrho}$.
Thus, if to a given $\varepsilon > 0$, there is for each
$f \in \mathcal{HF}_{\beta, d, B}$ a neural network
$\Phi_{f,\varepsilon} \in \mathcal{NN}_{M,K,d}^{\mathcal{B}}$ with
$\|
   f - \Realization_\varrho (\Phi_{f,\varepsilon})
 \|_{L^{\Felix{p}}\Felix{([-\nicefrac{1}{2},\nicefrac{1}{2}]^d)}}
 \leq \varepsilon$,
then the \emph{encoder-decoder pair} $(E^\ell, D^\ell)$ defined by
\[
  \begin{alignedat}{3}
      & E^\ell : &  & \mathcal{HF}_{\beta,d,B} \to \left\{ 0,1\right\} ^{\ell},
                            &  & f \mapsto \Gamma
                                           \left(
                                             \Realization_\varrho
                                               (\Phi_{f, \varepsilon})
                                           \right),\\
   & D^\ell:
   & \ & \left\{ 0,1\right\} ^{\ell}
         \to L^{\Felix{p}}\left(
                    \left[-\nicefrac{1}{2},\,\nicefrac{1}{2}\right]^{d}
                  \right),
   & \ & c \mapsto \Felix{[\Theta (c)]|_{[-\nicefrac{1}{2}, \nicefrac{1}{2}]^d}}
  \end{alignedat}
\]
\emph{achieves \Felix{$L^p$-}distortion $\varepsilon$}, that is, it satisfies
\[
  \sup_{f \in \mathcal{HF}_{\beta, d, B}}
     \, \| f - D^\ell (E^\ell (f)) \|_{L^{\Felix{p}}}
  \leq \varepsilon .
\]

From this, we obtain the desired lower bound by showing that each
encoder-decoder pair $(E^\ell, D^\ell)$ for $\mathcal{HF}_{\beta, d, B}$
which achieves \Felix{$L^p$-}distortion $\varepsilon$ necessarily has to satisfy
$\ell \gtrsim \varepsilon^{-\Felix{p}(d-1)/\beta}$.

Of course, this last statement is highly nontrivial; it is essentially a
lower bound on the \emph{description complexity} of the class
$\mathcal{HF}_{\beta, d, B}$.
As we will see, this description complexity---which is expressed using
encoder-decoder pairs---is closely related to the asymptotic behavior of the
so-called \emph{entropy numbers} of the class $\mathcal{HF}_{\beta, d, B}$.

Deriving a lower bound for these entropy numbers from first principles would be
quite difficult.
But luckily, we can use a trick to transfer known results from
\cite{ClementsLipschitzFunctionsL1Entropy} about the entropy numbers of the
class $C^{0,\beta}([-\nicefrac{1}{2},\nicefrac{1}{2}]^{d-1})$ to bounds on the
entropy numbers of the class of horizon functions.
This trick
is explained by the following lemma.
\begin{lemma}\label{lem:HFBounds}
For $d \in \N_{\geq 2}$, and an arbitrary Borel measurable function
$\gamma : [-\nicefrac{1}{2}, \nicefrac{1}{2}]^{d-1} \to \R$, define
\[
  \mathrm{HF}_\gamma :
  \left[-\nicefrac{1}{2}, \nicefrac{1}{2}\right]^d \to \{0,1\}, \quad
  (x_1, x_2, \dots, x_d) \mapsto H (
                                    x_1 + \gamma (x_2, \dots, x_d),
                                    x_2,
                                    \dots,
                                    x_d
                                   ),
\]
where $H = \chi_{[0,\infty) \times \R^{d-1}}$ denotes the Heaviside function.
Then, we have for arbitrary $p \in (0,\infty)$ and arbitrary measurable
$\psi, \gamma : [-\nicefrac{1}{2}, \nicefrac{1}{2}]^{d-1}
\to [-\nicefrac{1}{2}, \nicefrac{1}{2}]$ the identity
\[
  \|
    \mathrm{HF}_\gamma - \mathrm{HF}_\psi
  \|_{L^p ([-\nicefrac{1}{2}, \nicefrac{1}{2}]^d)}
  = \|
      \gamma - \psi
    \|_{L^1 ([-\nicefrac{1}{2}, \nicefrac{1}{2}]^{d-1})}^{\frac{1}{p}} \, .
\]
For measurable $\psi,\gamma : [-\nicefrac{1}{2}, \nicefrac{1}{2}]^{d-1} \to \R$,
we still have
$\|
   \mathrm{HF}_\gamma - \mathrm{HF}_\psi
 \|_{L^p ([-\nicefrac{1}{2}, \nicefrac{1}{2}]^d)}
\leq \|
       \gamma - \psi
     \|_{L^1 ([-\nicefrac{1}{2}, \nicefrac{1}{2}]^{d-1})}^{1/p}$.
\end{lemma}
\begin{proof}
For $x = (x_1,\dots, x_d) \in \R^d$, we write $\hat{x} := (x_2,\dots, x_d)$.
Then, \Felix{for $x \in [-\nicefrac{1}{2}, \nicefrac{1}{2}]^d$},
we have the following equivalence:
\[
  \mathrm{HF}_\gamma (x) = 1
  \qquad \Longleftrightarrow \qquad
  x_1 + \gamma (\hat{x}) \geq 0.
\]
Thus, $|\mathrm{HF}_\gamma - \mathrm{HF}_\psi|$ is $\{0,1\}$-valued with
\begin{align*}
  & |\mathrm{HF}_\gamma (x) - \mathrm{HF}_\psi (x)| = 1 \\
  \Longleftrightarrow \,\,
  & [
     x_1 + \gamma(\hat{x}) \geq 0
     \quad \text{ and } \quad
     x_1 + \psi(\hat{x}) < 0
    ]
    \qquad \text{ or } \qquad
    [
     x_1 + \gamma(\hat{x}) < 0
     \quad \text{ and } \quad
     x_1 + \psi(\hat{x}) \geq 0
    ] \\
  \Longleftrightarrow \,\,
  & x_1 \in [-\gamma(\hat{x}) , - \psi(\hat{x}))
    \quad \text{ or } \quad
    x_1 \in [-\psi(\hat{x}), -\gamma(\hat{x})).
\end{align*}
But since we have
$[-\gamma(\hat{x}) , - \psi(\hat{x}))
 \cap [-\psi(\hat{x}), -\gamma(\hat{x}))
 \subset [-\gamma(\hat{x}), -\gamma(\hat{x})) = \emptyset$,
and since $\gamma,\psi$ only take values in
$[-\nicefrac{1}{2},\nicefrac{1}{2}]$,
so that $[-\gamma(\hat{x}) , - \psi(\hat{x}))
\cup [-\psi(\hat{x}), -\gamma(\hat{x}))
\subset [-\nicefrac{1}{2},\nicefrac{1}{2}]$,
we get with the one-dimensional Lebesgue measure $\mu$ for each
$\hat{x} \in [-\nicefrac{1}{2}, \nicefrac{1}{2}]^{d-1}$ that
\begin{align}
&\mu (
     \{
        x_1 \in \left[-\nicefrac{1}{2}, \nicefrac{1}{2}\right]
        \, : \,
        |\mathrm{HF}_\gamma (x_1, \hat{x}) - \mathrm{HF}_\psi (x_1, \hat{x})|
        = 1
     \}
    ) \nonumber \\
& = \mu([-\gamma(\hat{x}), -\psi(\hat{x})))
    + \mu([-\psi(\hat{x}), -\gamma(\hat{x})))
  \label{eq:HFBoundsModificationForLargeGamma} \\
& = \max \{ 0, \gamma(\hat{x}) - \psi(\hat{x}) \}
    + \max \{ 0, \psi(\hat{x}) - \gamma(\hat{x}) \} \nonumber \\
& = |\gamma(\hat{x}) - \psi(\hat{x})| \nonumber.
\end{align}

Since $|\mathrm{HF}_\gamma - \mathrm{HF}_\psi|$ is $\{0,1\}$-valued,
this implies by Fubini's theorem
\begin{align}
  \|
    \mathrm{HF}_\gamma - \mathrm{HF}_\psi
  \|_{L^p ([-\nicefrac{1}{2}, \nicefrac{1}{2}]^d)}^p
  & = \int_{[-\nicefrac{1}{2}, \nicefrac{1}{2}]^{d-1}}
        \int_{-\nicefrac12}^{\nicefrac12}
          |
           \mathrm{HF}_\gamma (x_1, \hat{x})
           - \mathrm{HF}_\psi (x_1, \hat{x})
          |^p
        \, d x_1
      \, d \hat{x}
          \nonumber \\
  & = \int_{[-\nicefrac{1}{2}, \nicefrac{1}{2}]^{d-1}}
        \mu \left(
              \left\{
                x_1 \in \left[-\nicefrac{1}{2}, \nicefrac{1}{2}\right]
                \, : \,
                |
                 \mathrm{HF}_\gamma (x_1, \hat{x})
                 - \mathrm{HF}_\psi (x_1, \hat{x})
                | = 1
              \right\}
            \right)
      \, d \hat{x} \nonumber \\
  & = \int_{[-\nicefrac{1}{2}, \nicefrac{1}{2}]^{d-1}}
         | \gamma(\hat{x}) - \psi(\hat{x})|
      \, d \hat{x}
    = \| \gamma - \psi\|_{L^1 ([-\nicefrac{1}{2}, \nicefrac{1}{2}]^{d-1})},
  \label{eq:HFBoundsModificationForLargeGamma2}
\end{align}
as claimed.

If we have $\psi, \gamma : [-\nicefrac{1}{2}, \nicefrac{1}{2}]^{d-1} \to \R$
instead of
$\psi, \gamma : [-\nicefrac{1}{2}, \nicefrac{1}{2}]^{d-1}
\to [-\nicefrac{1}{2}, \nicefrac{1}{2}]$, then the equality
in \eqref{eq:HFBoundsModificationForLargeGamma}---and thus also the one in
\eqref{eq:HFBoundsModificationForLargeGamma2}---need to be
replaced by ``$\leq$'', but the remainder of the proof remains valid.
\end{proof}

Our next goal (see Lemma \ref{lem:HorizonFunctionsMinimaxCodeLength})
is to show that an $\ell$-bit encoder-decoder pair $(E^\ell, D^\ell)$
which achieves \Felix{$L^p$-}distortion $\varepsilon$ over the class
$\mathcal{HF}_{\beta, d, B}$ needs to satisfy
$\ell \gtrsim \varepsilon^{- \Felix{p}(d-1)/\beta}$.
Before we prove this, let us fix some notation and terminology:


\begin{definition}
Let $p \in (0,\infty)$, let $\Omega \subset \R^d$ be measurable, and let
$\mathcal{C}\subset L^{\Felix{p}}(\Omega)$ be an arbitrary function class.
For each $\ell\in \N$, we denote by \pp{$\mathfrak{E}^\ell
  := \{E: \mathcal{C} \to \{0,1\}^{\ell}\}$ the set of \emph{binary encoders mapping elements of $\mathcal{C}$
to bit-strings of length $\ell$}, and we let $\mathfrak{D}^\ell
  := \{D:\{0,1\}^{\ell} \to  L^{\Felix{p}}(\Omega)\}$ be the set of \emph{binary decoders mapping bit-strings of length $\ell$ to
elements of $L^{\Felix{p}}(\Omega)$}}.

\pp{An encoder-decoder pair
$(E^\ell, D^\ell) \in \mathfrak{E}^\ell \times \mathfrak{D}^\ell$
is said to {\em achieve \Felix{$L^p$-}distortion $\varepsilon >0$
over the function class $\mathcal{C}$}, if $\sup_{f\in \mathcal{C}}
  \|D^\ell (E^\ell (f)) - f \|_{L^{\Felix{p}}(\Omega)}
  \leq \varepsilon$.}
Finally, for $\varepsilon >0$ the \emph{minimax code length}
$L_{\Felix{p}}(\varepsilon, \mathcal{C})$ is
\[
  L_{\Felix{p}}(\varepsilon, \mathcal{C})
  := \min
     \left\{
       \ell\in \N
       \,:\,
       \exists \, (E^\ell, D^\ell) \in  \mathfrak{E}^\ell \times \mathfrak{D}^\ell:
         \sup_{f\in \mathcal{C}}
           \|D^\ell (E^\ell (f)) - f \|_{L^{\Felix{p}}(\Omega)}
         \leq \varepsilon
     \right\},
\]
with the interpretation $L_{\Felix{p}}(\varepsilon, \mathcal{C}) = \infty$ if
$\sup_{f\in \mathcal{C}} \|D^\ell (E^\ell (f)) - f \|_{L^{\Felix{p}}(\Omega)}
> \varepsilon$ for all
$(E^\ell, D^\ell) \in \mathfrak{E^\ell} \times \mathfrak{D}^\ell$ and arbitrary
$\ell \in \N$.
\end{definition}

Now that we have fixed the terminology, we derive a lower bound on the
asymptotic behavior of the minimax code length for the class
$\mathcal{HF}_{\beta,d,B}$ of horizon functions, by using
Lemma \ref{lem:HFBounds} to transfer results about the behavior of
the entropy numbers of $C^{0,\beta}([0,1]^{d-1})$ to the class
$\mathcal{HF}_{\beta,d,B}$.
We remark that this result is essentially folklore; see for example
\cite{Surflets,SurfletsTechnicalReport} for related, but less detailed proofs;
in fact, our proof is based on those two papers.

\begin{lemma}\label{lem:HorizonFunctionsMinimaxCodeLength}
Let $d \in \N_{\geq2}$, and $\Felix{p}, \beta, B > 0$ be arbitrary.
Then there are constants $C=C\left(d, \Felix{p}, \beta ,B\right) > 0$ and
$\varepsilon_0 = \varepsilon_0(d, \Felix{p}, \beta, B) > 0$, such that for each
$\varepsilon\in\left(0,\varepsilon_0\right)$,
the minimax code length
$L_{\Felix{p}}\left(\varepsilon,\mathcal{HF}_{\beta,d,B}\right)$
of the class $\mathcal{HF}_{\beta,d,B}$ of horizon functions satisfies \pp{$L_{\Felix{p}}(\varepsilon,\mathcal{HF}_{\beta,d,B})
  \geq C \cdot \varepsilon^{-\frac{\Felix{p} \left(d-1\right)}{\beta}}$}.
\end{lemma}
\begin{proof}
\textbf{Step 1}: We prove that there are constants
$C_{1} = C_{1}\left(d,\beta,B\right)>0$ and
$\varepsilon_1 = \varepsilon_1 (d,\beta, B) > 0$
such that for each $\varepsilon\in\left(0, \varepsilon_1\right)$, there is some
$N\geq\exp\left(C_{1}\cdot\varepsilon^{-\left(d-1\right)/\beta}\right)$,
and functions $f_{1},\dots,f_{N}\in\mathcal{F}_{\beta,d-1,B}$ satisfying
$\left\Vert f_{i}-f_{\ell}\right\Vert _{L^{1}}\geq\varepsilon$ for
$i\neq\ell$.

To show this, we need some preparation:
First, let us write $\beta = n + \sigma$ with $n \in \N_0$ and $\sigma \in (0,1]$.
It is easy to see from Lemma \ref{lem:DerivativeHoelderEstimate}
(by translating everything from $\left[0,1\right]^{d-1}$
to $\left[-\nicefrac{1}{2},\nicefrac{1}{2}\right]^{d-1}$)
that there is a constant $C_{2} = C_{2}\left(d,\beta\right) > 0$ such that
each $u\in C^{n}\left([-\nicefrac{1}{2}, \nicefrac{1}{2}]^{d-1}\right)$
satisfies
\begin{equation}
  \left\Vert u \right\Vert_{C^{0,\beta}}
  \leq C_{2} \cdot \left(
                     \left\Vert u \right\Vert_{\sup}
                     + \max_{\left|\alpha\right| = n}
                         \Lip_\sigma  \left(\partial^{\alpha}u\right)\right).
  \label{eq:IntermediateDerivativesApplication}
\end{equation}
Let $C_{3}:=B/\left(1+2C_{2}\right)$, and
set
\[
  F_{\beta}^{d-1}\left(C_{3}\right)
  := \left\{
       u\in C^{n}\left( [-\nicefrac{1}{2}, \nicefrac{1}{2}]^{d-1} \right)
       \,:\,
       \left\Vert u\right\Vert_{\sup}\leq C_{3}
       \text{ and }
       \max_{\left|\alpha\right| = n}
         \Lip_\sigma \left(\partial^{\alpha}u\right)
       \leq C_{3}
     \right\} ,
\]
as in \cite{ClementsLipschitzFunctionsL1Entropy}.
Actually, in \cite{ClementsLipschitzFunctionsL1Entropy},
the unit cube $\left[0,1\right]^{d-1}$ is used instead of
$\left[-\nicefrac{1}{2},\nicefrac{1}{2}\right]^{d-1}$,
but it is easy to see (by translation) that this makes no difference
for what follows.
Precisely, we want to use \cite[Theorem 3]{ClementsLipschitzFunctionsL1Entropy},
which ensures \Felix{the} existence of a large number of functions
$f_{1},\dots,f_{N}\in F_{\beta}^{d-1}\left(C_{3}\right)$
with $\left\Vert f_{i}-f_{\ell}\right\Vert _{L^{1}}\geq \varepsilon$
for $i\neq\ell$.
To see that this indeed follows from
\cite[Theorem 3]{ClementsLipschitzFunctionsL1Entropy},
we recall a few notions from
\cite[Page 1086]{ClementsLipschitzFunctionsL1Entropy}:
For a subset $U\subset X$ of a metric space $\left(X,d\right)$,
we say that $U$ is \textbf{$\varepsilon$-distinguishable} if
$d\left(x,y\right)\geq\varepsilon$ for all $x,y\in U$ with $x\neq y$.
Next, for $\emptyset\neq A\subset X$, we define
$M_{\varepsilon}\left(A\right)
 :=\max\left\{
          \left|U\right|
          \,:\,
          U\subset A\text{ is }\varepsilon\text{-distinguishable}
       \right\} $,
and we define the \textbf{capacity} of $A$
as\footnote{We remark that some authors use a logarithm with a different basis
than the natural logarithm. For us this does not matter, since we
will obtain a bound
$C_{\varepsilon}\left(A\right)\geq C\cdot\varepsilon^{-\left(d-1\right)/\beta}$,
so that a different choice of basis just leads to a different constant $C$.}
$C_{\varepsilon}\left(A\right)=\ln M_{\varepsilon}\left(A\right)$.
Additionally, there is also the notion of the \textbf{(metric) entropy}
$H_{\varepsilon}\left(A\right)$ of $A$, the precise definition of
which is immaterial for us; the only property of the entropy that
we will need is that
$C_{\varepsilon}\left(A\right)\geq H_{\varepsilon}\left(A\right)$.

Finally, \cite[Theorem 3]{ClementsLipschitzFunctionsL1Entropy} shows
that considering $A=F_{\beta}^{d-1}\left(C_{3}\right)$ as a
subset of the metric space
$X=L^{1}(\left[-\nicefrac{1}{2},\nicefrac{1}{2}\right]^{d-1})$
yields that the entropy of $F_{\beta}^{d-1} \! \left(C_{3}\right)$ satisfies
$H_{\varepsilon}\left(\smash{F_{\beta}^{d-1}} \! \left(C_{3}\right)\right)
 \!\geq \! C_{1}\cdot\varepsilon^{-\left(d-1\right)/\beta}$
for $\varepsilon\in\left(0,\varepsilon_1 \right)$ and certain constants
$C_{1} = C_{1}\left(d,\beta,C_{3}\right)=C_{1}\left(d,\beta,B\right) > 0$
and $\varepsilon_1 = \varepsilon_1 (d,\beta, B) > 0$.
Because of $\vphantom{\sum_j}\ln M_{\varepsilon}\left(A\right)
=C_{\varepsilon}\left(A\right)\geq H_{\varepsilon}\left(A\right)$,
and by definition of $M_{\varepsilon}\left(A\right)$, this implies
that there is some
$N\geq\exp\left(C_{1}\cdot\varepsilon^{-\left(d-1\right)/\beta}\right)$
and certain functions $f_{1},\dots,f_{N}\in F_{\beta}^{d-1}\left(C_{3}\right)$
with $\left\Vert f_{i}-f_{\ell}\right\Vert _{L^{1}}\geq\varepsilon$
for $i\neq\ell$.
To complete the proof of Step 1, we observe as a consequence of
Equation \eqref{eq:IntermediateDerivativesApplication}
that each $f_i \in F_\beta^{d-1}(C_3)$ satisfies
$f_i \in C^n ([-\nicefrac{1}{2}, \nicefrac{1}{2}]^{d-1})$, with
\[
  \| f_i \|_{C^{0,\beta}}
  \leq C_2 \cdot \left(
                   \| f_i \|_{\sup}
                   + \max_{|\alpha|=n}
                       \Lip_\sigma (\partial^\alpha f_i)
                 \right)
  \leq C_2 \cdot 2C_3 \leq B,
\]
that is, $f_i \in \mathcal{F}_{\beta, d-1, B}$.

\smallskip{}

\textbf{Step 2}: For simplicity, let
$B_{0}:=\min\left\{ \nicefrac{1}{2},B\right\} $.
Further, for $x\in\left[-\nicefrac{1}{2},\nicefrac{1}{2}\right]^{d}$,
let us write $x=\left(x_{1},\hat{x}\right)$, with
$x_{1}\in\left[-\nicefrac{1}{2},\nicefrac{1}{2}\right]$
and $\hat{x}\in\left[-\nicefrac{1}{2},\nicefrac{1}{2}\right]^{d-1}$.
Finally, recall from Lemma \ref{lem:HFBounds} that to every measurable function
$\gamma : [-\nicefrac{1}{2}, \nicefrac{1}{2}]^{d-1} \to \R$,
we associate the function
\begin{equation*}
  {\rm HF}_{\gamma} :
  \left[-\nicefrac{1}{2},\:\nicefrac{1}{2}\right]^{d}
  \to \left\{ 0,1\right\},
  \left(x_{1},\hat{x}\right)
  \mapsto H\left(x_{1}+\gamma\left(\hat{x}\right),\hat{x}\right) \, ,
  \Felix{\quad \text{where} \quad
  H = \chi_{[0,\infty) \times \R^{d-1}}} \, .
\end{equation*}
Now, each $\gamma \in \mathcal{F}_{\beta, d-1, B_0}$ satisfies
$\|\gamma\|_{\sup} \leq \|\gamma\|_{C^{0,\beta}}
                       \leq B_0 \leq \nicefrac{1}{2}$,
and thus $\gamma : [-\nicefrac{1}{2}, \nicefrac{1}{2}]^{d-1}
                   \to [-\nicefrac{1}{2}, \nicefrac{1}{2}]$.
Therefore, Lemma \ref{lem:HFBounds} shows
\begin{equation}
  \|
    \mathrm{HF}_\gamma - \mathrm{HF}_\psi
  \|_{L^{\Felix{p}} ([-\nicefrac{1}{2}, \nicefrac{1}{2}]^d)}
  \geq \|
         \gamma - \psi
       \|_{L^1 ([-\nicefrac{1}{2}, \nicefrac{1}{2}]^{d-1})}^{\nicefrac{1}{\Felix{p}}}
  \qquad \text{ for all } \gamma, \psi \in \mathcal{F}_{\beta, d-1, B_0}.
  \label{eq:HFDistanceLowerBound}
\end{equation}
Finally, we remark that directly from the definition, we have
$\mathrm{HF}_\gamma \in \mathcal{HF}_{\beta, d, B_0}
 \subset \mathcal{HF}_{\beta, d, B}$
for all $\gamma \in \mathcal{F}_{\beta, d-1, B_0}$.

\smallskip{}

\textbf{Step 3}: In this step, we actually prove the claim:
\Felix{Let $q := \max\{1, p^{-1} \}$.}
Step 1 (applied with $B_{0} = \min\left\{ \nicefrac{1}{2}, B\right\}$ instead of
$B$ and with \Felix{$(4^q \, \varepsilon)^p$} instead of $\varepsilon$)
yields constants $C_{1} = C_{1}\left(d,\beta,B\right)>0$ and
$\varepsilon_0 = \varepsilon_0 (d,\beta, \Felix{p}, B) > 0$, such that for
$\varepsilon\in\left(0,\varepsilon_0\right)$, there is some
$N \geq \exp\left(
               C_{1} \cdot
               \Felix{\left( 4^q \, \varepsilon \right)^{-p (d-1)/\beta}}
            \right)$
and $f_{1},\dots,f_{N} \in \mathcal{F}_{\beta,d-1,B_{0}}
                           \subset\mathcal{F}_{\beta,d-1,B}$
with
$\left\Vert f_{i} - f_{\ell}\right\Vert _{L^{1}} \geq \Felix{(4^q \, \varepsilon)^p}$
for $i\neq\ell$. With this constant $C_{1}$, we will show
\[
  L_{\Felix{p}}\left(\varepsilon,\mathcal{HF}_{\beta,d,B}\right)
  \geq \frac{C_{1}}{\Felix{4^{q p (d-1)/\beta}}}
       \cdot \varepsilon^{-\frac{\Felix{p} (d-1)}{\beta}},
  \qquad \text{ for all } \varepsilon \in (0, \varepsilon_0),
\]
which clearly implies the claim.

For the proof, let
$E^{\ell}:\mathcal{HF}_{\beta,d,B}\to\left\{0,1\right\}^{\ell}$
and $D^{\ell}:\left\{ 0,1\right\}^{\ell}
              \to L^{\Felix{p}}(\left[-\nicefrac{1}{2},\nicefrac{1}{2}\right]^{d})$
be any encoder-decoder pair which achieves \Felix{$L^p$-}distortion
$\varepsilon\in\left(0,\varepsilon_0\right)$
over the class $\mathcal{HF}_{\beta,d,B}$.
We need to show
\[
  \ell
  \geq \frac{C_{1}}{\Felix{4^{q p (d-1) / \beta}}}
       \cdot \varepsilon^{-\frac{\Felix{p}(d-1)}{\beta}}.
\]
Assume towards a contradiction that this fails. Thus,
$|\left\{ 0,1\right\} ^{\ell}| = 2^{\ell}\leq e^{\ell}
 < \exp \left(
          C_{1}
          \cdot \left( \Felix{4^q \, \varepsilon} \right)^{-\Felix{p}(d-1)/\beta}
        \right)$.
By the pigeonhole principle, with $f_{1},\dots,f_{N}$ as above, this
ensures existence of $i,j\in\left\{ 1,\dots,N\right\} $ with $i\neq j$,
but with
$E^{\ell}\left({\rm HF}_{f_{i}}\right) = E^{\ell}\left({\rm HF}_{f_{j}}\right)$.
But by Step 2 (Equation \eqref{eq:HFDistanceLowerBound})
\Felix{and by Equation \eqref{eq:PseudoTriangleInequality},} this entails
\begin{align*}
    4^q \, \varepsilon
  & \leq \left\Vert f_{i}-f_{j}\right\Vert_{L^{1}}^{\Felix{1/p}}
    \leq\left\Vert {\rm HF}_{f_{i}}-{\rm HF}_{f_{j}} \right\Vert_{L^{\Felix{p}}}
    \Felix{= \left\|
        {\rm HF}_{f_{i}}
        - D^{\ell}\left(E^{\ell}\left({\rm HF}_{f_{i}}\right)\right)
        + D^{\ell}\left(E^{\ell}\left({\rm HF}_{f_{j}}\right)\right)
        - {\rm HF}_{f_{j}}
      \right\|_{L^p}} \\
  & \leq \Felix{2^q \cdot
         \max}
         \left\{
           \left\Vert
             {\rm HF}_{f_{i}}
             - D^{\ell}\left(E^{\ell}\left({\rm HF}_{f_{i}}\right)\right)
           \right\Vert _{L^{\Felix{p}}} \, ,
           \left\Vert
             D^{\ell}\left(E^{\ell}\left({\rm HF}_{f_{j}}\right)\right)
             - {\rm HF}_{f_{j}}
           \right\Vert_{L^{\Felix{p}}}
         \right\} \\
  & \leq 2^q \cdot \varepsilon \, ,
\end{align*}
a contradiction. Here, we used in the last step that
the pair $\left(E^{\ell},D^{\ell}\right)$
achieves \Felix{$L^p$-}distortion $\varepsilon$ over $\mathcal{HF}_{\beta,d,B}
\supset \left\{ {\rm HF}_{f_{1}},\dots,{\rm HF}_{f_{N}}\right\}$.
This contradiction completes the proof.
\end{proof}

Now that we have a lower bound on the minimax code length of the class of
horizon functions, the next step of the program that was outlined at the
beginning of this subsection is to show that if each horizon function
$f \in \mathcal{HF}_{\beta, d, B}$ can be approximated with \Felix{$L^p$} error
$\leq \varepsilon$ by a neural network of bounded complexity,
then this yields an encoder-decoder pair for the class
$\mathcal{HF}_{\beta, d, B}$ of a certain (small) bit-length $\ell$.
\Felix{The} main idea for showing this is to encode the approximating neural
networks as bit-strings. Our next lemma shows that this is possible.

\begin{lemma}\label{lem:NNEncoding}
  Let $d \in \N$, and let $\mathcal{B}$ be an encoding scheme for real numbers.
  For $M, K \in \N$, let $\mathcal{NN}_{M,K,d}^{\mathcal{B}}$ be as in
  Definition \ref{def:WeightEncodability}.
  Let $\varrho : \R \to \R$ with $\varrho (0) = 0$, and define
  \[
    \mathcal{NN}_{M,K,d}^{\mathcal{B},\varrho}
    := \{
         \Realization_\varrho (\Phi)
         \,:\,
         \Phi \in \mathcal{NN}_{M,K,d}^{\mathcal{B}}
       \}.
  \]

  There is a universal constant $C = C(d) \in \N$, such that for arbitrary
  $M,K \in \N$, there is an injective map
  $\Gamma_{M,K,d}^{\mathcal{B}, \varrho} :
   \mathcal{NN}_{M,K,d}^{\mathcal{B},\varrho}
   \to \{ 0,1 \}^{C M (K + \lceil \log_2 M \rceil)}$.
\end{lemma}
%

\begin{proof}
The proof is similar to that of \cite[Theorem 2.7]{BoeGKP2017}.
However, since we define networks slightly differently in this work,
we repeat the main points of the proof with some simplifications.

\smallskip{}

In Lemma \ref{lem:WeightsAndNeuronsAssumptionJustification}, it is shown that
for each $f \in \mathcal{NN}_{M,K,d}^{\mathcal{B},\varrho}$, there is a neural
network $\Phi_f \in \mathcal{NN}_{M,K,d}^{\mathcal{B}}$
satisfying $f = \Realization_\varrho (\Phi_f)$ and furthermore
$N(\Phi_f) \leq M(\Phi_f) + d + 1$.

Therefore, it suffices to show for
\[
  \mathcal{NN}_{M,K}^\ast
  := \{
       \Phi \in \mathcal{NN}_{M,K,d}^\mathcal{B}
       \,:\,
       N(\Phi) \leq M(\Phi) + d + 1
     \}
\]
and $\ell := C \cdot M \cdot (K + \lceil \log_2 M \rceil)$
(with a suitable constant $C = C(d) \in \N$) that there is an injective map
$\Theta_{M,K}^{\mathcal{B}} : \mathcal{NN}_{M,K}^\ast \to \{0,1\}^{\ell}$,
since then the map
$\Gamma_{M,K,\Felix{d}}^{\mathcal{B}, \varrho} :
 \mathcal{NN}_{M,K,d}^{\mathcal{B}, \varrho}
 \to \{0,1\}^\ell, f \mapsto \Theta_{M,K}^{\mathcal{B}} (\Phi_f)$
is easily seen to be injective.

To prove \Felix{the} existence of $\Theta_{M,K}^{\mathcal{B}}$, we show that
each $\Phi \in \mathcal{NN}_{M,K}^\ast$ can be encoded (in a uniquely decodable
way) with $\ell$ bits.
To show this, we first observe that each such $\Phi$
satisfies for $L := L(\Phi)$ the estimates
\[
  L =    \sum_{\ell=1}^L 1
    \leq \sum_{\ell=1}^L N_\ell
    =    N(\Phi) - d
    \leq M(\Phi) + 1
    \leq M + 1,
  \text{ and }
  N(\Phi)
  \leq M(\Phi) + d + 1
  \leq M + d + 1
  \leq 3d \cdot M
  =:   T.
\]

Next, in the notation of Definition \ref{def:NeuralNetworks},
we can write $\Phi = ((A_1, b_1), \dots, (A_L, b_L))$, so that it suffices to
encode (in a uniquely decodable way) the integer $L \in \N$, the matrices
$A_1, \dots, A_L$ and the vectors $b_1, \dots, b_L$
using a bit-string of length $\ell$.
To show this, let $\mathcal{B} = (B_n)_{n \in \N}$.


Now, if $A \in \R^{n_1 \times n_2}$ with $1 \leq n_1, n_2 \leq T$ and
$\| A \|_{\ell^0} = m$ and with $A_{i,j} \in B_K (\{0,1\}^K)$ if
$A_{i,j} \neq 0$, then one can store $A$ by storing the values
$n_1, n_2 \Felix{\in \{1,\dots,T\}}$,
the value $\Felix{m \in \{0,\dots,T^2\}}$, the position of each of the $m$
nonzero entries of $A$, and the bit-string of length $K$ that is associated
(by $B_K$) to each nonzero weight.
Since one can always zero-pad the obtained bit-string to a larger length,
and since we have
\[
  \log_2 (T)
  =    \log_2 (3d) + \log_2 (M)
  \leq C_1 + \lceil \log_2 M \rceil
\]
and
$\log_2 (1 + T^2)
 \leq \log_2 (2 T^2)
 =    1 + 2 \log_2 (T)
 \leq 1 + 2C_1 + 2 \lceil \log_2 M \rceil$
for a suitable $C_1 = C_1(d) \in \N$, this can be done with
\begin{align*}
  &      \lceil \log_2 T \rceil
         + \lceil \log_2 T \rceil
         + \lceil \log_2 (T^2 +1) \rceil
         + m \cdot \left(
                     \lceil \log_2 T \rceil + \lceil \log_2 T \rceil + K
                   \right) \\
  & \leq 2C_1
         + 2 \lceil \log_2 M \rceil
         + 1
         + 2C_1
         + 2 \lceil \log_2 M \rceil
         + m \left(
               K + 2C_1 + 2 \lceil \log_2 M \rceil
             \right) \\
  & \leq 1
         + 4C_1
         + 4 \lceil \log_2 M \rceil
         + 2(1 + C_1) \cdot m \cdot (K + \lceil \log_2 M \rceil ) \\
  & \leq C_2
         + 4 \lceil \log_2 M \rceil
         + C_2 \cdot m \cdot (K + \lceil \log_2 M \rceil)
\end{align*}
bits, for a suitable constant $C_2 = C_2 (d) \in \N$.

Likewise, but easier, if $b \in \R^n$ with $1 \leq n \leq T$,
with $\| b \|_{\ell^0} = m$ and with $b_i \in B_K (\{0,1\}^K)$ if $b_i \neq 0$,
then one can store $b$ by storing the values $\Felix{n \in \{1,\dots,T\}}$ and
$m \in \{0,\dots,n\} \subset \{0,\dots,T\}$, and the position of each nonzero
entry \Felix{of $b$}, as well as the bit-string of length $K$ associated
(by $B_K$) to each \Felix{such} nonzero entry.
Because of $\log_2 (T+1) \leq \log_2 (2T) \leq 1 + \log_2 (T)$,
this can be done with
\begin{align*}
  \lceil \log_2 T \rceil
  + \lceil \log_2 (T+1) \rceil
  + m \cdot (K + \lceil \log_2 T \rceil)
  &\leq 1
        + 2 C_1
        + 2 \lceil \log_2 M \rceil
        + m \cdot (K + C_1 + \lceil \log_2 M \rceil ) \\
  &\leq C_2
        + 4 \lceil \log_2 M \rceil
        + C_2 \cdot m \cdot (K + \lceil \log_2 M \rceil )
\end{align*}
bits, after possibly enlarging the constant $C_2 = C_2 (d) \in \N$ from above.


Note that when decoding a given bit string, the values \Felix{of}
$M, K, d$---and thus also of $T$---are known.
Overall, our encoding scheme for encoding networks
$\Phi \in \mathcal{NN}_{M,K}^\ast$ now works as follows:

\textbf{Step 1:}
We store the number $\Felix{L \in \{1,\dots,M+1\}}$ in a bit-string of length
$\lceil\log_2(M+1)\rceil$.


\textbf{Step 2:}
We encode each $A_\ell$ using a bit string of length
$C_2 + 4 \lceil \log_2 M \rceil
 + C_2 \cdot \|A_\ell\|_{\ell^0} \cdot (K + \lceil \log_2 M \rceil )$,
and each $b_\ell$ using a bit string of length
$C_2 + 4 \lceil \log_2 M \rceil +
 C_2 \cdot \|b_\ell\|_{\ell^0} \cdot (K + \lceil \log_2 M \rceil )$.
As seen above, this can indeed be done in such a way that one can uniquely
reconstruct $A_1, \dots, A_L$ and $b_1, \dots, b_L$ from these bit-strings,
once one knows $M, K, d$ (which are given) and $L$,
which is given by the bit string from Step 1.

Overall, this encodes the network $\Phi = ((A_1, b_1),\dots, (A_L,b_L))$
in a uniquely decodable way using a bit-string of length
\begin{align*}
  &  \lceil \log_2 (M+1) \rceil
  + 2 \cdot \sum_{\ell = 1}^L (C_2 + 4 \lceil \log_2 M \rceil)
  + C_2 \cdot (K + \lceil \log_2 M \rceil)
        \cdot \sum_{\ell=1}^L
                (\|A_\ell\|_{\ell^0} + \|b_\ell\|_{\ell^0}) \\
  &\leq 1 + \lceil \log_2 M \rceil
       + 2 L \cdot (C_2 + 4 \lceil \log_2 M \rceil)
       + C_2 \cdot M \cdot (K + \lceil \log_2 M \rceil) \\
  &\leq K + \lceil \log_2 M \rceil
        + 4 \max\{4, C_2\} \cdot M \cdot (1 + \lceil \log_2 M \rceil)
        + C_2 \cdot M \cdot (K + \lceil \log_2 M \rceil) \\
  &\leq (1 + C_2 + 4 \max\{4, C_2\}) \cdot M
        \cdot (K + \lceil \log_2 M \rceil).
\end{align*}
Here, we used that $L \leq M+1 \leq 2M$ and that $M,K \geq 1$.
With $C := 1 + C_2 + 4 \max\{4, C_2\}$, we have thus proved the claim.
%
%
\end{proof}

Now, since we have a lower bound on the minimax code-length of the class of
horizon functions and since we know how to encode neural networks of limited
complexity, we can now prove our optimality result in the uniform setting,
by making precise the arguments that we sketched at the beginning
of the present subsection.

\begin{proof}[Proof of Theorem \ref{thm:Optimality}]
  We will use the notation $\mathcal{NN}_{M,K,d}^{\mathcal{B}}$ from
  Definition \ref{def:WeightEncodability} and the notation
  $\mathcal{NN}_{M,K,d}^{\mathcal{B}, \varrho}$ from Lemma \ref{lem:NNEncoding}.
  Recall from that lemma that there is an absolute constant
  $C_1 = C_1 (d) \in \N$, such that for arbitrary $M,K \in \N$,
  there is an injective map
  \[
    \Gamma :
    \mathcal{NN}_{M,K,d}^{\mathcal{B},\varrho}
    \to \{ 0, 1 \}^{C_1 \cdot M \cdot (K + \lceil \log_2 M \rceil)} \, .
  \]
  Furthermore, Lemma \ref{lem:HorizonFunctionsMinimaxCodeLength} yields
  constants $C_2 = C_2 (d, \Felix{p}, \beta ,B) > 0$ and
  $\nicefrac{1}{2} > \varepsilon_0 = \varepsilon_0(d, \Felix{p}, \beta, B) > 0$
  such that the minimax code length of $\mathcal{HF}_{\beta,d,B}$ satisfies
  $L_{\Felix{p}}(\varepsilon, \mathcal{HF}_{\beta,d,B})
   \geq C_2 \cdot \varepsilon^{- \Felix{p} (d-1) / \beta}$
  for all $\varepsilon \in (0,\varepsilon_0)$. Define
  \[
    C := \min \left\{
                     1,
                     C_2 \,
                     \bigg/ \, \left[
                                 2 C_1 \cdot
                                       \left(
                                         2 + \frac{\Felix{p}d}{\beta} + C_0
                                       \right)
                               \right]
              \right\}
      >  0 ,
  \]
  fix some $\varepsilon \in (0, \varepsilon_0)$, and define \pp{$K_0 := \left\lceil
               C_0 \cdot \log_2 \left( \nicefrac{1}{\varepsilon} \right)
             \right\rceil$ and $M_0 := \left\lfloor
                C \cdot \varepsilon^{-\Felix{p} (d-1)/\beta}
                  \big/ \log_2 \left(\nicefrac{1}{\varepsilon}\right)
             \right\rfloor$}. To prove the theorem, it suffices to show that there is
  $f_\varepsilon \in \mathcal{HF}_{\beta,d,B}$ such that for every
  $\Phi \in \mathcal{NN}_{M, K_0, d}^{\mathcal{B}}$ (for arbitrary $M \in \N$)
  with $\|f_\varepsilon - \Realization_\varrho(\Phi)\|_{L^{\Felix{p}}}
        \leq \varepsilon$,
  it already follows that $M > M_0$.

  Assume towards a contradiction that this fails; thus, for \emph{every}
  $f \in \mathcal{HF}_{\beta,d,B}$, there is
  $\Phi_f \in \mathcal{NN}_{M, K_0,d}^{\mathcal{B}}$ with
  $\| f - \Realization_\varrho (\Phi_f) \|_{L^{\Felix{p}}} \leq \varepsilon$,
  but such that $M \leq M_0$.
  In particular, $\Phi_f \in \mathcal{NN}_{M, K_0, d}^{\mathcal{B}}
  \subset \mathcal{NN}_{M_0, K_0, d}^{\mathcal{B}}$,
  so that
  $\Realization_\varrho (\Phi_f)
   \in \mathcal{NN}_{M_0, K_0, d}^{\mathcal{B}, \varrho}$.

  \pp{Let $\ell := C_1 \cdot M_0 \cdot (K_0 + \lceil \log_2 M_0 \rceil)$,} and recall from above (or from Lemma \ref{lem:NNEncoding}) that there is an
  injection
  $\Gamma : \mathcal{NN}_{M_0, K_0,d}^{\mathcal{B},\varrho} \to \{0,1\}^\ell$.
  Therefore, there is a left inverse
  $\Lambda : \{0,1\}^\ell \to \mathcal{NN}_{M_0, K_0, d}^{\mathcal{B}, \varrho}$
  for $\Gamma$. Using these \Felix{two maps}, we can now define an
  encoder-decoder pair for \Felix{the class} $\mathcal{HF}_{\beta,d,B}$,
  as follows:
  \[
    \begin{alignedat}{3}
      & E^\ell : &  & \mathcal{HF}_{\beta,d,B} \to \left\{ 0,1\right\} ^{\ell},
                 &  & f \mapsto \Gamma
                                \left(\Realization_\varrho (\Phi_{f}) \right),\\
      & D^\ell:
      & \ & \left\{ 0,1\right\} ^{\ell}
            \to L^{\Felix{p}}(\left[-\nicefrac{1}{2},\,\nicefrac{1}{2}\right]^{d}),
      & \ & c \mapsto \left[
                        \Lambda\left(c\right)
                      \right]|_{\left[-\nicefrac{1}{2},
                                       \nicefrac{1}{2}\right]^{d}} .
    \end{alignedat}
  \]
  With this definition, we have
  \[
    D^\ell (E^\ell (f))
    = [\Lambda (\Gamma ( \Realization_\varrho (\Phi_f) ))]
       |_{[- \nicefrac{1}{2}, \nicefrac{1}{2}]^d}
    = \Realization_\varrho (\Phi_f) |_{[-\nicefrac{1}{2}, \nicefrac{1}{2}]^d},
  \]
  and thus $\| f - D^\ell (E^\ell (f)) \|_{L^{\Felix{p}}} \leq \varepsilon$
  for all $f \in \mathcal{HF}_{\beta,d,B}$. By definition of the minimax code
  length $L_{\Felix{p}}(\varepsilon, \mathcal{HF}_{\beta,d,B})$, this implies
  \begin{equation}
    \ell
    \geq L_{\Felix{p}} (\varepsilon, \mathcal{HF}_{\beta,d,B})
    \geq C_2 \cdot \varepsilon^{-\Felix{p} (d-1)/\beta}.
    \label{eq:OptimalityEllLowerBound}
  \end{equation}

  In the remainder of the proof, we use elementary estimates to derive a
  contradiction to the preceding estimate for $\ell$.
  First, recall $\varepsilon < \nicefrac{1}{2}$, so that
  $\log_2 (\nicefrac{1}{\varepsilon}) \geq 1$,
  and hence
  $M_0 \leq C \cdot \varepsilon^{-\Felix{p} (d-1)/\beta}
       \leq \varepsilon^{-\Felix{p} (d-1)/\beta}$,
  which implies
  $\log_2 M_0 \leq \frac{\Felix{p} (d-1)}{\beta}
                   \cdot \log_2 (\nicefrac{1}{\varepsilon})
   \leq \frac{\Felix{p} d}{\beta} \cdot \log_2 (\nicefrac{1}{\varepsilon})$.
  Therefore, we get
  \[
    K_0 + \lceil \log_2 M_0 \rceil
    \leq 1
         + C_0 \cdot \log_2 \left(\frac{1}{\varepsilon}\right)
         + \left\lceil
             \frac{\Felix{p} d}{\beta}
             \cdot \log_2 \left(\frac{1}{\varepsilon}\right)
           \right\rceil
    \leq \left(2 + C_0 + \frac{\Felix{p} d}{\beta}\right)
         \cdot \log_2 \left(\frac{1}{\varepsilon}\right) \, ,
  \]
  where the last step used again that
  $\log_2 (\nicefrac{1}{\varepsilon})\geq 1$.
  All in all, recalling the definitions of $M_0$, \Felix{of $\ell$},
  and of $C$, we see
  \[
    \ell
    = C_1 \cdot M_0 \cdot (K_0 + \lceil \log_2 M_0 \rceil)
    \leq C_1 \cdot M_0 \cdot \left(2 + \frac{\Felix{p} d}{\beta} + C_0\right)
                       \cdot \log_2 \left(\frac{1}{\varepsilon}\right)
    \leq \frac{C_2}{2} \cdot \varepsilon^{-\Felix{p} (d-1)/\beta},
  \]
  which yields the desired contradiction, once we recall
  Equation \eqref{eq:OptimalityEllLowerBound}.
\end{proof}

\subsection{Lower bounds for the setting of instance optimality}

In the previous \Felix{sub}section, we showed (up to log factors) that
$M_{\varepsilon,\Felix{p}}^{\mathcal{B}, \varrho, C_0} (\mathcal{HF}_{\beta, d, B})
\gtrsim \varepsilon^{-\Felix{p} (d-1)/\beta}$.
Here, the quantity
$M = M_{\varepsilon,\Felix{p}}^{\mathcal{B}, \varrho, C_0} (\mathcal{HF}_{\beta, d, B})$
is the minimal $M \in \N$ such that \emph{every}
$f \in \mathcal{HF}_{\beta, d, B}$ can be approximated up to an
$L^{\Felix{p}}$-error of at most $\varepsilon$ using a neural network with
$M$ nonzero weights (and such that \Felix{each of these} weights can be encoded
with at most $\lceil C_0 \cdot \log_2 (\nicefrac{1}{\varepsilon}) \rceil$ bits,
using the encoding scheme $\mathcal{B}$).

In this section, we show that a similar lower bound holds if one is
interested in \Felix{approximating} a \emph{single} (judiciously chosen)
function $f \in \mathcal{HF}_{\beta, d, B}$, \Felix{and} not just if one is
interested in a uniform approximation over the whole class of horizon functions.

The proof idea is somewhat similar to the one that was used for the lower
bounds in the uniform setting:
We first obtain a lower bound regarding encoder-decoder pairs which achieve a
small $L^1$-error over the class $\mathcal{F}_{\beta, \Felix{d}, B}$, and then
we use the map $\gamma \mapsto \mathrm{HF}_\gamma$ to transfer
the result to the class of horizon functions.

Thus, our first step is the following lemma which uses Baire's category
theorem to ``upgrade'' the lower bound regarding encoder-decoder pairs with
\emph{uniform} error control to a lower bound concerning
encoder-decoder pairs with \emph{non-uniform} error control.

\begin{lemma}\label{lem:IndividualCnApproximation}
Let $d \in \N$ and $\beta, B > 0$ be arbitrary, and write
$\beta = n + \sigma$ with $n \in \N_0$ and $\sigma \in (0,1]$.
Define
\[
  X := \{
         u \in C^n ([-\nicefrac{1}{2}, \nicefrac{1}{2}]^d)
         \,:\,
         \| u \|_{C^{0,\beta}} \leq B < \infty
       \}.
\]
Let $\phi : \N \to (0,\infty)$ be arbitrary with
$\lim_{\ell \to \infty} \ell^{\beta/d} \cdot \phi (\ell) = 0$.
Finally, let $I \subset \N$ be infinite, and for each $\ell \in I$,
let $E^\ell : X \to \{0,1\}^\ell$ and
$D^\ell : \{0,1\}^\ell \to L^1 ([-\nicefrac{1}{2}, \nicefrac{1}{2}]^d)$
be arbitrary maps.

Then there is some $u \in X$, such that the sequence
$\big(
   \| u - D^\ell (E^\ell (u))\|_{L^1} \, \big/ \, \phi (\ell)
 \big)_{\ell \in I}$
is unbounded.
\end{lemma}

\begin{proof}
We assume towards a contradiction that the claim is false.
This means
\begin{equation}
  \forall \, u \in X:
          \left(
              \| u - D^\ell (E^\ell (u))\|_{L^1} \, \big/ \, \phi (\ell)
          \right)_{\ell \in I}
          \text{ is a bounded sequence}.
  \label{eq:IndividualCnApproximationContradictionAssumption}
\end{equation}
In the following, we consider the Banach space
\[
  C^{0,\beta} \left(\left[-\nicefrac{1}{2}, \nicefrac{1}{2}\right]^d\right)
  := \left\{
       u \in C^n \left(\left[-\nicefrac{1}{2},\nicefrac{1}{2}\right]^d\right)
       \,:\,
       \| u \|_{C^{0,\beta}} < \infty
     \right\},
\]
that is, all balls $B_\delta (u)$ or $\overline{B}_\delta (u)$ for
$u \in C^{0,\beta}$, and all closures $\overline{M}$ for $M \subset C^{0,\beta}$,
are to be understood with respect to the $\|\cdot\|_{C^{0,\beta}}$ norm.

We divide the proof into three steps.

\textbf{Step 1}: For $N \in \N$, let us set
\[
  G_{N}
  := \left\{
       u \in X
       \, : \,
       \forall \, \ell \in I :
          \| u - D^\ell (E^\ell (u)) \|_{L^1}
          \leq N \cdot \phi (\ell)
     \right\} \, .
\]
In this step, we show that there is some $N \in \N$ and certain $\delta > 0$ and
$u_0 \in X$ with
\begin{equation}
  \overline{B_\delta} (u_0)  \subset \overline{G_{N}}.
  \label{eq:IndividualCnApproximationStep1Result}
\end{equation}

To see this, first note that
Equation \eqref{eq:IndividualCnApproximationContradictionAssumption} simply says
$X = \bigcup_{N \in \N} G_{N}$.
But $X$ is a closed \Felix{subset} of the Banach space
$C^{0,\beta} ([-\nicefrac{1}{2}, \nicefrac{1}{2}]^d)$,
and thus a complete metric space. Therefore, the Baire category theorem
(see e.g.\@ \cite[Theorem 5.9]{FollandRA}) shows that
at least one of the \pp{$\overline{G_{N}}$} has nonempty interior
(with respect to $X$). In other words, Baire's theorem ensures \Felix{the}
existence of some $N \in \N$ and of $\delta_0 \in (0,1)$ and $v_0 \in X$
such that \pp{$X \cap B_{\delta_0} (v_0) \subset \overline{G_{N}}$}, where the ball $B_{\delta_0} (v_0)$ and the closure $\overline{G_{N}}$
are both formed with respect to the norm $\|\cdot\|_{C^{0,\beta}}$.

Now, set $u_0 := (1 - \delta_0 / (1+B)) \cdot v_0$ and note
\[
  \|u_0\|_{C^{0,\beta}}
  =    (1 - \delta_0 / (1+B)) \cdot \| v_0 \|_{C^{0,\beta}}
  \leq (1 - \delta_0 / (1+B)) \cdot B
  <    B,
\]
as well as
$\| u_0 - v_0 \|_{C^{0,\beta}}
 = \frac{\delta_0}{1+B} \cdot \| v_0 \|_{C^{0,\beta}}
 < \delta_0$.
These two properties easily imply that there is some $\delta > 0$ with
$\overline{B_\delta} (u_0) \subset X \cap B_{\delta_0} (v_0)$.
Because of $X \cap B_{\delta_0} (v_0) \subset \overline{G_{N}}$, this
establishes Equation \eqref{eq:IndividualCnApproximationStep1Result}.

\smallskip{}

\textbf{Step 2}: For brevity, set
\[
  Y := \overline{B_\delta} (0)
     = \{
         u \in C^{0,\beta} ([-\nicefrac{1}{2}, \nicefrac{1}{2}]^d)
         \,:\,
         \| u \|_{C^{0,\beta}} \leq \delta
       \}
     = \mathcal{F}_{\beta,d,\delta},
\]
where the notation $\mathcal{F}_{\beta,d,\delta}$ is as in
Equation \eqref{eq:SmoothFunctionClassDefinition}.
Our goal in this step is for each $\ell \in I \subset \N$ to construct modified
maps
$\widetilde{D^\ell} : \{0,1\}^\ell \to L^1 ([-\nicefrac{1}{2}, \nicefrac{1}{2}]^d)$
and $\widetilde{E^\ell} : Y \to \{0,1\}^\ell$ which satisfy
\begin{equation}
  \| u - \widetilde{D^\ell} (\widetilde{E^\ell} (u)) \|_{L^1}
  \leq N \cdot \phi (\ell)
  \qquad \text{ for all } u \in Y \text{ and all } \ell \in I.
  \label{eq:IndividualCnApproximationStep2Result}
\end{equation}

To this end, define
\[
  \widetilde{D^\ell} :
  \{0,1\}^\ell \to     L^1([-\nicefrac{1}{2}, \nicefrac{1}{2}]^d), \
  c            \mapsto D^\ell (c) - u_0.
\]
Now, since $\{0,1\}^\ell$ is finite, there is for each $u \in Y$ a certain
(not necessarily unique) coefficient sequence $c_u \in \{0,1\}^\ell$ with
\[
  \| u - \widetilde{D^\ell} (c_u) \|_{L^1}
  = \min_{c \in \{0,1\}^\ell} \| u - \widetilde{D^\ell} (c) \|_{L^1} \, .
\]
With this choice of $c_u$, we define
$\widetilde{E^\ell} : Y \to \{0,1\}^\ell, u \mapsto c_u$.
To prove Equation \eqref{eq:IndividualCnApproximationStep2Result}, recall for
$u \in Y$ from Step 1 that
$u + u_0 \in \overline{B_\delta}(u_0) \subset \overline{G_{N}}$.
Thus, there is a sequence $(u_k)_{k \in \N}$ in $G_{N}$
with $\| (u + u_0) - u_k \|_{C^{0,\beta}} \to 0$ as $k \to \infty$.
In particular, we get
\begin{align*}
  \| u - \widetilde{D^\ell} (\widetilde{E^\ell} (u))\|_{L^1}
  & =    \min_{c \in \{0,1\}^\ell}
            \| u - \widetilde{D^\ell} (c) \|_{L^1}
    \leq \| u - \widetilde{D^\ell} (E^\ell (u_k)) \|_{L^1} \\
  & =    \| (u + u_0) - D^\ell (E^\ell (u_k)) \|_{L^1} \\
  & \leq \| (u + u_0) - u_k \|_{L^1}
         + \| u_k - D^\ell (E^\ell (u_k)) \|_{L^1} \\
  ({\scriptstyle{\text{since } u_k \in G_{N}
                 \text{ and } \|\cdot\|_{L^1 ([-\nicefrac{1}{2},\nicefrac{1}{2}]^d)}
                              \leq \|\cdot\|_{C^{0,\beta}}}})
  & \leq \| (u + u_0) - u_k \|_{C^{0,\beta}} + N \cdot \phi (\ell)
    \xrightarrow[k \to \infty]{} N \cdot \phi (\ell),
\end{align*}
which is precisely what was claimed in
\eqref{eq:IndividualCnApproximationStep2Result}.

\smallskip{}

\textbf{Step 3}: In this step, we complete the proof.
To this end, recall from Step 1 of the proof of
Lemma \ref{lem:HorizonFunctionsMinimaxCodeLength} that there are constants
$C = C(\beta,d,\delta) > 0$ and
$\varepsilon_0 = \varepsilon_0 (\beta, d, \delta) > 0$
such that for every $\varepsilon \in (0,\varepsilon_0)$, there is some
$N \geq \exp(C \cdot \varepsilon^{-d/\beta})$ and certain functions
$u_1, \dots, u_N \in Y = \mathcal{F}_{\beta,d,\delta}$ satisfying
$\| u_i - u_j \|_{L^1} \geq \varepsilon$ for $i \neq j$.

\smallskip{}

We now apply this for every fixed, sufficiently large $\ell \in I$ with the
choice $\varepsilon = (C^{-1} \cdot \ell)^{-\beta/d}$.
Note that we indeed have $\varepsilon \in (0,\varepsilon_0)$, once $\ell$ is
large enough, which we always assume in the following;
since $I \subset \N$ is infinite, there exist arbitrarily large $\ell \in I$.
As just seen, there is
$N \geq \exp(C \cdot [(C^{-1} \cdot \ell)^{-\beta/d}]^{-d/\beta}) = e^\ell$,
and certain functions $u_1,\dots, u_N \in Y$ with
$\| u_i - u_j \|_{L^1} \geq \varepsilon = (C^{-1} \cdot \ell)^{-\beta/d}$
for $i \neq j$.

\smallskip{}

Because of $N \geq e^\ell > 2^\ell = |\{0,1\}^\ell|\vphantom{\sum_j}$,
the pigeonhole principle shows that there are $i, j \in \{1,\dots, N\}$
with $i \neq j$, but such that
$\widetilde{E^\ell} (u_i) = \widetilde{E^\ell} (u_j)$.
In view of Equation \eqref{eq:IndividualCnApproximationStep2Result},
this implies
\begin{align*}
  (C^{-1} \cdot \ell)^{-\beta/d}
  & \leq \| u_i - u_j \|_{L^1} \\
  & \leq \| u_i - \widetilde{D^\ell} (\widetilde{E^\ell} (u_i))\|_{L^1}
         + \| \widetilde{D^\ell} (\widetilde{E^\ell} (u_i))
              - \widetilde{D^\ell} (\widetilde{E^\ell} (u_j))\|_{L^1}
         + \| \widetilde{D^\ell} (\widetilde{E^\ell} (u_j)) - u_j \|_{L^1} \\
  & \leq N \cdot \phi (\ell) + 0 + N \cdot \phi (\ell).
\end{align*}
By rearranging, and by our assumption on $\phi$, this implies
\[
  \frac{C^{\beta/d}}{2N}
  \leq \ell^{\beta/d} \cdot \phi (\ell)
  \xrightarrow[\ell \in I, \ell \to \infty]{} 0 ,
\]
which is the desired contradiction, since the left-hand side is positive and
independent of $\ell$.
Note that we again used that $I$ is infinite to ensure that the limit
$\ell \in I, \ell \to \infty$ makes sense.
\end{proof}

\pp{Our next result transfers the lower bound of the previous lemma to the class
$\mathcal{HF}_{\beta,d,B}$ of horizon functions.}

\begin{lemma}\label{lem:IndividualHorizonApproximation}
  Let $d \in \N_{\geq 2}$ and $\Felix{p}, \beta, B > 0$ be arbitrary.
  Furthermore, let $\vartheta : \N \to (0,\infty)$ be arbitrary with
  $\lim_{\ell \to \infty}
     \ell^{\beta/(\Felix{p}(d-1))} \cdot \vartheta (\ell)
   = 0$.
  Finally, let $I \subset \N$ be infinite, and for each
  $\ell \in I$ let $E^\ell : \mathcal{HF}_{\beta,d,B} \to \{0,1\}^\ell$ and
  $D^\ell : \{0,1\}^\ell \to L^{\Felix{p}} ([-\nicefrac{1}{2}, \nicefrac{1}{2}]^d)$
  be arbitrary.

  Then there is some $f \in \mathcal{HF}_{\beta,d,B}$ such that the sequence \pp{
  $
    \big(
        \| f - D^\ell (E^\ell (f))\|_{L^{\Felix{p}}} \, \big/ \, \vartheta (\ell)
    \big)_{\ell \in I}
  $}
  is unbounded.
\end{lemma}

\begin{proof}
  Write $\beta = n + \sigma$ with $n \in \N_0$ and $\sigma \in (0,1]$.

  \textbf{Step 1}:
  We show for arbitrary $C > 0$ that the set
  \[
    K_C
    := \{
           f \in C^{n} ([-\nicefrac{1}{2}, \nicefrac{1}{2}]^{d-1})
           \, : \,
           \| f \|_{C^{0,\beta}} \leq C
        \}
  \]
  is a compact subset of $L^1 ([-\nicefrac{1}{2}, \nicefrac{1}{2}]^{d-1})$.
  To see this, let $(f_k)_{k \in \N}$ be an arbitrary sequence in $K_C$.
  Then, for each $\alpha \in \N_0^{d-1}$ with $|\alpha| < n$, we have
  \[
    \Lip_1 (\partial^\alpha f_k)
    \leq \| \nabla (\partial^\alpha f_k) \|_{L^\infty}
    \leq \sum_{j=1}^{d-1} \| \partial^{\alpha + e_j} f_k \|_{L^\infty}
    \leq (d-1) \cdot \| f_k \|_{C^{0,\beta}} \leq d \cdot C,
  \]
  and for $\alpha \in \N_0^{d-1}$ with $|\alpha| = n$, we have
  $\Lip_\sigma (\partial^\alpha f_k) \leq \| f_k \|_{C^{0,\beta}} \leq C$,
  where we emphasize that $\sigma > 0$.
  Furthermore, for $|\alpha| \leq n$ arbitrary, we have
  $\|\partial^\alpha f_k\|_{L^\infty} \leq \| f_k \|_{C^{0,\beta}} \leq C$.

  We have thus shown that each of the sequences
  $(\partial^\alpha f_k)_{k \in \N}$, for $|\alpha| \leq n$, is uniformly
  bounded and equicontinuous.
  By the Arzela-Ascoli theorem (see e.g.\@ \cite[Theorem 4.44]{FollandRA}),
  there is thus a common subsequence $(f_{k_t})_{t \in \N}$ such that
  $(\partial^\alpha f_{k_t})_{t \in \N}$ converges uniformly to a continuous
  function $g_\alpha \in C([-\nicefrac{1}{2}, \nicefrac{1}{2}]^{d-1})$
  for each $\alpha \in \N_0^{d-1}$ with $|\alpha|\leq n$.

  It is now a standard result (see for example
  \cite[Theorem 9.1 in XIII, §9]{LangRealFunctional}) that $f := g_0$ satisfies
  $f \in C^n ([-\nicefrac{1}{2}, \nicefrac{1}{2}]^{d-1})$ with
  $\partial^\alpha f = g_\alpha$ for $\alpha \in \N_0^{d-1}$ with
  $|\alpha| \leq n$. In particular, $f_{k_t} \to g_0 = f$ uniformly,
  and thus also in $L^1 ([-\nicefrac{1}{2}, \nicefrac{1}{2}]^{d-1})$.
  Thus, to prove \Felix{the} compactness of
  $K_C \subset L^1([-\nicefrac{1}{2}, \nicefrac{1}{2}]^{d-1})$,
  it suffices to show $f \in K_C$.
  But for $\alpha \in \N_0^{d-1}$ with $|\alpha| \leq n$, we have
  $\| \partial^\alpha f \|_{L^\infty} = \| g_\alpha \|_{L^\infty}
  = \lim_{t \to \infty} \|\partial^\alpha f_{k_t}\|_{L^\infty} \leq C$.
  Finally, for $|\alpha| = n$, and arbitrary
  $x,y \in [-\nicefrac{1}{2}, \nicefrac{1}{2}]^{d-1}$, we have
  \begin{align*}
          |\partial^\alpha f(x) - \partial^\alpha f(y)|
    &=    |g_\alpha (x) - g_\alpha (y)|
    =     \lim_{t \to \infty}
            |\partial^\alpha f_{k_t} (x) - \partial^\alpha f_{k_t}(y)| \\
    &\leq \limsup_{t \to \infty}
            \Lip_\sigma (\partial^\alpha f_{k_t}) \cdot |x-y|^\sigma
    \leq  \sup_{k \in \N}
            \|f_k\|_{C^{0,\beta}} \cdot |x-y|^\sigma
    \leq C \cdot |x-y|^\sigma.
  \end{align*}
  Therefore, $\Lip_\sigma (\partial^\alpha f) \leq C < \infty$.
  All in all, we have thus verified $\| f \|_{C^{0,\beta}} \leq C$, that is,
  $f \in K_C$.

  \smallskip{}

  \textbf{Step 2}: We observe with Lemma \ref{lem:HFBounds} that
  \[
    \Lambda :
    L^1 \left(\left[-\nicefrac{1}{2}, \nicefrac{1}{2}\right]^{d-1}\right)
    \to L^{\Felix{p}} \left(\left[-\nicefrac{1}{2}, \nicefrac{1}{2}\right]^d\right),
    \gamma \mapsto \mathrm{HF}_\gamma
    \,\, ,
  \]
  with $\mathrm{HF}_\gamma$ as in Lemma \ref{lem:HFBounds}, is continuous.

  \smallskip{}

  \textbf{Step 3}: Let $B_0 := \min \{B, \nicefrac{1}{2} \}$
  \Felix{and $q := \max \{1,p^{-1}\}$}.
  In this step, we construct modified encoding-decoding pairs
  $(\widetilde{E^\ell}, \widetilde{D^\ell})$ with
  $\widetilde{E^\ell} : \mathcal{F}_{\beta, d-1, B_0} \to \{0,1\}^\ell$ and
  $\widetilde{D^\ell} : \{0,1\}^\ell \to L^1([-\nicefrac{1}{2}, \nicefrac{1}{2}]^{d-1})$ such that
  \begin{equation}
    \|
      \gamma - \widetilde{D^\ell} (\widetilde{E^\ell} (\gamma))
    \|_{L^1 ([-\nicefrac{1}{2}, \nicefrac{1}{2}]^{d-1})}
    \leq \Felix{2^{pq}} \cdot
         \|
           \mathrm{HF}_\gamma - D^\ell (E^\ell (\mathrm{HF}_\gamma))
         \|_{L^{\Felix{p}}([-\nicefrac{1}{2}, \nicefrac{1}{2}]^d)}^{\Felix{p}}
    \quad \text{ for all } \gamma \in \mathcal{F}_{\beta, d-1, B_0} \, .
    \label{eq:IndividualHorizonApproximationStep3}
  \end{equation}

  For the construction, first note from Steps 1 and 2 that there is for each
  $g \in L^{\Felix{p}} ([-\nicefrac{1}{2}, \nicefrac{1}{2}]^d)$ some
  (not necessarily unique) $\gamma_g \in K_{B_0}$ with
  $\| g - \mathrm{HF}_{\gamma_g} \|_{L^{\Felix{p}}}
   = \min_{\gamma \in K_{B_0}} \| g - \mathrm{HF}_\gamma \|_{L^{\Felix{p}}}$.
  Now, for each $c \in \{0,1\}^\ell$, \Felix{define}
  $\theta_c := \gamma_{D^\ell (c)} \in K_{B_0}
   \subset L^1 ([-\nicefrac{1}{2}, \nicefrac{1}{2}]^{d-1})$,
  \Felix{and observe}
  \begin{equation}
    \| D^\ell (c) - \mathrm{HF}_{\theta_c} \|_{L^{\Felix{p}}}
    = \min_{\gamma \in K_{B_0}}
        \| D^\ell (c) - \mathrm{HF}_\gamma \|_{L^{\Felix{p}}}
    \qquad \text{ for all } c \in \{0,1\}^\ell .
    \label{eq:IndividualHorizonApproximationThetaChoice}
  \end{equation}
  With this choice, let
  \[
    \widetilde{D^\ell} :
    \{0,1\}^\ell \to     L^1 ([-\nicefrac{1}{2}, \nicefrac{1}{2}]^{d-1}), \,
    c            \mapsto \theta_c \, .
  \]

  Now, since $\{0,1\}^\ell$ is finite, there is for each
  $\gamma \in \mathcal{F}_{\beta, d-1,B_0}$ some (not necessarily unique)
  $c_\gamma \in \{0,1\}^\ell$ with
  $\| \gamma - \widetilde{D^\ell} (c_\gamma) \|_{L^1}
   = \min_{c \in \{0,1\}^\ell} \| \gamma - \widetilde{D^\ell} (c) \|_{L^1}$.
  With this choice, set
  \[
    \widetilde{E^\ell} :
    \mathcal{F}_{\beta, d-1, B_0} \to     \{0,1\}^\ell, \,
    \gamma                        \mapsto c_\gamma .
  \]

  Now that we have constructed $\widetilde{E^\ell}, \widetilde{D^\ell}$,
  it remains to establish Equation \eqref{eq:IndividualHorizonApproximationStep3}.
  Recall from Lemma \ref{lem:HFBounds} that all
  $\gamma, \psi \in L^1 ([-\nicefrac{1}{2}, \nicefrac{1}{2}]^{d-1})$ with
  $\| \gamma \|_{\sup}, \|\psi\|_{\sup} \leq \nicefrac{1}{2}\vphantom{\sum_j}$
  satisfy
  $\| \mathrm{HF}_\gamma - \mathrm{HF}_\psi \|_{L^{\Felix{p}}}^{\Felix{p}}
   = \| \gamma - \psi \|_{L^1}$.
  Therefore, we get for arbitrary $\gamma \in \mathcal{F}_{\beta,d-1,B_0}$ with
  $c^{(0)} := E^\ell (\mathrm{HF}_\gamma)$ that
  \begin{align*}
    \| \gamma - \widetilde{D^\ell} (\widetilde{E^\ell} (\gamma))\|_{L^1}
    & =    \min_{c \in \{0,1\}^\ell}
              \| \gamma - \widetilde{D^\ell} (c) \|_{L^1}
      \leq    \| \gamma - \widetilde{D^\ell} (c^{(0)}) \|_{L^1}
      =       \| \gamma -  \theta_{c^{(0)}} \|_{L^1} \\
    \left({\scriptstyle{\substack{
                          \|\gamma\|_{\sup} , \| \theta_{c^{(0)}}\|_{\sup}
                          \leq B_0 \leq \nicefrac{1}{2} \\
                          \text{since } \gamma \in \mathcal{F}_{\beta,d-1,B_0}
                          \text{ and } \theta_{c^{(0)}} \in K_{B_0}}}}\right)
    & = \|
          \mathrm{HF}_\gamma - \mathrm{HF}_{\theta_{c^{(0)}}}
        \|_{L^{\Felix{p}}}^{\Felix{p}} \\
    ({\scriptstyle{\text{Equation } \eqref{eq:PseudoTriangleInequality}}})
    & \leq \left(
             \Felix{2^q \cdot \max}
             \left\{
                \| \mathrm{HF}_\gamma - D^\ell (c^{(0)})\|_{L^{\Felix{p}}} \, ,
                \|
                  D^\ell (c^{(0)}) - \mathrm{HF}_{\theta_{c^{(0)}}}
                \|_{L^{\Felix{p}}}
             \right\}
           \right)^{\Felix{p}} \\
    ({\scriptstyle{\text{Equation }
                       \eqref{eq:IndividualHorizonApproximationThetaChoice}}})
    & =    \left(
              \Felix{2^q \cdot \max}
              \left\{
                \| \mathrm{HF}_\gamma - D^\ell (c^{(0)})\|_{L^{\Felix{p}}} \, ,
                \min_{\psi \in K_{B_0}}
                  \| D^\ell (c^{(0)}) - \mathrm{HF}_{\psi} \|_{L^{\Felix{p}}}
              \right\}
           \right)^{\Felix{p}} \\
    ({\scriptstyle{\text{since } \gamma \in \mathcal{F}_{\beta,d-1,B_0}
                                        =   K_{B_0}}})
    & \leq  \left(
              2^{\Felix{q}}
              \cdot \| \mathrm{HF}_\gamma - D^\ell (c^{(0)})\|_{L^{\Felix{p}}}
            \right)^{\Felix{p}}
      =     \Felix{2^{pq}}
            \cdot \|
                    \mathrm{HF}_\gamma - D^\ell ( E^\ell (\mathrm{HF}_\gamma ))
                  \|_{L^{\Felix{p}}}^{\Felix{p}} \, .
  \end{align*}
  This completes the proof of
  Equation \eqref{eq:IndividualHorizonApproximationStep3}.

  \smallskip{}

  \textbf{Step 4}: In this step, we complete the proof.
  To this end, let us assume towards a contradiction that the claim fails.
  Thus, for \emph{every} $f \in \mathcal{HF}_{\beta,d,B}$, we have
  \[
    \| f - D^\ell (E^\ell (f))\|_{L^{\Felix{p}}}
    \leq C_{f} \cdot \vartheta (\ell)
    \qquad \text{ for all } \ell \in I,
  \]
  for a finite constant $C_{f} > 0$.

  By Step 3, this implies for $\phi := \vartheta^{\Felix{p}}$ and arbitrary
  $\gamma \in \mathcal{F}_{\beta,d-1,B_0}$ because of
  $\mathrm{HF}_\gamma \in \mathcal{HF}_{\beta,d,B_0}
                      \subset \mathcal{HF}_{\beta,d,B}$
  that
  \[
    \| \gamma - \widetilde{D^\ell} (\widetilde{E^\ell} (\gamma)) \|_{L^1}
    \leq \Felix{2^{pq}}
         \cdot \|
                 \mathrm{HF}_\gamma - D^\ell (E^\ell (\mathrm{HF}_\gamma))
               \|_{L^{\Felix{p}}}^{\Felix{p}}
    \leq \Felix{2^{pq}}
         \cdot C_{\mathrm{HF}_\gamma}^{\Felix{p}}
         \cdot (\vartheta (\ell))^{\Felix{p}}
    =    \Felix{2^{pq}}
         \cdot C_{\mathrm{HF}_\gamma}^{\Felix{p}}
         \cdot \phi (\ell)
    \qquad \text{ for all } \ell \in I\!.
  \]
  But by assumption on $\vartheta$, we have
  $\lim_{\ell \to \infty} \ell^{\beta/(d-1)} \cdot \phi (\ell)
   = \lim_{\ell \to \infty}
       \big(\ell^{\beta/[\Felix{p} (d-1)]} \cdot \vartheta (\ell)\big)^{\Felix{p}}
   = 0$,
  so that Lemma \ref{lem:IndividualCnApproximation} yields the desired
  contradiction.
\end{proof}

With the preceding lemma, we have shown that, given a sequence of
encoder-decoder pairs for the class of horizon functions, one can always find a
\emph{single function} which is not ``too well approximated'' by the sequence.
We now use this result to prove the claimed lower bound in the setting of
instance optimality.

\begin{proof}[Proof of Theorem \ref{thm:SingleFunctionOptimality}]
  \textbf{Step 1}: For technical reasons, we first need to study for fixed,
  but arbitrary $\nu > 0$ the monotonicity of the function
  \[
    \phi
    : (2, \infty) \to     (0,\infty),
      x           \mapsto \frac{x^\nu}{\log_2 (x) \cdot \log_2 (\log_2 (x))} .
  \]
  We claim that there is some $x_0 = x_0 (\nu)$ such that
  $\phi|_{[x_0, \infty)}$ is strictly increasing;
  since we clearly have $\phi (x) \to \infty$ as $x \to \infty$,
  we can then choose $x_0$ so that also $\phi(x_0) \geq 4$.

  To show \Felix{the} existence of $x_0$, first note from a direct computation
  that
  \[
    \phi ' (x)
    = x^{\nu - 1}
      \cdot \left[
              \nu \cdot \log_2 (x)
                  \cdot \log_2 (\log_2 (x))
              - \frac{\log_2 (\log_2 (x))}{\ln 2}
              - (\ln 2)^{-2}
            \right]
      \bigg/ [\log_2 (x) \cdot \log_2 (\log_2 (x))]^2.
  \]
  Here, the denominator is positive.
  Furthermore, the first term in the numerator dominates the other two terms
  for $x$ large enough.
  Therefore, $\phi '(x)$ is positive for $x$ large enough.
  This establishes the claim of Step 1.

  \smallskip{}

  \textbf{Step 2}: In this technical step, we construct quantities
  $\Omega_\varepsilon, K_\varepsilon, \ell_\varepsilon \in \N$
  for $0 < \varepsilon \leq \varepsilon_0$, for a certain
  $\varepsilon_0 \in (0, \nicefrac14]$, and use these quantities
  to define an infinite set $I \subset \N$.
  The relevance of these constructions will become apparent in Steps 3 and 4.

  Let $\phi, x_0$ be as in Step 1, with
  $\nu := \nicefrac{\Felix{p} (d-1)}{\beta}$.
  By possibly enlarging $x_0$, we can (and will) assume $x_0 \geq 4$.
  Set $\varepsilon_0 := x_0^{-1}$, choose the constant $C = C(d) \in \N$ as
  provided by Lemma \ref{lem:NNEncoding}, and let
  $C_1 = C_1 (d, \Felix{p}, \beta, C_0) \in \N$ with
  $C_1 \geq C_0^{-1} \cdot \left( 1 + \nicefrac{\Felix{p} (d-1)}{\beta} \right)$.
  Furthermore, set $C_2 := C \cdot (1 + C_1) \in \N$.

  Next, for $\varepsilon \in (0, \varepsilon_0]$, define
  \[
    \Omega_\varepsilon
    := \lceil \phi (\varepsilon^{-1}) \rceil
    =  \left\lceil
          \varepsilon^{-\Felix{p} (d-1)/\beta}
          \big/ \left[
                  \log_2 \left(\varepsilon^{-1}\right)
                  \cdot \log_2 \left(\log_2 \left(\varepsilon^{-1}\right)\right)
                \right]
       \right\rceil
    \in \N
  \]
  and
  $K_\varepsilon :=  \lceil C_0 \cdot \log_2 (\nicefrac{1}{\varepsilon}) \rceil
                 \in \N$,
  and set
  $\ell_\varepsilon := C_2 \cdot \Omega_\varepsilon \cdot K_\varepsilon \in \N$.

  First, note because of
  $0 < \varepsilon \leq \varepsilon_0 \leq \nicefrac{1}{4}$ that
  $\varepsilon^{-1} \geq 4$ and hence
  $\log_2 (\nicefrac{1}{\varepsilon}) \geq 2$ and
  $\log_2 (\log_2 (\nicefrac{1}{\varepsilon})) \geq 1$,
  as well as $\varepsilon^{-\Felix{p} (d-1)/\beta} \geq 1$,
  Hence,
  $\Omega_\varepsilon \leq \lceil \varepsilon^{-\Felix{p} (d-1)/\beta} \rceil
                      \leq 1 + \varepsilon^{-\Felix{p} (d-1)/\beta}
                      \leq 2 \cdot \varepsilon^{-\Felix{p} (d-1)/\beta}$,
  which implies
  \[
    \log_2 (\Omega_\varepsilon)
    \leq 1 + \frac{\Felix{p} (d-1)}{\beta}
             \cdot \log_2 \left(\frac{1}{\varepsilon}\right)
    \leq \left( 1 + \frac{\Felix{p} (d-1)}{\beta}\right)
         \cdot \log_2 \left(\frac{1}{\varepsilon}\right)
    \leq C_1 \cdot C_0 \cdot \log_2 \left(\frac{1}{\varepsilon}\right)
    \leq C_1 \cdot K_\varepsilon.
  \]
  Therefore,
  \begin{equation}
    C \cdot \Omega_\varepsilon
      \cdot (K_\varepsilon + \lceil \log_2 \Omega_\varepsilon \rceil)
    \leq C \cdot (1 + C_1) \cdot \Omega_\varepsilon \cdot K_\varepsilon
    =    \ell_\varepsilon.
    \label{eq:SingleFunctionOptimalityProofEllEpsilonLargeEnoughToEncode}
  \end{equation}

  Our last goal in this step is to show that the map
  $\ell_\varepsilon \mapsto (\Omega_\varepsilon, K_\varepsilon)$ is
  well-defined.
  To see this, first recall from Step 1, that if
  $0 < \varepsilon \leq \varepsilon'$, then
  $\Omega_\varepsilon \geq \Omega_{\varepsilon '}$.
  By contraposition, this shows that if
  $\Omega_\varepsilon < \Omega_{\varepsilon '}$, then
  $\varepsilon > \varepsilon'$ and hence $K_\varepsilon \leq K_{\varepsilon'}$,
  so that
  \[
    \ell_\varepsilon
    = C_2 \cdot \Omega_\varepsilon \cdot K_\varepsilon
    \leq C_2 \cdot \Omega_\varepsilon \cdot K_{\varepsilon'}
    < C_2 \cdot \Omega_{\varepsilon'} \cdot K_{\varepsilon'}
    = \ell_{\varepsilon'}.
  \]
  Again by contraposition, we have shown that
  $\Omega_\varepsilon = \Omega_{\varepsilon'}$ if
  $\ell_\varepsilon = \ell_{\varepsilon'}$.
  Even more, if $\ell_\varepsilon = \ell_{\varepsilon'}$, we just saw
  $\Omega_\varepsilon = \Omega_{\varepsilon'}$, but this also implies
  $K_\varepsilon = \ell_\varepsilon / (C_2 \Omega_\varepsilon)
                 = \ell_{\varepsilon'} / (C_2 \Omega_{\varepsilon'})
                 = K_{\varepsilon'}.$
  Hence, $I := \{\ell_\varepsilon \, : \, \varepsilon \in (0, x_0]\} \subset \N$
  is clearly an infinite set, and for $\ell = \ell_\varepsilon \in I$,
  it makes sense to write $\Omega_\varepsilon, K_\varepsilon$,
  since these quantities are independent of the precise choice of
  $\varepsilon \in (0, x_0]$ with $\ell = \ell_\varepsilon$.

  \smallskip{}

  \textbf{Step 3}: In this step, we define for each $\ell \in I$ a certain
  encoder-decoder pair $(E^\ell, D^\ell)$.
  More precisely, we recall from Lemma \ref{lem:NNEncoding} by our choice of
  $C = C(d)$ in Step 2 and because of
  Equation \eqref{eq:SingleFunctionOptimalityProofEllEpsilonLargeEnoughToEncode}
  that for each $\ell = \ell_\varepsilon \in I$,
  there is an injective function
  $\Gamma_{\ell} :
  \mathcal{NN}_{\Omega_\varepsilon, K_\varepsilon,d}^{\mathcal{B}, \varrho}
  \to \{0,1\}^\ell$.
  Let us fix some left-inverse
  $\Psi_\ell :
  \{0,1\}^\ell
  \to \mathcal{NN}_{\Omega_\varepsilon, K_\varepsilon, d}^{\mathcal{B},\varrho}$
  for $\Gamma_\ell$.

  Next, for each $\ell = \ell_\varepsilon \in I$ and each
  $f \in \mathcal{HF}_{\beta,d,B}$ we can use \Felix{the} finiteness of
  $\mathcal{NN}_{\Omega_\varepsilon, K_\varepsilon,d}^{\mathcal{B},\varrho}$
  (which follows from the injectivity of $\Gamma_\ell$)
  to choose a (not necessarily unique) neural network $\Phi_{f,\ell}
  \in \mathcal{NN}_{\Omega_\varepsilon, K_\varepsilon, d}^{\mathcal{B}}$
  which satisfies
  \pp{\[
    \| f - \Realization_{\varrho} (\Phi_{f,\ell}) \|_{L^{\Felix{p}}}
    = \min_{\Phi \in \mathcal{NN}^{\mathcal{B}}_{\Omega_\varepsilon,
                                                 K_\varepsilon, d}}
        \| f - \Realization_{\varrho} (\Phi) \|_{L^{\Felix{p}}}.
  \]}
  With this choice, we can finally define
  \[
    \begin{alignedat}{3}
      & E^\ell : &  & \mathcal{HF}_{\beta,d,B} \to \left\{ 0,1\right\} ^{\ell},
                 &  & f \mapsto \Gamma_\ell
                                \left(
                                  \Realization_\varrho (\Phi_{f, \ell})
                                \right),\\
      & D^\ell :
      & \ & \left\{ 0,1\right\}^{\ell}
            \to L^{\Felix{p}}(\left[-\nicefrac{1}{2},\,\nicefrac{1}{2}\right]^{d}),
      & \ & c \mapsto \left[ \Psi_\ell (c) \right]
                        |_{\left[-\nicefrac{1}{2},\nicefrac{1}{2}\right]^{d}} .
    \end{alignedat}
  \]
  Note that this definition implies
  \begin{equation}
    \|f - D^\ell (E^\ell (f))\|_{L^{\Felix{p}}}
    = \min_{\Phi \in \mathcal{NN}^{\mathcal{B}}_{\Omega_\varepsilon,
                                                 K_\varepsilon, d}}
        \| f - \Realization_{\varrho} (\Phi)  \|_{L^{\Felix{p}}}
    \qquad \text{ for all } \ell = \ell_\varepsilon \in I
           \text{ and }     f \in \mathcal{HF}_{\beta, d, B}.
    \label{eq:SingleFunctionOptimalityConstructedEncoderIsOptimal}
  \end{equation}

  \smallskip{}

  \textbf{Step 4}: In Step 5, we will invoke
  Lemma \ref{lem:IndividualHorizonApproximation} with
  \[
    \vartheta :
    \N \to (0,\infty),
    \ell \mapsto \ell^{-\beta/(\Felix{p} (d-1))}
                 \,\big/\,
                 [\log_2 (\log_2 (\max \{4 , \ell\}))]^{\beta/(\Felix{p} (d-1))}.
  \]
  As a preparation,  in this step, we derive some elementary estimates
  concerning $\Omega_\varepsilon, K_\varepsilon$ and $\ell_\varepsilon$,
  and then also for $\vartheta (\ell_\varepsilon)$.

  First, note for $\varepsilon \in (0, \varepsilon_0 ]$ because of
  $\varepsilon_0 \leq \nicefrac{1}{4}$
  that $\log_2 \left(\nicefrac{1}{\varepsilon}\right) \geq 2 \geq 1$, and hence
  \[
    K_\varepsilon
    = \left\lceil
        C_0 \cdot \log_2 \left(\frac{1}{\varepsilon}\right)
      \right\rceil
    \leq 1 + C_0 \log_2 \left(\frac{1}{\varepsilon}\right)
    \leq (1 + C_0) \cdot \log_2 \left(\frac{1}{\varepsilon}\right).
  \]
  Next, since
  $\phi(\varepsilon^{-1}) \geq \phi(\varepsilon_0^{-1}) = \phi(x_0) \geq 4$
  for $\varepsilon \in (0, \varepsilon_0]$,
  we have
  $\Omega_\varepsilon = \lceil \phi(\varepsilon^{-1}) \rceil
   \leq 1 + \phi(\varepsilon^{-1}) \leq 2 \phi(\varepsilon^{-1})$.
  All in all, this yields for a suitable constant
  $C_3 = C_3 (d, \Felix{p}, \beta, C_0) \in \N$ that
  \[
    \ell_\varepsilon
    =    C_2 \cdot \Omega_\varepsilon \cdot K_\varepsilon
    \leq 2 \cdot (1 + C_0) \cdot C_2
           \cdot \phi\left(\frac{1}{\varepsilon}\right)
           \cdot \log_2 \left(\frac{1}{\varepsilon}\right)
    = C_3 \cdot \varepsilon^{-\Felix{p} (d-1)/\beta}
          \cdot \left[
                  \log_2 \left(\log_2 \left(\frac{1}{\varepsilon}\right)\right)
                \right]^{-1} .
  \]
  Furthermore, because of $\log_2 (\log_2 (\nicefrac{1}{\varepsilon})) \geq 1$,
  we get for a suitable constant $C_4 = C_4 (d, \Felix{p}, \beta, C_0) \in \N$
  that
  \[
    \log_2 (\ell_\varepsilon)
    \leq \log_2 (C_3 \cdot \varepsilon^{-\Felix{p} (d-1)/\beta})
    = \log_2 (C_3)
      + \frac{\Felix{p} (d-1)}{\beta}
        \cdot \log_2 \left(\frac{1}{\varepsilon}\right)
    \leq C_4 \cdot \log_2 \left(\frac{1}{\varepsilon}\right),
  \]
  so that
  $\log_2 (\log_2 (\ell_\varepsilon))
   \leq \log_2 (C_4) + \log_2 (\log_2 (\nicefrac{1}{\varepsilon}))
   \leq C_5 \cdot \log_2 (\log_2 (\nicefrac{1}{\varepsilon}))$
  for some $C_5 = C_5 (d, \Felix{p}, \beta, C_0) > 0$.

  From the preceding estimates, because of
  $\ell_\varepsilon =    C_2 \cdot \Omega_\varepsilon \cdot K_\varepsilon
                    \geq \Omega_\varepsilon
                    \geq \phi(\varepsilon^{-1})
                    \geq 4$,
  and from the definition of $\vartheta$, we get a constant
  $C_6 = C_6 (d, \Felix{p}, \beta, C_0) > 0$ with
  \begin{equation}
    \begin{split}
            \frac{1}{\vartheta (\ell_\varepsilon)}
      &=    \ell_\varepsilon^{\beta/(\Felix{p} (d-1))}
            \cdot [
                   \log_2(\log_2(\ell_\varepsilon))
                  ]^{\beta / (\Felix{p} (d-1))} \\
      &\leq C_3^{\beta/(\Felix{p} (d-1))} \cdot \varepsilon^{-1} \cdot
            [
             \log_2 (\log_2 \left(\nicefrac{1}{\varepsilon}\right))
            ]^{-\beta / (\Felix{p} (d-1))}
            \cdot C_5^{\beta/(\Felix{p} (d-1))}
            \cdot [
                   \log_2 (\log_2 \left(\nicefrac{1}{\varepsilon}\right))
                  ]^{\beta / (\Felix{p} (d-1))} \\
      &=    C_6 \cdot \varepsilon^{-1}.
    \end{split}
    \label{eq:IndividualFunctionOptimalityThetaEstimate}
  \end{equation}

  \smallskip{}

  \textbf{Step 5}: Now, we complete the proof. First, we note
  \[
      \lim_{\ell \to \infty}
        \ell^{\beta / (\Felix{p} (d-1))} \cdot \vartheta (\ell)
    = \lim_{\ell \to \infty}
        [\log_2 (\log_2 (\max \{4, \ell\}))]^{-\beta / (\Felix{p} (d-1))}
    = 0,
  \]
  as required in Lemma \ref{lem:IndividualHorizonApproximation}.
  Hence, using that lemma, we obtain a horizon function
  $f \in \mathcal{HF}_{\beta,d,B}$ which satisfies
  \begin{align*}
    \infty
    &= \sup_{\ell_\varepsilon \in I}
           \frac{\|
                   f - D^{\ell_\varepsilon} (E^{\ell_\varepsilon} (f))
                 \|_{L^{\pp{p}}}}
                {\vartheta (\ell_\varepsilon)} \\
    ({\scriptstyle{\text{by Equations }
                   \eqref{eq:SingleFunctionOptimalityConstructedEncoderIsOptimal}
                   \text{ and }
                   \eqref{eq:IndividualFunctionOptimalityThetaEstimate}}})
    &\leq C_6 \cdot
          \sup_{ 0 < \varepsilon \leq \varepsilon_0}
            \left[
                \min
                \{
                  \| f - \Realization_{\varrho} (\Phi) \|_{L^{\Felix{p}}}
                  \, : \,
                  \Phi
                  \in \mathcal{NN}^{\mathcal{B}}_{\Omega_\varepsilon,
                                                  K_\varepsilon, d}
                \}
                \cdot \varepsilon^{-1}
            \right].
  \end{align*}

  For brevity, let us set
  $\delta_\varepsilon
  := \min
     \{
       \| f - \Realization_{\varrho} (\Phi) \|_{L^{\Felix{p}}}
       \,:\,
       \Phi \in \mathcal{NN}_{\Omega_\varepsilon, K_\varepsilon,d}^{\mathcal{B}}
     \}$.
  Then the preceding estimate yields a sequence $(\varepsilon_k)_{k \in \N}$
  with $0 < \varepsilon_k \leq \varepsilon_0 \leq \nicefrac{1}{4}$ and such that
  $\varepsilon_k^{-1} \cdot \delta_{\varepsilon_k} \geq 2k$ for all $k \in \N$.
  In particular, $\delta_{\varepsilon_k} > 0$.

  But for $0 < \varepsilon \leq \varepsilon' \leq \varepsilon_0$,
  we have $\Omega_\varepsilon \geq \Omega_{\varepsilon'}$ and
  $K_\varepsilon \geq K_{\varepsilon '}$, see also Step 2.
  Therefore, and since we require the encoding scheme
  $\mathcal{B} = (B_\ell)_{\ell \in \N}$ to be consistent, that is, to satisfy
  $\mathrm{Range}(B_\ell) \subset \mathrm{Range}(B_{\ell + 1})$ for all
  $\ell \in \N$, we have
  $\mathcal{NN}_{\Omega_\varepsilon, K_\varepsilon,d}^{\mathcal{B}}
   \supset \mathcal{NN}^{\mathcal{B}}_{\Omega_{\varepsilon'},
                                       K_{\varepsilon'},d}$,
  and thus $\delta_{\varepsilon} \leq \delta_{\varepsilon'}$.
  In particular, we get
  \[
    0 <    \varepsilon_k
      \leq \frac{\delta_{\varepsilon_k}}{2k}
      \leq \frac{\delta_{\varepsilon_0}}{2k}
      \xrightarrow[k \to \infty]{} 0.
  \]
  We have thus constructed the function $f \in \mathcal{HF}_{\beta, d, B}$
  and the sequence $(\varepsilon_k)_{k \in \N}$, so that it remains to show
  that these have the desired properties.
  To see this, pick any $k \in \N$, and let $M \in \N$ such that there exists
  $\Phi
   \in \mathcal{NN}^{\mathcal{B}}_{M,
                                   \lceil
                                     C_0 \log_2 (\nicefrac{1}{\varepsilon}_k)
                                   \rceil,
                                   d}
   = \mathcal{NN}_{M, K_{\varepsilon_k},d}^{\mathcal{B}}$
  with
  $\| f - \Realization_{\varrho} (\Phi)\|_{L^{\Felix{p}}} \leq \varepsilon_k$.
  Then we get
  \[
    \| f - \Realization_{\varrho} (\Phi)\|_{L^{\Felix{p}}}
    \leq \varepsilon_k
    \leq \frac{\delta_{\varepsilon_k}}{2k}
    < \delta_{\varepsilon_k}
    = \min
      \{
        \| f - \Realization_{\varrho} (\Psi) \|_{L^{\Felix{p}}}
        \,:\,
        \Psi \in \mathcal{NN}^{\mathcal{B}}_{\Omega_{\varepsilon_k},
                                             K_{\varepsilon_k},d}
      \}.
  \]
  But in case of $M \leq \Omega_{\varepsilon_k}$, we would have (as above) that
  $\Phi \in \mathcal{NN}_{M, K_{\varepsilon_k},d}^{\mathcal{B}}
   \subset \mathcal{NN}^{\mathcal{B}}_{\Omega_{\varepsilon_k},
                                       K_{\varepsilon_k}, d}$,
  which then yields a contradiction to the preceding inequality.
  Therefore, we must have
  \[
    M
    > \Omega_{\varepsilon_k}
    = \left\lceil
          \varepsilon_k^{-\Felix{p} (d-1)/\beta}
          \big/
          [\log_2 (1/\varepsilon_k) \cdot \log_2 (\log_2 (1/\varepsilon_k))]
      \right\rceil
    \geq \frac{\varepsilon_k^{-\Felix{p} (d-1)/\beta}}
              {\log_2 (1/\varepsilon_k)
               \cdot \log_2 (\log_2 (1/\varepsilon_k))}.
  \]
  Since $M \in \N$ was chosen arbitrarily, only subject to the restriction
  that there is
  $\Phi \in \mathcal{NN}^{\mathcal{B}}_{M,
                                        \lceil
                                          C_0 \log_2 (1/\varepsilon_k)
                                        \rceil,
                                        d}$
  with
  $\| f - \Realization_{\varrho} (\Phi)\|_{L^{\Felix{p}}} \leq \varepsilon_k$,
  this implies
  $M_{\varepsilon_k, \Felix{p}}^{\mathcal{B},\varrho,C_0}(f)
   > \Omega_{\varepsilon_k}$, as claimed.
\end{proof}

\section{Depth matters: Fast approximation needs deep networks}
\label{sec:DepthMatters}

In this section, we provide the proofs for the theorems from
Subsection \ref{sub:OptimalDepth}.
In the whole section, $\varrho$ will be the ReLU function
$\varrho (x) = \max \{0,x\}$, and all realizations \Felix{of networks}
are \Felix{computed} using this activation function.

The overall proof strategy in this section is heavily inspired by
Yarotsky \cite{YAROTSKY2017103}:
At first, we exclusively work in dimension $d=1$.
For this setting, we begin by establishing (in
Lemma \ref{lem:SquareApproxLowerBound}) a lower bound on the $L^p$ approximation
quality of affine-linear functions to the square function.
By locally approximating a nonlinear $C^3$ function by its Taylor polynomial
of degree two, this then implies (see Corollary \ref{cor:C3ApproxLowerBound})
a lower bound on the $L^p$ approximation quality of affine-linear functions to
nonlinear $C^3$ functions.

We then move to dimension $d > 1$ by saying that $g : \R^d \to \R$ is
\emph{$P$-piecewise slice affine} for some $P \in \N$ if each of the ``slices''
$t \mapsto g(x_0 + t v_0)$ for arbitrary $x_0, v_0 \in \R^d$ is piecewise
affine-linear with at most $P$ pieces.
By applying a ``Fubini-type argument'', we lift the one-dimensional lower
bounds to a lower bound for the $L^p$ approximation quality that can be achieved
for approximating a nonlinear, $d$-dimensional $C^3$ function using
$P$-piecewise slice affine functions, see
Proposition \ref{prop:PPiecewiseAffineFunctionsAreBad}.

We then complete the proof (see Theorem \ref{thm:DepthLowerBoundAppendix})
by invoking known results of Telgarsky \cite{TelgarskySmallPaper} which show
that realizations of ReLU neural networks are always $P$-piecewise slice affine,
for $P \lesssim N^L$, where $N$ is the number of neurons of the network, and $L$
is its depth.

The main difference to the results by Yarotsky \cite{YAROTSKY2017103}
is that Yarotsky considers approximation in $L^\infty$,
while we are interested in approximation in the $L^p$-sense, with $p < \infty$.
In this case, the reduction of the $d$-dimensional case to the one-dimensional
case is more involved; see the proof of
Proposition \ref{prop:PPiecewiseAffineFunctionsAreBad}.

Finally, we remark that shortly after the first version of the present article
appeared on the arXiv, we became aware of the paper \cite{pmlr-v70-safran17a}
and its longer arXiv version \cite{DepthWidthArxiv}, in which a result
very similar to ours was developed.
The main difference is that our approach works for approximation in $L^p$ for
arbitrary \Felix{$p \in (0,\infty)$}, while in
\cite{pmlr-v70-safran17a, DepthWidthArxiv} only the case $p=2$ is considered.
Furthermore, our proof is more elementary, since it does not rely on properties
of Legendre polynomials, which are used crucially in \cite{DepthWidthArxiv}.

\medskip{}

After this high-level overview, let us turn to the details:

\begin{lemma}\label{lem:SquareApproxLowerBound}
  \Felix{For each $p \in (0,\infty)$, there is}
  a constant $C_{0} = C_0 (p) > 0$ with the following property:
  For arbitrary $\alpha, a, b \in \R$ with $a < b$, we have
  \[
    \inf_{\beta,\gamma\in\R}
      \left\Vert
        \alpha\cdot x^{2}-\left(\beta x+\gamma\right)
      \right\Vert_{L^{p}\left(\left[a,b\right];dx\right)}
    \geq C_{0} \cdot \left| \alpha \right| \cdot (b-a)^{2+\frac{1}{p}}.
  \]
\end{lemma}
\begin{proof}
For $\alpha=0$, the claim is trivial.
Next, for $\alpha\neq0$, we have
\begin{align*}
  \inf_{\beta,\gamma\in\R}
    \left\Vert
      \alpha\cdot x^{2}-\left(\beta x+\gamma\right)
    \right\Vert_{L^{p}\left(\left[a,b\right];dx\right)}
   & =\left| \alpha \right|
      \cdot
      \inf_{\beta,\gamma\in\R}
        \left\Vert
          x^{2} - \left(\frac{\beta}{\alpha}x+\frac{\gamma}{\alpha}\right)
        \right\Vert_{L^{p}\left(\left[a,b\right];dx\right)}\\
   & =\left| \alpha \right|
      \cdot
      \inf_{\beta',\gamma'\in\R}
        \left\Vert
          x^{2} - \left(\beta'\cdot x+\gamma'\right)
        \right\Vert_{L^{p}\left(\left[a,b\right];dx\right)}.
\end{align*}
This easily shows that it suffices to consider the case $\alpha = 1$.

Next, let us consider the case $a=0$ and $b=1$.
\Felix{%
The space $V_{\pp{0}} := \spann \{1, \identity_{[0,1]} \} \subset L^p ([0,1])$
is finite-dimensional, and hence closed; see \cite[Theorem 1.21]{RudinFA}.
Since $f_0 : [0,1] \to \R, x \mapsto x^2$ satisfies
$f_0 \in L^p([0,1]) \setminus V_0$, there is thus some $C_0 = C_0 (p) > 0$
with $B_{C_0}^{\|\mybullet\|_{L^p}}(f_0) \subset L^p ([0,1]) \setminus V_0$,
that is, $\|x^2 - (\beta x + \gamma)\|_{L^p([0,1];dx)} \geq C_0$ for all
$\beta, \gamma \in \R$.
Because of $(b-a)^{2 + p^{-1}} = 1$ and $\alpha = 1$, this
proves the claim in case of $a = 0$ and $b = 1$.}

Finally, for the general case, first note by a straightforward application
of the change-of-variables formula that
$
 \left\Vert f \right\Vert_{L^{p}\left(\left[a,b\right]\right)}
 = \left(b-a\right)^{1/p}
   \cdot \left\Vert
           f \left(a+\left(b-a\right)y\right)
         \right\Vert_{L^{p}\left(\left[0,1\right];dy\right)}$
for measurable $f:\left[a,b\right]\to\R$.
Applied to our specific setting, this implies for arbitrary $\beta,\gamma\in\R$
that
\begin{align*}
  \left\Vert
    x^{2}-\left(\beta x+\gamma\right)
  \right\Vert_{L^{p}\left(\left[a,b\right];dx\right)}
  & = (b-a)^{\frac{1}{p}}
      \cdot \left\Vert
              \left(a + (b-a) y\right)^{2}
              - \left[\beta\left(a + (b-a) y\right) + \gamma\right]
            \right\Vert_{L^{p}\left([0,1];dy\right)} \\
  & = (b-a)^{2+\frac{1}{p}}
      \cdot \left\Vert
              y^{2}
              - \left[
                  \frac{\beta - 2a}{b - a} y
                  + \frac{\beta a + \gamma - a^{2}}{(b - a)^{2}}
                \right]
            \right\Vert_{L^{p}\left([0,1];dy\right)}
    \geq C_{0} \cdot (b-a)^{2+\frac{1}{p}}.
\end{align*}
As seen at the beginning of the proof, this yields the claim.
\end{proof}
The preceding lemma shows that affine-linear functions cannot approximate
the square function too well. By approximating $C^{3}$ functions
by their Taylor polynomial of degree $2$, this implies that $C^{3}$
functions with nonvanishing second derivative are not approximated
too well by linear functions. This is made precise by the following lemma:

\begin{lemma}\label{lem:C3ApproxLowerBoundUnitInterval}
  Let $f \in C^{3}\left([0,1]\right)$
  with $\left|f''(x)\right| \geq c > 0$ for all $x \in [0,1]$
  and with $\left\Vert f'''\right\Vert_{\sup} \leq C$, for some $C > 0$.
  Then, \Felix{for $p \in (0,\infty)$} and with $C_{0} \Felix{= C_0 (p)}$ as in
  Lemma \ref{lem:SquareApproxLowerBound}, we have
  \begin{align*}
    \inf_{\beta,\gamma\in\R}
      \left\Vert
        f(x) - (\beta x+\gamma)
      \right\Vert_{L^{p}\left([0,1];dx\right)}
    & \geq \min
           \left\{
             \frac{C_{0}}{\Felix{2^{1 + \max\{1,p^{-1}\}}}} \cdot c
             \,\, , \,\,
             \frac{C_{0}^{3}}{\Felix{2^{3 \max \{1,p^{-1}\}}}}
             \cdot \frac{c^{3}}{C^{2}}
           \right\} \\
    & \Felix{
      \geq 2^{-3 \max \{1,p^{-1}\}}
           \cdot \min \left\{
                        C_0 \, c
                        \,\,,\,\,
                        C_0^3 \cdot \frac{c^3}{C^2}
                      \right\}
           \, .
    }
  \end{align*}
\end{lemma}
\begin{proof}
\Felix{Let $q := \min \{1,p\}$ and} set
$N := \left\lceil \nicefrac{2^{\Felix{1/q}}}{(3C_{0})}
      \cdot \nicefrac{C}{c}\right\rceil
   \in \N$.
For $i \in \underline{N+1}$ let $x_{i}:=\nicefrac{(i-1)}{N} \in [0,1]$.
By Taylor's theorem, we know for each $i \in \underline{N}$ and
$x \in (x_{i},x_{i+1})$ that there is some
$\xi_{x} \in (x_{i},x) \subset (x_{i}, x_{i+1})$ with
\begin{align*}
  f(x)
  & =  f (x_{i})
     + f'(x_{i}) \cdot (x-x_{i})
     +\frac{f''(x_{i})}{2} \cdot (x-x_{i})^{2}
     +\frac{f'''(\xi_{x})}{6} \cdot (x-x_{i})^{3}\\
  & =  \frac{f'' \! (x_{i})}{2} \cdot x^{2}
     + x \cdot \left[
                 f' \! (x_{i})
                 - x_{i} \cdot f'' \! (x_{i})
               \right]
     + \! \left[\!
            f(x_{i})
            \!-\!  f' \! (x_{i}) \cdot x_{i}
            +      \frac{1}{2} \cdot f''(x_{i}) \cdot x_{i}^{2}
          \right] \!
     + \frac{f''' \! (\xi_{x})}{6} \cdot (x \! - \! x_{i})^{3} \\
   & =: \alpha_{i} \cdot x^{2}
        + \beta_{i} \cdot x
        + \gamma_{i}
        + \frac{f'''(\xi_{x})}{6} \cdot (x-x_{i})^{3}.
\end{align*}
Hence, since $\left| f'''(\xi_{x}) \right| \leq C$, we get
\begin{align*}
  \left\Vert
    f(x)
    \! - \!
    \left[ \alpha_{i} \cdot x^{2} + \beta_{i}\cdot x + \gamma_{i} \right]
  \right\Vert_{L^{p}\left([x_{i},x_{i+1}];dx\right)}
  & \leq \! (x_{i+1} \! - \! x_{i})^{\frac{1}{p}}
         \cdot \left\Vert
                 f(x)
                 \!-\!
                 \left[
                   \alpha_{i} \cdot x^{2} + \beta_{i} \cdot x + \gamma_{i}
                 \right]
               \right\Vert_{L^\infty ([x_i, x_{i+1}]; dx)}\\
   & \leq N^{-\frac{1}{p}} \cdot \frac{C}{6} \cdot N^{-3}
     =    \frac{C}{6} \cdot N^{-\left(3+\frac{1}{p}\right)}\:.
\end{align*}

\Felix{As noted before Equation \eqref{eq:PseudoTriangleInequality}, we have
$\|f+g\|_{L^p}^q \leq \|f\|_{L^p}^q + \|g\|_{L^p}^q$ for all $f,g \in L^p$,
so that $d(f,g) := \|f-g\|_{L^p}^q$ defines a metric on $L^p$.
The reverse triangle inequality for this metric shows
$\|f+g\|_{L^p}^q \geq \|f\|_{L^p}^q - \|g\|_{L^p}^q$.}
Therefore, by applying Lemma \ref{lem:SquareApproxLowerBound} and by noting
$\left|\alpha_{i}\right|=\left|\nicefrac{f''\left(x_{i}\right)}{2}\right|
 \geq\nicefrac{c}{2}$,
we get for arbitrary $\beta,\gamma\in\R$ and $1\leq i\leq N$ the
estimate
\begin{align*}
  & \left\Vert
      f(x) - (\beta x + \gamma)
    \right\Vert_{L^{p}\left([x_{i},x_{i+1}];dx\right)}^{\Felix{q}} \\
  & \geq \left\Vert
           \alpha_{i} \cdot x^{2}
           + \beta_{i}\cdot x
           + \gamma_{i} - (\beta x + \gamma)
         \right\Vert_{L^{p}\left([x_{i},x_{i+1}];dx\right)}^{\Felix{q}}
         - \left\Vert
             f(x)
             - \left[
                 \alpha_{i}\cdot x^{2} + \beta_{i}\cdot x + \gamma_{i}
               \right]
           \right\Vert_{L^{p}\left([x_{i},x_{i+1}];dx\right)}^{\Felix{q}} \\
  & \geq \Felix{
         \left( \frac{c}{2} \cdot C_{0}\right)^q
         \cdot N^{-\left(2+\frac{1}{p}\right)q}
         - \left(\frac{C}{6}\right)^q \cdot N^{-\left(3+\frac{1}{p}\right)q}
    = \left( \frac{c}{2} \cdot C_{0}\right)^q \cdot N^{-\left(2+\frac{1}{p}\right)q}
                  \cdot \left(
                          1 - \left( \frac{1}{3C_{0}}\cdot\frac{C}{c} \right)^q
                          \cdot N^{-q}
                        \right).
               }
\end{align*}
But \Felix{$\frac{1}{3 C_0} \frac{C}{c} N^{-1} \leq 2^{-1/q}\vphantom{\sum_j}$}
by choice of $N$, whence
$\left\Vert
   f(x) - (\beta x+\gamma)
 \right\Vert_{L^{p}\left([x_{i},x_{i+1}];dx\right)}
 \geq \Felix{\frac{c}{2} C_0 \cdot N^{-\left(2 + \frac{1}{p}\right)} \cdot 2^{-1/q}}$.
Therefore,\vspace{-0.3cm}
\[
  \left\Vert
    f(x) - (\beta x + \gamma)
  \right\Vert_{L^{p}\left([0,1];dx\right)}
  = \left[
      \sum_{i=1}^{N}
         \left\Vert
             f(x) - \left(\beta x+\gamma\right)
         \right\Vert_{L^{p}\left([x_{i},x_{i+1}];dx\right)}^{p}
    \right]^{\frac{1}{p}}
  \geq \frac{C_{0}}{\Felix{2^{1+q^{-1}}}} \cdot c \cdot N^{-2}.
\]

For brevity, set $\theta := \nicefrac{\Felix{2^{1/q}}}{(3C_{0})} \cdot \nicefrac{C}{c}$,
\Felix{so that $N = \lceil \theta \rceil$.}
There are now two cases: First, if $\theta<1$, then $N=1$, so that we get
$\left\Vert f(x) - (\beta x+\gamma) \right\Vert_{L^{p}\left([0,1];dx\right)}
 \geq \Felix{2^{-(1+q^{-1})} \cdot C_0 \, c}$;
that is, the claim is valid in this case.
Finally, if $\theta \geq 1$,
then $N=\left\lceil \theta \right\rceil \leq 1 + \theta \leq 2\theta$,
and thus\vspace{-0.2cm}
\[
  \left\Vert
    f(x) - \left(\beta x+\gamma\right)
  \right\Vert_{L^{p}\left([0,1];dx\right)}
  \geq \frac{C_{0}}{\Felix{2^{1+q^{-1}}}} \cdot c \cdot N^{-2}
  \geq \frac{C_{0}}{\Felix{2^{1+q^{-1}}}} \cdot c \cdot (2\theta)^{-2}
  =    \frac{C_{0}}{\Felix{2^{3+q^{-1}}}} \cdot c
       \cdot \left(
               \frac{3C_{0}}{\Felix{2^{1/q}}}\cdot\frac{c}{C}
             \right)^{2}
  \geq \frac{C_{0}^{3}}{\Felix{2^{3/q}}} \cdot \frac{c^{3}}{C^{2}},
\]
so that the claim also holds in this case.
\end{proof}

The next lemma generalizes the preceding estimate from the
interval $\left[0,1\right]$ to general intervals $\left[a,b\right]$.

\begin{corollary}\label{cor:C3ApproxLowerBound}
Let $c, C, \Felix{p} > 0$ be arbitrary and let $C_{0} \Felix{= C_0(p)} > 0$
as in Lemma \ref{lem:SquareApproxLowerBound}.
Further, let $a,b \in \R$ with
$0 < b-a \Felix{\leq C_0 \cdot \frac{c}{C}} \vphantom{\sum_j}$.

Then, each function $f\in C^{3}\left(\left[a,b\right]\right)$ with
$\left\Vert f'''\right\Vert_{\sup}\leq C$ and with
$\left|f''\left(x\right)\right|\geq c$ for all $x\in\left[a,b\right]$ satisfies
\[
  \inf_{\beta,\gamma\in\R}
    \left\Vert f(x) - (\beta x+\gamma)\right\Vert_{L^{p}\left([a,b];dx\right)}
  \geq \frac{C_{0}}{\Felix{2^{3 \max \{1,p^{-1}\}}}}
       \cdot c \cdot (b-a)^{2+\frac{1}{p}} \, .
\]
\end{corollary}
\begin{proof}
Define
\[
  \widetilde{f}:
  [0,1] \to     \R,
  x     \mapsto f\left(a + (b-a)x\right),
\]
and note $\widetilde{f}\in C^{3}\left(\left[0,1\right]\right)$ with
$\left|\smash{\widetilde{f}}''\left(x\right)\right|
 = (b-a)^{2} \cdot \left|f''\left(a + (b - a) x\right)\right|
 \geq (b-a)^{2}\cdot c
 =:c'$,
as well as
\[
       \left\Vert \smash{\widetilde{f}}'''\right\Vert _{\sup}
  =    (b-a)^{3} \cdot \left\Vert f'''\right\Vert_{\sup}
  \leq (b-a)^{3}\cdot C
  =: C'.
\]
By our assumptions on $a,b,c,C$, we then have
\[
  \left(
    \Felix{C_{0}^{3}} \cdot \frac{(c')^{3}}{(C')^{2}}
  \right)
  \!\!\bigg/\!\!
  \left( \Felix{C_0} \cdot c' \right)\!
  = \Felix{C_0^3}
    \cdot \frac{\left(b\!-\!a\right)^{6} \cdot c^{3}}
               {\left(b\!-\!a\right)^{6} \cdot C^{2}}
    \cdot \frac{\Felix{1}}{C_{0}}
    \cdot\frac{1}{(b\!-\!a)^{2} \cdot c}
  = \Felix{C_0^2}
    \cdot \frac{c^{2}}{C^{2}}
    \cdot\frac{1}{(b\!-\!a)^{2}}
  \Felix{=} \!
    \left(
      \Felix{C_0}
      \cdot \frac{c}{C}
      \cdot \frac{1}{b\!-\!a}
    \right)^{\! 2}
  \! \Felix{\geq} 1,
\]
so that Lemma \ref{lem:C3ApproxLowerBoundUnitInterval} shows
\begin{align*}
  \inf_{\beta,\gamma\in\R}
    \left\Vert
      \widetilde{f}\left(x\right)-\left(\beta x+\gamma\right)
    \right\Vert_{L^{p}\left([0,1];dx\right)}
  & \geq \Felix{2^{-3 \max \{1,p^{-1}\}}}
         \min
         \left\{
           \Felix{C_0} \cdot c' \,\,,\,\,
           \Felix{C_0^3} \cdot \frac{(c')^{3}}{(C')^{2}}
         \right\} \\
  & =   \Felix{\frac{C_0 \cdot c'}{2^{3 \max\{1,p^{-1}\}}}}
    =   \frac{C_{0}}{\Felix{2^{3 \max\{1,p^{-1}\}}}} \cdot c \cdot (b-a)^{2}.
\end{align*}

To complete the proof, we note from a direct application of the
change-of-variables formula for arbitrary $\beta,\gamma\in\R$ that
\begin{align*}
  \left\Vert
    f(x) - (\beta x + \gamma)
  \right\Vert_{L^{p}\left([a,b];dx\right)}
  & = (b-a)^{\frac{1}{p}}
      \cdot \left\Vert
              \widetilde{f}(y)
              - \left[
                  \beta\left((b - a) y + a\right) + \gamma
                \right]
            \right\Vert_{L^{p}\left([0,1];dy\right)} \\
  & \geq \frac{C_{0}}{\Felix{2^{3 \max \{1,p^{-1}\}}}}
         \cdot c \cdot (b-a)^{2+\frac{1}{p}}.
  \qedhere
\end{align*}
\end{proof}
Before we progress further, we introduce a convenient terminology:
\begin{definition}
\label{def:PiecewiseSliceAffine}Let $P\in\N$. A function $g:\R^{d}\to\R$
is called \textbf{$P$-piecewise slice affine} if for arbitrary
$x_{0},v\in\R^{d}$ the function
$g_{x_{0},v}:\R\to\R,t\mapsto g\left(x_{0}+tv\right)$
is piecewise affine-linear with at most $P$ pieces.
Precisely, this means that there are $-\infty=t_{0}<t_{1}<\dots<t_{P}=\infty$
such that $g_{x_{0},v}|_{\left(t_{i},t_{i+1}\right)}$ is affine-linear
for each $i\in\left\{ 0,\dots,P-1\right\} $.
\end{definition}
\begin{remark*}
Note that we allow $g_{x_0, v}$ to even be discontinuous at the
``break points'' $t_1, \dots, t_{P-1}$.
\end{remark*}
Our next result shows that if a $P$-piecewise slice affine function
approximates a nonlinear function $f\in C^{3}\left(\Omega\right)$
very well, then $P$ needs to be large. This result will then imply
that ReLU networks need to have a certain minimal depth in order to
achieve a given approximation rate for nonlinear functions, once we
show that if $g=\Realization_{\varrho}\left(\Phi\right)$,
then $g$ is $P$-piecewise slice affine for
$P \asymp \left[N\left(\Phi\right)\right]^{L\left(\Phi\right)}$.
Actually, we will not derive this claim from first principles, but
rather use existing results of Telgarsky \cite{TelgarskySmallPaper}.
But first, let us consider the case of general $P$-piecewise slice
affine functions:
\begin{proposition}\label{prop:PPiecewiseAffineFunctionsAreBad}
Let $\Omega\subset\R^{d}$ be nonempty, open, bounded and connected, and let
$f\in C^{3}\left(\Omega\right)$ be nonlinear,
that is, there do \emph{not} exist $y_{0}\in\R$ and $w\in\R^{d}$
with $f\left(x\right)=y_{0}+\left\langle w,x\right\rangle $ for all
$x\in\Omega$.
\Felix{Finally, let $p \in (0,\infty)$.}
Then there is a constant \Felix{$C = C(f,p) > 0$} with the following property:

If $g:\R^{d}\to\R$ is measurable and $P$-piecewise slice affine for some
$P \in \N$, then we have
\[
  \left\Vert f-g \right\Vert_{L^{p}\left(\Omega\right)}
  \geq \Felix{C} \cdot P^{-2} \, .
\]
\end{proposition}
\begin{proof}
Let ${\rm Hess}\,f=D\left(\nabla f\right)$ denote the Hessian of
$f$. If we had ${\rm Hess}\,f\equiv0$, then it would follow by standard
results of multivariable calculus (since $\Omega$ is connected) that
$\nabla f$ is constant, and then that
$f\left(x\right)
 =f(x_{0}) + \left\langle \nabla f(x_{0}) ,\, x - x_{0}\right\rangle$
for all $x\in\Omega$, where $x_{0}\in\Omega$ is fixed, but arbitrary.
Since $f$ is assumed nonlinear, this is impossible.

Hence, let $x_{0} \in \Omega$ with ${\rm Hess}\,f (x_{0}) \neq 0$.
Since $A:={\rm Hess}\,f(x_{0})$ is symmetric, the spectral
theorem shows that there is an orthonormal basis
$\left(b_{1},\dots,b_{d}\right)$ of $\R^{d}$ that consists of eigenvectors for
$A$, and at least one of these eigenvectors needs to correspond to a
nonzero eigenvalue;
by rearranging we can assume $A \, b_{d} = \lambda \cdot b_{d}$ for some
$\lambda\in\R\setminus\left\{ 0\right\} $.
Since $\Omega$ is open,
there is some $\varepsilon\in\left(0,\nicefrac{1}{2}\right)$ with
$\overline{B_{d\varepsilon}}\left(x_{0}\right)\subset\Omega$.
Since
$\left|\left\langle {\rm Hess} \, f(x_{0})b_{d} , b_{d}\right\rangle \right|
 = \left| \lambda \right| \neq 0$,
and since ${\rm Hess}\,f$ is continuous, we can possibly shrink $\varepsilon$
to achieve
$\left|\left\langle {\rm Hess} \, f(x)b_{d} ,\, b_{d}\right\rangle \right|
 \geq c := \nicefrac{\left|\lambda\right|}{2}$
for all $x\in\overline{B_{d\varepsilon}}\left(x_{0}\right)$.
Furthermore, since $\overline{B_{d\varepsilon}}(x_{0})\subset\Omega$
is compact, the constant
\[
  C := \max
       \left\{
          1, \,\,
          d^{3}
          \cdot \sup_{x\in\overline{B_{d\varepsilon}}\left(x_{0}\right)} \,\,
                  \max_{\left|\alpha\right|=3} \,\,
                    \left| \partial^{\alpha} f(x) \right|
       \right\}
\]
is finite. Finally, by again shrinking $\varepsilon$ (which can
at most \emph{shrink} $C$), we can assume
$2\varepsilon < \Felix{C_0 \cdot \frac{c}{C}}$,
where $C_{0} \Felix{= C_0 (p)}$ is the constant from
Lemma \ref{lem:SquareApproxLowerBound}.

Now, for $y=\left(y_{1},\dots,y_{d-1}\right)\in\R^{d-1}$ let us set
$z_{y}:=x_{0}+\sum_{i=1}^{d-1}y_{i}b_{i}$.
Note $z_{y}+t\cdot b_{d}\in\overline{B_{d\varepsilon}}\left(x_{0}\right)$
for all $y\in\left[-\varepsilon,\varepsilon\right]^{d-1}$ and
$t\in\left[-\varepsilon,\varepsilon\right]$.
Therefore, since $\left(b_{1},\dots,b_{d}\right)$ is an orthonormal
basis, an application of the change-of-variables formula and of Fubini's
theorem shows
\[
  \left\Vert f-g\right\Vert _{L^{p}\left(\Omega\right)}^{p}
  \geq \int_{B_{d\varepsilon}\left(x_{0}\right)}
         \left|f\left(x\right)-g\left(x\right)\right|^{p}
       dx
  \geq \int_{\left[-\varepsilon,\varepsilon\right]^{d-1}}
         \int_{-\varepsilon}^{\varepsilon}
           \left|
             f \left(z_{y} + t\cdot b_{d}\right)
             - g \left(z_{y} + t\cdot b_{d}\right)
           \right|^{p}
         \,dt
       \,dy.
\]
Note that the choice of $x_{0},\varepsilon,\left(b_{1},\dots,b_{d}\right)$
and $\lambda,c,C,C_{0}$ are all independent of $g$ and $P$.

Now, let $y\in\left[-\varepsilon,\varepsilon\right]^{d-1}$ be fixed,
but arbitrary. Since $g$ is $P$-piecewise slice affine, we know
that the map $g_{z_{y},b_{d}}:\R\to\R,t\mapsto g\left(z_{y}+t\cdot b_{d}\right)$
is piecewise affine-linear, with at most $P$ pieces, that is, there
is a partition $\R=\biguplus_{i=1}^{N}I_{y}^{(i)}$ (up to a null-set) into
open intervals $I_{y}^{(1)},\dots,I_{y}^{(N)}$ with $N\in\underline{P}$ such
that $g_{z_{y},b_{d}}$ is affine-linear on each $I_{y}^{(i)}$. Hence,
with $\lambda$ denoting the one-dimensional Lebesgue measure, we conclude
$2\varepsilon
 = \lambda \left(
             \biguplus_{i=1}^{N}
               (I_{y}^{(i)}\cap\left[-\varepsilon,\varepsilon\right])
           \right)
 = \sum_{i=1}^{N} \lambda(\left[-\varepsilon,\varepsilon\right]\cap I_{y}^{(i)})
 = \sum_{i=1}^N t_y^{(i)}$,
with $t_y^{(i)} := \lambda(I_y^{(i)} \cap [-\varepsilon,\varepsilon])$.
Since intersections of intervals are intervals again, we have
$ (a_y^{(i)}, b_y^{(i)}) \subset I_y^{(i)} \cap [-\varepsilon,\varepsilon]
                         \subset [a_y^{(i)}, b_y^{(i)}]$
for certain
$a_y^{(i)}, b_y^{(i)} \in \R$ with
$-\varepsilon \leq a_y^{(i)} \leq b_y^{(i)} \leq \varepsilon$ and
$t_y^{(i)} = b_y^{(i)} - a_y^{(i)}$.

Now, for each $i \in \underline{N}$, since $g_{z_y, b_d}$ is affine-linear on
$I_y^{(i)} \supset (a_y^{(i)}, b_y^{(i)})$,
there are certain \Felix{$\beta_{y}^{(i)}, \gamma_{y}^{(i)} \in \R$} with
\begin{align*}
  \int_{-\varepsilon}^{\varepsilon}
    \left|
      f \left(z_{y} + t\cdot b_{d}\right)
      - g \left(z_{y} + t\cdot b_{d}\right)
    \right|^{p}
  \,dt
  &= \sum_{i=1}^N
      \int_{a_y^{(i)}}^{b_y^{(i)}}
          | f( z_y + t \cdot b_d) - g(z_y + t \cdot b_d)|^p
      \, dt \\
  &= \sum_{i=1}^N
      \left\Vert
          f \left(z_{y} + t\cdot b_{d}\right)
          -\left(\beta_{y}^{(i)}t + \gamma_{y}^{(i)}\right)
      \right\Vert_{L^{p}([a_{y}^{(i)},b_{y}^{(i)}]; \, dt)}^{p}.
\end{align*}
But for $i \in \underline{N}$ with $t_y^{(i)} > 0$, and
\[
  f_{y}^{(i)}:[a_{y}^{(i)},b_{y}^{(i)}] \to     \R,
              t                         \mapsto f\left(z_{y}+t\cdot b_{d}\right),
\]
we have $f_{y}^{(i)} \in C^{3}([a_{y}^{(i)},b_{y}^{(i)}])\vphantom{\sum_j}$, and
$|(f_{y}^{(i)}) '' \! \left(t\right)| \!
 = \! \left|
        \left\langle
          {\rm Hess}\,f \! \left(z_{y}+t\cdot b_{d}\right)\cdot b_{d},\,b_{d}
        \right\rangle
      \right|
 \geq c$
for
$t\in [a_{y}^{(i)},b_{y}^{(i)} ] \subset [-\varepsilon, \varepsilon]$,
since $z_{y}+t\cdot b_{d}\in\overline{B_{d\varepsilon}}(x_{0})$,
as we saw above.
Finally, by an iterated application of the chain rule, we also have
\[
  \left|(f_{y}^{(i)})'''\left(t\right)\right|
  = \left|
      \sum_{i,j,\ell=1}^{d}
        (b_{d})_{i} (b_{d})_{j} (b_{d})_{\ell}
        \cdot \left(\partial_{i}\partial_{j}\partial_{\ell}f\right)
                \left(z_{y}+t\cdot b_{d}\right)
    \right|
  \leq d^{3}
       \cdot \sup_{x\in\overline{B_{d\varepsilon}}\left(x_{0}\right)} \,\,
               \max_{\left|\alpha\right|=3} \,\,
                 \left|\partial^{\alpha}f(x)\right|
  = C,
\]
where we used that $\left|\left(b_{d}\right)_{i}\right|\leq\left|b_{d}\right|=1$
for all $i\in\underline{d}$.
All in all, \Felix{setting $r := 3 \max \{1,p^{-1}\}$}, an application of
Corollary \ref{cor:C3ApproxLowerBound} shows because of
$b_{y}^{(i)} - a_{y}^{(i)} \leq 2 \varepsilon \Felix{\leq C_0 \cdot \frac{c}{C}}$
that
\[
  \left\Vert
    f \left(z_{y}+t\cdot b_{d}\right)
    - \left(\beta_{y}^{(i)} t + \gamma_{y}^{(i)}\right)
  \right\Vert_{L^{p}\left([a_{y}^{(i)},b_{y}^{(i)}];dt\right)}^{p}
  \geq
  \left[
      \frac{C_{0}}{\Felix{2^r}}
      \cdot c
      \cdot \left(b_{y}^{(i)} - a_{y}^{(i)}\right)^{2+\frac{1}{p}}
  \right]^{p}
  = \left(
      \frac{C_{0}}{\Felix{2^r}}
      \cdot c
    \right)^{p}
    \cdot (t_y^{(i)})^{1 + 2 p}.
\]
In case of $t_y^{(i)} = 0$, this estimate holds trivially.

Now, Hölder's inequality, applied with the exponent
$q := 1 + 2 p \in (\Felix{1}, \infty)$ shows
\[
  2\varepsilon
  = \sum_{i=1}^N (t_y^{(i)} \cdot 1)
  \leq \left(\sum_{i=1}^N
         (t_y^{(i)})^q\right)^{1/q} \cdot N^{1 - \frac{1}{q}}
  \leq \left(\sum_{i=1}^N
         (t_y^{(i)})^{1+2p}\right)^{1/q} \cdot P^{1 - \frac{1}{q}},
\]
and hence
$\sum_{i=1}^N (t_y^{(i)})^{1 + 2p}
 \geq (2\varepsilon)^q \cdot P^{1 - q}
 =    (2\varepsilon)^{1 + 2p} \cdot P^{-2p}$.
Therefore,
\begin{align*}
  \int_{-\varepsilon}^{\varepsilon}
      \left|
        f \left(z_{y} + t\cdot b_{d}\right)
        - g \left(z_{y} + t\cdot b_{d}\right)
      \right|^{p}
  \,dt
  & \geq
      \left(
          \frac{C_{0}}{\Felix{2^r}} \cdot c
      \right)^{p}
      \cdot \sum_{i=1}^N (t_y^{(i)})^{1 + 2p}
   \geq
      \left(
          \frac{C_{0}}{\Felix{2^r}} \cdot c
      \right)^{p}
      \cdot (2\varepsilon)^{1 + 2p} \cdot P^{-2p}.
\end{align*}

By putting everything together,
we thus arrive at
\begin{align*}
  \left\Vert f-g \right\Vert_{L^{p}(\Omega)}^{p}
  & \geq \int_{[-\varepsilon,\varepsilon]^{d-1}}
           \int_{-\varepsilon}^{\varepsilon}
             \left|
               f\left(z_{y} + t\cdot b_{d}\right)
               - g \left(z_{y}+t\cdot b_{d}\right)
             \right|^{p}
           \,dt
         \,dy\\
  &\geq
      \left(
          \frac{C_{0}}{\Felix{2^r}} \cdot c
      \right)^{p}
      \cdot (2\varepsilon)^{d + 2p} \cdot P^{-2p}
  \Felix{=}
      \left(
          \frac{C_{0}}{\Felix{2^r}}
          \cdot c
          \cdot (2\varepsilon)^{\Felix{\frac{d}{p}} + 2} \cdot P^{-2}
      \right)^{p},
\end{align*}
which yields the claim if we set
\Felix{$C = C(f,p) := (2\varepsilon)^{\frac{d}{p} + 2} \cdot 2^{-r} \cdot C_0 \cdot c$},
which is indeed independent of $g$ and $P$.
\end{proof}

By using the results of Telgarsky\cite{TelgarskySmallPaper} which
show that functions represented by neural ReLU networks are $P$-piecewise
slice affine for $P\asymp\left[N\left(\Phi\right)\right]^{L\left(\Phi\right)}$,
we can now derive a lower bound on the number of layers that are needed
to achieve a given approximation rate for nonlinear $C^{3}$ functions:

\begin{theorem}\label{thm:DepthLowerBoundAppendix}
Let $\Omega\subset\R^{d}$ be nonempty, open, bounded, and connected.
Furthermore, let $f\in C^{3}\left(\Omega\right)$ be nonlinear,
\Felix{and let $p \in (0,\infty)$}.
Then there is a constant $C_{f,\Felix{p}} > 0$ satisfying
\begin{align*}
  \left\Vert f-\Realization_{\varrho}(\Phi) \right\Vert_{L^{p} (\Omega)}
  & \geq C_{f,\Felix{p}} \cdot \left(N(\Phi) - 1\right)^{-2 \cdot L(\Phi)}, \\
  \left\Vert f - \Realization_{\varrho}(\Phi) \right\Vert_{L^{p} (\Omega)}
  & \geq C_{f,\Felix{p}} \cdot \left(M(\Phi)+d\right)^{-2 \cdot L(\Phi)}
\end{align*}
for each ReLU neural network $\Phi$
with input dimension $d$ and output dimension $1$.
\end{theorem}
\begin{remark*}
By adapting the given arguments (mostly the proof of
Lemma \ref{lem:C3ApproxLowerBoundUnitInterval}), one can show that the same
claim remains true for $f \in C^{2+\varepsilon} (\Omega)$,
with fixed but arbitrary $\varepsilon > 0$.
For the sake of brevity, we omitted this generalization.
\end{remark*}

Before we give the proof of Theorem \ref{thm:DepthLowerBoundAppendix}, we
observe the following corollary:
\begin{corollary}
\label{cor:OptimalRateRequiresDepthAppendix}Let $\Omega\subset\R^{d}$ be
nonempty, open, bounded, and connected.
Furthermore, let $f\in C^{3}\left(\Omega\right)$ be nonlinear,
\Felix{and let $p \in (0,\infty)$}.
If there are constants $C,\theta > 0$, a null-sequence
$\left(\varepsilon_{k}\right)_{k\in\N}$ of positive numbers, and
a sequence $\left(\Phi_{k}\right)_{k\in\N}$ of ReLU neural networks satisfying
\[
  \left\Vert f-\Realization_{\varrho}\left(\Phi_{k}\right)\right\Vert _{L^{p}}
  \leq C \cdot \varepsilon_{k}
  \quad \text{ and } \quad
  \left[
    M\left(\Phi_{k}\right) \leq C \cdot \varepsilon_{k}^{-\theta}
    \text{ or }
    N\left(\Phi_{k}\right) \leq C \cdot \varepsilon_{k}^{-\theta}
  \right]\quad
\]
for all $k\in\N$, then
\[
  \liminf_{k\to\infty}\:
    L\left(\Phi_{k}\right)
  \geq \frac{1}{2 \theta}.
\]
\end{corollary}
\begin{proof}
Let us assume that the claim is false, that is, we have
$\liminf_{k\to\infty}L\left(\Phi_{k}\right)<(2\theta)^{-1}$.
By switching to a subsequence, we can then assume that there is some
$\delta>0$ with $L\left(\Phi_{k}\right)\leq (2\theta)^{-1}-\delta=:L$
for all $k\in\N$.
Note that $1\leq L\left(\Phi_{k}\right)\leq L$.

Next, since $\varepsilon_{k} \to 0$, and since $f \neq 0$ (because $f$
is nonlinear), we have
$\left\Vert f - \Realization_{\varrho}\left(\Phi_{k}\right)\right\Vert_{L^{p}}
<\left\Vert f \right\Vert _{L^{p}}$
for $k$ large enough (which we will assume in the following).
In particular, $\Realization_{\varrho}\left(\Phi_{k}\right)\neq0$ and hence
$M\left(\Phi_{k}\right)\geq1$, so that
$M\left(\Phi_{k}\right)+d \leq (d+1)\cdot M\left(\Phi_{k}\right)$.

Now, there are two cases for (large) $k\in\N$:
If $M\left(\Phi_{k}\right) \leq C \cdot \varepsilon_{k}^{-\theta}$,
then the second part of Theorem \ref{thm:DepthLowerBoundAppendix} shows
\Felix{that there is a constant $C_{f,p} > 0$ (independent of $k \in \N$)
satisfying}
\begin{equation}
\begin{split}
  C_{f,\Felix{p}} \cdot (1+d)^{-2 L} C^{-2L} \cdot \varepsilon_{k}^{2L\theta}
  & \leq C_{f,\Felix{p}} \cdot (1+d)^{-2L} \cdot \left[ M(\Phi_{k}) \right]^{-2L}
    \leq C_{f,\Felix{p}} \cdot \left(M(\Phi_{k}) + d\right)^{-2L} \\
  & \leq C_{f,\Felix{p}} \cdot \left(M(\Phi_{k}) + d\right)^{-2L(\Phi_{k})}
    \leq \left\Vert
           f - \Realization_{\varrho}\left(\Phi_{k}\right)
         \right\Vert_{L^{p}}
    \leq C \cdot \varepsilon_{k}.
\end{split}
\label{eq:DepthCorollaryFirstCase}
\end{equation}
If otherwise $N(\Phi_{k}) \leq C \cdot \varepsilon_{k}^{-\theta}$,
then the first part of Theorem \ref{thm:DepthLowerBoundAppendix} shows
\Felix{with the same constant $C_{f,p}$ as above that}
\begin{equation}
\begin{split}
  C_{f,\Felix{p}} \cdot C^{-2L} \cdot \varepsilon_{k}^{2L\theta}
  & \leq C_{f,\Felix{p}} \cdot\left(N\left(\Phi_{k}\right)\right)^{-2L}
    \leq C_{f,\Felix{p}} \cdot \left( N(\Phi_{k}) \right)^{-2 L(\Phi_{k})} \\
  & \leq C_{f,\Felix{p}} \cdot \left(N(\Phi_{k}) - 1\right)^{-2L(\Phi_{k})}
    \leq \left\Vert
           f - \Realization_{\varrho}(\Phi_{k})
         \right\Vert_{L^{p}}
    \leq C \cdot \varepsilon_{k}.
\end{split}
\label{eq:DepthMattersCorollaryCase2}
\end{equation}

At least one of the equations (\ref{eq:DepthCorollaryFirstCase})
or (\ref{eq:DepthMattersCorollaryCase2}) holds for infinitely many $k \in \N$.
Since $\varepsilon_{k} \to 0$ and $\varepsilon_{k} > 0$,
this easily yields $2 L \theta \geq 1$, and hence
$(2\theta)^{-1} - \delta = L \geq (2\theta)^{-1}$,
which is the desired contradiction.
\end{proof}

We close this section with the proof of Theorem \ref{thm:DepthLowerBoundAppendix}.
\begin{proof}[Proof of Theorem \ref{thm:DepthLowerBoundAppendix}]
\textbf{Step 1}: In this step, we
show\footnote{Essentially, this is already contained in the statement of
\cite[Lemma 2.1]{TelgarskySmallPaper}, but
\Felix{the paper \cite{TelgarskySmallPaper}} uses a slightly
different definition of neural networks than we do.
Therefore, and for the convenience of the reader, we provide a proof.}
that if $\Phi$ is a neural network
\Felix{with $d$-dimensional input and $1$-dimensional output}
of depth $L$ and with $N$ neurons, then $\Realization_{\varrho}(\Phi)$ is
$P$-piecewise slice affine with $P\leq (\nicefrac{2}{L})^L \cdot (N-1)^L$.

To this end, we first introduce some terminology:
As in \cite{TelgarskySmallPaper}, let us call a \emph{continuous}
function $f:\R\to\R$ \textbf{$t$-sawtooth} (with $t\in\N$) if $f$
is piecewise affine-linear with at most $t$ pieces, that is, there are
$-\infty = x_{0} < x_{1} < \dots < x_{t} = \infty$ such that
$f|_{(x_{i-1},x_{i})}$ is affine-linear for each $i \in \underline{t}$.
Note that there are no issues at the boundary points of the affine-linear
``pieces'', since (in slight contrast to \cite{TelgarskySmallPaper}), we assume
$f$ to be continuous.
Using this terminology, \cite[Lemma 2.3]{TelgarskySmallPaper}
states that if $f,g:\R\to\R$ are $k$-sawtooth and $\ell$-sawtooth,
respectively, then $f+g$ is $k+\ell$-sawtooth, and $f\circ g$ is
$k\ell$-sawtooth. Note that the ReLU $\varrho$ is $2$-sawtooth.

Now, let $\Phi =
\left(\left(A^{(1)},b^{(1)}\right),\dots,\left(A^{(L)},b^{(L)}\right)\right)$
be a neural network with $d$-dimensional input and one-dimensional input,
with $L$ layers and $N$ neurons. \Felix{Thus}, 
$A^{\left(\ell\right)}\in\R^{N_{\ell}\times N_{\ell - 1}}$ and
$b^{\left(\ell\right)}\in\R^{N_{\ell}}$, where $N_{0}=d$ and $N_{L}=1$,
and $N = \sum_{j=0}^L N_j$.
Further, let $g:=\Realization_{\varrho}\left(\Phi\right)$ and let $x,v\in\R^{d}$
be arbitrary.
We want to show that $g_{x,v}:\R\to\R,t\mapsto g\left(x+t\cdot v\right)$
is $P$-sawtooth, with $P \leq (\nicefrac{2}{L})^L \cdot (N-1)^L$.
To see this, inductively define
$g^{\left(0\right)},g^{\left(1\right)},\dots,g^{\left(L\right)}$
as follows: $g^{\left(0\right)}:\R\to\R^{d},x\mapsto x+t\cdot v$,
\[
  g^{(\ell+1)}:
  \R \to \R^{N_{\ell+1}},
  t \mapsto \varrho \left(A^{(\ell+1)} \cdot g^{(\ell)}(t) + b^{(\ell+1)}\right)
  \quad\text{ for } 0 \leq \ell \leq L-2,
\]
and $g^{(L)} : \R \to \R, t \mapsto A^{(L)} \cdot g^{(L-1)}(t) + b^{(L)}$.
We clearly have $g^{(L)} = g_{x,v}$.

We will show by induction on $\ell \in \left\{ 0, \dots ,L\right\} $
that each component function $g_{i}^{(\ell)}$ for $i \in \underline{N_{\ell}}$
is $M_{\ell}$-sawtooth, with $M_{\ell} :=\ prod_{j=0}^{\ell-1} 2N_{j}$,
where (by the convention for empty products) $M_{0} = 1$.
Indeed, for $\ell=0$, we have $g_{i}^{(0)}(t) = x_{i} + t \cdot v_{i}$,
which is affine-linear.
Hence, $g_{i}^{(0)}$ is $M_{0}$-sawtooth, since $M_{0}=1$.
For the induction step, assume that all $g_{i}^{(\ell)}$,
$i \in \underline{N_{\ell}}$ are $M_{\ell}$-sawtooth for some
$0 \leq \ell \leq L-1$.
In case of $\ell = L - 1$, let $\theta := \identity_{\R}$, and otherwise
let $\theta := \varrho$.
In either case, we have that $\theta$ is $2$-sawtooth, and
\[
  g_{i}^{(\ell+1)}(t)
  = \theta
    \left(
      b_{i}^{(\ell + 1)}
      + \sum_{j=1}^{N_{\ell}}
          A_{i,j}^{(\ell+1)} \cdot g_{j}^{(\ell)}(t)
    \right)
  \qquad \text{ for } t \in \R
  \quad  \text{ and } i \in \underline{N_{\ell + 1}} .
\]
But since each $g_{j}^{(\ell)}$ is $M_{\ell}$-sawtooth,
so is $t\mapsto A_{i,j}^{(\ell+1)} \cdot g_{j}^{(\ell)}(t)$, so that
$t \mapsto \sum_{j=1}^{N_{\ell}} A_{i,j}^{(\ell+1)} \cdot g_{j}^{(\ell)}(t)$
is $(M_{\ell+1}/2)$-sawtooth, since
$\sum_{j=1}^{N_{\ell}} M_{\ell} = N_{\ell} \cdot M_{\ell}
 = M_{\ell+1} / 2$.
Thus, since $\theta$ is $2$-sawtooth, $g_{i}^{\left(\ell+1\right)}$
is $M_{\ell+1}$-sawtooth, as claimed.
\Felix{Overall}, we have shown that $g_{x,v} = g^{(L)}$ is $M_{L}$-sawtooth,
where $M_{L}=\prod_{j=0}^{L-1} 2N_{j} = 2^{L} \cdot \prod_{j=0}^{L-1} N_{j}$.
Now, by concavity of the natural logarithm and because of $N_{L}=1$,
we have
\[
       \ln \left( \frac{N(\Phi) - 1}{L} \right)
  =    \ln\left(\frac{1}{L} \sum_{j=0}^{L-1} N_{j}\right)
  \geq \frac{1}{L} \sum_{j=0}^{L-1} \ln N_{j}
  =    \frac{1}{L} \cdot \ln\left( \prod_{j=0}^{L-1} N_{j} \right),
\]
and hence
$\prod_{j=0}^{L-1}N_{j}
 \leq e^{L \cdot \ln\left( \left(N(\Phi) - 1\right) / L \right)}
 =\left( \nicefrac{(N(\Phi) - 1)}{L} \right)^{L}$,
so that all in all $M_{L}
\leq \left(\nicefrac{2}{L}\right)^{L} \cdot \left(N(\Phi) - 1\right)^{L}$.

\medskip{}

\textbf{Step 2}:
Now, an application of Proposition \ref{prop:PPiecewiseAffineFunctionsAreBad}
yields a constant $C_{f,\Felix{p}}^{(0)} > 0$ (independent of $\Phi,L$)
satisfying
\[
  \left\Vert
    f \! - \Realization_{\varrho}\left(\Phi\right)
  \right\Vert_{L^{p}}
  \geq C_{f,\Felix{p}}^{(0)} \cdot M_{L}^{-2}
  \geq C_{f,\Felix{p}}^{(0)}
       \cdot \left(\frac{L}{2}\right)^{2L} \!\!
       \cdot \left(N(\Phi) - 1\right)^{-2L}
  \! \geq \! \frac{1}{4}
             \cdot C_{f,\Felix{p}}^{(0)}
             \cdot \left(N(\Phi) - 1\right)^{-2L}.
\]
Here, the estimate $\left(\nicefrac{L}{2}\right)^{2L}\geq\nicefrac{1}{4}$
can be easily seen to be true by distinguishing the cases $L = 1$ and
$L \geq 2$.
This yields the first claim, with $C_{f,\Felix{p}} = C_{f,\Felix{p}}^{(0)}/4$,
since $L = L(\Phi)$.

\medskip{}

\textbf{Step 3}:
Finally, to prove the second claim, recall from
Lemma \ref{lem:WeightsAndNeuronsAssumptionJustification} that there is a neural
network $\Phi'$ with
$\Realization_{\varrho}(\Phi) = \Realization_{\varrho}(\Phi')$ and such that
$M(\Phi') \leq M(\Phi)$ and $N(\Phi') \leq M(\Phi') + d + 1$,
as well as $L(\Phi') \leq L(\Phi)$.
By applying the first claim of the current theorem to $\Phi'$ instead of $\Phi$
(and with $L' = L(\Phi')$ instead of $L=L(\Phi)$), we get
\begin{align*}
         \left\Vert
           f - \Realization_{\varrho}\left(\Phi\right)
         \right\Vert_{L^{p}}
  & =    \left\Vert
          f - \Realization_{\varrho}(\Phi')
         \right\Vert_{L^{p}}
    \geq C_{f,\Felix{p}} \cdot \left(N(\Phi') - 1\right)^{-2 L'}
    \geq C_{f,\Felix{p}} \cdot \left(N(\Phi') - 1\right)^{-2L}\\
  & \geq C_{f,\Felix{p}} \cdot \left(M(\Phi') + d\right)^{-2L}
    \geq C_{f,\Felix{p}} \cdot \left(M(\Phi) + d\right)^{-2L}.
   \qedhere
\end{align*}
\end{proof}

\section{Composition with smooth submersions}\label{sec:SmoothSubmersions}

In Lemma \ref{lem:CompositionWithSubmersion} we claimed that the map
$f \mapsto f \circ \tau$, with a \emph{smooth submersion} $\tau$,
is a bounded map from $L^p$ to $L^p$.
This might be a well-known fact in the right communities, but since we were
unable to find a reference, we provide a full proof. The main ingredients are
the \emph{coarea formula} from geometric measure theory, and the
\emph{constant rank theorem} from differential geometry.

%

\begin{proof}[Proof of Lemma \ref{lem:CompositionWithSubmersion}]
  \textbf{Step 1 (Shrinking $U$):}
  If a matrix $A \in \R^{n \times m}$ has full rank, then
  $\range A = \R^n$, since $n \leq m$. Therefore, if $y \in \kernel A^T$,
  then $0 = \langle x, A^T y \rangle = \langle Ax, y \rangle$ for all
  $x \in \R^m$, whence $y = 0$; that is, $\kernel A^T = \{0\}$.
  This implies that $A A^T$ is positive definite.
  By applying these observations with $A = D \tau(x)$, we see
  $\gamma (x) > 0$ for all $x \in K$, with the continuous function
  \begin{equation*}
    \gamma : U \to [0,\infty),
             x \mapsto \sqrt {
                               \det \left(
                                      D \tau(x) \cdot [D \tau(x)]^T
                                    \right)
                             } \, .
  \end{equation*}
  Conversely, for any $x \in U$ with $\gamma (x) > 0$, it follows that
  $\kernel [D \tau(x)]^T = \{0\}$, so that $D \tau(x)$ has full rank,
  since $n \leq m$. By compactness of $K$ and continuity of $\gamma$, we see
  $C_0 := \min_{x \in K} \gamma (x) > 0$. Thus,
  \[
    U^{(0)}
    := \left\{
         x \in U
         \,:\,
         \gamma (x) > \frac{C_0}{2}
       \right\}
    \subset U \subset \R^m
  \]
  is open with $K \subset U^{(0)}$.
  From now on, we will only be working on the set $U^{(0)}$.

  \medskip{}

  \textbf{Step 2 (Applying the coarea formula):}
  Let $r := \frac{1}{3} \cdot \dist (K, \R^m \setminus U^{(0)}) > 0$, and define
  \begin{equation*}
      U' := \{ x \in \R^m \,:\, \dist (x, K) < r \}
      \quad \text{and} \quad
      K_0 := \{ x \in \R^m \,:\, \dist (x, K) \leq 2r \} \, ,
  \end{equation*}
  so that $K \subset U' \subset K_0$, with $K_0 \subset U^{(0)}$ being compact
  and $U'$ open.

  We claim that $\tau|_{U'}$ is Lipschitz continuous.
  To prove this, set $C_1 := \max_{x \in K_0} \|D \tau (x)\|$, and furthermore
  $C_2 := \max_{x \in K_0} |\tau (x)|$.
  Then, for $x,y \in U' \subset K_0$, there are two cases:
  \smallskip{}

  \textbf{Case 1}: $|x-y| \geq r$. This implies
  \[
      |\tau(x) - \tau(y)|
      \leq 2 C_2
      = \frac{2 C_2}{r} \cdot r
      \leq \frac{2 C_2}{r} \cdot |x-y| \, .
  \]

  \textbf{Case 2}: $|x-y| < r$. Because of $x \in U'$, we then have
  $x,y \in B_r (x) \subset K_0$. Since $B_r (x)$ is convex with
  $\| D \tau (z)\| \leq C_1$ for all $z \in B_r (x)$, standard estimates
  from multi-variable calculus yield $|\tau(x)-\tau(y)| \leq C_1 \cdot |x-y|$.
  \medskip{}

  Taken together, the two cases \Felix{prove} the Lipschitz continuity of
  $\tau|_{U'}$.
  By the Kirszbraun-Valientine theorem
  (see \cite{KirszbraunTheorem,KirszbraunTheoremModern}, or
  \cite[Section 3.1, Theorem 1]{EvansGariepy} for a simpler version),
  there is thus a Lipschitz continuous map $\tau_0 : \R^m \to \R^n$ extending
  $\tau|_{U'}$.
  Thus, an application of the \emph{coarea formula}
  (see \cite[Theorem 2 in Section 3.4.3]{EvansGariepy},
  and see \cite[Section 3.2.2 and Theorem 3 in Section 3.2]{EvansGariepy} for
  a justification of the identity $J_{\tau_0} (x) = J_\tau (x) = \gamma (x)$
  for $x \in U ' \supset K$) shows
  \begin{align*}
    \int_{K} f \circ \tau \, dx
    & = \int_{K}
        f(\tau(x))
        \cdot J_{\tau_0} (x) \,\big/\, \gamma (x) \, dx
     \leq C_0^{-1} \! \int_{\R^m}
                              \Indicator_K (x)
                              \cdot f(\tau_0(x))
                              \cdot J_{\tau_0}(x)
                          \, dx \\
    & = C_0^{-1} \! \int_{\R^n}
                          \int_{\tau_0^{-1}(\{t\})} \!\!\!
                            \Indicator_K (x)
                            \cdot f(\tau_0(x))
                          \, d\mathcal{H}^{m-n} (x)
                       \, dt
     \leq C_0^{-1} \! \int_{\tau(K)}
                             f(t)
                             \cdot \mathcal{H}^{m-n}
                                   \big(\tau^{-1} (\{t\}) \cap K\big)
                          \, dt \, ,
  \end{align*}
  where $\mathcal{H}^{m-n}$ denotes the $m-n$-dimensional Hausdorff-measure in
  $\R^m$. The last step above used that if $x \in \tau_0^{-1}(\{t\}) \cap K$,
  then $x \in K \subset U'$, so that $\tau(x) = \tau_0 (x) = t$,
  that is, $x \in K \cap \tau^{-1}(\{t\})$, and $t = \tau(x) \in \tau(K)$.

  From the above estimate, we see that we are done once we show
  $\mathcal{H}^{m-n} \big(\tau^{-1}(\{t\}) \cap K\big) \leq C$ for all
  $t \in \tau(K)$.
  \medskip{}

  \textbf{Step 3 (Estimating $\mathcal{H}^{m-n} (\tau^{-1}(\{t\}) \cap K)$):}
  We saw in Step 1 that $\rank D\tau (x) = n$ for all
  $x \in U^{(0)}$. Thus, by the constant rank theorem
  (see Theorem \ref{thm:ConstantRankTheorem} below),
  for each $x \in K \subset U^{(0)}$ there is an open
  neighborhood $U_x \subset U^{(0)}$ of $x$, an open neighborhood
  $V_x \subset \R^m$ of $0$, and a $C^1$-diffeomorphism
  $\varphi_x : V_x \to U_x$ with $\varphi_x (0) = x$,
  and an invertible affine-linear map $\Felix{T}_x : \R^n \to \R^n$ satisfying
  \[
      \Felix{T}_x \circ \tau \circ \varphi_x = \pi |_{V_x} \, ,
      \quad \text{where} \quad
      \pi (y_1,\dots, y_m) := (y_1,\dots,y_n)
      \, \text{ for } \, y = (y_1,\dots, y_m) \in \R^m \, .
  \]

  Because $0 \in V_x$, we have $[-r_x, r_x]^m \subset V_x$ for some
  $r_x > 0$, and $W_x := \varphi_x ((-r_x, r_x)^m) \subset U_x$
  is an open neighborhood of $x$. By compactness of $K$,
  there are thus $x_1,\dots, x_N \in K$ with $K \subset\bigcup_{i=1}^N W_{x_i}$.
  Now, let
  \[
      C_3 := \max_{i = 1,\dots,N} \,\,
                \max_{x \in [-r_{x_i}, r_{x_i}]^m}
                  \|
                     D \varphi_{x_i} (x)
                  \| \, .
  \]
  Since $(-r_{x_i}, r_{x_i})^m$ is convex, standard arguments show that
  $\varphi_{x_i} : [-r_{x_i}, r_{x_i}]^m \to\R^m$
  is Lipschitz continuous with Lipschitz constant no \Felix{larger} than $C_3$,
  for each $i = 1,\dots,N$.  Therefore, elementary properties of the
  Hausdorff-measure (see \cite[Theorem 7.5]{MattilaGeometryOfSets})
  imply for arbitrary $z \in [-r_{x_i}, r_{x_i}]^n$ that
  \[
      \mathcal{H}^{m-n}
      \big(
        \varphi_{x_i} (\{z\} \times [-r_{x_i}, r_{x_i}]^{m-n})
      \big)
      \leq C_3^{m-n} \cdot \mathcal{H}^{m-n}
                             \big(
                                 \{z\} \times [-r_{x_i}, r_{x_i}]^{m-n}
                             \big)
      =    (2 \cdot C_3 \cdot r_{x_i})^{m-n} \, .
  \]

  Finally, note for arbitrary $i \in \{1,\dots,N\}$, $t \in \tau (K)$, and
  $y \in W_{x_i} \cap \tau^{-1} (\{t\}) \subset U_{x_i}$ with $x := x_i$ that
  \begin{equation}
      \Felix{T}_x^{-1} (\pi (\varphi_x^{-1} (y)))
      = \Felix{T}_x^{-1}
        \circ \Felix{T}_x
        \circ \tau
        \circ \varphi_x (\varphi_x^{-1} (y))
      = \tau (y) = t \, ,
      \label{eq:InverseImageRewriting}
  \end{equation}
  and that
  $z := \Felix{T}_x (t) = \pi (\varphi_x^{-1} (y)) \in [-r_x, r_x]^{n}$,
  since $y \in W_x$, so that $\varphi_x^{-1} (y) \in [-r_x, r_x]^m$.
  Now, with this choice of $z$, Equation \eqref{eq:InverseImageRewriting}
  implies $\varphi_x^{-1} (y) \in \{z\} \times \R^{m-n}$, from which we get
  $\varphi_x^{-1} (y) \in \{z\} \times [-r_x, r_x]^{m-n}$, since
  $y \in W_x = \varphi_x ((-r_x, r_x)^m)$. In summary, we have thus shown
  \[
      W_{x_i} \cap \tau^{-1}(\{t\})
      \subset \varphi_{x_i} (\{z\} \times [-r_{x_i}, r_{x_i}]^{m-n})
      \quad \text{with} \quad
      z = z_{i,t} := \Felix{T}_{x_i} (t) \in [-r_{x_i}, r_{x_i}]^n \, .
  \]
  All in all, we get for arbitrary $t \in \tau(K)$ that
  \begin{align*}
    \hspace{-0.8cm}
    \mathcal{H}^{m-n} \big(K \! \cap \! \tau^{-1}(\{t\})\big)
    & \leq \mathcal{H}^{m-n}
         \Big(
           \tau^{-1}(\{t\}) \! \cap \! \bigcup_{i=1}^N W_{x_i}
         \Big)
    \leq \sum_{i=1}^N
              \mathcal{H}^{m-n} \big(\tau^{-1}(\{t\}) \cap W_{x_i}\big) \\
    & \leq \sum_{i=1}^N
              \mathcal{H}^{m-n}
              \big(
                  \varphi_{x_i}
                  (
                      \{z_{i, t}\} \! \times \! [-r_{x_i}, r_{x_i}]^{m-n}
                  )
                \big)
      \leq \sum_{i=1}^N
             (2 \cdot C_3 \cdot r_{x_i})^{m-n}
      =: C < \infty \, .
      \qedhere
  \end{align*}
\end{proof}

In the above proof, we used the following version of the constant rank theorem:
\begin{theorem}\label{thm:ConstantRankTheorem}
  Let $n,m\in \N$, $\emptyset \neq U \subset \R^m$ be open,
  and let $\tau : U \to \R^n$ be continuously differentiable with
  $\rank D \tau (x) = n$ for all $x \in U$; in particular $n \leq m$.

  Then, for each $x \in U$, there is an invertible affine-linear map
  $\Felix{T}_x : \R^n \to \R^n$, an open neighborhood $U_x \subset U$ of $x$,
  an open $V_x \subset \R^m$ with $0 \in V_x$,
  and a $C^1$-diffeomorphism $\varphi_x : V_x \to U_x$
  with $\varphi_x (0) = x$ which satisfies
  \[
      \Felix{T}_x \big( \tau (\varphi_x (y) ) \big) = (y_1,\dots, y_n)
      \quad \text{for all} \quad y = (y_1,\dots, y_m) \in V_x \, .
  \]
\end{theorem}
\begin{proof}
  We derive the claim as a consequence of a slightly different version of the
  constant rank theorem, namely of \cite[Theorem 9.32]{RudinPrinciples}.
  In the notation of that theorem, we have $A = D \tau (x)$ and therefore
  $Y_1 = \range A = \R^n$, so that $P = \identity_{\R^n}$ is the unique
  projection of $\R^n$ onto $Y_1$, and thus $Y_2 = \kernel P = \{0\}$.
  Therefore, the map $\varphi$ from \cite[Theorem 9.32]{RudinPrinciples}
  satisfies $\varphi \equiv 0$, so that there is a $C^1$ diffeomorphism
  $H : W_x \to U_x$ satisfying $\tau(H(y)) = Ay$ for all $y \in W_x$,
  with $W_x \subset \R^m$ open and $U_x \subset U$ an open neighborhood of $x$.

  Writing $A = O D Q$ for the singular value decomposition of $A$, we have
  because of $\rank A = n$ that
  $D = \diag (\sigma_1, \dots, \sigma_n) \in \R^{n \times m}$ for
  certain $\sigma_1, \dots, \sigma_n > 0$, and the matrices
  $O \in \R^{n \times n}$ and $Q \in \R^{m \times m}$ are orthogonal.
  Now, setting
  $D^+ := \diag (\sigma_1^{-1},\dots, \sigma_n^{-1}) \in \R^{n\times n}$,
  we have $D^+ D v = (v_1,\dots, v_n) =: \pi v$ for all
  $v = (v_1,\dots,v_m) \in \R^m$.

  Setting $x_0 := H^{-1} (x) \in W_x$, $V_x := Q(W_x - x_0)$,
  as well as $\varphi_x : V_x \to U_x, y \mapsto H(Q^{-1}y + x_0)$ and
  $\Felix{T}_x : \R^n \to \R^n, z \mapsto D^+ O^{-1} z - \pi Q x_0$, we have
  \begin{align*}
    (\Felix{T}_x \circ \tau \circ \varphi_x)(y)
    & = D^+ O^{-1} (\tau (H (Q^{-1} y + x_0))) - \pi Q x_0
      = D^+ O^{-1} A (Q^{-1} y + x_0) - \pi Q x_0 \\
    & = D^+ D Q (Q^{-1} y + x_0) - \pi Q x_0
      = \pi y + \pi Q x_0 - \pi Q x_0
      = \pi y
  \end{align*}
  for all $y \in V_x$, as required.
  It is easy to see that $\Felix{T}_x$ is indeed an invertible affine-linear map,
  and that $\varphi_x : V_x \to U_x$ is a $C^1$-diffeomorphism.
\end{proof}

\section{Approximation of high-dimensional functions}\label{sec:HighDimApprox}

To ultimately prove Theorem \ref{thm:ApproximationOfSymmetricFunctions},
we start by establishing the following auxiliary lemma.
\begin{lemma}\label{lem:L1concatenations}
  Let $d,D, T \in \N$, $\eps \in (0,1)$, and $\Felix{p}, \kappa, C > 0$.
  Further, let $U \subset \R^d$ and $V \subset \R^D$ be measurable.
  Also, let $\tau : V \to U$ be measurable, and such that
  \begin{align} \label{eq:BoundedConcatenations}
    \left\|g \circ \tau\right\|_{L^{\Felix{p}}(V)}
    \leq  \kappa \cdot \left\|g\right\|_{L^{\Felix{p}}(U)}
    \text{ for all } g \in L^{\Felix{p}}(U).
  \end{align}
  Additionally, let $f : U \to \R$ \Felix{and $\widehat{f} : \R^d \to \R$,
  as well as $\widehat{\tau} : V \to \R^d$} be measurable, and assume that
  $\widehat{f}$ is Lipschitz continuous with Lipschitz constant
  $\varepsilon^{-T}$. 
  Finally, assume that
  \[
    \left\| f-\widehat{f}\right\|_{L^{\Felix{p}}(U)}
    \leq \frac{\varepsilon}{\Felix{2^{\max\{1,p^{-1}\}}} \cdot \kappa},
    \quad \text{and} \quad
    \left\| \tau - \widehat{\tau} \right\|_{L^{\Felix{p}}(V)}
    \leq \frac{\varepsilon^{T+\Felix{1}}}{\Felix{2^{\max\{1,p^{-1}\}}}} \, .
  \]
  Then
  \[
  \left\|
    f \circ \tau - \widehat{f} \circ \widehat{\tau}
  \right\|_{L^{\Felix{p}}(V)}
  \leq \eps \, .
\]
\end{lemma}
\begin{proof}
  \Felix{Setting $q := \max\{1,p^{-1}\}$,
  Equation \eqref{eq:PseudoTriangleInequality} shows}
  \[
    \left\|
      f \circ \tau - \widehat{f} \circ \widehat{\tau}
    \right\|_{L^{\Felix{p}}(V)}
    \leq \Felix{2^q \cdot \max}
         \left\{
           \left\|
             f \circ \tau - \widehat{f} \circ \tau
           \right\|_{L^{\Felix{p}}(V)}
           \,\,,\,\,
           \left\|
             \widehat{f} \circ \tau - \widehat{f} \circ \widehat{\tau}
           \right\|_{L^{\Felix{p}}(V)}
         \right\}
    =:   \Felix{2^q \cdot \max} \{ I_1 \,,\, I_2\} \, .
  \]
  We start by estimating $I_1$. Precisely, as a consequence of
  Equation \eqref{eq:BoundedConcatenations}, we have
  \begin{align*}
    \Felix{2^q \cdot} I_1
    = \Felix{2^q} \cdot \left\|
                          f \circ \tau - \widehat{f} \circ \tau
                        \right\|_{L^{\Felix{p}}(V)}
    \leq \Felix{2^q} \, \kappa
         \cdot \left\| f  - \widehat{f} \right\|_{L^{\Felix{p}}(U)}
    \leq \Felix{2^q} \, \kappa \cdot \frac{\varepsilon}{\Felix{2^q} \, \kappa}
    =    \Felix{\varepsilon} \, .
  \end{align*}
  We proceed by estimating $I_2$.
  To this end, first note by \Felix{the} Lipschitz-continuity of
  $\widehat{f}$ that
  \[
    | \widehat{f} (\tau(x)) - \widehat{f} (\widehat{\tau} (x))|^{\Felix{p}}
    \leq \big( \eps^{-T} \cdot |\tau(x) - \widehat{\tau}(x)| \big)^{\Felix{p}}
    \, .
  \]
  This implies
  \[
    \Felix{\big( 2^q \cdot I_2^2 \big)^p}
    = \Felix{2^{q p}} \cdot
      \int_{V}
         | \widehat{f} (\tau(x)) - \widehat{f} (\widehat{\tau} (x))|^{\Felix{p}}
      \, dx
    \leq \Felix{2^{q p}} \cdot \eps^{-T \Felix{p}}
         \cdot \| \tau - \widehat{\tau} \|_{L^{\Felix{p}}(V)}^{\Felix{p}}
    \leq
    \Felix{
    \Big( 2^q \, \varepsilon^{-T} \, \frac{\varepsilon^{T+1}}{2^q} \Big)^p
    = \varepsilon^p \, ,
    }
  \]
  and hence $\Felix{2^q} \, I_2 \leq \Felix{\eps}$.

  \medskip{}

  All in all, we have shown
  $\| f \circ \tau - \widehat{f} \circ \widehat{\tau}\|_{L^{\Felix{p}} (V)}
   \leq \Felix{2^q \cdot \max} \{ I_1 \,,\, I_2\} \leq \eps$,
  as claimed.
\end{proof}

Proving Theorem \ref{thm:ApproximationOfSymmetricFunctions} is now simply a
matter of constructing networks the activations of which satisfy the assumptions
of Lemma \ref{lem:L1concatenations}.
\begin{proof}[Proof of Theorem \ref{thm:ApproximationOfSymmetricFunctions}]
  Let $f\in\mathcal{SE}_{r, \beta, a, B, \kappa, d, D}^{\Felix{p}}$ and
  $f = g \circ \tau$, where $g \in \mathcal{E}_{r, \beta, d, B}^{\Felix{p}}$ and
  $\tau \in \mathcal{S}_{\kappa, d, D, a}$.
  By Corollary \ref{cor:PiecewiseSmoothFunctions}
  \Felix{and Remark \ref{rem:QuantisationConversion}}, there exist constants
  $c' \Felix{= c'(\beta,d,p) \in \N}$,
  $c= c(d, r, \Felix{p}, \beta, B, \kappa) > 0$,
  and $s = s(d, r, \Felix{p}, \beta, \Felix{\kappa}, B) \in \N$
  independent of $g$ such that for any
  $\varepsilon \in (0, \nicefrac{1}{2})\vphantom{\sum_j}$,
  there is a neural network $\Phi^g_\varepsilon$ with at most
  \Felix{$c'$} layers, and at most
  $c \cdot \varepsilon^{-\Felix{p} (d-1)/\beta}$ nonzero,
  $(s,\varepsilon)$-quantized weights such that
  \[
    \left\|\Realization_\varrho(\Phi^g_\varepsilon) -g\right\|_{L^{\Felix{p}}}
    < \frac{\varepsilon}{\Felix{2^{\max\{1,p^{-1}\}}} \, \kappa}
    \, .
  \]
  Since the ReLU $\varrho$ is Lipschitz with Lipschitz constant $1$,
  since all weights of $\Phi^g_\varepsilon$ are bounded (in absolute value)
  by $\varepsilon^{-s}$, and since there are at most
  $c \cdot \varepsilon^{-\Felix{p} (d-1)/\beta}$ weights arranged in a
  bounded number of layers, there exists \Felix{a number}
  $T = T(d, r, \Felix{p}, \beta, B, \kappa) \in \N$
  \Felix{independent of $g$}
  such that $\Realization_\varrho(\Phi^g_\varepsilon) \Felix{: \R^d \to \R}$ is
  Lipschitz continuous with Lipschitz constant $\varepsilon^{-T}$.
  Note that this uses that $\eps \in (0, \nicefrac{1}{2})$,
  \Felix{so that $\eps^{-T} \to \infty$ as $T \to \infty$}.

  Now, set
  $\beta_0 := \left\lceil
                \beta \cdot \frac{D (T+\Felix{1})}{\Felix{p}(d-1)}
              \right\rceil$.
  Since $\tau_i \in \mathcal{F}_{\beta_0, D, a_{\beta_0}}$ for all
  $i = 1, \dots, d$, we can apply
  \Felix{a combination of
  Remarks \ref{rem:QuantisationConversion} and \ref{rem:DeepIdentity} and}
  Theorem \ref{thm:ApproxOfSmoothFctn}
  (\Felix{applied} with $\tau_i$ instead of $f$, with
  $\varepsilon^{T+\Felix{1}} / \Felix{(2d)^{\max\{1,p^{-1}\}}}$
  instead of $\varepsilon$, with $D$ instead of $d$, with $\beta_0$ instead of
  $\beta$, \Felix{and with $a_{\beta_0}$ instead of $B$})
  to obtain neural networks $\Phi_1, \dots, \Phi_d$
  \Felix{of a common depth
  $L' = L'(D,\beta_0,p) = L'(D,\beta,d,p,T) = L'(D,\beta,d,p,r,B,\kappa) \in \N$},
  such that their parallelization
  \Felix{$\Phi^\tau_\varepsilon
          = P(\Phi_1, P(\Phi_2, \dots, P(\Phi_{d-1},\Phi_d) \dots ))$}
  has at most
  \[
         c'' \cdot (\varepsilon^{T+\Felix{1}}
             /     \Felix{(2d)^{\max\{1,p^{-1}\}}})^{- D / \beta_0}
    \leq c''' \cdot \varepsilon^{- D(T + \Felix{1}) / \beta_0}
    \leq c''' \cdot \varepsilon^{\Felix{p} (d-1)/\beta}
  \]
  many, $(s',\varepsilon)$-quantized weights,
  and satisfies
  \begin{align*}
          \left\|
              \Realization_\varrho(\Phi^\tau_\varepsilon) - \tau
          \right\|_{L^{\Felix{p}} ([-\nicefrac{1}{2}, \nicefrac{1}{2}]^D; \R^d)}
    & \Felix{
      \overset{\text{Eq. } \eqref{eq:PseudoTriangleInequality}}{\leq}
           d^{\max\{1,p^{-1}\}}
           \cdot \max \left\{
                        \|
                          \Realization_\varrho (\Phi_i)
                          - \tau_i
                        \|_{L^p ([-\nicefrac{1}{2}, \nicefrac{1}{2}]^D)}
                        \,:\,
                        i = 1,\dots,d
                      \right\} } \\
    & \Felix{
      \overset{\phantom{\text{Eq. } \eqref{eq:PseudoTriangleInequality}}}{\leq}
           \frac{\varepsilon^{T+1}}{2^{\max\{1,p^{-1}\}}} } \, .
  \end{align*}
  \pp{The above constants satisfy
  \begin{align*}
    c'' &  = c''(d,D,\Felix{p},\beta_0,a_{\beta_0})
           = c''(d,\Felix{p},D,\beta,T,a)
           = c''(d,\Felix{p},D,r,B,a,\beta,\kappa) > 0 \, ,\\
    c''' & = c'''(\Felix{d,p},D,\beta_0,c'')
           = c'''(d,\Felix{p},D,B,r,a,\beta,\kappa) > 0,
    \text{ and }\\
    s' &   = s' (D,\Felix{p},\beta_0,a_{\beta_0},\Felix{T,d})
           = s' (D,\Felix{p},\beta,d,r,a,\kappa,B) \in \N.
  \end{align*}}
  Finally, invoking Lemma \ref{lem:L1concatenations}
  with $g$ and $\Realization_\varrho(\Phi^g_\varepsilon)$,
  instead of $f$ and $\widehat{f}$, and with $\tau$ and
  $\Realization_\varrho(\Phi^\tau_\varepsilon)$ instead
  of $\tau$ and $\widehat{\tau}$ shows that
  $\Phi^f_\varepsilon : = \Phi^g_\varepsilon \sconc \Phi^\tau_\varepsilon$
  satisfies \eqref{eq:TargetApproxSymmetricFunctions}.
\end{proof}

\section{An estimate of intermediate derivatives}
\label{sec:IntermediateDerivativesEstimate}

\pp{
\begin{lemma}\label{lem:DerivativeHoelderEstimate}
For $n \in \N_0$, $d \in \N$ and $\sigma \in (0,1]$ there is a constant $C = C(n,d,\sigma) > 0$
such that every $f \in C^n([0,1]^d)$ satisfies
\[
    \| \partial^\gamma f \|_{\sup}
    \leq C \cdot
         \big(
            \| f \|_{\sup}
            + \sum_{|\alpha| = n}
                \Lip_\sigma (\partial^\alpha f)
         \big)
    \qquad \text{ for all } \gamma \in \N_0^d \text{ with } |\gamma| \leq n.
\]
\end{lemma}
}
\begin{proof}
Note: This proof is heavily based on that of \cite[Lemmas 4.10 and 4.12]{AdamsSobolevSpacesOld},
where a related, but different estimate is established.

\textbf{Step 1}: We claim for $f \in C^1([0,1]^d)$ and arbitrary $N \in \N$ that
\begin{equation}
    \| \partial_\ell f \|_{\sup}
    \leq 4 \cdot N^{\frac{1}{\sigma}} \cdot \| f \|_{\sup}
         + \frac{1}{N} \cdot \Lip_\sigma (\partial_\ell f)
    \qquad \text{ for all }\ell \in \{1, \dots, d\}.
    \label{eq:DerivativeHoelderEstimateStep1}
\end{equation}
By symmetry (that is, by relabeling the coordinates), we can assume $\ell = 1$.
Define $K := \lceil N^{1/\sigma} \rceil$, and let $x = (x_1,\dots,x_d) \in [0,1]^d$ be arbitrary.
Choose $i \in \{0,\dots,K-1\}$ with $x_1 \in [\nicefrac{i}{K}, \nicefrac{(i+1)}{K}]$.
By the mean value theorem, there is some $\xi \in (\nicefrac{i}{K}, \nicefrac{(i+1)}{K})$ with
\[
    |\partial_1 f (\xi, x_2,\dots, x_d)|
    =    \left|
            \frac{f(\frac{i+1}{K}, x_2,\dots, x_d) - f(\frac{i}{K}, x_2,\dots, x_d)}{\frac{i+1}{K} - \frac{i}{K}}
         \right|
    \leq 2K \cdot \| f \|_{\sup}
    \leq 4 \cdot N^{\frac{1}{\sigma}} \cdot \| f \|_{\sup},
\]
where we used $K = \lceil N^{1/\sigma} \rceil \leq 1+ N^{1/\sigma} \leq 2 \cdot N^{1/\sigma}$.
Since $|(\xi, x_2,\dots, x_d) - x| \leq |\xi - x_1| \leq K^{-1} \leq N^{-1/\sigma}$, the preceding
estimate implies
\begin{align*}
    |\partial_1 f(x)|
    \leq &~|\partial_1 f(x) - \partial_1 f(\xi, x_2,\dots,x_d)| + |\partial_1 f (\xi, x_2, \dots, x_d)| \\
    \leq &~(N^{-\frac{1}{\sigma}})^\sigma \cdot \Lip_\sigma (\partial_1 f) + 4 \cdot N^{\frac{1}{\sigma}} \cdot \| f \|_{\sup}
    =    4 \cdot N^{\frac{1}{\sigma}} \cdot \| f \|_{\sup} + \frac{1}{N} \cdot \Lip_\sigma (\partial_1 f),
\end{align*}
as claimed (since we assumed $\ell = 1$).

\smallskip{}

\textbf{Step 2}: For brevity, set $|f|_\ell := \sum_{|\alpha| = \ell} \| \partial^\alpha f \|_{\sup}$
and $|f|_{\ell, \sigma} := \sum_{|\alpha| = \ell} \Lip_\sigma (\partial^\alpha f) \in [0,\infty]$
for $\ell \in \N_0$ and $f \in C^\ell ([0,1]^d)$.
In this step, we show by induction on $k \in \N_0$ that for each $k \in \N_0$ and $N \in \N$, there is a constant
$C_{\sigma, d, k, N} > 0$ with
\begin{equation}
    | f |_k \leq \frac{1}{N} \cdot |f|_{k,\sigma} + C_{\sigma, d, k, N} \cdot \| f \|_{\sup}
    \qquad \text{ for all }f \in C^k ([0,1]^d).
    \label{eq:DerivativeHoelderEstimateStep2}
\end{equation}

Before we begin with the induction, we first show the following estimate:
\begin{equation}
    |f|_{k, \sigma} \leq d^2 \cdot |f|_{k+1} \qquad \text{ for all } k \in \N_0 \text{ and } f \in C^{k+1}([0,1]^d).
    \label{eq:EstimatingHoelderNormByHigherDerivative}
\end{equation}
To prove Equation \eqref{eq:EstimatingHoelderNormByHigherDerivative}, first note because of
$\mathrm{diam}([0,1]^d) = \sqrt{d}$ that each Lipschitz continuous function $f \in C([0,1]^d)$ satisfies
$|f(x) - f(y)| \leq |x-y|^\sigma \cdot |x-y|^{1-\sigma} \cdot \Lip_1 (f)
\leq |x-y|^\sigma \cdot d^{(1-\sigma)/2} \cdot \Lip_1 (f)$. Therefore, each $f \in C^1([0,1]^d)$
fulfills $\Lip_\sigma (f) \leq d^{(1-\sigma)/2} \cdot \Lip_1 (f) \leq d^{(1-\sigma)/2} \cdot \| \nabla f \|_{\sup}
\leq d \cdot \sum_{\ell=1}^d \| \partial_\ell f \|_{\sup}$, which finally yields for $f \in C^{k+1}([0,1]^d)$ that
\[
    |f|_{k, \sigma}
    = \sum_{|\alpha| = k} \Lip_\sigma (\partial^\alpha f)
    \leq d \sum_{\ell=1}^d \sum_{|\alpha| = k}  \| \partial_\ell \partial^\alpha f\|_{\sup}
    \leq d^2 \cdot |f|_{k+1},
\]
which is nothing but \eqref{eq:EstimatingHoelderNormByHigherDerivative}.

Now we properly begin with the proof of Equation \eqref{eq:DerivativeHoelderEstimateStep2}.
For $k = 0$, Equation \eqref{eq:DerivativeHoelderEstimateStep2} is trivial with $C_{\sigma, d, 0, N} = 1$,
since $|f|_0 = \| f \|_{\sup}$.
For $k = 1$, Equation \eqref{eq:DerivativeHoelderEstimateStep2} is a consequence of
Equation \eqref{eq:DerivativeHoelderEstimateStep1}, which yields
\begin{align*}
    |f|_k = |f|_1
    & =    \sum_{\ell = 1}^d \| \partial_\ell f \|_{\sup}
      \leq 4 \cdot N^{\frac{1}{\sigma}} \cdot \| f \|_{\sup} \cdot \sum_{\ell = 1}^d 1
           + \frac{1}{N} \sum_{\ell=1}^d \Lip_\sigma (\partial_\ell f) \\
    & =    \frac{1}{N} \cdot |f|_{1, \sigma} + 4d \cdot N^{\frac{1}{\sigma}} \cdot \| f \|_{\sup}
      =    \frac{1}{N} \cdot |f|_{k, \sigma} + 4d \cdot N^{\frac{1}{\sigma}} \cdot \| f \|_{\sup} ,
\end{align*}
so that $C_{\sigma, d, 1, N} = 4d \cdot N^{1/\sigma}$ makes Equation \eqref{eq:DerivativeHoelderEstimateStep2}
true for $k = 1$.

For the induction step, note that if $f \in C^{k+1}([0,1]^d)$, and if we apply the case $k = 1$ (with $M$ instead
of $N$) to each of the partial derivatives $\partial^\alpha f$ with $|\alpha| = k$, then we get
\begin{equation}
    \begin{split}
    |f|_{k+1}
    & \leq \sum_{|\alpha| = k}
                |\partial^\alpha f|_1
      \leq \sum_{|\alpha| = k}
             \left(
                \frac{1}{M} |\partial^\alpha f|_{1, \sigma} + C_{\sigma, d, M} \cdot \| \partial^\alpha f\|_{\sup}
             \right) \\
    & \overset{(\ast)}{\leq}
           \frac{d^k}{M} \cdot |f|_{k+1, \sigma} + C_{\sigma, d, M} \cdot |f|_k \\
    ({\scriptstyle{\text{by induction}}})
    & \leq \frac{d^k}{M} \cdot |f|_{k+1, \sigma}
           + C_{\sigma, d, M} \cdot
             \left(
                \frac{1}{N} \cdot |f|_{k,\sigma} + C_{\sigma, d, k, N} \cdot \| f \|_{\sup}
             \right) \\
    ({\scriptstyle{\text{by Eq. } \eqref{eq:EstimatingHoelderNormByHigherDerivative}
                   \text{ since } f \in C^{k+1}([0,1]^d)}})
    & \leq \frac{d^k}{M} \cdot |f|_{k+1, \sigma}
           + C_{\sigma, d, M} \cdot
           \left(
               \frac{d^2}{N} \cdot |f|_{k+1} + C_{\sigma, d, k, N} \cdot \| f \|_{\sup}
           \right),
    \end{split}
    \label{eq:DerivativeHoelderEstimateStep2Intermediate}
\end{equation}
where $M, N \in \N$ can be chosen arbitrarily. In the above calculation, the step marked with $(\ast)$ used the
elementary estimates
$|\partial^\alpha f|_{1,\sigma} = \sum_{\ell=1}^d \Lip_\sigma (\partial_\ell \partial^\alpha f)
\leq \sum_{|\gamma| = k+1} \Lip_\sigma (\partial^\gamma f) = |f|_{k+1, \sigma}$, which is valid for all
$\alpha \in \N_0^d$ with $|\alpha| = k$; \Felix{furthermore, we used that}
$|\{\alpha \in \N_0^d \,:\, |\alpha | = k\}| \leq d^k$.

Finally, note that \eqref{eq:DerivativeHoelderEstimateStep2} is trivially satisfied (for $k+1$ instead of $k$) if
$|f|_{k+1, \sigma} = \infty$. Therefore, we can assume $|f|_{k+1, \sigma} < \infty$. If we now choose
$N = N(\sigma, d, M) \in \N$ to satisfy $N \geq 1 + 2 d^2 C_{\sigma, d, M}$, so that
$C_{\sigma, d, M} \cdot \frac{d^2}{N} \leq \nicefrac{1}{2}$, then we get from
Equation \eqref{eq:DerivativeHoelderEstimateStep2Intermediate} by rearranging that
\begin{align*}
    |f|_{k+1}
    & \leq 2 \cdot
             \left(
                \frac{d^k}{M} \cdot |f|_{k+1,\sigma} + C_{\sigma, d, M} C_{\sigma, d, k, N} \cdot \|f\|_{\sup}
             \right) 
     \leq \frac{2 d^k}{M} \cdot |f|_{k+1, \sigma} + C_{\sigma, d, k, M}' \cdot \| f \|_{\sup}.
\end{align*}
Since $M \in \N$ can be chosen arbitrarily, this establishes Equation \eqref{eq:DerivativeHoelderEstimateStep2}
for $k+1$ instead of $k$, and thus completes the induction.

\smallskip{}

\textbf{Step 3}: For arbitrary $k \in \N$, we prove by induction on $0 \leq j \leq k-1$ that there is a constant $C_{\sigma, d, k, j} > 0$
with
\begin{equation}
    | f |_{k-j} \leq |f|_{k, \sigma} + C_{\sigma, d, k, j} \cdot \| f \|_{\sup}
    \qquad \text{ for all } f \in C^k ([0,1]^d).
    \label{eq:DerivativeHoelderEstimateStep3}
\end{equation}
For $j = 0$, this is a direct consequence of Equation \eqref{eq:DerivativeHoelderEstimateStep2} (with $N = 1$).
For the induction step, assume that \eqref{eq:DerivativeHoelderEstimateStep3} holds for some $0 \leq j \leq k-2$, and note
\begin{align*}
    |f|_{k - (j+1)}
    & =    |f|_{k - j - 1} \\
    ({\scriptstyle{ \text{Eq. } \eqref{eq:DerivativeHoelderEstimateStep2} \text{ with } k-j-1 \text{ instead of } k \text{ and with } N = d^2}})
    & \leq \frac{1}{d^2} \cdot |f|_{k-j-1, \sigma} + C_{\sigma, d, k, j}' \cdot \| f \|_{\sup} \\
    ({\scriptstyle{\text{Eq. } \eqref{eq:EstimatingHoelderNormByHigherDerivative} \text{ since } f \in C^{k} \subset C^{(k-j-1)+1}}})
    & \leq |f|_{k-j} + C_{\sigma, d, k, j}' \cdot \| f \|_{\sup} \\
    ({\scriptstyle{\text{by induction}}})
    & \leq |f|_{k, \sigma} + (C_{\sigma, d,k,j} + C_{\sigma, d, k, j}') \cdot \| f \|_{\sup}.
\end{align*}

\smallskip{}

\textbf{Step 4}: In this step, we prove the actual claim. For $n=0$, this is trivial, so that we can assume $n \geq 1$.
Thus, let $f \in C^n([0,1]^d)$, and let $\gamma \in \N_0^d$ with $|\gamma| \leq n$. For $\gamma = 0$, the claim is trivial, so that
we can assume $1 \leq |\gamma| \leq n$. Hence, $j  := n - |\gamma|$ satisfies $0 \leq j \leq n-1$. Therefore, we can apply Step 3
with $k = n$ to conclude
\[
    \| \partial^\gamma f \|_{\sup} \leq | f |_{|\gamma|} = |f|_{n-j} \leq |f|_{n, \sigma} + C_{\sigma, d, n, j} \cdot \| f \|_{\sup}.
\]
This easily implies the claim, with $C = \max\{1, \, \max \{ C_{\sigma, d, n, j} \, : \, 0 \leq j \leq n-1 \}\}$.
\end{proof}

\section{Reducing the number of neurons}
\label{sec:StandingAssumptionJustification}

In this short technical appendix, we prove that for each neural network $\Phi$ with one-dimensional output and $d$-dimensional input,
one can assume essentially without loss of generality that $N(\Phi) \leq M(\Phi) + d + 1$.
This observation is important for the proof of Lemma \ref{lem:NNEncoding}, where we encode the functions represented by a
class of neural networks using a fixed number of bits. It is also used in the proof of Theorem \ref{thm:DepthLowerBoundAppendix}.

\begin{lemma}\label{lem:WeightsAndNeuronsAssumptionJustification}
    Let $\varrho : \R \to \R$ with $\varrho (0) = 0$.
    Then, for every neural network $\Phi$ with input dimension $d \in \N$ and output dimension $1$, there is
    a neural network $\Phi'$ with the same input and output dimension
    and with the following additional properties:
    \begin{itemize}
        \item We have $\Realization_\varrho (\Phi') = \Realization_\varrho (\Phi)$.
        \item We have $N(\Phi') \leq M(\Phi') + d + 1$.
        \item We have $M(\Phi') \leq M(\Phi)$ and $L(\Phi') \leq L(\Phi)$.
        \item If $I \subset \R$ contains the values of all nonzero weights of $\Phi$, then the same holds for $\Phi'$.
    \end{itemize}
\end{lemma}
\begin{proof}
    \Felix{In case of}
    $N(\Phi) \Felix{\leq} M(\Phi) + d + 1$, \Felix{the network $\Phi' = \Phi$ satisfies}
    all properties \Felix{required in} the statement of the lemma.
    We \Felix{will} show that \Felix{for $N(\Phi) > M(\Phi) + d + 1$},
    one can always find a network $\Phi'$ with $N(\Phi') < N(\Phi)$ and such that $\Phi'$ has the
    same input and output dimension as $\Phi$, such that $M(\Phi') \leq M(\Phi)$, $L(\Phi') \leq L(\Phi)$
    and $\Realization_\varrho (\Phi') = \Realization_\varrho (\Phi)$,
    and such that if $I \subset \R$ contains the values of all nonzero weights of $\Phi$, then the same holds for $\Phi'$.
    Iterating this observation yields the result.

    For $n_1,n_2 \in \N$ and $A \in \R^{n_1 \times n_2}$, as well as $i \in \{1, \dots, n_1\}$ we denote (in case of $n_1 > 1$) by
    $A_{\hat{i}} \in \R^{(n_1 - 1) \times n_2}$ the matrix resulting from removing the $i$-th row of $A$. Likewise, for $i \in \{1, \dots, n_2\}$ we write
    (in case of $n_2 > 1$) $A^{\hat{i}}$ for the matrix resulting from removing the $i$-th column of $A$.
    Similarly, for $b \in \R^{n_1}$ with $n_1 > 1$, we denote by $b_{\hat{i}} \in \R^{n_1 - 1}$ the vector resulting from removing the $i$-th entry of $b$.

    Let $\Phi = ((A_1, b_1), \dots, (A_L, b_L))$ with $A_\ell \in \R^{N_\ell \times N_{\ell-1}}$ and $b_\ell \in \R^{N_\ell}$
    for $\ell \in \{1, \dots, L\}$.
    Since
    \[
        \sum_{\ell=1}^L N_\ell - 1
        = N(\Phi) - d - 1
        > M(\Phi)
        = \sum_{\ell=1}^L \left( \|A_\ell\|_{\ell^0} + \|b_\ell\|_{\ell^0} \right),
    \]
    there exist more rows of $[A_1, b_1], \dots, [A_{L}, b_{L}]$ than nonzero entries
    in all these matrices. Hence, there exists $\ell \in \{1, \dots, L\}$ and $i \in \{1, \dots, N_\ell\}$ such that
    the $i$-th row of $A_{\ell}$ and the $i$-th entry of $b_\ell$ vanish.
    In fact, let us choose $\ell \in \{1,\dots, L\}$ maximal with the property that there is some $i \in \{1, \dots, N_\ell\}$
    such that the $i$-th row of $A_{\ell}$ and the $i$-th entry of $b_\ell$ vanish.
    Now we distinguish three cases:

    \smallskip
    \noindent
    \textbf{Case 1}: If $N_\ell>1$ (so that in particular $\ell < L$, since $N_L = 1$), then we set
    \[
        \Phi' :=
        ((A_1,b_1),
         \dots,
         (A_{\ell-1}, b_{\ell - 1}),
         ( (A_{\ell})_{\hat{i}}, (b_{\ell})_{\hat{i}} ),
         ( (A_{\ell+1})^{\hat{i}}, b_{\ell+1} ),
         (A_{\ell+2}, b_{\ell+2}),
         \dots,
         (A_L, b_L)).
    \]
    We have that $(A_{\ell+1})^{\hat{i}}x_{\hat{i}} = A_{\ell+1}x$ for all $x = (x_1,\dots, x_{N_\ell}) \in \R^{N_\ell}$ with $x_i = 0$,
    and furthermore for all $x \in \R^{N_{\ell - 1}}$ that $(\varrho (A_\ell \, x + b_\ell) )_{\hat{i}} = \varrho ((A_\ell)_{\hat{i}} \, x + (b_\ell)_{\hat{i}})$.
    Since $\varrho(0) = 0$, we see that the $i$-th entry of $\varrho(A_{\ell} \, x + b_\ell)$ is zero, for arbitrary $x \in \R^{N_{\ell - 1}}$.
    All in all, these observations show $\Realization_\varrho (\Phi') = \Realization_\varrho (\Phi)$.
    Moreover, $N(\Phi')<N(\Phi)$, $M(\Phi') \leq M(\Phi)$, and $L(\Phi') = L(\Phi)$ follow from the construction.
    The statement regarding the values of the nonzero weights being contained in $I$ is also clearly satisfied.

    \smallskip
    \noindent
    \textbf{Case 2}: If $N_\ell = 1$, but $\ell >1$, then we have $A_\ell = 0$ and $b_\ell = 0$.
    We set $\tilde{A}_1 :=  0 \in \R^{1\times d}, \tilde{b}_1 := 0 \in \R$.

    If $\ell < L$ we set
    \[
        \Phi' := ((\tilde{A}_1, \tilde{b}_1), (A_{\ell+1}, b_{\ell + 1}), \dots, (A_L, b_L)).
    \]
    By construction and because of $\varrho(0) = 0$, we have
    $\Realization_\varrho (\Phi') = \Realization_\varrho (\Phi)$ and
    $N(\Phi')<N(\Phi)$ \Felix{(here we use that $\ell > 1$)}, as well as
    $M(\Phi') \leq M(\Phi)$ and $L(\Phi') \leq L(\Phi)$.
    The statement regarding the values of the nonzero weights being contained in
    $I$ is also clearly satisfied.

    If $\ell = L$, then $\Realization_{\varrho}(\Phi) \equiv 0$.
    \Felix{Hence, we have $\Realization_\varrho(\Phi) = \Realization_\varrho (\Phi')$}
    for
    \[
        \Phi' := \big( (\tilde{A}_1, \tilde{b}_1) \big).
    \]
    \Felix{Furthermore}, we have $N(\Phi') = d+1 \leq M(\Phi) + d + 1 < N(\Phi)$, as well as $M(\Phi') = 0 \leq M(\Phi)$ and
    $L(\Phi') = 1 \leq L(\Phi)$. Finally, since $\Phi'$ only has weights with value zero, the
    statement regarding the values of the nonzero weights being contained in $I$ is trivially satisfied.

    \smallskip
    \noindent
    \textbf{Case 3}:  If $\ell = 1$ and $N_1 = 1$, then $A_1 = b_1 = 0$. Thus we have
    \[
        \sum_{\ell = 2}^L N_\ell
        = \sum_{\ell=1}^L N_\ell - 1
        = N(\Phi) - d - 1 > M(\Phi)
        = \sum_{\ell=2}^L (\|A_\ell\|_{\ell^0} + \|b_\ell\|_{\ell^0}),
    \]
    and therefore there exists some $\ell' \in \{ 2, \dots, L\}$ and some $j \in \{1, \dots, N_{\ell'}\}$ such that the
    $j$-th row of $A_{\ell'}$ and the $j$-th entry of $b_{\ell'}$ vanish.
    This contradicts the maximality of $\ell$, so that this case cannot occur.
    %
\end{proof}

\section*{Acknowledgements}

The authors would like to thank Gitta Kutyniok, Philipp Grohs,
Stephan W\"aldchen, \pp{and Nadav Cohen} for fruitful discussions on the topic,
and Dimitri Bytchenkoff for boosting our morale.
F.V. acknowledges support by the European Commission-Project DEDALE
(contract no. 665044) within the H2020 Framework.
P.P acknowledges support by the DFG Collaborative Research Center TRR 109
``Discretization in Geometry and Dynamics".

\small

\bibliographystyle{plain}
\bibliography{references}

\end{document}